\documentclass[10pt]{article}

\usepackage{amsmath,amsfonts,amsthm,amssymb,mathrsfs,stmaryrd}
\usepackage[T1]{fontenc}
\usepackage[utf8]{inputenc}
\usepackage{natbib} 
\usepackage{float}

\usepackage{xcolor}
\usepackage{fullpage}
\usepackage{mathtools}
\usepackage{graphicx}
\usepackage{tabularx}
\usepackage[margin=1cm,font=footnotesize]{caption}
\usepackage{subcaption}
\usepackage{authblk}
\usepackage{comment}
\DeclareFontFamily{U}{mathx}{\hyphenchar\font45}
\DeclareFontShape{U}{mathx}{m}{n}{
      <5> <6> <7> <8> <9> <10>
      <10.95> <12> <14.4> <17.28> <20.74> <24.88>
      mathx10
      }{}
\DeclareSymbolFont{mathx}{U}{mathx}{m}{n}
\DeclareFontSubstitution{U}{mathx}{m}{n}
\DeclareMathAccent{\widecheck}{0}{mathx}{"71}
\DeclareMathAccent{\wideparen}{0}{mathx}{"75}

\usepackage{thmtools}
\usepackage{thm-restate}

\newcommand\numberthis{\addtocounter{equation}{1}\tag{\theequation}}

\DeclareMathOperator*{\argsup}{arg\,sup}

\usepackage{relsize}

\newtheorem{theorem}{Theorem}[section]
\newtheorem{corollary}{Corollary}[section]
\newtheorem{lemma}{Lemma}[section]

\theoremstyle{definition} 
\newtheorem{fact}{Fact}
\newtheorem{remark}{Remark}[section]

\usepackage{enumitem}
\setlist{  
  listparindent=\parindent,
  parsep=0pt,
}
\setlist[itemize]{leftmargin=*}
\usepackage{mathtools}
\usepackage{algorithmic}
\usepackage{algorithm}
\usepackage{bm}
\usepackage{bbm}
\usepackage{empheq}
\usepackage{amsfonts}

\usepackage[colorlinks,backref=page,hypertexnames=false]{hyperref}
\hypersetup{%
  colorlinks=true,
  linkcolor=blue,
  citecolor = blue,
  urlcolor = blue
}
\usepackage[capitalize]{cleveref}
\crefname{fact}{Fact}{assumptions}
\makeatletter
\AddToHook{cmd/appendix/before}{\def\cref@section@alias{appendix}\def\cref@subsection@alias{appendix}}
\makeatother

\setcitestyle{round}

\newcommand{\eps}{\ensuremath{\varepsilon}}
   
\newcommand{\s}{^*}
\renewcommand{\t}{^{\top}}

\newcommand{\inv}{^{-1}}

\newcommand{\p}{\mathbb{P}}

\newcommand{\per}{\uperp\uperp\t}

\renewcommand{\hat}{\widehat}
\renewcommand{\tilde}{\widetilde}

\newcommand{\tr}{{\sf{Tr}}}
\newcommand{\utilde}{\bm{\tilde U}}
\newcommand{\uhat}{\bm{\hat U}}
\newcommand{\uhatj}{\bm{\hat u}_j}
\renewcommand{\a}{\bm{a}}
\newcommand{\uj}{\bm{u}_j}
\newcommand{\uk}{\bm{u}_k}
\newcommand{\bE}{\bm{E}}
\newcommand{\bM}{\bm{M}}
\newcommand{\bS}{\bm{S}}
\newcommand{\bN}{\bm{N}}

\newcommand{\x}{\bm{X}}
\newcommand{\y}{\bm{Y}}
\newcommand{\bSigma}{\bm{\Sigma}}
\newcommand{\mhat}{\bm{\hat M}}
\newcommand{\uperp}{\bm{U}_{\perp}}

\newcommand{\U}{\bm{U}}
\newcommand{\rmd}{\mathcal{R}^{{\sf MD}}}
\newcommand{\md}{^{{\sf MD}}}
\newcommand{\pca}{^{{\sf PCA}}}
\newcommand{\rpca}{\mathcal{R}^{{\sf PCA}}}
\newcommand{\lamtilde}{\bm{\tilde{\Lambda}}}
\newcommand{\sigmahat}{\bm{\hat \Sigma}}
\renewcommand{\check}{\widecheck}

\usepackage{natbib} 
\bibpunct{(}{)}{;}{a}{,}{,}
\title{Statistical Inference for Linear Functions of Eigenvectors with Small Eigengaps}
\author{Joshua Agterberg\thanks{Department of Statistics, University of Illinois Urbana-Champaign}}
\date{\today}
\begin{document}

\maketitle

\begin{abstract}
Spectral methods have myriad applications in high-dimensional statistics and data science, and while previous works have primarily focused on $\ell_2$ or $\ell_{2,\infty}$ eigenvector and singular vector perturbation theory, in many settings these analyses fall short of providing the fine-grained guarantees required for various inferential tasks.  In this paper we study statistical inference for linear functions of eigenvectors and principal components with a 
 particular emphasis on the setting where gaps between eigenvalues may be extremely small relative to the corresponding spiked eigenvalue, a regime which has been oft-neglected in the literature.  
First, we prove the approximate Gaussianity for debiased linear forms in the matrix denoising model and the spiked principal component analysis model, both under Gaussian noise.   Based on this limiting behavior, we propose estimators for the appropriate bias and variance quantities resulting in approximately valid confidence intervals.  We then investigate the optimality of these confidence intervals and show that their widths are minimax optimal up to constant factors.  Of note, our proposed confidence intervals can be computed directly from data without the need for any sample-splitting.
\end{abstract}

\tableofcontents

\section{Introduction}
Spectral methods, or algorithms and procedures that rely on eigenvectors, singular vectors, and related quantities, are a useful collection of tools for analyzing high-dimensional data, with applications to clustering \citep{loffler_optimality_2021}, network analysis \citep{jin_mixed_2023,mao_estimating_2020},  ranking \citep{chen_spectral_2019,chen_partial_2022}, principal component analysis \citep{johnstone_consistency_2009}, nonconvex optimization \citep{chi_nonconvex_2019}, and tensor data analysis \citep{zhang_tensor_2018}, to name a few.  With the growing ubiquity of spectral methods, there has been an increased interest in providing statistical analyses of these techniques that take into account both the structural properties of the problem at hand as well as the particularities of the noise, in some cases resulting in provably optimal guarantees for various spectral-based algorithms \citep{abbe_entrywise_2020,abbe_ell_p_2022,ling_near-optimal_2022,nguyen_novel_2025,zhang_exact_2024,zhang_fundamental_2024}.

One common paradigm to analyze spectral methods is via classical deterministic perturbation theory, which studies the $\ell_2$ errors between eigenvectors of some underlying matrix $\bm{M} \in \mathbb{R}^{n\times n}$ and the eigenvectors of a perturbed matrix $\bm{\hat M} = \bm{M} + \bm{E}$ as a function of the perturbation matrix $\bm{E}$ and the eigenvalues or singular values of $\bm{M}$.  By combining these techniques with tools from nonsaymptotic random matrix theory one can derive high-probability or expected $\ell_2$ error rates for eigenvectors and singular vectors \citep{cai_rate-optimal_2018}.  Furthermore, there has been a growing interest in $\ell_{2,\infty}$ eigenspace and singular subspace perturbation theory \citep{cape_two--infinity_2019}, 
which quantifies the entrywise and row-wise fluctuations of subspaces through careful probabilistic analysis of the noise.  By carefully tailoring these analyses to the distributional properties of the noise matrix $\bm{E}$, estimation error rates can be obtained that go beyond what classical perturbation theory offers.  In some cases, these entrywise guarantees can yield asymptotic normality for estimators of subspaces \citep{agterberg_overview_2023,cape_signal-plus-noise_2019}, and these tools can be applied in various contexts, resulting in statistically sound inferential procedures \citep{chen_spectral_2021,gao_uncertainty_2023,yan_inference_2021}.

However, in many contexts entrywise perturbation and distributional theory is still not sufficiently fine-grained to quantify inherent uncertainty in the estimate.  Therefore, in this paper we contribute to the literature by studying statistical inference for a general linear function of a single eigenvector.  Explicitly, given a pre-specified unit vector $\bm{a} \in \mathbb{R}^{n}$ and an eigenvector of interest $\uj \in \mathbb{R}^{n}$, we study how to construct confidence intervals for the linear form $\a\t \uj$.  There have been a few works in this area \citep{chen_asymmetry_2021,cheng_tackling_2021,koltchinskii_perturbation_2016,koltchinskii_asymptotics_2016,koltchinskii_efficient_2020}, though we provide a more thorough review of related work in \cref{sec:previouswork}.  
One key downside of existing theory is that theoretical results often require highly pessimistic assumptions on the \emph{eigengap}, the minimal separation between distinct eigenvalues.  Therefore, in this work we contribute to the literature by studying statistical inference in regimes where the eigengap can be \emph{significantly smaller} than previously assumed (possibly up to order $n^{-1/2}$ smaller, ignoring logarithmic terms).  Our analysis is motivated by the findings of \citet{li_minimax_2025} who provide  perturbation bounds for estimators of $\a\t\uj$ in the presence of small eigengaps.   While their work demonstrated that it is possible to obtain minimax-optimal estimators of linear functions of eigenvectors, in this work we go one step further and analyze the distributional fluctuations of these estimators and study how to construct confidence intervals for $\a\t\uj$.

In the following subsections we first describe the two models studied in this work, and we then describe the primary contributions of this paper  in greater detail.

\subsection{Two Canonical Settings}

Throughout this work we focus on two fundamental statistical models: matrix denoising and principal component analysis.

\begin{itemize}
    \item \textbf{Matrix Denoising}. Consider the setting where one is interested in the eigenvectors of some  symmetric rank $r$ matrix $\bS \in \mathbb{R}^{n\times n}$  corrupted with some symmetric Gaussian noise matrix $\bN$, where $\bN_{ij} \sim \mathcal{N}(0,\sigma^2)$ for $i < j$ and $\bN_{ii} \sim \mathcal{N}(0,2\sigma^2)$ are independent (i.e., $\bN$ belongs to the \emph{Gaussian Orthogonal Ensemble}).  Given a pre-specified deterministic vector $\a$, in this paper we consider developing confidence intervals for $\a\t\uj$, where $\uj$ is the eigenvector associated to the $j$'th largest in magnitude eigenvalue of $\bS$.  
    \item \textbf{Principal Component Analysis}.  Suppose one has $n$ i.i.d. observations $\bm{X}_i \in \mathbb{R}^p$ with $\bm{X}_i \sim \mathcal{N}(0,\bSigma)$, where $\bSigma$ belongs to the class of \emph{spiked covariance matrices}; i.e., $\bSigma = \bSigma_0 + \sigma^2 \bm{I}_p$, where $\bSigma_0$ is a positive semidefinite rank $r$ matrix and $\bm{I}_p$ denotes the $p\times p$ identity.  As in  matrix denoising we consider computing confidence intervals for $\a\t\uj$ for some pre-specified unit vector $\a$, where $\uj$ is the $j$'th eigenvector of $\bSigma_0$ (and hence also $\bSigma$). 
\end{itemize}
These models have been studied explicitly in the literature previously.  For example, \citet{bao_singular_2021,knowles_outliers_2014,xia_normal_2021}
 study the asymptotic properties of empirical eigenvectors and singular vectors, and 
\citet{bao_statistical_2022,koltchinskii_asymptotics_2016,koltchinskii_efficient_2020} study the asymptotic properties of principal components, and  both of these models were also studied in \citet{li_minimax_2025} in a similar context. However, unlike the aforecited works, our emphasis is on developing optimal confidence intervals. 

\subsection{Contributions}
Throughout this work we provide analogous results for both matrix denoising  and principal component analysis.  Our main contributions are as follows:
\begin{itemize}
    \item We establish  distributional theory for linear functions of eigenvectors $\a\t\uhatj$ in the presence of small eigengaps about the quantity $\a\t\uj\uj\t\uhatj$, which is a  biased centering term.  We also provide similar results for a bias-corrected estimator about the quantity $\a\t\uj$, where the bias correction is computed directly from data assuming knowledge of the noise variance $\sigma^2$.    
    \item We then study the problem of statistical inference for $\a\t\uj$, and we provide fully data-driven estimators for the bias and variance without assuming knowledge of the noise variance $\sigma^2$.  We show that these estimators yield approximately valid confidence intervals.  
    \item We provide lower bounds  demonstrating that the resulting confidence intervals achieve the optimal length up to universal constants over all confidence intervals with a desired level of coverage.
\end{itemize}
Unlike previous work,  our procedures are fully data-driven and do not require any sample splitting, and to the best of our knowledge, these results are the first to provide valid confidence intervals for linear functions of eigenvectors in these settings under nearly optimal signal-to-noise ratio conditions without any sample splitting.   It is worth emphasizing that our results continue to hold  even in the ``large eigengaps'' regime.

\subsection{Organization}
The rest of this paper is organized as follows.  We first present results for matrix denoising in \cref{sec:MD},  including distributional theory, confidence interval construction, and lower bounds.  We then repeat the study for PCA in \cref{sec:pca}.  In \cref{sec:previouswork} we discuss previous work, and in \cref{sec:discussion} we include discussion.  In \cref{sec:proofoverview} we discuss the proof ideas.  More detailed proofs and numerical studies are deferred to the appendices.

\subsection{Notation}
We let $[n]$ denote the set of integers $\{1, \dots, n\}$. For two numbers $a$ and $b$ we write $a\wedge b$ to mean the minimum of $a$ and $b$ and $a \vee b$ as the maximum. For two sequences $a_n$ and $b_n$, we write $a_n \lesssim b_n$ or $a_n = O(b_n)$ if there exists a universal constant $C > 0$ such that $a_n \leq C b_n$, and we write $a_n \ll b_n$ if $a_n/b_n \to 0$ as $n\to\infty$.  We write $a_n \asymp b_n$ and if $a_n \lesssim b_n$ and $b_n \lesssim a_n$. We set $a_n = \Omega(b_n)$ if $b_n = O(a_n)$.  We write $a_n = \tilde O(b_n)$ if there exists some $c > 0$ not depending on $n$ such that $a_n \lesssim b_n \log^c(n)$, and similarly for $a_n = \tilde \Omega(b_n)$.  For a matrix $\U$ with orthonormal columns we write $\uperp$ to denote the (not necessarily unique) matrix satisfying $\uperp\t \U = 0$ such that $[\U,\uperp]$ is an orthogonal matrix.  We write $\bm{I}_{n}$ for the $n \times n$ identity matrix,  $\|\cdot\|$ as the spectral norm for matrices and Euclidean norm for vectors, $\|\cdot\|_F$ as the Frobenius norm on matrices, and we let $\|\cdot\|_{\infty}$ denote the $\ell_{\infty}$ norm on vectors.  We let $\bm{e}_i$ denote the $i$'th standard basis vector whose dimension will be clear from context.  For a given vector $\bm{v}$, we let $\bm{v}(i)$ denote its $i$'th entry.  For a random variable $X$, we let $\|X \|_{\psi_\alpha}$ be its  Orlicz $\psi_{\alpha}$ norm.

\section{Matrix Denoising}
\label{sec:MD}
Suppose one is given a  matrix $\bm{\hat S} \in\mathbb{R}^{n\times n}$ of the form
\begin{align*}
    \bm{\hat S} &= \bS + \bN \numberthis \label{MD}
    \end{align*}
 where $\bS$ is a rank $r$ signal matrix with eigendecomposition given by
\begin{align*}
    \bS &= \U \bm{\Lambda} \U\t = \sum_{i=1}^{r} \lambda_i \bm{u}_i \bm{u}_i\t,
\end{align*}
and $\bN$ is a symmetric noise matrix with independent entries given by
\begin{align*}
\bN_{ij} \sim \begin{cases}  \mathcal{N}(0,\sigma^2) & i < j , \\
\mathcal{N}(0,2\sigma^2) & i = j.
\end{cases}
\end{align*}
The matrix $\bN$ is said to belong to the \emph{Gaussian Orthogonal Ensemble} (GOE). 
We define
\begin{align*}
    \lambda_{\min} = \min_{1\leq i\leq r} |\lambda_i|; \qquad \lambda_{\max} = \max_{1\leq i \leq r} |\lambda_i|,
\end{align*}
representing the magnitude of the smallest  and largest nonzero eigenvalues of $\bS$, and we set $\kappa$ as the \emph{reduced condition number} $\kappa = \lambda_{\max}/\lambda_{\min}$. Similarly, we let $\bm{\hat S}$ have eigenvectors $\bm{\hat u}_k$ and eigenvalues $\hat \lambda_k$. For a given index $j$ with $1\leq j \leq r$, we define the eigenvalue gap (or eigengap) via
\begin{align*}
    \Delta_j &= \min_{k\neq j,1\leq k\leq r} |\lambda_j - \lambda_k|,
\end{align*}
with the convention that $\Delta_j = \lambda_{\min}$ if $r = 1$. We also set
$\Delta_{\min} = \min_{1\leq j \leq r} \Delta_j.$
We let $\bm{\hat U}$ denote the $n \times r$ matrix of leading eigenvectors of $\bm{\hat S}$ organized according to the magnitude of the eigenvalues of $\bm{\hat S}$.  We set $\uhatj$ to be the $j$'th column of $\bm{\hat U}$. We study how to construct confidence intervals for $\a\t\uj$, where $\a$ is a prespecified unit vector.  Our confidence intervals will be based on the distributional theory we develop in the subsequent section.

\subsection{Distributional Theory}
First, it is well-known that $\a\t\uhatj$ is biased for $\a\t\uj$.  Define the \emph{debiasing} quantity $b_j\md$ via
\begin{align*}
    b_j\md \coloneqq \sum_{k > r} \frac{\sigma^2}{(\hat \lambda_j - \hat \lambda_k)^2}.
\end{align*}
The following result establishes the asymptotics for $\a\t\uhatj$ and $\a\t\uhatj \sqrt{1 + b_j\md}$.

\begin{theorem}\label{thm:distributionaltheory_MD}
Consider the model \eqref{MD}, and suppose $\lambda_j$ is unique.  Suppose that 
\begin{align*}
    \lambda_{\min}/\sigma \geq C_0 \sqrt{rn} \qquad \Delta_j/\sigma \geq C_0  r \log(n)
\end{align*}
where $C_0$ is some sufficiently large constant.  Suppose further that $r \leq c_0 n/\log^2(n)$, where $c_0$ is a sufficiently small constant.  Let $\a$ be any determinstic unit vector such that $\a \neq \pm \uj$. Let $\U^{(k)}$ denote the matrix whose columns are the eigenvectors corresponding to $\lambda_k$. Define
\begin{align*}
    (s_{\a,j}\md)^2 \coloneqq \sum_{\substack{k\neq j\\k\leq r}'} \frac{\sigma^2 \|\a\t \U^{(k)} \|^2}{(\lambda_j - \lambda_k)^2} + \frac{\sigma^2 \|\uperp\t \a\|^2}{\lambda_j^2},
\end{align*}
where the summation is over all $r' \leq r$ distinct eigenvalues. 
Let $\Phi(z)$ denote the cumulative distribution function of a standard Gaussian random variable.  Then it holds that
\begin{align*}
\sup_{z\in \mathbb{R}}    \bigg| \p\bigg\{ \frac{1}{s_{\a,j}\md} \big(\a\t \uhatj - \a\t \uj \uj\t \uhatj \big) \leq z \bigg\} - \Phi(z) \bigg| &\lesssim {\sf ErrMD} \numberthis \label{md_firstresult}
\end{align*}
where
\begin{align*}
    {\sf ErrMD} \coloneqq \underbrace{\frac{r\sqrt{n\log(n)}}{\lambda_{\min}/\sigma}}_{\text{Subspace Estimation Effect}} +  \underbrace{\frac{r^{3/2}\log(n)}{\Delta_j/\sigma}}_{\text{Small Eigengaps Effect}} + \  n^{-9}.
\end{align*}
Furthermore, if ${\sf ErrMD} = o(1)$ and in addition it holds that
\begin{align*}
    |\a\t\uj| \bigg[ \frac{\sigma^2 \sqrt{n\log(n)}}{\lambda_j^2} + \frac{\sigma^2 r\log(n)}{\Delta_j^2} \bigg] \ll s_{\a,j}\md. \numberthis \label{atujcondition}
\end{align*}
Then it further holds that
\begin{align*}
    \sup_{z \in \mathbb{R}} \bigg| \mathbb{P} \bigg\{ \frac{1}{s_{\a,j}\md} \big( \a\t \uhatj \sqrt{1 + b_j\md} - \a\t\uj \big) \leq z \bigg\} - \Phi(z) \bigg| = o(1). \numberthis \label{md_secondresult}
\end{align*}
\end{theorem}
\noindent 
Several features of \cref{thm:distributionaltheory_MD} are worth noting.
\begin{itemize}
    \item \textbf{Signal-Strength conditions.}  Observe that the bound in \cref{thm:distributionaltheory_MD} is only non-vacuous if $\lambda_{\min}/\sigma \gg r\sqrt{n\log(n)}$. It has been demonstrated in \citet{cai_rate-optimal_2018} that $\lambda_{\min}/\sigma \gg \sqrt{rn}$ is both neccessary and sufficient for consistent subspace estimation in $\|\cdot\|_F$ when $r \ll n$; consequently the condition $\lambda_{\min}/\sigma \gg r\sqrt{n\log(n)}$ is optimal up to the factor of  $\sqrt{r\log(n)}$.  
   The reason for this deficiency stems from the fact that the residual terms in the analysis must be shown to be $o(s_{\a,j}\md)$, and our proof occasionally relies on certain concentration inequalities for these residual terms using $\eps$-net arguments that may not be the most sharp bounds available.  It may be possible to weaken these assumptions, but such an analysis may require more involved techniques  but is unlikely to yield substantive changes in the results.
   As the purpose of this paper is optimal statistical inference, we leave such examinations to future work.  Nonetheless, when $r$ is relatively small, this condition is not significantly stronger than the optimal condition for consistency.
    \item \textbf{Eigengap conditions.} Again the bound is only non-vacuous if $\Delta_j/\sigma \gg r^{3/2}\log(n)$.  Together with our assumption on the eigengap $\Delta_j / \sigma\geq C_0 r \log(n)$, we see that a sufficient condition for normality is $\Delta_j/\sigma \gg r^{3/2} \log(n)$.  However, this eigengap condition is significantly weaker than what is required in classical theory.  For example, the Davis-Kahan Theorem requires that $\Delta_j \gtrsim \| \bN \| \asymp \sigma \sqrt{n},$
which is more stringent by a factor of $\tilde\Omega\big(\sqrt{\frac{n}{r^3}}\big)$. 
    \item \textbf{Rate of convergence.} Ignoring the additional factor of $n^{-9}$, when $r \asymp 1$, the rate of convergence to asymptotic normality is of the form
\begin{align*}
    \max\bigg\{ \frac{\sqrt{n}}{\lambda_{\min}/\sigma} , \frac{1}{\Delta_j/\sigma} \bigg\},
\end{align*}
ignoring logarithmic terms, where we dub the first term as the ``subspace estimation effect'' and the second term as the ``small eigengaps effect.'' 
Consequently, \cref{thm:distributionaltheory_MD} reveals a novel phenomenon -- when the signal strength is large but eigengaps are small, the primary difficulty in estimation is separating the $j$'th eigenvector from the other leading $r - 1$ eigenvectors, whereas when the eigengaps are large but the signal strength is small, the primary difficulty lies in estimating the subspace $\U$, which corresponds to separating the leading $r$ eigenvectors from the bottom $n - r$ eigenvectors. 
    \item \textbf{Biased estimation.}   Observe that the centering term for $\a\t\uhatj$ in \cref{thm:distributionaltheory_MD} is given by $\a\t\uj\uj\t\uhatj$.  In general $\uhatj\t\uj$ need not equal one; in fact, in the regime $\lambda_{\min}/\sigma \gg \sqrt{nr}$ it can be shown that $\uhatj\t \uj = 1 - o(1)$ and when $\lambda_{\min}/\sigma \asymp \sqrt{nr}$ then $\uhatj\t\uj \approx c < 1$ (e.g., \citet{benaych-georges_eigenvalues_2011}).  Consequently, \cref{thm:distributionaltheory_MD} demonstrates that $\a\t\uhatj$ is a \emph{biased} estimate of $\a\t\uj$ with bias given by $\a\t\uj(1 - \uhatj\t\uj)$, a phenomenon that has also been demonstrated for asymmetric noise in \citet{chen_asymmetry_2021}. 
    \item \textbf{Allowable size of $|\a\t\uj|$}.  In order for the debiased estimator $\a\t\uhatj \sqrt{1 + b_j\md}$ to admit an asymptotically Gaussian distribution, the magnitude of $|\a\t\uj|$ must be bounded away from one according to the condition \eqref{atujcondition}.  This assumption is quite intuitive: if $|\a\t\uj|$ is too close to 1, the bias from $\a\t\uj$ is too large relative to the variance. Informally, this assumption means that $\a$ must have some of its ``mass'' aligned with other columns of either $\U$ or $\uperp$.     
    However, $|\a\t\uj|$ is still permitted to be quite large: as large as $\min\big\{ \frac{\lambda_j^2}{\sigma^2 \sqrt{n}}, \frac{\Delta_j^2}{\sigma^2} \big\}$ relative to $s_{\a,j}\md$. A sufficient condition for \eqref{atujcondition} to hold is that  $|\a\t\uj| \leq (1 - \eps)$ for some fixed constant $\eps$ and $\kappa$ bounded.  
    \item \textbf{Perturbation bounds.}  The proof of \cref{thm:distributionaltheory_MD} immediately implies that with probability at least $1- O(n^{-10})$,
    \begin{align*}
        |\a\t\uhatj \sqrt{1 + b_j\md} - \a\t\uj | \lesssim s_{\a,j}\md \sqrt{\log(n)} + |\a\t\uj| \frac{\sigma^2 r\log(n)}{\Delta_j^2} + |\a\t\uj| \frac{\sigma^2 \sqrt{n\log(n)}}{\lambda_j^2}.
    \end{align*}
  This result improves  upon the upper bound developed in \citet{li_minimax_2025} slightly.  Furthermore, this upper bound attains the minimax rate established in \citet{cheng_tackling_2021}.  Thus, \cref{thm:distributionaltheory_MD} implies a rate-optimal upper bound (albeit under slightly stronger assumptions than in prior works).  
\end{itemize}

\cref{thm:distributionaltheory_MD} is perhaps most similar to results in the works \citet{li_minimax_2025} and \citet{koltchinskii_perturbation_2016}, who study perturbation bounds for the estimator $\a\t\uhatj$ in similar contexts, and our assumptions on the signal strength $\lambda_{\min}/\sigma$ match theirs up to logarithmic terms and factors of $r$.  However, only \citet{li_minimax_2025} consider the small eigengaps regime. Our condition on $\Delta_j$ matches theirs up to logarithmic terms and factors of $r$, while simultaneously providing a sharper characterization of the estimator  $\a\t\uhatj$ in terms of nonasymptotic distributional theory. 

In order to prove our results we must  must identify the leading-order term, compute its variance $(s_{\a,j}\md)^2$, and then demonstrate that the residual terms are of order $o(s_{\a,j}\md)$.  In contrast, both of these previous works only needed to show that the difference is of order $O(s_{\a,j}\md)$.  Therefore, our analysis must extend further (implicitly considering higher-order terms) to show that these residuals are all sufficiently small even with small eigengaps. To accomplish this goal we leverage heavily two deterministic matrix analysis results (\cref{thm:mainexpansion,lem:uhatjuperpidentity}) that allow us to decouple several statistically dependent quantities as well as isolate certain ``eigenvalue bias'' terms that arise in the analysis. Identifying how these eigenvalue bias terms arise plays a key role in providing the sharp concentration guarantees required to prove \cref{thm:distributionaltheory_MD}.

\begin{remark}[Extension to Asymmetric Case] Our results herein can be extended to the asymmetric setting using the standard Hermitian dilation trick.  Suppose $\bm{S}$ and $\bm{N}$ are asymmetric, each with dimension $n_1 \times n_2$. Define
\begin{align*}
\bm{\tilde S} := \begin{pmatrix}
    0 & \bm{S} \\ \bm{S}\t & 0.
\end{pmatrix}    
\end{align*}
Then $\bm{\tilde S}$ has eigenvectors the same as the singular vectors of $\bm{S}$ up to a factor of $\sqrt{2}$.  By redefining $n = \max \{ n_1, n_2\}$, and $\bm{\tilde a} = (\a\t, 0\t) \in \mathbb{R}^{n_1 + n_2}$, one can obtain asymptotics for linear functions of $\a\t \uj$, where $\uj$ is the $j$'th left singular vector. 
\end{remark}

\subsection{Confidence Intervals}
We now discuss how to construct confidence intervals based on the limiting result in \cref{thm:distributionaltheory_MD}.  First, we consider estimating the noise variance $\sigma^2$.  Define
\begin{align*}
    \hat \sigma^2 &\coloneqq \binom{n}{2} \inv\| \mathcal{P}_{{\sf upper-diag}} \big( \bm{\hat S} - \bm{\hat S}_r \big) \|_F^2, \numberthis \label{alg:noisevarMD}
\end{align*}
where $\bm{\hat S}_r$ is the best rank $r$ approximation of $\bm{\hat S}$, and where $\mathcal{P}_{{\sf upper-diag}}$ denotes the projection onto the upper off-diagonal. To understand the intuition behind this estimator, when $\bm{\hat S}_r$ is sufficiently close to $\bm{S}$, we expect that $\bm{\hat S} - \bm{\hat S}_r$ is approximately equal to the noise $\bm{N}$.    Therefore, $\hat \sigma^2$ can be understood as a plug-in estimator for $\sigma^2$ using $\bm{\hat S} - \bm{\hat S}_r$ as a proxy for $\bm{N}$. Using this estimator of the noise variance we can reliably estimate $b_j\md$ with $\hat{b_j\md}$, defined via
\begin{align*}
         \hat{b_j\md} \coloneqq \sum_{k > r} \frac{\hat \sigma^2}{(\hat \lambda_j - \hat \lambda_k)^2}. \numberthis \label{alg:biasMD}
\end{align*} 
Next, when the eigenvalues are distinct, 
in order to estimate $\big(s\md_{\a,j}\big)^2$ we require estimates of $(\uk\t\a)^2$ for $k \neq j, k\leq r$; however,  by the theory in the previous section the plug-in estimator $\bm{\hat u}_k\t\a$ will be biased for $\uk\t\a$.  Therefore, instead of using $(\bm{\hat u}_k\t\a)^2$ for our estimate of $(\uk\t\a)^2$, we propose to use $(\bm{\hat u}_k\t\a)^2 (1 + \hat{b_k\md})$, where $\hat{b_k\md}$ is the estimate of the debiasing parameter for the $k$'th eigenvector obtained via \cref{alg:biasMD}.  

Finally, estimating $s\md_{\a,j}$ also requires estimating the leading $r$ eigenvalues.  Given $k \leq r$, a natural estimate for $\lambda_k$ is the empirical eigenvalue $\hat \lambda_k$; however, $\hat \lambda_k$ is a biased estimate of $\lambda_k$.  
 Therefore, motivated by the form of this eigenvalue bias,  we propose to use a debiased estimate of $\lambda_k$, defined via 
\begin{align*}
    \check \lambda_k \coloneqq \hat \lambda_k - \sum_{i > r} \frac{\hat \sigma^2}{\hat \lambda_k - \hat \lambda_i}.
\end{align*}
With all these individual estimators in place, we then define
\begin{align*}
    \big( \hat{s_{\a,j}\md} \big)^2 &= \sum_{k\neq j, k\leq r} \frac{\hat \sigma^2}{(\check \lambda_j - \check \lambda_k)^2} (\bm{\hat u}_k\t\a)^2 (1 + \hat{b_k\md}) + \frac{\hat \sigma^2}{\check\lambda_j^2} \| \bm{\hat U}_{\perp}\t\a\|^2.
\end{align*}
This estimator is entirely data-driven and does not require any sample-splitting.  With this estimator we can define a $1-\alpha$ confidence interval for $\a\t\uj$; the full procedure is summarized in \cref{alg:ciMD}.  The following result demonstrates the approximate validity of the resulting confidence intervals.

\begin{algorithm}[t]
\begin{algorithmic}[1]
\caption{Confidence Interval for $\a\t\uj$ -- Matrix Denoising}
\label{alg:ciMD}
    \REQUIRE   Input matrix $\bm{\hat S}$, rank $r$, index $j \in [r]$, unit vector $\a$, desired confidence level $\alpha$.
       \STATE Compute the eigendecomposition of $\bm{\hat S} = \big[\bm{\hat U}, \bm{\hat U}_{\perp}\big]  \bm{\hat \Lambda}  \big[\bm{\hat U}, \bm{\hat U}_{\perp}\big]\t$, where $\bm{\hat U}$ has columns consisting of the $r$ eigenvectors corresponding to the largest in magnitude nonzero eigenvalues.  Let    
       $\bm{\hat u}_k$ denote the $k$'th column of $\bm{\hat U}$ and $\hat \lambda_k$ denote the $k$'th largest in magnitude eigenvalue of $\bm{\hat S}$. 
     \STATE Set $\hat \sigma^2$ via \cref{alg:noisevarMD}.
    \STATE Set $\hat{b_k\md}$ via \cref{alg:biasMD} for $1 \leq k \leq r$.
    \STATE Compute, for $1 \leq k \leq r$, $\check \lambda_k \coloneqq \hat \lambda_k - \sum_{i > r} \frac{\hat \sigma^2}{\hat \lambda_k - \hat \lambda_i}.$
    \STATE Set
    \begin{align*}
        \big( \hat{s_{\a,j}\md}\big)^2 \coloneqq \sum_{k\leq r,k\neq j} \frac{\hat \sigma^2}{(\check \lambda_j - \check \lambda_k)^2} ( \bm{\hat u}_k\t \a)^2 (1 + \hat{b_k\md}) + \frac{\hat \sigma^2}{\check \lambda_j^2} \| \bm{\hat U}_{\perp}\t\a\|^2.
    \end{align*}
    \STATE Compute  ${\sf C.I.}^{(\alpha)}(\uhatj\t\a)$ defined via
    \begin{align*}
        {\sf C.I.}^{(\alpha)}(\uhatj\t\a) \coloneqq \bigg[ \uhatj\t\a \sqrt{1 + \hat{b_j\md}} \pm \Phi\inv\big(1 - \alpha/2\big) \hat{s_{\a,j}\md}\bigg],
    \end{align*}
    where $\Phi\inv$ is the inverse of the cumulative density function for a standard Gaussian random variable.
    \RETURN Estimated confidence interval ${\sf C.I.}^{(\alpha)}(\uhatj,\a)$.
\end{algorithmic} 
\end{algorithm}

\begin{theorem}\label{thm:civalidity_MD}
Consider the model \eqref{MD}, and suppose $\lambda_1$ through $\lambda_r$ are unique.  Suppose further that
\begin{align*}
    \lambda_{\min}/\sigma \gg  r \kappa \sqrt{n} \log(n); \qquad \Delta_{\min}/\sigma \gg C_0   r^{3/2} \log(n). \numberthis \label{snrconds_inference}
\end{align*}
In addition, assume that $r \leq c_0 \frac{n}{\kappa^2 \log^2(n)}$, where $c_0$ is a sufficiently small constant.  Suppose further that
\begin{align*}
    \max_{\substack{k\neq j\\k\leq r}} \frac{s_{\a,k}\md}{s_{\a,j}\md} \frac{\sigma \sqrt{r} \log(n)}{\Delta_j}  = o(1).  \numberthis \label{relativeratiocondition}
\end{align*}
Furthermore, assume that
\begin{align*}
    \frac{|\a\t\uj|}{s_{\a,j}\md} \bigg( \frac{\sigma^2 r \log(n)}{\Delta_j^2} + \frac{\sigma^2 \kappa \sqrt{r n\log(n)}}{\lambda_{\min}^2} \bigg) = o(1). \numberthis \label{atujcondition_ci}
\end{align*}
Let ${\sf C.I.}^{(\alpha)}(\uhatj\t\a)$ denote the output of \cref{alg:ciMD}.  If either $\|\U\t\a\| > \|\uperp\t\a\|$ with $|\a\t\uj| \lesssim |\a\t\uk|$ for $\substack{k\neq j\\k\leq r}$ or $\|\uperp\t\a\| \geq \|\U\t\a\|$, it holds that
\begin{align*}
\bigg|    \p\bigg\{ \uj\t\a \in {\sf C.I.}^{(\alpha)}(\uhatj\t\a) \bigg\} - (1 - \alpha) \bigg| = o(1).
\end{align*}
\end{theorem}

To the best of our knowledge, \cref{thm:civalidity_MD} is the first result in the literature providing asymptotically valid confidence intervals  for linear functions of eigenvectors in the matrix denoising context considered herein. Moreover, our results allow $r$ and $\kappa$ to grow with $n$.  

Note that \cref{thm:civalidity_MD} requires a condition on the minimum eigengap $\Delta_{\min}$, whereas \cref{thm:distributionaltheory_MD,thm:distributionaltheory_MD} only impose conditions on $\Delta_j$. In essence, this extra assumption is imposed to ensure a faithful estimate of $(\a\t \uk)^2$ for all  $k \leq r$ with $k \neq j$.  In addition, the new condition \eqref{relativeratiocondition} requires that the variance of $\uj\t\a$ is not significantly smaller than the variance of $\uk\t\a$.  This is to ensure that the error in estimating the quantity $(\uk\t\a)^2$ does not overwhelm the error in estimating $s_{\a,j}\md$.   
It may be possible to extend the result to settings with repeated eigenvalues assuming \emph{a priori} knowledge of which eigenvalues are repeated.

\begin{remark}[Knowledge of $r$]
Our results require that $r$ is known \emph{a priori}, though in general it must be estimated from data.  Our theory predicts that the leading $r$ eigenvalues will be well-separated from the bottom $n- r$ eigenvalues.  If this is not the case, then the spikes are not sufficiently strong for our theory to apply directly.  If the true rank of $\bS$ is $r$, but there are $r - r\s$ eigenvalues that are below the threshold $\sigma\sqrt{n}$, then  these eigenvalues are asymptotically indistinguishable from noise, so the asymptotic variance will not be significantly different as if the rank is chosen to be $r$ versus $r\s$.  It may be possible to design a procedure that still maintains validity, but this may require different techniques, so we leave designing confidence intervals for these settings to future work.
\end{remark}

\subsection{Lower Bounds and Optimality}
In this subsection we provide lower bounds on the length of any level $1 - \alpha$ honest confidence intervals.  Define the parameter space
\begin{align*}
    \mathcal{P}(r,\lambda_{\min},\Delta_j) \coloneqq \bigg\{ \bm{S}: \bm{S} = \U\bm{\Lambda} \U\t = \sum_{k=1}^{r} \lambda_k \uk\uk\t: |\lambda_k - \lambda_j| \geq \Delta_j; |\lambda_r| \geq \lambda_{\min}; \U\t\U = \bm{I}_r \bigg\}.
\end{align*}
Define the set of $1-\alpha$ honest confidence intervals based on the observation $\bS + \bN$ via
\begin{align*}
    \mathcal{I}_{\alpha,\a}(\mathcal{P}) \coloneqq \bigg\{ {\sf C.I.}\big( \bm{S} + \bm{N} \big) = [l,u]: \inf_{\bm{S}\in \mathcal{P}(r,\lambda_{\min},\Delta_j)} \mathbb{P}_{\bS}\big\{ \pm \a\t \uj(\bS) \in {\sf C.I.}\big( \bm{S} + \bm{N} \big) \big\} \geq 1 - \alpha \bigg\},
\end{align*}
where $\mathbb{P}_{\bS}(\cdot)$ denotes the probability with $\bS$ as the ground truth, and $\uj(\bS)$ denotes the $j$'th eigenvector of $\bS$.
For simplicity we suppress the dependence of ${\sf C.I.}( \bS + \bN)$ on $\bS + \bN$.  
Denoting $| {\sf C.I.} | = u - l$ as the length of the confidence interval ${\sf C.I.}$, the expected length  is denoted via
\begin{align*}
    {\sf L}_{{\sf C.I.}}(\bS) \coloneqq \mathbb{E}_{\bS}| {\sf C.I.} |.
\end{align*}
The following result gives a lower bound on the length of any level $1 - \alpha$ confidence interval based on the observation $\bS + \bN$.  The proof can be found in \cref{sec:minimaxlowerbound}.
\begin{theorem} \label{thm:minimaxlowerbd}
    Suppose that all $r$ nonzero eigenvalues of $\bS$ are distinct, and let $\Delta_j$ and $\lambda_{\min}$ denote its $j$'th eigengap and magnitude of its smallest nonzero eigenvalue respectively.  Suppose that $\Delta_j/\sigma \gg 1$ and $\lambda_{\min}/\sigma \gg \sqrt{n}$.  
    Then for any $\alpha \in (0,1/4)$, it holds that
    \begin{align*}
        \inf_{{\sf C.I.} \in \mathcal{I}_{\alpha,\a}(\mathcal{P}(r,\lambda_{\min},\Delta_j))} {\sf L}_{{\sf C.I.}}(\bS) \gtrsim   s_{\a,j}\md,
    \end{align*} 
    where the implicit constant depends only on $\alpha$.
\end{theorem}
 The parameter space $\mathcal{P}(r,\lambda_{\min},\Delta_j)$ can be understood as the parameter space ``generated'' by the underlying ground truth $\bS$; that is, it is the set of rank $r$ matrices with eigengap at least $\Delta_j$ and smallest eigenvalue $\lambda_{\min}$. Therefore, any confidence interval that has level $1-\alpha$ coverage uniformly over the parameter space generated by $\bS$ must have expected length of order $s_{\a,j}\md$ for the observation $\bS$.   Our lower bound construction is based on considering a null and alternative hypothesis space such that the eigenvalues are shared between both models, but the eigenvectors are rotated within each model.  The rotation construction is novel to the best of our knowledge.

Finally, the following corollary shows that the length of the confidence intervals generated by \cref{alg:ciMD} are optimal in this sense.
\begin{corollary}
    Under the conditions of \cref{thm:civalidity_MD}, it holds that
    \begin{align*}
     \mathbb{E}_{\theta} | \hat{{\sf C.I.}} | \lesssim s_{\a,j}\md + n^{-8}.
    \end{align*}
\end{corollary}
\begin{proof}
    The result follows immediately from \cref{lem:sigma_md_approx}.
\end{proof}

\section{Principal Component Analysis}
\label{sec:pca}
Consider the setting where one is given $n$ observations $\bm{X}_i \in \mathbb{R}^{p}$ satisfying
\begin{align*}
    \bm{X}_i \sim \mathcal{N}(0, \bSigma), \numberthis \label{pcadef}
\end{align*}
where $\bSigma$ has the form
\begin{align*}
    \bSigma &= \bSigma_0 + \sigma^2 \bm{I}_p; \\
    \bSigma_0 &= \bm{U} \bm{\Lambda} \bm{U}\t = \sum_{i=1}^{r} \lambda_i \bm{u}_i \bm{u}_i\t,
\end{align*}
where each $\lambda_i$ satisfies $0 <\lambda_r \leq \lambda_{r-1} \leq \cdots \leq \lambda_1$.  We denote $\lambda_{\min} = \lambda_r, \lambda_{\max} = \lambda_1,$
and we define $\kappa$ as the \emph{reduced condition number} $\kappa \coloneqq \lambda_{\max}/\lambda_{\min}$.  Similar to matrix denoising, for a given index $j$ with $1\leq j\leq r$, we denote the eigengap associated to the index $j$ via $\Delta_j = \min_{k\neq j,1\leq k\leq r} |\lambda_j - \lambda_k|,$
with $\Delta_j = \lambda_{\min}$ if $r =1$. We also set $\Delta_{\min} = \min_{1\leq j \leq r} \Delta_j.$ 
We consider estimating $\bSigma$ with the sample covariance $\sigmahat$ defined via
\begin{align*}
    \sigmahat &\coloneqq \frac{1}{n} \sum_{i=1}^{n} \bm{X}_i \bm{X}_i\t= \frac{\bm{XX}\t}{n},
\end{align*}
where $\bm{X} \in \mathbb{R}^{p\times n}$ is the matrix whose columns are the observations $\bm{X}_i$.  We let $\sigmahat$ have eigenvectors and eigenvalues $\{\bm{\hat u}_k,\hat \lambda_k\}_{k=1}^{p}$.  We let $\bm{\hat U}$ denote the $p \times r$ matrix of leading eigenvectors of $\bm{\hat \Sigma}$, and we set $\uhatj$ to be the $j$'th column of $\bm{\hat U}$.

\subsection{Distributional Theory}
As in the matrix denoising setting, the quantity $\a\t\uhatj$ is known to be biased.  Define the debiasing quantity $b_j\pca$ via
\begin{align*}
    b_j\pca &\coloneqq \begin{cases} \dfrac{\hat \lambda_j}{n + \sum_{r< i \leq n} \frac{\hat \lambda_i}{\hat \lambda_j -\hat \lambda_i}} \mathlarger{\sum}_{i: r < i \leq n} \dfrac{\hat \lambda_i}{(\hat \lambda_j - \hat \lambda_i)^2} & n \geq p; \\
    \dfrac{\frac{\sigma^2 p}{n}}{\hat \lambda_j - \frac{\sigma^2 p}{n}} + \dfrac{\hat \lambda_j}{\hat \lambda_j - \frac{\sigma^2 p}{n}} \dfrac{\hat \lambda_j}{n + \sum_{i: r< i \leq n} \frac{\hat \lambda_i}{\hat \lambda_j - \hat\lambda_i}} \mathlarger{\sum}_{i: r< i \leq n} \dfrac{\hat \lambda_i - \frac{\sigma^2 p}{n}}{(\hat \lambda_j - \hat \lambda_i)^2} & n < p,
    \end{cases} \numberthis \label{alg:biasPCA}
\end{align*}
where we include additional zero eigenvalues in the case that $n \geq p$.  The following result studies the asymptotic normality of $\a\t\uhatj$ and $\a\t\uhatj \sqrt{1 + b_j\pca}$.

\begin{theorem}\label{thm:distributionaltheory_PCA}
Consider the model in \eqref{pcadef} and suppose $\lambda_j$ is unique.     Suppose that $ r \leq c_0 \frac{n}{\kappa^2 \log^4(n\vee p)}$ for some sufficiently small constant $c_0$, that $\log(p) \lesssim n$, and that
    \begin{align*}
        \lambda_{\min} &\geq  C_1 \sigma^2 \log^3(n\vee p)\bigg( \kappa \frac{p}{n} + \sqrt{\frac{p}{n}} \bigg);  
        \numberthis \label{noiseassumption:pca}\\
        \Delta_j &\geq C_1 \big( \lambda_{\max} + \sigma^2\big) \sqrt{\frac{r}{n}} \log(n \vee p),
    \end{align*}
    where $C_1$ is a sufficiently large constant.  Let $\a$ be any determinstic unit vector such that $\a \neq \pm \uj$. Let $\U^{(k)}$ denote the eigenspace corresponding to $\lambda_k$. Define
    \begin{align*}
         (s_{\a,j}\pca)^2 &= \sum_{\substack{k\neq j\\k\leq r}'} \frac{(\lambda_j + \sigma^2) (\lambda_k + \sigma^2) \| \a\t \U^{(k)}\|^2 }{n(\lambda_j - \lambda_k)^2} + \frac{(\lambda_j + \sigma^2)\sigma^2 \|\uperp\t \a\|^2}{n\lambda_j^2},
    \end{align*}
    where the summation is over all $r' \leq r$ distinct eigenvalues.  Then it holds that
\begin{align*}
  \sup_{z\in \mathbb{R}}  \bigg| \p\bigg\{ \frac{1}{s_{\a,j}\pca} \bigg( \a\t \uhatj - \a\t \uj\uj\t\uhatj \bigg) \leq z \bigg\} - \Phi(z) \bigg| &\lesssim {\sf ErrPCA}, \numberthis \label{pca_firstresult}
  \end{align*}
  where
  \begin{align*}
      {\sf ErrPCA} &\coloneqq  \underbrace{\kappa^{3/2} r \log^3(n\vee p)\Bigg(  \frac{\sigma^2}{\lambda_j} \bigg( \frac{p}{n} + \sqrt{\frac{p}{n}} + \sqrt{\frac{\log(n\vee p)}{n}} \bigg)  + \frac{\sigma}{\sqrt{\lambda_j}} \sqrt{\frac{p}{n}}\Bigg) }_{\text{Subspace Estimation Effect}} \\
     &\quad +   \underbrace{\kappa r^{3/2} \log^{5/2}(n\vee p)\frac{(\lambda_{\max} + \sigma^2) }{\Delta_j \sqrt{n}}}_{\text{Small  Eigengaps Effect}} + \underbrace{\kappa^2 r^{3/2} \log^{5/2}(n\vee p)\frac{1}{\sqrt{n}}}_{\text{Parametric Effect}}.
  \end{align*}
  Next, suppose that ${\sf ErrPCA} = o(1)$.  Suppose further that
\begin{align*}
    |\a\t\uj| \frac{(\lambda_{\max} + \sigma^2)^2 r \log(n\vee p)}{\Delta_j^2 n} \ll s_{\a,j}\pca \numberthis \label{atujconditionpca1}
\end{align*}
and that 
\begin{align*}
|\a\t\uj| \times \Bigg(
\begin{cases} 
  \frac{\sigma^2 \kappa \sqrt{pr \log(n\vee p)}}{\lambda_j n}  & p > n; \\
    \frac{\sigma^2 \kappa p \sqrt{r\log(n\vee p)}}{\lambda_j n^{3/2}} \bigg( 1 + \frac{\sigma^2}{\lambda_j} \bigg) + \frac{\sigma^2 \sqrt{p} \log^2(n\vee p)}{\lambda_j n}  \bigg( 1 + \frac{\sigma^2}{\lambda_j} \bigg)    & p \leq n
    \end{cases}\Bigg) \ll s_{\a,j}\pca. \numberthis \label{atujconditionpca2}
\end{align*}
Then it further holds that
\begin{align*}
    \sup_{z\in \mathbb{R}} \bigg| \mathbb{P} \bigg\{ \frac{1}{s_{\a,j}\pca}\big( \a\t\uhatj \sqrt{1 + b_j\pca} - \a\t\uj \big) \leq z \bigg\} - \Phi(z) \bigg| = o(1). \numberthis \label{pca_secondresult}
\end{align*}
\end{theorem}
\noindent
Again, we single out several features of \cref{thm:distributionaltheory_PCA}.
\begin{itemize}
    \item \textbf{Signal-Strength conditions.} 
Up to logarithmic terms, factors of $\kappa$, and factors of $r$, our assumption on $\lambda_{\min}$ in \eqref{noiseassumption:pca}  requires that
\begin{align*}
    \lambda_{\min}/\sigma^2 \gtrsim \max\bigg\{ \frac{p}{n}, \sqrt{\frac{p}{n}} \bigg\},
\end{align*}
which matches the assumption in \citet{li_minimax_2025}.  By \citet{cai_sparse_2013}, the minimax rate of estimation for $\U$ in Frobenius norm is of the form
\begin{align*}
    \frac{pr}{n} \frac{(\lambda_{\min}/\sigma^2) + 1}{(\lambda_{\min}/\sigma^2)^2} \asymp \frac{pr}{n} \max\bigg\{ \frac{1}{\lambda_{\min}/\sigma^2}, \frac{1}{(\lambda_{\min}/\sigma^2)^2}\bigg\}. \numberthis \label{minimaxrate}
\end{align*}
It is likely possible that the dependence on $r,\kappa$, and logarithmic terms can be improved, but such analysis requires additional bookkeeping, and may require introducing new tools. Again, this deficiency arises due to our proof technique, which requires us to demonstrate that the residual quantities are $o(s_{\a,j}\pca)$ (and not simply $O(s_{\a,j}\pca)$). 
\item \textbf{Eigengap conditions}. In order for the bound to be non-vacuous, when $\kappa = O(1)$, our result requires $\Delta_j \gtrsim (\lambda_{\max} + \sigma^2) \frac{r^{5/2} \log^{5/2}(n\vee p)}{\sqrt{n}}$.
Were we to apply the Davis-Kahan Theorem to $\|\uhatj \pm \uj\|$, we would essentially need to have 
$\Delta_j \gg \| \bm{\hat \Sigma} - \bm{\Sigma} \|$.  Using the concentration bounds of \citet{koltchinskii_concentration_2017}, when $r,\kappa = O(1)$, this condition translates to $\max\bigg\{ \frac{\lambda_{\max} + p \sigma^2}{n(\lambda_{\max} + \sigma^2)}, \sqrt{\frac{\lambda_{\max} + p \sigma^2}{n(\lambda_{\max} + \sigma^2)}} \bigg\} \ll \frac{\Delta_j}{\lambda_{\max} + \sigma^2}$.  In the regime $\lambda_{\max} \lesssim \sigma^2$, our condition improves by a factor of $\tilde \Omega(\sqrt{p})$, and when $\sigma^2 \ll \lambda_{\max} \ll p \sigma^2$, our condition improves by a factor of $\tilde \Omega\big( \sqrt{\frac{\sigma^2 p}{\lambda_{\max}}} \big)$. Thus, our condition improves upon the implicit Davis-Kahan Theorem condition except for the ``extreme spike'' regime $\lambda_{\max} \gtrsim p \sigma^2$, in which case it is comparable.  
  
\item \textbf{Rate of convergence.}  When  $\kappa,r = O(1)$, it holds that
\begin{align*}
     \bigg| \p\bigg\{ \frac{1}{s_{\a,j}\pca} \bigg( \a\t \uhatj - \a\t \uj\uj\t\uhatj \bigg) \leq z \bigg\} - \Phi(z) \bigg|
     &= \tilde O \Bigg\{   
    \frac{(\lambda_{\max} + \sigma^2) }{\Delta_j \sqrt{n}} + \frac{1}{\sqrt{n}} +   \frac{\sigma^2}{\lambda_j} \bigg( \frac{p}{n} + \sqrt{\frac{p}{n}} \bigg) +   \frac{\sigma}{\sqrt{\lambda_j}} \sqrt{\frac{p}{n}} \Bigg\}.
\end{align*}
Similar to the matrix denoising setting, the quantities $\frac{\sigma^2}{\lambda_j} \sqrt{\frac{p}{n}}$ and $\frac{\sigma}{\lambda_j^{1/2}}\sqrt{\frac{p}{n}}$ quantify the ``subspace estimation effect,'' or how easy it is to separate the leading $r$ principal components from the remaining components, where the two different regimes arise according to the relationship between $\lambda/\sigma^2$ and $\frac{p}{n}$ (see the minimax rate given in \eqref{minimaxrate}), and the quantity $\frac{\lambda_{\max} + \sigma^2}{\Delta_j \sqrt{n}}$ quantifies the ``small eigengaps effect'' via how well it is to separate the $j$'th eigenvector from the other eigenvectors.  Unlike matrix denoising, PCA exhibits an additional factor of $\frac{1}{\sqrt{n}}$, representing the classical parametric rate that cannot be overcome even when the noise is relatively small. 
\item \textbf{Biased estimation.} As in matrix denoising, the centering term in \cref{thm:distributionaltheory_PCA} is given by $\a\t\uj \uj\t\uhatj$.   In previous work (e.g., \citet{paul_asymptotics_2007}) it has been established that $\bm{\hat u}_1\t \bm{u}_1 \to c \in (0,1)$ when $\lambda_{\min}/\sigma^2 \asymp 1$ and $p/n$ converges to a constant (provided $\lambda_{\min}/\sigma^2$ is above a certain threshold), and $\bm{\hat u}_1\t \bm{u}_1 = 1 - o(1)$ when $\lambda_{\min}/\sigma^2$ grows sufficiently quickly.  Therefore, \cref{thm:distributionaltheory_PCA} demonstrates that $\a\t\uhatj$ is a biased estimate of $\a\t\uj$, with bias given by $\a\t \uj(1 - \uhatj\t \uj)$. 
\item \textbf{Allowable size of $|\a\t\uj|$}.  Similar to the matrix denoising setting, $
|\a\t\uj|$ must be bounded away from one, with larger values permitted when there are larger eigengaps and signal strengths.  Singling out the case $p > n$, the condition \eqref{atujconditionpca2} requires that $|\a\t\uj| \frac{\sigma^2 \sqrt{p}}{\lambda_{\min} n} \ll s_{\a,j}\pca$ modulo factors of $\kappa,r$, and logarithmic terms.  Thus, since $\frac{\sigma^2 p}{\lambda_{\min}n} = o(1)$ is required for consistency, this condition is relatively mild.
\item \textbf{Perturbation bounds}. As in matrix denoising, the proof of \cref{thm:distributionaltheory_PCA} can be modified to yield an upper bound of the form
\begin{align*}
    |\a\t\uhatj \sqrt{1 + b_j\pca} - \a\t\uj| \lesssim s_{\a,j}\md \sqrt{\log(n\vee p)} + 
|\a\t\uj| \frac{(\lambda_{\max} + \sigma^2)^2 r\log(n \vee p)}{\Delta_j^2 n} + |\a\t\uj| \times \mathcal{E}
    \end{align*}
where 
\begin{align*}
    \mathcal{E} = \begin{cases} 
  \frac{\sigma^2 \kappa \sqrt{pr \log(n\vee p)}}{\lambda_j n}  & p > n; \\
    \frac{\sigma^2 \kappa p \sqrt{r\log(n\vee p)}}{\lambda_j n^{3/2}} \bigg( 1 + \frac{\sigma^2}{\lambda_j} \bigg) + \frac{\sigma^2 \sqrt{p} \log^2(n\vee p)}{\lambda_j n}  \bigg( 1 + \frac{\sigma^2}{\lambda_j} \bigg)    & p \leq n
    \end{cases}.
\end{align*}
This bound is slightly tighter than the bound derived in \citet{li_minimax_2025} in terms of the dependence on $r$, and thus is minimax-optimal.
\end{itemize}

\cref{thm:distributionaltheory_PCA} is most similar to the works \citet{koltchinskii_asymptotics_2016,koltchinskii_efficient_2020}, both of whom study estimating $\a\t \uj$.   Our model differs significantly from the model considered in these prior works: we require more specificity through the assumption of a low-rank spike model, whereas they only require a condition on the ``effective rank'' defined as  $\frac{{\sf Tr}(\bm{\Sigma})}{\|\bm{\Sigma}\|}$. Translating this to our setting,  when \citet{koltchinskii_asymptotics_2016} require that 
 \begin{align*}
     \frac{(\lambda_{\max} + \sigma^2)^2}{\Delta_j^2 \sqrt{n}} \sqrt{\frac{r \lambda_{\max} + p \sigma^2}{\lambda_{\max} + \sigma^2}} = o(1).
 \end{align*}
 When $r,\kappa = O(1)$, our eigengap condition is weaker by a factor of $\tilde\Omega\big( \frac{n (\lambda_{\max}/\sigma^2 + p)}{\lambda_{\max}/\sigma^2 + 1 } \big)^{1/4}$, thus permitting significantly smaller eigengaps (depending on the signal-to-noise ratio $\lambda_{\max}/\sigma^2$).  
In addition, our asymptotic variance matches the variance obtained in \citet{koltchinskii_efficient_2020}, who demonstrate that this variance is optimal over a sufficiently broad class of covariance matrices.  Our estimator thus achieves the semi-parametric lower bound without the need for sample-splitting (albeit under a simpler statistical model). However, it is worth noting that  \citet{koltchinskii_efficient_2020} also provide similar finite-sample results under their more general setup.

\subsection{Confidence Intervals}
\label{sec:pca_inference}
Similar to the matrix denoising setting, in order to obtain data-driven confidence intervals, we require an estimator for the noise variance $\sigma^2$.  We therefore consider two different estimators based on the aspect ratio $\frac{p}{n}$.  Define
\begin{align*}
    \hat \sigma^2 \coloneqq \begin{cases}\dfrac{{\sf Tr} \big( \sigmahat - \sigmahat_r \big)}{p-r} & p \geq \dfrac{n}{\log^4(n\vee p)}; \\
      \hat \lambda_{r+1} & p < \dfrac{n}{\log^4(n\vee p)}. \end{cases} \numberthis \label{alg:noisevarPCA}
\end{align*}
The intuition of this estimator can be understood as follows: when $p \ll n$, the eigenvalues of $\sigmahat$ are already consistent for those of $\bSigma$, and hence $\hat \lambda_{r+1}$ will approximate $\sigma^2$.  However, when $p$ is relatively large, the approximation fails, and a more careful approximation must be done by leveraging the independence of the noise, which becomes more effective when $p$ is large.  The choice of the cutoff $n/\log^4(n\vee p)$ is primarily for technical reasons; it is possible that a similar argument will work with any choice for $p$ strictly smaller than $n$.  
With the estimator $\hat \sigma^2$ defined, we can also define the estimated debiasing parameter $\hat{b_j\pca}$ with the same definition as in \cref{thm:distributionaltheory_PCA}, except with $\hat \sigma^2$ as a replacement for $\sigma^2$.

Next, as in the matrix denoising setting, estimating the asymptotic variance requires estimating $(\a\t\uk)$ for $k \leq r$ and $\lambda_k$ for $k \leq r$. Since the empirical eigenvalues $\hat \lambda_j$ exhibit a multiplicative bias (see \cref{lem:eigenvalueconcentration_PCA}), we propose to use a debiased estimator for $\lambda_k + \sigma^2$, defined via: 
\begin{align*}
    \check{\lambda}_k &\coloneqq \frac{\hat \lambda_k}{1 + \hat \gamma\pca(\hat \lambda_k)},
\end{align*}
where
\begin{align*}
    \hat \gamma\pca(\hat \lambda_k) &= \frac{1}{n} \sum_{i=r+1}^{\min(p-r,n)} \frac{\hat \lambda_i}{\hat \lambda_k - \hat \lambda_i}.
\end{align*}
Under our assumptions we can demonstrate that $\check\lambda_k$ is a strong estimator of $\lambda_k + \sigma^2$.  To estimate $\lambda_j^2$, we simply use the estimator $\check \lambda_j - \hat \sigma^2$.  If $\hat \sigma^2 \approx \sigma^2$, then it holds that $\check \lambda_j - \hat \sigma^2 \approx \lambda_j$.   

Finally, to estimate $s_{\a,j}\pca$ we also need to estimate $(\a\t\uk)^2$ for $k \leq r$ and $k \neq j$. 
As in the matrix denoising case, the plug-in estimators for $\a\t \uk$ may not be sufficiently close to merit a strong estimate. We therefore propose to use the debiasing parameter $\hat{b_k\pca}$ to form the estimator $(\a\t \bm{\hat u}_k)^2 (1 + \hat{b_k\pca})$, which, from the previous analysis,  is a strong estimator of $(\a\t \uk)^2$ under conditions on $\Delta_{\min}$.  

Combining these estimators, we set
\begin{align*}
    \big(\hat{s\pca_{\a,j}} \big)^2 &\coloneqq \sum_{k\neq j} \frac{\check \lambda_k \check \lambda_j (\a\t \bm{\hat u}_k)^2 \big( 1 + \hat{b_k\pca} \big)}{n(\check \lambda_j - \check \lambda_k)^2} + \frac{\hat \sigma^2 \check \lambda_j \| \uhat_{\perp}\t \a \|^2}{n(\check \lambda_j - \hat \sigma^2)^2},
\end{align*}
and we use this estimator to construct a $1-\alpha$ confidence interval for $\a\t\uj$.  The full procedure is summarized in \cref{alg:ciPCA}. The following result demonstrates the approximate validity of the resulting confidence intervals.

\begin{algorithm}[t]
\begin{algorithmic}[1]
\caption{Confidence Interval for $\a\t\uj$ -- PCA}
\label{alg:ciPCA}
    \REQUIRE  Sample data $\bm{X} \in \mathbb{R}^{p\times n}$, rank $r$, index $j \in [r]$, unit vector $\a \in \mathbb{R}^p$, desired confidence level $\alpha$.
    \STATE  Compute the eigendecomposition of $\sigmahat = \frac{\bm{XX}\t}{n} = \big[ \uhat, \uhat_{\perp}\big] \bm{\hat \Lambda} \big[\uhat, \uhat_{\perp} \big]\t$, where $\uhat$ has columns consisting of the $r$ eigenvectors corresponding to the largest nonzero eigenvalues.  Let $\bm{\hat u}_k$ denote the $k$'th column of $\uhat$ and $\hat \lambda_k$ denote the $k$'th largest eigenvalue of $\sigmahat$.
    \STATE Set $\hat \sigma^2$ via \cref{alg:noisevarPCA}. 
    \STATE Set $\hat{b_k\pca}$ via \cref{alg:biasPCA} for $1 \leq k \leq r$, using $\hat \sigma^2$ for $\sigma^2$. 
    \STATE Compute, for $1 \leq k \leq r$, $\hat \gamma\pca(\hat \lambda_k) \coloneqq \frac{1}{n} \sum_{i=r+1}^{\min(p-r,n)} \frac{\hat \lambda_i}{\hat \lambda_k - \hat \lambda_i}.$
    \STATE Compute, for $1 \leq k\leq r$, $\check \lambda_k \coloneqq \frac{\hat \lambda_k}{1 + \hat \gamma\pca(\hat \lambda_k)}.$
    \STATE Set
    \begin{align*}
        \big( \hat{s_{\a,j}\pca} \big)^2 \coloneqq \sum_{k\leq r,k\neq j} \frac{\check \lambda_k \check \lambda_j (\a\t \bm{\hat u}_k)^2 (1 + \hat{b_k\pca})}{n(\check\lambda_j - \check \lambda_k)^2} + \frac{\hat \sigma^2 \check \lambda_j \|\uhat_{\perp}\t\a\|^2}{n(\check \lambda_j - \hat \sigma^2)^2}.
   \end{align*}
    \STATE Compute ${\sf C.I.}^{(\alpha)}(\uhatj\t\a)$ defined via
 \begin{align*}
        {\sf C.I.}^{(\alpha)}(\uhatj\t\a) \coloneqq \bigg[ \uhatj\t\a \sqrt{1 + \hat{b_j\pca}} \pm \Phi\inv\big(1 - \alpha/2\big) \hat{s\pca_{\a,j}} \bigg]
    \end{align*}
    \RETURN Estimated confidence interval ${\sf C.I.}^{(\alpha)}(\uhatj\t\a)$. 
\end{algorithmic} 
\end{algorithm}

\begin{theorem}\label{thm:civalidity_pca}
Consider the model in \eqref{pcadef} and suppose $\lambda_1$ through $\lambda_r$ are unique.  Suppose further that 
\begin{align*}
      \lambda_{\min} &\geq  C_1  \sigma^2 \kappa^{3} r^2 \log^7(n\vee p) \bigg( \frac{p}{n} + \sqrt{\frac{p}{n}}\bigg)  
      \numberthis \label{noisecondpca:inf} \\ 
         \Delta_{\min} &\geq C_1 \big( \lambda_{\max} + \sigma^2) \frac{r^{5/2}}{\sqrt{n}} \log^4(n\vee p), \numberthis \label{eigengapcondpca:inf}
\end{align*}
where $C_1$ is some sufficiently large constant.  In addition, assume that $r^3 \leq c_1 \frac{n}{\kappa^6 \log^{12}(n\vee p)}$, where $c_1$ is a sufficiently small constant, and that $\log(p) \lesssim n^{1/2}$.   Suppose that
\begin{align*}
    \max_{\substack{k\neq j\\k\leq r}} \frac{s_{\a,k}\pca}{s_{\a,j}\pca} \frac{(\lambda_{\max} + \sigma^2) \sqrt{r} \log^2(n\vee p)}{\Delta_j \sqrt{n}} = o(1). \numberthis \label{sakassumption:pca}
\end{align*}
Finally, assume that \eqref{atujconditionpca1} and \eqref{atujconditionpca2} hold.  Then it holds that 
\begin{align*}
    \bigg| \p\bigg\{ \a\t\uj \in {\sf C.I.}^{(\alpha)}(\a\t \uhatj) \bigg\} - (1- \alpha) \bigg| = o(1).
\end{align*}
\end{theorem}
The assumptions in \cref{thm:civalidity_pca} are similar to those of \cref{thm:civalidity_MD}.  %

\subsection{Lower Bounds}
We now provide a lower bounds for the length of any level $1 - \alpha$ honest confidence interval.  Define the parameter space
\begin{align*}
    \mathcal{P}(r,\lambda_{\min},\Delta_j,\sigma^2) \coloneqq \bigg\{ \bm{\Sigma} = \bm{U\Lambda U}\t + \sigma^2 \bm{I}_p =  \sum_{k=1}^{r} \lambda_i \uk\uk\t + \sigma^2 \bm{I}_p: \lambda_r \geq \lambda_{\min}, \min_{k\neq j} |\lambda_i - \lambda_j| \geq \Delta_j; \U\t\U = \bm{I}_r \bigg\}.
\end{align*}
Define the set 
\begin{align*}
    \mathcal{I}_{\alpha,\a}(\mathcal{P}) \coloneqq \bigg\{ {\sf C.I.}(\bm{X}) = [l,u]: \inf_{\bm{\Sigma} \in \mathcal{P}} \mathbb{P}_{\bm{\Sigma}} \big\{ \pm \a\t\uj(\bm{\Sigma}) \in {\sf C.I.}(\bm{X}) \big\} \geq 1 - \alpha \bigg\},
\end{align*}
where all the notation is similar to the matrix denoising case. Again $| {\sf C.I.} | = u - l$ denotes the length of the confidence interval ${\sf C.I.}$, and the expected length  is denoted via
${\sf L}_{{\sf C.I.}}(\bm{\Sigma}) \coloneqq \mathbb{E}_{\bm{\Sigma}}| {\sf C.I.} |.$
The following result gives a lower bound on the length of any level $1 - \alpha$ confidence interval based on the observation $\bm{X}$.  
\begin{theorem} \label{thm:minimaxlowerbound_PCA}
   Suppose that the leading $r$ eigenvalues of $\bSigma$ are distinct with $\lambda_r = \lambda_{\min}$ and eigengap $\Delta_j$, and assume that $n \gg \max\big\{ \frac{(\lambda_j + \sigma^2)\sigma^2}{\lambda_j^2}, \frac{(\lambda_{\max} + \sigma^2)^2}{\Delta_j^2} \big\}$.  Then for any $\alpha \in (0,1/4)$, it holds that 
    \begin{align*}
        \inf_{{\sf C.I.} \in \mathcal{I}_{\alpha,\a}(\mathcal{P}(r,\lambda_{\min},\Delta_j,\sigma^2))} {\sf L}_{{\sf C.I.}}(\bm{\Sigma}) \gtrsim s_{\a,j}\pca,
    \end{align*}
    where the implicit constant depends only on $\alpha$.
\end{theorem}
Finally, the following corollary shows that the confidence intervals attained by \cref{thm:civalidity_pca} are optimal.
\begin{corollary}
    Under the conditions of \cref{thm:civalidity_pca}, it holds that
    \begin{align*}
     \mathbb{E}_{\theta} | \hat{{\sf C.I.}} | \lesssim s_{\a,j}\pca + (n\vee p)^{-8}
    \end{align*}
\end{corollary}
\begin{proof}
    The result follows immediately from \cref{lem:sapproxpca} en route to the proof of \cref{thm:civalidity_pca}.
\end{proof}

\section{Previous Work} \label{sec:previouswork}

The analysis of spectral methods has a rich history, going back to $\ell_2$ perturbation theory such as the Davis-Kahan Theorem \citep{davis_rotation_1970} or Wedin's Theorem \citep{wedin_perturbation_1972}, and much of this theory now forms the basis of book-level treatments (e.g. \citet{stewart_matrix_1990,kato_perturbation_1995}).  Classical deterministic matrix perturbation theory is remarkably effective in a number of statistical settings, and improvements on the deterministic theory have been obtained \citep{cai_rate-optimal_2018,luo_schatten-q_2021,yu_useful_2015,zhang_exact_2024}, as well as improvements tailored explicitly for various types of random noise \citep{orourke_random_2018,koltchinskii_perturbation_2016,nadler_finite_2008}. Notably, a number of works have studied the $\ell_2$ perturbation in principal component analysis \citep{cai_optimal_2021,cai_sparse_2013,koltchinskii_normal_2017}.  Many of these works focus on providing bounds for estimated subspaces (or for eigenvectors with large eigengaps), though there have been several related works in the context of small eigengaps \citep{orourke_matrices_2024,braun_accurate_2006,diaconu_eigenstructure_2023,ostrovskii_affine_2019,mas_high-dimensional_2015,jirak_perturbation_2019,jirak_relative_2023}.

Beyond  $\ell_2$ perturbation theory, a number of works have also considered $\ell_{2,\infty}$ perturbation theory, which quantifies the entrywise (or row-wise) fluctuations of eigenvectors (or subspaces). The works \citet{cape_two--infinity_2019,fan_ell_infty_2018,eldridge_unperturbed_2018,damle_uniform_2020,bhardwaj_entry-wise_2023} provide deterministic $\ell_{2,\infty}$ bounds for the leading subspaces of low-rank matrices and apply their results to problems arising in high-dimensional statistics.  Furthermore, by taking into account both the structural assumptions and the probabilistic nature of the noise, a number of works have also derived $\ell_{2,\infty}$ perturbation bounds tailored to different probabilistic settings.  For example, \citet{lei_unified_2019}  considers symmetric matrices with independent noise, \citet{cai_subspace_2021} and \citet{zhou_deflated_2025} study highly unbalanced matrices,~\citet{agterberg_entrywise_2022-1}~focus on matrices with additional sparse structure, \citet{agterberg_estimating_2025} examine problems arising from tensor data analysis, and \citet{chen_spectral_2019,chen_partial_2022} study ranking problems, though this list is still incomplete.  A general survey on applications of $\ell_{2,\infty}$ perturbation theory in various statistical contexts can be found in \citet{chen_spectral_2021}.

The $\ell_{2,\infty}$ subspace perturbation bounds considered in previous works are a special type of  the more general linear functions studied in this work.  The works \citet{koltchinskii_perturbation_2016} and \citet{xia_sup-norm_2019} study perturbation bounds for linear functions of eigenvectors, with the latter focusing on the ``unbalanced'' setting, wherein the column dimension is significantly larger than the row dimension.  Similarly, \citet{koltchinskii_asymptotics_2016} provides perturbation bounds for linear functions of principal components under Gaussian noise in a general Hilbert space.  However, these previous results primarily operate in the ``large eigengaps'' regime, or at least have stronger eigengap conditions than in this work.  The works \citet{cheng_tackling_2021,chen_asymmetry_2021,li_minimax_2025} all provide perturbation bounds for linear functions of eigenvectors in the ``small eigengaps'' regime, though both \citet{cheng_tackling_2021} and \citet{chen_asymmetry_2021} require \emph{asymmetric} noise, something that we do not consider herein.

This work is also closely related to a number of works studying the asymptotic theory for low-rank matrix models.  For example, \citet{xia_normal_2021,xia_confidence_2019,bao_singular_2021} consider asymptotic theory for the $\sin\Theta$ distances in matrix denoising, and \citet{bao_statistical_2022,koltchinskii_normal_2017} study the asymptotic distribution of $\sin\Theta$ distances for the spiked principal component analysis model.  These works also pertain to the so-called ``supercritical'' regime in random matrix theory, for which there are a number of works studying the asymptotics for eigenvalues and eigenvectors in matrix denoising \citep{ding_high_2020,capitaine_limiting_2018,benaych-georges_eigenvalues_2011,benaych-georges_singular_2012,knowles_outliers_2014,capitaine_central_2012,capitaine_largest_2009}
and PCA \citep{ding_spiked_2021,paul_asymptotics_2007,bloemendal_principal_2016,baik_phase_2005,ding_spiked_2021-1,onatski_asymptotics_2012,johnstone_consistency_2009}, though these lists are incomplete.

Turning to asymptotic distributional theory, the works \citet{koltchinskii_concentration_2017,ding_spiked_2021,bao_statistical_2022} study the asymptotics for bilinear forms of spectral projectors in spiked covariance models, and  \citet{koltchinskii_efficient_2020} study efficient estimation in these settings.  In a matrix denoising setting, \citet{fan_asymptotic_2020} study linear functions of eigenvectors of symmetric matrices when eigenvalues diverge.  However, in the aforementioned works, the emphasis has primarily been on the ``large eigengaps'' regime, and the results typically do not hold for finite samples, with the exception of \citet{koltchinskii_efficient_2020}, who operate under a different context than the one herein.

Finally, statistical inference in low-rank matrix models is a relatively nascent field, with only a few results in specific contexts \citep{agterberg_entrywise_2022,yan_inference_2021,carpentier_uncertainty_2019,silin_hypothesis_2020,carpentier_signal_2015,carpentier_constructing_2018,chernozhukov_inference_2023,cai_uncertainty_2023,xia_statistical_2021}.  Considering estimating linear functionals, \citet{koltchinskii_efficient_2020} show that it is possible to find asymptotically efficient confidence intervals, though their estimators still rely on sample splitting.  The work \citet{cheng_tackling_2021} provides inferential  procedures for linear functions of eigenvectors under asymmetric, heteroskedastic noise. 
However, in our work we also prove lower bounds for both models, thereby showing that our confidence intervals attain the optimal width.

\section{Discussion} \label{sec:discussion}

In this work we have considered inference for linear forms of the form $\a\t \uj$ for a pre-specified unit vector $\a$.  We have seen how  small eigengaps and signal strength affect the approximate Gaussianity of both the plugin and debiased estimators, and we have provided asymptotically valid and statistically efficient confidence intervals for $\a\t \uj$.  All of our results hold under nearly minimal signal-strength conditions, and our proposed procedures are fully data-driven and do not require any sample splitting.

There are a number of potential future works.  First, our analysis relies heavily on the assumption of Gaussian noise, which may not hold in general.  It is of interest to develop similar statistical theory under general noise mechanisms such as subgaussian or heteroskedastic noise.  \textcolor{black}{It is likely possible that our results will continue to hold under rotational invariance, but may require additional considerations without this assumption.}  Furthermore, our results are likely suboptimal with respect to logarithmic terms, factors of $r$, and factors of the condition number $\kappa$, and it would be interesting to study the  optimal dependence on these parameters.  For example, in matrix denoising, our analysis requires that $\Delta_j/\sigma \gg \mathrm{polylog}(n)$ assuming $r,\kappa = O(1)$. 
 Does asymptotic normality hold even when $\Delta_j \asymp \sigma$ provided certain necessary signal-strength conditions are met? 

Beyond these immediate future directions, it would also be of interest to study similar problems in other settings.  For example, the work \citet{xia_inference_2022} studies statistical inference for linear functions of tensor singular vectors; it would be interesting to see if similar guarantees can be developed when the tensor singular value gaps are of similarly small order as considered herein.  Finally, the work \citet{xia_confidence_2019} considers confidence regions for singular subspaces in low-rank trace regression; it would be of interest to develop similar guarantees for individual singular subspaces in the presence of small singular value gaps.

\section{Technical Tools and Proof Overview}
\label{sec:proofoverview}
The analysis for our main results for both matrix denoising and PCA follows along similar arguments.  First, our main results will require understanding the bias (i.e., the difference $\uj\t \uhatj -1$), so in our first step we characterize this quantity as well as study the approximation of the estimated bias $b_j\md$  and $b_j\pca$.  Next, we study how the empirical eigenvalues $\hat \lambda_k$ for $k \leq j$ behave, as well as study certain approximately ``debiased'' eigenvalues that (assuming knowledge of the noise $\sigma^2$) match the estimated eigenvalues $\check \lambda_k$.  These two first steps largely follow the analysis in \citet{li_minimax_2025}, with some straightforward modifications.

Our next step requires several novel considerations relative to \citet{li_minimax_2025}.  For simplicity we focus on the matrix denoising case. For ease of comparison, we first discuss the analysis in \citet{li_minimax_2025}.  By modifying their proof slightly, one arrives at the decomposition
\begin{align*}
    \a\t \uhatj - \a\t \uj \uj\t \uhatj &= \underbrace{\uhatj\t \uperp \uperp\t \a}_{T_1} + \underbrace{\sum_{k\neq j} \a\t \uk \uk\t \uj^{\perp} \big( \hat \lambda_j \bm{I}_{n-1}- \bm{\hat S}^{(j)} \big)\inv (\uj^{\perp})\t \bm{N} \uj}_{T_2}, \numberthis \label{lidecomp}
\end{align*}
where $\uj^{\perp}$ is the $n \times (n-1)$ matrix with its $j$'th column removed, and $\bm{\hat S}^{(j)} = (\uj^{\perp})\t \big( \bm{S} + \bm{N} \big) \uj^{\perp}$.  The analysis in \citet{li_minimax_2025} is based on the following two observations: first,  the term $T_1$ is approximately a product of a vector that is uniform on the subspace spanned by $\bm{U}_{\perp}$ with a deterministic vector, and hence concentrates due to the delocalization phenomenon; and, next, the term $T_2$ is a product of terms that are approximately independent of $\bm{u}_j^{\perp} \bm{N} \uj$, so  concentrates at a rate $O(s_{\a,j}\md)$.

Unfortunately, these observations are not enough to show that these quantities are $o(s_{\a,j}\md)$ as required for our distributional theory.  
From our analysis, it turns out that both $T_1$ and $T_2$ above contain a non-negligible Gaussian term, and the analysis in \citet{li_minimax_2025} is not sufficiently fine-grained to characterize the distributional fluctuations and isolate the leading-order term. Therefore, a major technical contribution of our work is in finding the ``correct'' form of the  residual terms in the presence of small eigengaps.  It is worth emphasizing that the focus of \citet{li_minimax_2025} is not on distributional theory, but rather on perturbation bounds, and hence their analysis is sufficient for their goals.  

In order to identify the leading-order term and the appropriate residual terms, we first use the following deterministic decomposition. 
\begin{lemma}
\label{thm:mainexpansion}
Let $\{\uhatj,\hat \lambda_j\}$ denote the  eigenvectors and eigenvalues of the symmetric matrix $\mhat \coloneqq \bM + \bE$, where the symmetric matrix $\bM$ has eigenvectors and eigenvalues $\uj$ and $\lambda_j$ respectively.
Suppose $\lambda_j$ is unique.    Then the following expansion is valid always:
   \begin{align*}
    \a\t \uhatj - \a\t \uj \uj\t \uhatj - \sum_{k\neq j} \frac{ \uj\t \bE \uk}{\lambda_j - \lambda_k} \a\t \uk &= \sum_{k\neq j} \frac{( \uhatj - \uj)\t \bE \uk }{\lambda_j - \lambda_k} \a\t \uk + \sum_{k\neq j} \frac{\lambda_j - \hat \lambda_j  }{\lambda_j - \lambda_k} \uhatj\t \uk \uk\t \a.
\end{align*}
\end{lemma}
\begin{proof}
    See \cref{sec:mainexpansionproof}.
\end{proof}
The goal is to demonstrate that each term on the right hand side is $o(s_{\a,j}\md)$ (with $\bE = \bN$ and $\bM = \bS$).  However, from the analysis in \citet{li_minimax_2025}, the leading empirical eigenvalues are biased.  Therefore, we further decompose the right hand side given by \cref{thm:mainexpansion} via
\begin{align*}
\sum_{k\neq j} \frac{(\uhatj- \uj)\t \bN \uk}{\lambda_j - \lambda_k} &\a\t \uk + \sum_{k\neq j} \frac{\lambda_j - \hat \lambda_j}{\lambda_j - \lambda_k} \uhatj\t \uk\uk\t\a \\
&= \sum_{k\neq j, k\leq r } \frac{(\uhatj - \uj)\t \bN \uk - \gamma\md(\hat \lambda_j) \uhatj\t \uk}{\lambda_j - \lambda_k} \a\t \uk + \frac{(\uhatj - \uj)\t \bN \uperp \uperp\t \a}{\lambda_j} \\&\quad + \sum_{\substack{k\neq j\\k\leq r} } \frac{\lambda_j - \hat \lambda_j + \gamma\md(\hat \lambda_j)}{\lambda_j - \lambda_k} \uhatj\t \uk\uk\t\a+ \frac{\lambda_j - \hat \lambda_j}{\lambda_j} \uhatj\t \uperp \uperp\t \a \\
    &\coloneqq \rmd_1 + \rmd_2 + \rmd_3 + \rmd_4, \numberthis \label{decomp2}
\end{align*}
where we have added and subtracted $\gamma\md(\hat \lambda_j)\uhatj\t\uk \uk\t\a$ for each $k \leq r$, where $\gamma\md(\hat \lambda_j)$ is an eigenvalue bias quantity (see \cref{lem:eigenvalues_MD}).  

Each term above comes with unique challenges, so we single $\rmd_1$ as a representative example, since it is also related to the eigenvalue bias quantity $\gamma\md(\hat \lambda_j)$.   We can write
\begin{align*}
    (\uhatj - \uj)\t \bN \uk &= (\uhatj - \uj)\U\U\t \bN \uk + \uhatj\t \uperp \uperp\t \bN \uk.
\end{align*}
The first term can be shown to be of sufficiently small order directly.  However, the second term above depends on $\uhatj\t \uperp$.  To analyze this quantity we introduce the following deterministic lemma. 
\begin{lemma}
    \label{lem:uhatjuperpidentity}
Let $\mhat = \bM + \bE$, and let $\{\uhatj,\hat \lambda_j\}$ denote the eigenvectors and eigenvalues of $\mhat$ and $\uj,\lambda_j$ denote the eigenvectors and eigenvalues of $\bM$. Suppose $\bM$ is rank $r$ with orthogonal complement $\uperp$, such that $\bM = \U \bm{\Lambda} \bm{U}\t$.  Then it holds that 
    \begin{align*}
    \uhatj\t \uperp &= \uhatj\t  \U \U\t \bE \uperp\t \big( \hat \lambda_j \bm{I}_{n-r} - \uperp\t \bE \uperp \big)\inv \\
    &\equiv \uhatj\t \uk \uk\t \bE \uperp \big( \hat \lambda_j \bm{I}_{n-r} - \uperp\t  \bE \uperp \big)\inv + \uhatj \U^{-k}(\U^{-k})\t \bE \uperp  \big( \hat \lambda_j \bm{I}_{n-r} - \uperp\t \bE \uperp \big)\inv, 
\end{align*}
provided the inverse is defined.  Here $\U^{-k}$ corresponds to the matrix $\U$ with its $k$'th column removed.  
\end{lemma}
\begin{proof}
    See \cref{sec:uhatjuperpproof}.
\end{proof}
With this lemma in hand, we can write
\begin{align*}
    \uhatj\t \uperp \uperp\t \bN \uk &=  \uhatj\t \uk \uk\t \bN \uperp \big( \hat \lambda_j \bm{I}_{n-r} - \uperp\t  \bN \uperp \big)\inv  \uperp\t \bN \uk \\
    &\quad + \uhatj \U^{-k}(\U^{-k})\t \bN \uperp  \big( \hat \lambda_j \bm{I}_{n-r} - \uperp\t \bN \uperp \big)\inv \uperp\t \bN \uk.
\end{align*}
The second term above concentrates around zero.  Remarkably, the first term above concentrates about $\gamma\md(\hat \lambda_j) \uhatj\t \uk$, which justifies its subtraction in \eqref{decomp2}.  Note that it is not obvious \emph{a priori} that this term concentrates; it is due to our representation in \cref{lem:uhatjuperpidentity} that enables us to identify this concentration.

Finally, to prove the asymptotic validity of our proposed confidence intervals, we must demonstrate that $$\bigg|\frac{\hat {s_{\a,j}\md}}{s_{\a,j}\md} - 1\bigg| = o(1)$$ with high probability.  
Since $\hat{s_{\a,j}\md}$ uses the debiased estimators of $\a\t\uk$, this analysis  further relies on our previously established distributional theory applied to all $k \leq r$.  The full details can be found in \cref{sec:mdproof,sec:pcamainproof} for matrix denoising and PCA respectively.

\section*{Acknowledgements}
The author thanks Carey Priebe, Alex Modell,  and Patrick Rubin-Delanchy for productive discussions on close eigenvalues in the context of network analysis problems.  In addition, the author thanks Carey Priebe for encouragement in studying this problem.  Finally, the author thanks Jes\'us Arroyo who provided valuable feedback on an early draft of this manuscript.

\appendix

\section{Proofs of Matrix Analysis Results} \label{sec:matrixanalysisproofs}
In this section we prove our main deterministic results.

\subsection{Proof of  \cref{thm:mainexpansion}} \label{sec:mainexpansionproof}
\begin{proof}
First, note that by the eigenvector-eigenvalue equation, it holds that $(\lambda_k - \hat \lambda_j) \uhatj\t  \uk = -\uhatj\t \bE \uk.$
Therefore,
\begin{align*}
    (\lambda_k - \lambda_j) \uhatj\t \uk + (\lambda_j - \hat \lambda_j) \uhatj\t \uk &= -\uhatj\t \bE \uk.
\end{align*}
Diving through by $(\lambda_k - \lambda_j)$ and rearranging yields
\begin{align*}
    \uhatj\t \uk
    &= \frac{\uhatj\t \bE \uk}{\lambda_j - \lambda_k} + \frac{\lambda_j - \hat \lambda_j}{\lambda_j - \lambda_k} \uhatj\t \uk.
\end{align*}
Therefore,
\begin{align*}
    \a\t \uhatj - \a\t \uj \uj\t \uhatj &= \a\t ( \bm{I} - \uj \uj\t) \uhatj \\&= 
    \sum_{k\neq j} \a\t \uk \uk\t \uhatj \\
    &= \sum_{k\neq j} \a\t \uk \frac{\uj\t \bE \uk}{\lambda_j - \lambda_k} + \sum_{k\neq j} \frac{(\uhatj - \uj)\t \bE \uk}{\lambda_j - \lambda_k} \a\t \uk + \sum_{k\neq j} \frac{\hat \lambda_j - \lambda_j}{\lambda_j - \lambda_k} \uhatj\t \uk \uk\t \a.
\end{align*}
Rearranging completes the proof.
\end{proof}

\subsection{Proof of \cref{lem:uhatjuperpidentity}} \label{sec:uhatjuperpproof}
\begin{proof}
The second line is immediate as $\U\U\t = \uk\uk\t + \U^{-k} (\U^{-k})\t$.  Therefore, we start with the observation
\begin{align*}
   \hat \lambda_j \uhatj\t \uperp = \uhatj\t ( \bM + \bE) \uperp = \uhatj\t \bE \uperp = \uhatj \t \U\U\t \bE \uperp + \uhatj\t \uperp \uperp\t \bE \uperp,
\end{align*}
which holds since $\bM\uperp = 0$.  By rearranging this reveals that
\begin{align*}
    \uhatj\t \uperp \big( \hat \lambda_j \bm{I}_{n-r} - \uperp\t \bE \uperp \big) &= \uhatj\t \U \U\t \bE \uperp.
\end{align*}
   Right multiplying through by $\big( \hat \lambda_j \bm{I}_{n-r} - \uperp\t \bE \uperp \big)\inv$ gives the result.
\end{proof}

\section{Proofs for Matrix Denoising} \label{sec:mdproof}
This section contains the full proof of \cref{thm:distributionaltheory_MD,thm:civalidity_MD}.  
 Throughout our proofs  we assume that $\lambda_1$ through $\lambda_r$ are unique.  If not, the extensions are only more notationally cumbersome and not significantly different.

We first state several preliminary facts that are standard from nonasymptotic random matrix theory and several results that follow immediately from the analysis in \citet{li_minimax_2025}.  We will use these facts repeatedly in our proofs. Without loss of generality we assume that $\lambda_j$ is positive; otherwise repeat the analysis with $-\bm{\hat S}$ and $-\bm{S}$.  

\begin{fact} \label{fact1_MD} It holds that $\|\bN\| \lesssim \sigma \sqrt{n}$ with probability at least $1 - \exp(-c n)$.  See, for example \citet{vershynin_high-dimensional_2018}.  
\end{fact}
\begin{fact} \label{fact2_MD}
It holds that $\| \U\t \bN \U \| \lesssim \sigma \big( \sqrt{r} + \sqrt{\log(n)} \big)$ with probability at least $1 - O(n^{-10})$.  This can be proven via standard $\eps$-net arguments, or one can appeal directly to the fact that $\U\t \bN \U $ is a symmetric $r\times r$ Gaussian Wigner matrix.
\end{fact}
\begin{fact} \label{fact3_MD}
    By virtue of \cref{fact1_MD} and Weyl's inequality, it holds that $|\lambda_k| + C \sigma \sqrt{n} \geq |\hat \lambda_k| \geq |\lambda_k| - C\sigma \sqrt{n} $ for all $k \leq r$ and $|\hat \lambda_k | \lesssim \sigma \sqrt{n}$ for all $k \geq r+ 1$.  Therefore, on the event $\|\bN\| \lesssim \sigma \sqrt{n}$, it holds that $|\hat \lambda_j| \in \big[ 2 |\lambda_j|/3, 4 |\lambda_j|/3\big]$. 
\end{fact}

\begin{fact}
    \label{fact4_MD} By virtue of \cref{fact3_MD}, it holds that $|\hat \lambda_j| \geq 2|\lambda_{\min} |/3 \geq \| \bN \| \geq \| \uperp\t \bN \uperp \|$, and hence $\lambda \bm{I}_{n-r} - \uperp\t \bN \uperp$ is invertible for all $|\lambda| \geq 2 \lambda_{\min}/3$.  Specifically, $\hat \lambda_j \bm{I}_{n-r}- \uperp\t \bN \uperp$ is invertible on the event $\|\bN \| \lesssim \sigma \sqrt{n}$.  In addition, $\| \big( \hat \lambda_j \bm{I}_{n-r} - \uperp\t \bN \uperp \big)\inv \| \lesssim \lambda_j\inv$.  
\end{fact}
Next, we require several results from \citet{li_minimax_2025}. The following result characterizes the bias of the quantity $\uj\t\uhatj$.

\begin{lemma}
   \label{lem:bias_MD}
Instate the conditions in \cref{thm:distributionaltheory_MD}. Define
\begin{align*}
     b_j\md &\coloneqq \sum_{k > r} \frac{\sigma^2}{(\hat \lambda_j - \hat \lambda_k)^2}.
\end{align*}
Then with probability at least $1 - O(n^{-10})$ it holds that
\begin{align*}
| 1 - (\uhatj\t \uj)^2 |&\lesssim \frac{\sigma^2 n}{\lambda_j^2} + \frac{\sigma^2 r\log(n)}{\Delta_j^2}; \\
\big| 1 - \sqrt{1 + b_j\md} \uj\t\uhatj \big| &\lesssim \frac{\sigma^2 r\log(n)}{\Delta_j^2} + \frac{\sigma^2 \sqrt{n\log(n)}}{\lambda_j^2}.
\end{align*}
\end{lemma}
\begin{proof}
The result follows immediately from the analysis in \citet{li_minimax_2025} (see equation 5.38 therein).
\end{proof}

The next result shows that the eigenvalues concentrate around this same quantity.  
\begin{lemma} \label{lem:eigenvalues_MD}
Instate the conditions of \cref{thm:distributionaltheory_MD}.  Define
\begin{align*}
    \gamma\md(\hat \lambda_j) \coloneqq \sigma^2 \tr\bigg( \big( \hat \lambda_j \bm{I}_{n-r} - \uperp\t \bN \uperp \big)\inv \bigg); \qquad \hat \gamma\md(\hat \lambda_j) \coloneqq \sum_{k > r} \frac{\sigma^2}{\hat \lambda_j - \hat \lambda_k}.
\end{align*}
Then with probability at least $1 -  O(n^{-10})$ it holds that
\begin{align*}
    | \hat \lambda_j - \lambda_j - \gamma\md(\hat \lambda_j) | &\lesssim \sigma\big( \sqrt{r}  + \sqrt{\log(n)} \big) \coloneqq \delta\md.
\end{align*}
Furthermore, with this same probability it holds that
\begin{align*}
    |\gamma\md(\hat \lambda_j) - \hat \gamma\md(\hat \lambda_j) | &\lesssim \frac{\sigma^2 r }{\lambda_j}.
\end{align*} 
\end{lemma}
\begin{proof}
The result can be obtained by repeating the argument in the proof of Theorem 7 of \citet{li_minimax_2025} under the slightly stronger assumption $\lambda_{\min} \geq C_0 \sigma \sqrt{rn}$.  The only part that does not directly follow is the final inequality. Let $\lambda_k^{\perp}$ denote the eigenvalues of $\uperp\t \bN \uperp$.  Then 
\begin{align*}
\gamma\md(\hat \lambda_j) &= \sigma^2\tr \bigg( \frac{1}{n} \big(\hat \lambda_j \bm{I}_{n-r} - \uperp\t \bN \uperp \big)\inv \bigg) = \sum_{k=1}^{n-r} \frac{\sigma^2}{\hat \lambda_j - \lambda_k^{\perp}}.
\end{align*}
Therefore, it suffices to show that
\begin{align*}
    \bigg| \sum_{k=1}^{n-r} \frac{\sigma^2}{\hat \lambda_j - \lambda_k^{\perp}} - \sum_{k > r} \frac{\sigma^2}{\hat \lambda_j - \hat \lambda_k} \bigg| \lesssim \frac{r \sigma^2 }{\lambda_j}.
\end{align*}
We will demonstrate this result assuming $\lambda_1$ through $\lambda_r$ are all positive; this is with no significant loss of generality as the argument is straightforward to adapt and significantly more cumbersome if this is not the case.

By the Poincare Separation Theorem (Corollary 4.3.37 of \citet{horn_matrix_2012}), it holds that
\begin{align*}
   \hat \lambda_{k+r} \leq  \lambda_k^{\perp}  \leq \hat \lambda_k.
\end{align*}
Therefore  $\hat \lambda_j - \lambda_k^{\perp} \geq \hat \lambda_j - \hat \lambda_k$ for $k \geq r+ 1$, where we have used the fact that $\hat \lambda_j$ is positive.  Therefore,
\begin{align*}
    \sum_{k=1}^{n-r} \frac{\sigma^2}{\hat \lambda_j - \lambda_{k}^{\perp}} \leq \sum_{k=1}^{r} \frac{\sigma^2}{\hat \lambda_j - \lambda_k^{\perp}} + \sum_{k= r+1}^{n-r} \frac{\sigma^2}{\hat \lambda_j - \hat \lambda_k } &\leq\sum_{k=1}^{r} \frac{\sigma^2}{\hat \lambda_j - \lambda_k^{\perp}} + \sum_{k=r+1}^{n} \frac{\sigma^2}{\hat \lambda_j - \hat \lambda_k}.
\end{align*}
Similarly,
\begin{align*}
    \sum_{k=1}^{n-r} \frac{\sigma^2}{\hat \lambda_j - \lambda_{k}^{\perp}} &\geq \sum_{k=1}^{n-r}\frac{\sigma^2}{\hat \lambda_j - \hat\lambda_{k+r}} = \sum_{k=r+1}^{n}\frac{\sigma^2}{\hat \lambda_j - \hat \lambda_{k}}.
\end{align*}
Therefore,
\begin{align*}
   \bigg|  \sum_{k=1}^{n-r} \frac{\sigma^2}{\hat \lambda_j - \lambda_{k}^{\perp}}  - \sum_{k=r+1}^{n} \frac{\sigma^2}{\hat \lambda_j - \hat \lambda_k} \bigg| &\leq \sum_{k=1}^{r} \frac{\sigma^2}{\hat \lambda_j - \lambda_k^{\perp}}.
\end{align*}
The result follows by noting that $\hat \lambda_j - \lambda_k^{\perp} \gtrsim \lambda_j$ by \cref{fact1_MD}.  
\end{proof}
\noindent

Our analysis will also rely on the following technical result which is used as an intermediate step in \citet{li_minimax_2025}, which we state as a lemma for ease of reference.
 Define the matrices
\begin{align*}
    \bm{G}(\lambda) &\coloneqq \sigma^2 \tr\bigg( \big( \lambda \bm{I}_{n-r} - \uperp\t \bN \uperp \big)\inv \bigg) \bm{I}_r, \numberthis \label{glambda} \\
    \bm{\tilde G}(\lambda) &\coloneqq \U\t \bN \uperp\bigg( \lambda \bm{I}_{n-r} - \uperp\t \bN \uperp \bigg)\inv \uperp\t \bN \U, \numberthis \label{gtildelambda}
\end{align*}
where both quantities are understood as functions of $\lambda$ conditional on $\uperp\t \bN \uperp$.

\begin{lemma} \label{lem:gammaapprox}
In the setting of \cref{thm:distributionaltheory_MD}, with probability at least $1 - O(n^{-10})$  it holds that
    \begin{align*}
        \sup_{\lambda: |\lambda| \in [2 |\lambda_j|/3, 4|\lambda_j|/3]} \| \bm{G}(\lambda) - \bm{\tilde G}(\lambda) \| &\lesssim \frac{\sigma^2}{\lambda_{\min}}  \sqrt{rn\log(n)}.
    \end{align*}
\end{lemma}

\begin{proof}[Proof of \cref{lem:gammaapprox}]
This result follows immediately from the proof of Lemma 1 of \citet{li_minimax_2025} with the assumption $r\lesssim n/\log^2(n)$.
\end{proof}

\subsection{Isolating the Leading-Order Term}
Invoking \cref{thm:mainexpansion}, we have that
\begin{align*}
    \a\t\uhatj - \a\t \uj \uj\t \uhatj - \sum_{k\neq j} \frac{\uj\t \bN \uk}{\lambda_j - \lambda_k} \uk\t \a 
    &= \sum_{k\neq j} \frac{(\uhatj- \uj)\t \bN \uk}{\lambda_j - \lambda_k} \a\t \uk + \sum_{k\neq j} \frac{\lambda_j - \hat \lambda_j}{\lambda_j - \lambda_k} \uhatj\t \uk\uk\t\a \\
    &= \sum_{k\neq j, k\leq r } \frac{(\uhatj - \uj)\t \bN \uk - \gamma\md(\hat \lambda_j) \uhatj\t \uk}{\lambda_j - \lambda_k} \a\t \uk + \frac{(\uhatj - \uj)\t \bN \uperp \uperp\t \a}{\lambda_j} \\&\quad + \sum_{\substack{k\neq j\\k\leq r} } \frac{\lambda_j - \hat \lambda_j + \gamma\md(\hat \lambda_j)}{\lambda_j - \lambda_k} \uhatj\t \uk\uk\t\a+ \frac{\lambda_j - \hat \lambda_j}{\lambda_j} \uhatj\t \uperp \uperp\t \a \\
    &\coloneqq \rmd_1 + \rmd_2 + \rmd_3 + \rmd_4,
\end{align*}
 where we have subtracted $\gamma\md(\hat \lambda_j) \uhatj\t\uk$ from the term containing $(\uhatj- \uj)\t \bN \uk$ and added it to the term containing $\lambda_j - \hat \lambda_j$ for $k \leq r$. We will now demonstrate that each $\rmd_i$ is $o(s_{\a,j}\md)$. 

    \begin{itemize}
    \item \textbf{Bounding $\rmd_1$}.

 Write
\begin{align*}
\rmd_1 &= \sum_{\substack{k\neq j\\k\leq r}} \frac{(\uhatj - \uj)\t \bN \uk - \gamma\md(\hat \lambda_j) \uhatj\t \uk}{\lambda_j - \lambda_k} \uk\t \a \\
&=\sum_{\substack{k\neq j\\k\leq r}} \frac{\uhatj \t \uperp\uperp\t \bN \uk- \gamma\md(\hat \lambda_j) \uhatj\t \uk}{\lambda_j - \lambda_k} \uk\t \a + \sum_{\substack{k\neq j\\k\leq r}} \frac{(\uhatj - \uj)\t \U \U\t \bN \uk}{\lambda_j - \lambda_k} \uk\t \a,
\end{align*}
where we have implicitly used the fact that $\uj\t\uperp = 0$. 
 By \cref{lem:uhatjuperpidentity} and \cref{fact4_MD} it holds that
\begin{align*}
    \uhatj\t \uperp &= \uhatj\t \uk \uk\t \bN \uperp \big( \hat \lambda_j \bm{I}_{n-r} - \uperp\t \bN \uperp \big)\inv + \uhatj\t \U^{-k} (\U^{-k})\t \big( \hat \lambda_j \bm{I}_{n-r} - \uperp\t \bN \uperp \big)\inv.
\end{align*}
Plugging this in yields 
\begin{align*}
  \rmd_1  &= \sum_{\substack{k\neq j\\k\leq r}} \bigg(   \uk\t \bN \uperp \big( \hat \lambda_j \bm{I}_{n-r} - \uperp\t \bN \uperp \big)\inv \uperp\t \bN \uk- \gamma\md(\hat \lambda_j)  \bigg) \frac{\uhatj\t \uk}{\lambda_j - \lambda_k} \uk\t \a  \\
    &\quad +   \sum_{\substack{k\neq j\\k\leq r}} \frac{ \uhatj\t \U^{-k} (\U^{-k})\t \bN \uperp \big( \hat \lambda_j \bm{I}_{n-r} - \uperp\t \bN \uperp \big)\inv \uperp\t \bN \uk}{\lambda_j - \lambda_k} \uk\t \a \\
    &\quad + \sum_{\substack{k\neq j\\k\leq r}} \frac{(\uhatj - \uj)\t \U \U\t \bN \uk}{\lambda_j - \lambda_k} \uk\t \a \\
    &=: \alpha_1 + \alpha_2 + \alpha_3.
\end{align*}
Observe that  
$ \uk\t \bN \uperp \bigg( \hat \lambda_j \bm{I}_{n-r} - \uperp\t \bN \uperp \bigg)\inv \uperp\t \bN \uk - \gamma(\hat \lambda_j) $ is a submatrix of the matrix $\bm{G}(\hat \lambda_j) - \bm{\tilde G}(\hat \lambda_j)$ as defined in \eqref{glambda} and \eqref{gtildelambda} respectively. 
Therefore, by \cref{fact4_MD}  it holds that $|\hat \lambda_j| \in[2|\lambda_j|/3, 4|\lambda_j|/3] $ and hence
\begin{align*}
\bigg|    \uk\t \bN \uperp \bigg( \hat \lambda_j \bm{I}_{n-r} - \uperp\t \bN \uperp \bigg)\inv \uperp\t \bN \uk - \gamma(\hat \lambda_j) \bigg| 
&\leq \sup_{\lambda: |\lambda| \in [2|\lambda_j|/3, 4|\lambda_j|/3]} \| \bm{G}(\lambda) - \bm{\tilde G}(\lambda) \| \\
&\lesssim 
\frac{\sigma^2}{\lambda_{\min}} \sqrt{rn\log(n)},
    \end{align*}
    which holds with probability $1 - O(n^{-10})$ by \cref{lem:gammaapprox}.   As a result,
\begin{align*}
|\alpha_1 | 
    &\lesssim \frac{\sigma^2 r\sqrt{n\log(n)} }{\lambda_{\min}}   \sqrt{ \sum_{k\neq j} \frac{(\uk\t \a)^2}{(\lambda_j - \lambda_k)^2}}, \numberthis \label{alpha1boundmd}
\end{align*}
where we have applied Cauchy-Schwarz. 

Next, to bound $\alpha_2$ we proceed via a similar argument.  Note that
\begin{align*}
|\alpha_2 |
    &\leq  \max_{k\neq j} \big\| (\U^{-k})\t \bN \uperp\big( \hat \lambda_j \bm{I}_{n-r} - \uperp\t \bN \uperp \big)\inv \uperp\t \bN \uk \big\| \sum_{k\neq j} \frac{|\uk\t \a|}{|\lambda_j - \lambda_k|}.
\end{align*}
The first term above is again a submatrix of $\bm{\tilde G}(\hat \lambda_j) - \bm{G}(\hat \lambda_j)$
and hence by \cref{lem:gammaapprox} and \cref{fact4_MD}, 
    \begin{align*}
        \big\| (\U^{-k})\t \bN \uperp \big( \hat \lambda_j \bm{I}_{n-r} - \uperp\t \bN \uperp \big)\inv \uperp\t \bN \uk \big \| 
        &\leq \sup_{\lambda: |\lambda| \in [2|\lambda_j|/3, 4|\lambda_j|/3]} \| \bm{G}(\lambda) - \bm{\tilde G}(\lambda) \| \\
        &\lesssim \frac{\sigma^2}{\lambda_{\min}} \sqrt{rn\log(n)},
    \end{align*}
Therefore, it holds that
\begin{align*}
    |\alpha_2| &\lesssim  \frac{\sigma^2}{\lambda_{\min}} \sqrt{rn\log(n)}   \sum_{k\neq j} \frac{|\uk\t \a|}{|\lambda_j - \lambda_k|} \lesssim \frac{\sigma^2 r\sqrt{n\log(n)} }{\lambda_{\min}}   \sqrt{ \sum_{k\neq j} \frac{(\uk\t \a)^2}{(\lambda_j - \lambda_k)^2}}  \numberthis \label{r2mdbound1}
\end{align*}
where the final line follows from Cauchy-Schwarz.

To bound $\alpha_3$, we have that
by Cauchy-Schwarz,
    \begin{align*}
    |\alpha_3| 
        &\leq \big\| \uhatj - \uj \big\| \bigg\| \sum_{\substack{k\neq j\\k\leq r}} \frac{  \U\t \bN \uk}{\lambda_j - \lambda_k} \uk\t \a \bigg\|.
    \end{align*}
Observe that by \cref{lem:bias_MD} we have that
\begin{align*}
    \| \uhatj - \uj \|^2 &= 2 - 2 \langle \uhatj, \uj \rangle^2 = 2 \big( 1 - (\uhatj\t \uj)^2 \big) \lesssim \frac{\sigma^2 n}{\lambda_j^2} + \frac{\sigma^2 r\log(n)}{\Delta_j^2}.
\end{align*}
Consequently,
\begin{align*}
    \|\uhatj -\uj \| &\lesssim \frac{\sigma \sqrt{n}}{\lambda_j } + \frac{\sigma \sqrt{r\log(n)}}{\Delta_j}. \numberthis \label{uhatjminusuj}
\end{align*}
Next, the term $\sum_{\substack{k\neq j\\k\leq r}} \frac{\U\t \bN \uk}{\lambda_j - \lambda_k} \uk\t \a$ is a sum of mean-zero independent random matrices.  The following lemma bounds this term by appealing to a standard concentration argument.  
\begin{lemma} \label{lem:simpleconcentration}
    Instate the conditions of \cref{thm:distributionaltheory_MD}.  Then with probability at least $1 - O(n^{-10})$ it holds that
    \begin{align*}
    \bigg\|  \sum_{\substack{k\neq j\\k\leq r}} \frac{ \U\t \bN \uk}{\lambda_j - \lambda_k} \uk\t \a \bigg\| \lesssim \sigma r \sqrt{\log(n)}\sqrt{ \sum_{\substack{k\neq j\\k\leq r}} \frac{(\uk\t \a)^2}{(\lambda_j - \lambda_k)^2} }
\end{align*}
 \end{lemma}
\begin{proof}
    See \cref{sec:simpleconcentration_proof}.  
\end{proof}
Therefore, with probability at least $1 - O(n^{-10})$,
\begin{align*}
  |\alpha_3|  &\lesssim \bigg( \frac{\sigma \sqrt{n}}{\lambda_j} + \frac{\sigma \sqrt{r\log(n)}}{\Delta_j} \bigg) \sigma r \sqrt{\log(n)}  \sqrt{\sum_{\substack{k\neq j\\k\leq r}} \frac{(\uk\t\a)^2}{(\lambda_j - \lambda_k)^2}}.\numberthis \label{r2mdbound2}
 \end{align*}
Combining \eqref{alpha1boundmd} , \eqref{r2mdbound1} and \eqref{r2mdbound2}  results in the bound
\begin{align*}
    |\rmd_1| \lesssim\bigg( \frac{\sigma r \sqrt{\log(n)}}{\lambda_{\min}} + \frac{\sigma r^{3/2} \log(n)}{\Delta_j} \bigg)  \sqrt{\sum_{\substack{k\neq j\\k\leq r}} \frac{\sigma^2(\uk\t\a)^2}{(\lambda_j - \lambda_k)^2}},\numberthis \label{eq:rmd1}
\end{align*}
which holds with probability at least $1 - O(n^{-10})$.  
 
\item \textbf{Bounding $\rmd_2$}.  We decompose via
\begin{align*}
    \frac{(\uhatj - \uj)\t \bN \uperp \uperp\t \a}{\lambda_j} &= \frac{(\uhatj - \uj) \U \U\t \bN \uperp \uperp\t\a}{\lambda_j} + \frac{(\uhatj - \uj)\t \uperp \uperp\t \bN \uperp \uperp\t \a}{\lambda_j} \\
    &= \frac{(\uhatj - \uj) \U \U\t \bN \uperp \uperp\t\a}{\lambda_j} + \frac{\uhatj \t \uperp \uperp\t \bN \uperp \uperp\t \a}{\lambda_j} \\
    &=: \beta_1 + \beta_2.
\end{align*}
We bound each in turn.  To bound $\beta_1$, by a similar argument as in \eqref{uhatjminusuj}, with probability at least $1 - O(n^{-10})$, we have that
    \begin{align*}
    |\beta_1| &= \bigg|        \frac{(\uhatj - \uj)\t \U \U\t \bN \uperp \uperp\t \a }{\lambda_j} \bigg| \\
    &\lesssim \| \uhatj - \uj \| \frac{\|\U\t \bN \uperp \uperp\t \a \|}{\lambda_j} \\&\lesssim\bigg( \frac{\sigma \sqrt{n}}{\lambda_j} + \frac{\sigma \sqrt{r\log(n)}}{\Delta_j} \bigg)\frac{\|\U\t \bN \uperp \uperp\t \a \|}{\lambda_j} \\
&\lesssim \bigg( \frac{\sigma \sqrt{n}}{\lambda_j} + \frac{\sigma \sqrt{r\log(n)}}{\Delta_j} \bigg) \sigma \sqrt{r\log(n)}  \frac{\| \uperp\t \a \|}{\lambda_j}, \numberthis \label{r2mdbound3}
    \end{align*}
    where the final bound holds from standard Gaussian concentration inequalities and the fact that $\frac{1}{\|\uperp\t\a\|} \U\t \bN \uperp \uperp\t\a$ is equal in distribution to a standard $r\times 1$ Gaussian random vector (note that if $\uperp\t\a = 0$, the bound is trivial).   

    To handle $\beta_2$, by \cref{lem:uhatjuperpidentity} and \cref{fact4_MD}  have that
    \begin{align*}
    |\beta_2| &=      \bigg|    \frac{\uhatj\t \uperp \uperp\t \bN \uperp \uperp\t \a}{\lambda_j } \bigg| \\
    &= \bigg| \frac{ \uhatj\t \U \U\t \bN \uperp \big( \hat \lambda_j - \uperp\t \bN \uperp \big)\inv \uperp\t \bN \uperp \uperp\t \a }{\lambda_j} \bigg| \\ 
     &\leq  \frac{ \big\|\U\t \bN \uperp \big( \hat \lambda_j - \uperp\t \bN \uperp \big)\inv \uperp\t \bN \uperp \uperp\t \a  \big\|}{\lambda_j}.
    \end{align*}
 The following result bounds the numerator.  
    \begin{lemma} \label{lem:residual2_MD}
Instate the conditions of \cref{thm:distributionaltheory_MD}.  With probability at least $1 - O(n^{-10})$ it holds that        \begin{align*}
            \big\| \U\t \bN \uperp \big( \hat \lambda_j - \uperp\t \bN \uperp \big)\inv \uperp\t \bN \uperp \uperp\t \a \big\| &\lesssim \frac{\sigma^2 \sqrt{rn\log(n)}}{\lambda_{\min}} \|\uperp\t \a \|.
        \end{align*}
    \end{lemma}
\begin{proof}
    See \cref{sec:residual2proof}.
\end{proof}
\noindent As a result of this lemma it holds with probability at least $1 - O(n^{-10})$ that
    \begin{align*}
|\beta_2| &\lesssim \frac{\sigma^2 \sqrt{rn\log(n)}}{\lambda_{\min}} \frac{\|\uperp\t \a \|}{\lambda_j}. \numberthis \label{r2mdbound4}
    \end{align*}
Therefore, combining the bounds in \eqref{r2mdbound3} and \eqref{r2mdbound4} we obtain that
\begin{align*}
   | \rmd_2 |&\lesssim  \bigg( \frac{\sigma \sqrt{rn\log(n)}}{\lambda_{\min}} + \frac{\sigma r \log(n)}{\Delta_j} \bigg) \frac{\sigma \|\uperp\t \a\|}{\lambda_j} 
\numberthis \label{rmd1bound}
\end{align*}
which holds with probability at least $1 - O(n^{-10})$.
\item 
 \textbf{Bounding $\rmd_3$}. When it comes to $\rmd_3 = \sum_{\substack{k\neq j\\k\leq r} } \frac{\lambda_j - \hat \lambda_j + \gamma\md(\hat \lambda_j)}{\lambda_j - \lambda_k} \a\t \uj$, we may appeal directly to the eigenvalue concentration results in \cref{lem:eigenvalues_MD} to observe that with  probability at least $1 - O(n^{-10})$,
\begin{align*}
    |\rmd_3| &= \bigg|\sum_{\substack{k\neq j\\k\leq r} } \frac{\lambda_j - \hat \lambda_j + \gamma\md(\hat \lambda_j)}{\lambda_j - \lambda_k}  \uhatj\t \uk \a\t \uj\bigg| \\
    &\leq |\lambda_j - \hat \lambda_j  + \gamma\md(\hat \lambda_j)|  \bigg| \sum_{\substack{k\neq j\\k\leq r}} \uhatj\t \uk \frac{\a\t \uj}{\lambda_j - \lambda_k} \bigg| \\ 
    &\leq \delta\md \bigg| \sum_{\substack{k\neq j\\k\leq r}} \uhatj\t \uk \frac{\a\t \uj}{\lambda_j - \lambda_k} \bigg|. 
\end{align*}
Therefore, it suffices to study $\uhatj\t \uk$ for $k \leq r$ and $k\neq j$. 
The following lemma studies this error.
\begin{lemma} \label{lem:innerproduct_MD}
Instate the conditions of \cref{thm:distributionaltheory_MD}. 
  Then simultaneously for all $k \neq j$, with $k,j\leq r$, with probability at least $1 - O(n^{-9})$ it holds that
    \begin{align*}
        |\uhatj\t \uk | &\lesssim \frac{\delta\md}{|\lambda_j -\lambda_k|}.
    \end{align*}
\end{lemma}
\begin{proof}
    See \cref{sec:innerproduct_MD_proof}.   
\end{proof}
\noindent
Therefore, 
\begin{align*}
    |\rmd_3| &\lesssim   \frac{(\delta\md)^2}{\Delta_j}  \sum_{\substack{k\neq j\\k\leq r}} \frac{|\a\t \uj|}{|\lambda_j - \lambda_k|} 
    \lesssim \frac{\sigma^2 r^{3/2}\log(n)}{\Delta_j}  \sqrt{\sum_{\substack{k\neq j\\k\leq r}} \frac{(\a\t \uk)^2}{(\lambda_j - \lambda_k)^2} }
    \numberthis \label{rmd2bound}
\end{align*}
with probability at least $1 - O(n^{-9})$.
\item 
   \textbf{Bounding $\rmd_4$}.  By \cref{lem:uhatjuperpidentity} via \cref{fact4_MD} it holds that
   \begin{align*}
       \rmd_4 &= \frac{\lambda_j - \hat \lambda_j}{\lambda_j} \uhatj\t \uperp \uperp\t \a = \frac{\lambda_j - \hat \lambda_j}{\lambda_j} \uhatj\t\U \U\t \bN \uperp\big( \hat \lambda_j \bm{I}_{n-r} - \uperp\t\bN \uperp \big)\inv \uperp\t \a.
   \end{align*}
By Weyl's inequality it holds that $|\lambda_j - \hat \lambda_j| \lesssim \sigma \sqrt{n}$ by \cref{fact1_MD}.  Consequently, 
\begin{align*}
      |\rmd_4| &\lesssim  \frac{\sigma \sqrt{n}}{\lambda_{\min}} \bigg\| \U\t \bN \uperp\big( \hat \lambda_j \bm{I}_{n-r} - \uperp\t\bN \uperp \big)\inv \uperp\t \a \bigg\|.
\end{align*}
By the exact same analysis as in \cref{lem:residual2_MD},  by replacing the appearance of $\uperp\t \bN \uperp\uperp\t\a$ with $\uperp \uperp\t\a$, it is straightforward to show that
with probability at least $1 - O(n^{-10})$ it holds that 
\begin{align*}
    \bigg\| \U\t \bN \uperp\big( \hat \lambda_j \bm{I}_{n-r} - \uperp\t\bN \uperp \big)\inv \uperp\t \a \bigg\| &\lesssim \frac{\sigma \sqrt{r\log(n)}}{\lambda_j} \|\uperp\t \a \|.
\end{align*}
Therefore, with probability at least $1 - O(n^{-10})$,
\begin{align*}
    |\rmd_4| \lesssim \frac{\sigma \sqrt{rn\log(n)}}{\lambda_{\min}} \frac{\|\uperp\t \a \|}{\lambda_j}.  \numberthis \label{eq:rmd4}
\end{align*}
\end{itemize}
By  \eqref{eq:rmd1}, \eqref{rmd1bound}, \eqref{rmd2bound}, and \eqref{eq:rmd4} we have that with probability at least $1 - O(n^{-9})$,
\begin{align*}
    |\rmd_1 + \rmd_2 + \rmd_3 + \rmd_4| 
    &\lesssim \bigg( \frac{\sigma r \sqrt{n \log(n)}}{\lambda_{\min}} + \frac{\sigma r^{3/2} \log(n)}{\Delta_j} \bigg)s_{\a,j}\md,
\end{align*}
where we recall the definition of $s\md_{\a,j}$ in \cref{thm:distributionaltheory_MD}.

\subsection{Completing the Proofs of \cref{thm:distributionaltheory_MD,thm:civalidity_MD}}

We have shown thus far that with probability at least $1 - O(n^{-9}),$
\begin{align*}
   \frac{1}{s_{\a,j}\md}\bigg( \a\t\uhatj - \a\t \uj \uj\t \uhatj \bigg) =\frac{1}{s_{\a,j}\md} \sum_{k\neq j} \frac{\uj\t \bN \uk}{\lambda_j - \lambda_k} \uk\t \a +  {\sf ErrMD}
\end{align*}
In addition, it is straightforward to observe that the leading-order term $\sum_{k\neq j} \frac{\uj\t \bN \uk}{\lambda_j - \lambda_k} \uk\t\a$ is Gaussian with variance $(s_{\a,j}\md)^2$, as $\uj\t \bN \uk$ is independent from $\uj\t \bN \bm{u}_l$ for $l \neq k$ by rotational invariance. As a consequence, letting $\Phi(\cdot)$ denote the CDF of the Gaussian distribution, it holds that
\begin{align*}
    \bigg| \p\bigg\{ \frac{1}{s_{\a,j}\md}& \big(\a\t \uhatj - \a\t \uj \uj\t \uhatj \big) \leq z \bigg\} - \Phi(z) \bigg| \\
    &= \bigg| \p\bigg\{ \frac{1}{s_{\a,j}\md} \bigg( \sum_{k\neq j} \frac{\uj\t \bN \uk}{\lambda_j - \lambda_k} \uk\t\a + \rmd_1 + \rmd_2 + \rmd_3 + \rmd_4 \bigg) \leq z \bigg\} - \Phi(z)  \bigg| \\
    &\leq  \bigg| \p\bigg\{ \frac{1}{s_{\a,j}\md} \bigg( \sum_{k\neq j} \frac{\uj\t \bN \uk}{\lambda_j - \lambda_k} \uk\t\a  \bigg) \leq z  \pm {\sf ErrMD} \bigg\} - \Phi(z)  \bigg| + O(n^{-9}) \\
    &\leq \bigg| \Phi(z \pm {\sf ErrMD}) - \Phi(z) \bigg| + O(n^{-9}) \\
    &\lesssim {\sf ErrMD} + n^{-9},
\end{align*}
where the final line follows from the Lipschitz property of $\Phi$, where $\Phi(z \pm x)$ is interpreted as applying the result to both $\Phi(z + x)$ and $\Phi(z-x)$ separately.  This completes the proof of \eqref{md_firstresult}.

We now prove \eqref{md_secondresult}. 
 By \cref{lem:bias_MD} with probability at least $1 - O(n^{-10})$ it holds that
\begin{align*}
    | \sqrt{1 + b_j\md} \uhatj\t \uj - 1 | &\lesssim \frac{\sigma^2 r\log(n)}{\Delta_j^2} + \frac{\sigma^2 \sqrt{n\log(n)}}{\lambda_j^2} = {\sf ErrBiasMD},
\end{align*}
where again we have absorbed implicit constants. As a result, by a similar analysis to before we can demonstrate that 
\begin{align*}
    \bigg| \p\bigg\{ &\frac{1}{s_{\a,j}\md}\bigg( \a\t \uhatj \sqrt{1 + b_j\md} - \a\t \uj \bigg) \leq z \bigg\} - \Phi(z) \bigg| \\
 &\leq \bigg| \p\bigg\{ \frac{1}{s_{\a,j}\md}\bigg( \a\t \uhatj  - \a\t \uj \uj\t \uhatj     \bigg)  \leq \frac{z}{\sqrt{1 +  b_j\md  }} \pm \frac{\a\t \uj {\sf ErrBiasMD}}{\sqrt{1 +   b_j\md }} \bigg\}- \Phi(z) \bigg| + O(n^{-9})\\ 
&\lesssim {\sf ErrMD} + \frac{|\a\t \uj| {\sf ErrBiasMD}}{s_{\a,j}\md}+ n^{-9},
\end{align*}
where we have applied the first part of \cref{thm:distributionaltheory_MD} in \eqref{md_firstresult}. The condition \eqref{atujcondition} implies that the right hand side above is $o(1)$, which yields \eqref{md_secondresult}.

In order to prove \cref{thm:civalidity_MD} we need to demonstrate that the estimated variance
\begin{align*}
    (\hat s_{\a,j}\md)^2 \coloneqq \hat \sigma^2 \sum_{k\leq r,k\neq j } \frac{(\bm{\hat u}_k\t \a)^2 (1 + \hat{b_k\md})}{(\check\lambda_k - \check \lambda_j)^2} + 2 \frac{\hat \sigma^2\|\uhat_{\perp}\t \a \|^2}{\check \lambda_j^2}
\end{align*}
yields a strong estimate of $(s_{\a,j}\md)^2$.   This analysis requires two steps. 
First we require the following result concerning the estimated noise variance $\hat \sigma^2$.
\begin{lemma} \label{lem:sigmahatoversigma}
    Under the conditions of \cref{thm:civalidity_MD}, with probability at least $1- O(n^{-10})$ it holds that
    \begin{align*}
    \bigg| \frac{\hat \sigma}{\sigma} - 1 \bigg| &\lesssim \kappa \sqrt{\frac{r}{n}}; \qquad
    \bigg| \frac{\hat \sigma^2}{\sigma^2} - 1 \bigg| \lesssim \kappa \sqrt{\frac{r}{n}}.
    \end{align*}
\end{lemma}
\begin{proof}
    See \cref{sec:sigmahatoversigmaproof}.
\end{proof}

Next we study the approximated bias $\hat{b_k\md}$. 

\begin{lemma} \label{lem:biashatmd} Let $k \in [r]$ be fixed.     Under the conditions of \cref{thm:civalidity_MD}, with probability at least $1 - O(n^{-10})$ it holds that
    \begin{align*}
\bigg|        \sqrt{1 + \hat{b_k\md}} - \sqrt{1 + b_k\md} \bigg| \lesssim \kappa \frac{\sigma^2 \sqrt{rn}}{\lambda_{\min}^2}.
    \end{align*}
\end{lemma}
\begin{proof}
    See \cref{sec:biashatmdproof}.
\end{proof}
As a consequence,
\begin{align*}
    \bigg| \frac{ \a\t \uhatj \sqrt{1 + \hat{b_j\md}} - \a\t \uj}{s\md_{\a,j}} -\frac{ \a\t \uhatj \sqrt{1 +  b_j\md} - \a\t \uj}{s\md_{\a,j}} \bigg| &\lesssim \frac{|\a\t \uhatj|}{s\md_{\a,j}} \kappa \frac{\sigma^2 \sqrt{rn}}{\lambda_{\min}^2}.
\end{align*}
Furthermore, we note that the analysis leading up to the proof of \cref{thm:distributionaltheory_MD} implies that
\begin{align*}
    \big| \a\t \uhatj \big| \lesssim \big| \a\t \uj \big| + \sqrt{\log(n)} s\md_{\a,j} + {\sf ErrMD}\times  s\md_{\a,j} \lesssim |\a\t\uj| + \sqrt{\log(n)} s\md_{\a,j} 
\end{align*}
with probability at least $1 - O(n^{-8})$, provided that ${\sf ErrMD} \lesssim \sqrt{\log(n)}$, which holds by the assumptions in \cref{thm:civalidity_MD}.  Therefore, with this same probability,
\begin{align*}
     \bigg| \frac{ \a\t \uhatj \sqrt{1 + \hat{b_j\md}} - \a\t \uj}{s\md_{\a,j}} -\frac{ \a\t \uhatj \sqrt{1 +  b_j\md} - \a\t \uj}{s\md_{\a,j}} \bigg| 
     &\lesssim \frac{|\a\t \uj|}{s\md_{\a,j}} \kappa \frac{\sigma^2 \sqrt{rn}}{\lambda_{\min}^2} + \sqrt{\log(n)}\kappa\frac{\sigma^2 \sqrt{rn}}{\lambda_{\min}^2}. \numberthis \label{toplugin54}
\end{align*}
Next we study the approximated variance $(\hat{s\md_{\a,j}})^2$, which is accomplished through the following lemma.
\begin{lemma}\label{lem:sigma_md_approx}
    Instate the conditions in \cref{thm:civalidity_MD}.      Then with probability at least $1- O(n^{-8})$ it holds that
    \begin{align*}
\bigg|        \frac{\hat{s_{\a,j}\md}}{s_{\a,j}\md}  - 1 \bigg| &\ll \frac{1}{\sqrt{\log(n)}}.
    \end{align*}
\end{lemma}

\begin{proof}
    See \cref{sec:sigma_approx_MD_proof}.
\end{proof}

Let the error in \cref{lem:sigma_md_approx} be denoted as $ {\sf ErrCIMD}$.   As a consequence of this lemma, with probability at least $1 - O(n^{-8})$,
\begin{align*}
     \bigg| &\frac{ \a\t \uhatj \sqrt{1 + \hat{b_j\md}} - \a\t \uj}{\hat{s\md_{\a,j}}} -\frac{ \a\t \uhatj \sqrt{1 +  b_j\md} - \a\t \uj}{s\md_{\a,j}} \bigg| \\
     &\lesssim  \bigg| \frac{ \a\t \uhatj \sqrt{1 + \hat{b_j\md}} - \a\t \uj}{s\md_{\a,j}} -\frac{ \a\t \uhatj \sqrt{1 +  b_j\md} - \a\t \uj}{s\md_{\a,j}} \bigg| + \bigg| \frac{ \a\t \uhatj \sqrt{1 + \hat{b_j\md}} - \a\t \uj}{s\md_{\a,j}}\bigg(1 - \frac{s\md_{\a,j}}{\hat{s\md_{\a,j}}} \bigg) \bigg|  \\
     &\lesssim \frac{|\a\t\uj|}{s\md_{\a,j}} \kappa \frac{\sigma^2 \sqrt{rn}}{\lambda_{\min}^2} + \sqrt{\log(n)} \frac{\kappa \sigma^2 \sqrt{rn}}{\lambda_{\min}^2} +\bigg( \frac{|\a\t\uj|}{s\md_{\a,j}} \kappa \frac{\sigma^2 \sqrt{rn}}{\lambda_{\min}^2} + \sqrt{\log(n)} \frac{\kappa \sigma^2 \sqrt{rn}}{\lambda_{\min}^2} \bigg) {\sf ErrCIMD}\\
     &\qquad + {\sf ErrCIMD} \bigg| \frac{\a\t \uhatj \sqrt{1 + b_j\md} - \a\t \uj}{s\md_{\a,j}} \bigg| \\
     &\lesssim \frac{|\a\t\uj|}{s\md_{\a,j}} \kappa \frac{\sigma^2 \sqrt{rn}}{\lambda_{\min}^2} + \sqrt{\log(n)} \frac{\kappa \sigma^2 \sqrt{rn}}{\lambda_{\min}^2} + {\sf ErrCIMD}\sqrt{\log(n)}, \numberthis \label{boundedboy}
\end{align*}
as long as ${\sf ErrCIMD} = o(1)$.  Here we have used the fact that the analysis leading to \cref{thm:distributionaltheory_MD} demonstrates that
\begin{align*}
     \bigg| \frac{\a\t \uhatj \sqrt{1 + b_j\md} - \a\t \uj}{s\md_{\a,j}} \bigg| \lesssim \sqrt{\log(n)}
\end{align*}
with probability at least $1 - O(n^{-8})$, as long as each of the quantities in \cref{thm:distributionaltheory_MD} are $o(1)$, which is immediate from our assumptions.  In particular, the right hand side of \eqref{boundedboy} is $o(1)$ directly.  

We are now prepared to prove \cref{thm:civalidity_MD}. By a similar analysis to the previous arguments, we can show that 
\begin{align*}
\Bigg|  &\p\bigg\{ \frac{\a\t \uhatj \sqrt{1 + \hat{b_j\md}} - \a\t \uj}{\hat{s\md_{\a,j}}} \leq z \bigg\} -\Phi(z) \Bigg| \\
&= \Bigg|  \p\bigg\{ \frac{\a\t \uhatj \sqrt{1 +  b_j\md} - \a\t \uj}{ s\md_{\a,j}}  \leq z + \bigg(  \frac{ \a\t \uhatj \sqrt{1 + \hat{b_j\md}} - \a\t \uj}{\hat{s\md_{\a,j}}} -\frac{ \a\t \uhatj \sqrt{1 +  b_j\md} - \a\t \uj}{s\md_{\a,j}} \bigg) \bigg\} - \Phi(z) \Bigg| \\
&\lesssim {\sf ErrMD} +  \frac{|\a\t \uj| {\sf ErrBiasMD}}{s_{\a,j}\md}+\frac{|\a\t\uj|}{s\md_{\a,j}} \kappa \frac{\sigma^2 \sqrt{rn}}{\lambda_{\min}^2}  + \sqrt{\log(n)} \frac{\kappa \sigma^2 \sqrt{rn}}{\lambda_{\min}^2} + {\sf ErrCIMD}\sqrt{\log(n)}  + n^{-8} \\
&= o(1),
\end{align*}
where in the penultimate line we have applied both \cref{thm:distributionaltheory_MD} and the Lipschitz property of $\Phi(\cdot)$.  The final line is $o(1)$ directly due to our assumptions \eqref{atujcondition_ci} and \eqref{snrconds_inference} together with \cref{lem:sigma_md_approx}.  
The proof is therefore concluded by taking $\pm z = \Phi\inv(1-\alpha/2)$.

\subsection{Proofs of Additional Matrix Denoising Lemmas}

This section contains all of the proofs of the intermediate lemmas required in the prior subsections. 

\subsubsection{Proof of Lemma \ref{lem:simpleconcentration}}
\label{sec:simpleconcentration_proof}

\begin{proof}
    First we will devise a concentration inequality for $\sum_{\substack{k\neq j\\k\leq r}} \frac{\bm{x}\t \U\t \bN \uk}{\lambda_j - \lambda_k} \uk\t \a$ for a  fixed deterministic vector $\bm{x} \in \mathbb{R}^r$, and then complete the proof via $\eps$-net.  
    \\ \ \\ 
    \noindent \textbf{Step 1: Concentration for a fixed vector}. Let $\bm{x} \in \mathbb{R}^r$ be any deterministic unit vector.  Note that
    \begin{align*}
        \sum_{\substack{k\neq j\\k\leq r}} \frac{\bm{x}\t \U\t \bN \uk}{\lambda_j - \lambda_k} \uk\t \a &= \sum_{i,i'} \bN_{ii'} \big( \U \bm{x} \big)_{i} \bigg( \sum_{\substack{k\neq j\\k\leq r}} \frac{\uk\t \a}{\lambda_j - \lambda_k} (\uk)_{i'} \bigg) \\
        &=: \sum_{i\leq i'} \bN_{ii'} \big( \bm{c}_i \bm{\tilde c}_{i'} + \bm{c}_{i'} \bm{\tilde c}_i \big), 
    \end{align*}
    where we set $\bm{c}_i \coloneqq \big( \U \bm{x} \big)_{i}$ and $\bm{\tilde c}_{i'} \coloneqq \bigg( \sum_{\substack{k\neq j\\k\leq r}} \frac{\uk\t \a}{\lambda_j - \lambda_k} (\uk)_{i'} \bigg)$. By Hoeffding's inequality, with probability at least $1 -c  \exp(- c t ^2)$ that
    \begin{align*}
       \bigg| \sum_{i\leq i'} \bN_{ii'} \big( \bm{c}_i \bm{\tilde c}_{i'} + \bm{c}_{i'} \bm{\tilde c}_i \big) \bigg| 
       &\lesssim \sigma t \sqrt{r} \sqrt{ \sum_{\substack{k\neq j\\k\leq r}} \frac{(\uk\t \a)^2}{(\lambda_j - \lambda_k)^2} }.
    \end{align*}
Therefore, for any deterministic unit vector $\bm{x} \in \mathbb{R}^r$ it holds that
\begin{align*}
    \bigg| \sum_{\substack{k\neq j\\k\leq r}} \frac{\bm{x}\t \U\t \bN \uk}{\lambda_j - \lambda_k} \uk\t \a \bigg| \lesssim \sigma t \sqrt{r} \sqrt{ \sum_{\substack{k\neq j\\k\leq r}} \frac{(\uk\t \a)^2}{(\lambda_j - \lambda_k)^2} }.
\end{align*}
with probability at least $1 - c \exp( -c t ^2)$.  
\\ \ \\ \noindent \textbf{Step 2: Union bound}.  By taking an $1/4$-net of the sphere in $r$ dimensions (see \citet{vershynin_high-dimensional_2018} for details on $\eps$-nets), the bound holds uniformly for all vectors $\bm{x}$ of with probability at least $c 9^{r} \exp(-c t^2)$.  
Taking $t \geq C \sqrt{r\log(n)}$ demonstrates that 
\begin{align*}
    \bigg\|  \sum_{\substack{k\neq j\\k\leq r}} \frac{ \U\t \bN \uk}{\lambda_j - \lambda_k} \uk\t \a \bigg\| \lesssim \sigma r \sqrt{\log(n)}\sqrt{ \sum_{\substack{k\neq j\\k\leq r}} \frac{(\uk\t \a)^2}{(\lambda_j - \lambda_k)^2} }
\end{align*}
with probability at least $1 - O(n^{-10})$ as desired.
\end{proof}

\subsubsection{Proof of Lemma \ref{lem:residual2_MD}} \label{sec:residual2proof}
\begin{proof}
We will proceed in steps.  First, note that by \cref{fact3_MD} it holds that
\begin{align*}
  \bigg\|  \U\t \bN &\uperp\big( \hat \lambda_j \bm{I}_{n-r} - \uperp\t \bN \uperp \big)\inv \uperp\t \bN \uperp \uperp\t \a \bigg\| \\&\leq \sup_{|\lambda| \in [2|\lambda_j|/3, 4|\lambda_j|/3}  \bigg\|  \U\t \bN \uperp\big(  \lambda \bm{I}_{n-r} - \uperp\t \bN \uperp \big)\inv \uperp\t \bN \uperp \uperp\t \a \bigg\|.
\end{align*}
We will first deduce a concentration inequality for any fixed value $\lambda$ and then take an $\eps$-net.  
\\ \ \\ \noindent
\textbf{Step 1: Concentration for fixed $\lambda$.} Define the matrix
\begin{align*}
    \bm{H}(\lambda) \coloneqq \U\t \bN \uperp\big( \lambda \bm{I}_{n-r} - \uperp\t \bN \uperp \big)\inv \uperp\t \bN \uperp \uperp\t \a.
\end{align*}
 Let $\bN_{\|\perp} \in \mathbb{R}^{r\times(n-r)}$ denote a random matrix with independent $\mathcal{N}(0,\sigma^2)$ entries, and let $\bN_{\perp\perp} \in \mathbb{R}^{(n-r)\times (n-r)}$ be defined similarly, independently from $\bN_{\|\perp}$, except with diagonal elements having variance $2\sigma^2$.  Then 
 \begin{align*}
     ( \U\t \bN \uperp, \uperp\t \bN \uperp) \overset{{{\sf d}}}{=} ( \bN_{\|\perp}, \bN_{\perp\perp} ).
 \end{align*}
 Consequently, $\bm{H}(\lambda)$ has the same distribution as
\begin{align*}
    \bN_{\| \perp} \big( \lambda \bm{I}_{n-r} - \bN_{\perp\perp} \big)\inv \bN_{\perp\perp} \uperp\t \a.
\end{align*}
 We will use an $\eps$-net argument for the $r$-dimensional sphere. Let $\bm{x}$ denote a deterministic $r$-dimensional unit vector.  Then by Hoeffding's inequality, with probability at least $1 - c \exp(- c t^2),$
\begin{align*}
  \bigg|  \bm{x}\t \bN_{\|\perp}\big( \lambda \bm{I}_{n-r} - \bN_{\perp\perp} \big)\inv \bN_{\perp\perp} \uperp\t \a \bigg| &=  \bigg| \sum_{i=1}^{r}\sum_{i'=1}^{n-r} \bm{x}_{i} \big(\bN_{\|\perp}\big)_{ii'} \bigg( \big( \lambda \bm{I}_{n-r} - \bN_{\perp\perp} \big)\inv \bN_{\perp\perp} \uperp\t \a \bigg)_{i'}\bigg| \\
  &\lesssim \sigma t \| \bm{x} \| \bigg\| \big( \lambda \bm{I}_{n-r} - \bN_{\perp\perp} \big)\inv \bN_{\perp\perp} \uperp\t \a \bigg\| \\
  &\lesssim \sigma t \bigg\| \big( \lambda \bm{I}_{n-r} - \bN_{\perp\perp} \big)\inv \bN_{\perp\perp} \uperp\t \a \bigg\|.
\end{align*}
 By taking a $1/4$-net of the sphere in $r$ dimensions, we obtain that 
 \begin{align*}
     \bigg\| \bN_{\|\perp}\big( \lambda \bm{I}_{n-r} - \bN_{\perp\perp} \big)\inv \bN_{\perp\perp} \uperp\t \a \bigg\| &\lesssim \sigma t  \bigg\| \big( \lambda \bm{I}_{n-r} - \bN_{\perp\perp} \big)\inv \bN_{\perp\perp} \uperp\t \a \bigg\|
 \end{align*}
 with probability at least $1 - c 9^{r} \exp(- c t^2)$.  Take $t = C\sqrt{r\log(n)}$ to yield 
 \begin{align*}
     \bigg\| \bN_{\|\perp}\big( \lambda \bm{I}_{n-r} - \bN_{\perp\perp} \big)\inv \bN_{\perp\perp} \uperp\t \a \bigg\| &\lesssim \sigma \sqrt{r\log(n)}  \bigg\| \big( \lambda \bm{I}_{n-r} - \bN_{\perp\perp} \big)\inv \bN_{\perp\perp} \uperp\t \a \bigg\|
 \end{align*}
 with probability at least $1 - O(n^{-11})$. Consequently, for any $\lambda$ satisfying $2|\lambda_j|/3 \leq |\lambda| \leq 4|\lambda_j|/3$ we obtain, with probability at least $1 - O(n^{-11})$,
 \begin{align*}
     \bigg\| \bN_{\|\perp}\big( \lambda \bm{I}_{n-r} - \bN_{\perp\perp} \big)\inv \bN_{\perp\perp} \uperp\t \a \bigg\| &\lesssim \sigma \sqrt{r\log(n)} \frac{\sigma \sqrt{n}}{\lambda_{\min}} \|\uperp\t \a \|,
 \end{align*}
 where we used the fact that by \cref{fact1_MD} it holds that the eigenvalues of $\lambda \bm{I}_{n-r} - \uperp\t \bN \uperp$
 are at least $|\lambda| - \| \bN\| \geq 2|\lambda_j|/3 - \|\bN\| \geq \lambda_{\min}/3$ together with the Gaussian concentration inequality $\|\bN\| \lesssim \sigma \sqrt{n}$.
  \\ \ \\ \noindent
 \textbf{Step 2: Completing the argument}.  Let $\eps = c |\lambda_j|/n$, and let $\mathcal{E}_{\eps}$ denote the $\eps$-net for the region $[-4|\lambda_j|/3, -2|\lambda_j|/3] \cup [2|\lambda_j|/3, 4|\lambda_j|/3]$.  Then $|\mathcal{E}_{\eps}| \lesssim |\lambda_j|/\eps \asymp n$.  Therefore, by the union bound, for any $\lambda \in \mathcal{E}_{\eps}$, with probability at least $1 - O(n^{-10})$, 
 \begin{align*}
     \big\| \bm{H}(\lambda) \big\| \lesssim \sigma \sqrt{r\log(n)} \frac{\sigma \sqrt{n}}{\lambda_{\min}} \|\uperp\t \a \|.
 \end{align*}
 Let $\lambda'$ be any fixed value, and let $\lambda \in \mathcal{E}_{\eps}$  satisfy $|\lambda -  \lambda'| \leq \eps$.  Consequently, conditional on $\|\bN\|\leq \lambda_{\min}/3$,  letting $\eta_i$ denote the eigenvalues of $\bN_{\perp\perp}$, 
 \begin{align*}
     \bigg\|  \bm{H}(\lambda) - \bm{H}(\lambda') \bigg\| &\leq \bigg\| \bN_{\| \perp} \big( \lambda \bm{I}_{n-r} - \bN_{\perp\perp} \big)\inv \bN_{\perp\perp} \uperp\t \a -  \bN_{\| \perp} \big( \lambda' \bm{I}_{n-r} - \bN_{\perp\perp} \big)\inv \bN_{\perp\perp} \uperp\t \a \bigg\| \\
     &\leq \| \bN_{\|\perp}\| \|\bN_{\perp\perp} \| \uperp\t \a \| \max_{i} \bigg| \frac{1}{\lambda - \eta_i} - \frac{1}{\lambda' - \eta_i} \bigg| \\
     &\lesssim  \sigma^2 n \|\uperp\t \a \| \max_i \frac{|\lambda - \lambda' |}{\lambda_j^2} \\
     &\lesssim \frac{\sigma^2}{\lambda_j} \|\uperp\t \a \|,
 \end{align*}
 where the final inequality uses the fact that $\eps =   c|\lambda_j|/n$.  Therefore, define 
 \begin{align*}
 \lambda^* \coloneqq \argsup_{|\lambda| \in [2|\lambda_j|/3, 4|\lambda_j|/3]}  \bigg\|  \U\t \bN \uperp\big(  \lambda \bm{I}_{n-r} - \uperp\t \bN \uperp \big)\inv \uperp\t \bN \uperp \uperp\t \a \bigg\|,
 \end{align*}
 which is attained due to the fact that the set is compact and the fact that the function in question is continuous as a function of $\lambda$ (on the set in question, conditional on $\|\uperp\t \bN \uperp\| \leq \lambda_{\min}/3$).  Let $\lambda^{*'}$ be such that $|\lambda^* - \lambda^{*'}| \leq \eps$.  Then
 \begin{align*}
     \sup_{|\lambda| \in [2|\lambda_j|/3, 4|\lambda_j|/3]}  \bigg\|  \U\t \bN \uperp\big(  \lambda \bm{I}_{n-r} - \uperp\t \bN \uperp \big)\inv \uperp\t \bN \uperp \uperp\t \a \bigg\| 
     &= \big\| \bm{H}(\lambda^*) \big\| \\
     &\leq \sup_{\lambda \in \mathcal{E}_{\eps}} \big\| \bm{H}(\lambda) \big\| + \big\| \bm{H}(\lambda^*) - \bm{H}(\lambda^{*'} )\big\| \\
      &\lesssim \frac{\sigma^2 \sqrt{rn\log(n)}}{\lambda_{\min}} \|\uperp\t \a \|,
 \end{align*}
 which holds with probability at least $1 - O(n^{-10})$, which completes the proof.
\end{proof}
\subsubsection{Proof of Lemma \ref{lem:innerproduct_MD}}
\label{sec:innerproduct_MD_proof}

\begin{proof}
Let $\gamma\md(\hat \lambda_j)$ be as in \cref{lem:eigenvalues_MD}, and let $k\neq j$ be fixed.  By the eigenvalue-eigenvector equation it holds that
\begin{align*}
    \uhatj\t \uk (\hat \lambda_j - \lambda_k - \gamma\md(\hat \lambda_j) ) &= \uhatj\t \bN \uk - \gamma\md(\hat \lambda_j) \uhatj\t \uk.
\end{align*}
Consequently,
\begin{align*}
    \uhatj\t \uk 
    &= \frac{1}{\hat \lambda_j - \lambda_k - \gamma\md(\hat \lambda_j)} \uhatj\t \uperp \uperp\t \bN \uk + \frac{1}{\hat \lambda_j - \lambda_k - \gamma\md(\hat \lambda_j)} \uhatj\t \U \U\t  \bN \uk -  \frac{\gamma\md( \hat \lambda_j)}{\hat \lambda_j - \lambda_k - \gamma\md(\hat \lambda_j)}\uhatj\t \uk \\
    &= \frac{\uhatj\t \uk}{\hat \lambda_j - \lambda_k - \gamma\md(\hat \lambda_j)}   \bigg( \uk\t \bN \uperp \big( \hat \lambda_j \bm{I}_{n-r} - \uperp\t \bN \uperp \big)\inv  \uperp\t \bN \uk -  \gamma\md \hat (\lambda_j) \bigg)\\
    &\quad + \frac{1}{\hat \lambda_j - \lambda_k - \gamma\md(\hat \lambda_j)} \uhatj\t \U^{-k} (\U^{-k})\t \bN \uperp \big( \hat \lambda_j \bm{I}_{n-r} - \uperp\t \bN \uperp \big)\inv  \uperp\t \bN \uk \\&\quad 
     +   \frac{1}{\hat \lambda_j - \lambda_k - \gamma\md(\hat \lambda_j)} \uhatj\t \U \U\t  \bN \uk,
\end{align*}
where in the second line we used the fact that $\bm{I} = \U\U\t + \uperp \uperp\t$, and in the final line we have implicitly invoked \cref{thm:mainexpansion}  on the event in \cref{fact1_MD}, where invertibility of $\hat \lambda_j \bm{I}_{n-r} - \uperp\t \bN \uperp$ is guaranteed by \cref{fact4_MD}.   We have also implicitly assumed that the deminator is nonzero; however, by \cref{lem:eigenvalues_MD} with probability at least $1 - O(n^{-10})$ it holds that
\begin{align*}
    |\hat \lambda_j - \lambda_k - \gamma\md(\hat \lambda_j) | &\geq | \lambda_j - \lambda_k | - | \hat \lambda_j - \lambda_j - \gamma\md(\hat \lambda_j) | \\
    &\geq |\lambda_j - \lambda_k | - \delta\md \\
    &\gtrsim |\lambda_j - \lambda_k|,
\end{align*}
since $\delta\md \lesssim \Delta_j \leq |\lambda_j - \lambda_k|$, with $\delta\md$ defined in \cref{lem:eigenvalues_MD}.     Consequently, with probability at least $1 - O(n^{-9})$ it holds that
\begin{align*}
    |\uhatj\t \uk| &\lesssim \frac{1}{| \lambda_j - \lambda_k |}   \bigg\| \uk\t \bN \uperp \big( \hat \lambda_j \bm{I}_{n-r} - \uperp\t \bN \uperp \big)\inv  \uperp\t \bN \uk -  \gamma\md \hat (\lambda_j) \bigg\|\\
    &\quad +\frac{1}{| \lambda_j - \lambda_k |}  \bigg \| (\U^{-k})\t \bN \uperp \big( \hat \lambda_j \bm{I}_{n-r} - \uperp\t \bN \uperp \big)\inv  \uperp\t \bN \uk\bigg\| \\&\quad 
 +       \frac{1}{| \lambda_j - \lambda_k |}  \| \U\t  \bN \uk \|.
\end{align*}
By the previous analysis it holds that
\begin{align*}
    \bigg\| \uk\t \bN \uperp \big( \hat \lambda_j \bm{I}_{n-r} - \uperp\t \bN \uperp \big)\inv  \uperp\t \bN \uk -  \gamma\md \hat (\lambda_j) \bigg\| &\lesssim \frac{\sigma^2 \sqrt{rn\log(n)}}{\lambda_{\min}}; \\  
    \bigg \| (\U^{-k})\t \bN \uperp \big( \hat \lambda_j \bm{I}_{n-r} - \uperp\t \bN \uperp \big)\inv  \uperp\t \bN \uk\bigg\| &\lesssim \frac{\sigma^2 \sqrt{rn\log(n)}}{\lambda_{\min}}
\end{align*}
with probability $1 - O(n^{-10})$.  In addition, note that $\U\t \bN \uk$ is a submatrix of $\U\t \bN \U$, and by \cref{fact2_MD} it holds that $\|\U\t \bN \U\| \lesssim \delta\md$ with this same probability.  Therefore, combining bounds we arrive at
\begin{align*}
     |\uhatj\t \uk| &\lesssim \frac{1}{| \lambda_j - \lambda_k |}  \bigg( \frac{\sigma^2 \sqrt{rn\log(n)}}{\lambda_{\min}} + \delta\md \bigg) \lesssim \frac{\delta\md}{|\lambda_j - \lambda_k|},
\end{align*}
where the final inequality follows from the assumption that $\lambda_{\min} \geq C_0 \sigma \sqrt{nr}$ for some sufficiently large constant $C_0$.      The result is completed with a union bound over all $k\neq j$.   
\end{proof}

\subsubsection{Proof of Lemma \ref{lem:sigmahatoversigma}}
\label{sec:sigmahatoversigmaproof}

\begin{proof}
Recall that $\mathcal{P}_{{\sf Upper-diag}}$ denotes the projection onto the upper diagonal of a matrix.  We start by noting at the outset that
\begin{align*}
    \|\mathcal{P}_{{\sf upper-diag}}\big(\bm{\hat S} - \uhat \uhat\t\bS \uhat \uhat\t \big)\|_F &= \| \mathcal{P}_{{\sf Upper-diag}}\big(\bS - \uhat\uhat\t \bS \uhat \uhat\t + \bN - \uhat \uhat\t \bN \uhat \uhat\t \big)\|_F,
\end{align*}
and hence that
\begin{align*}
    \bigg| \|\mathcal{P}_{{\sf upper-diag}}\big(\bm{\hat S} &- \uhat \uhat\t\bm{\hat S}  \uhat \uhat\t \big)\|_F - \|\mathcal{P}_{{\sf upper-diag}}\big( \bN\big)\|_F \bigg| \\
    &\leq \|\mathcal{P}_{{\sf upper-diag}}\big( \uhat \uhat\t \bN \uhat \uhat\t\big) \|_F + \|\mathcal{P}_{{\sf upper-diag}}\big( \bS - \uhat\uhat\t \bS \uhat \uhat\t \big)\|_F \\
    &\leq  \| \uhat \uhat\t \bN \uhat \uhat\t \|_F + \|\bS - \uhat\uhat\t \bS \uhat \uhat\t \|_F.
\end{align*}
Furthermore, by \cref{fact1_MD},
\begin{align*}
    \| \uhat \uhat\t \bN \uhat \uhat\t \|_F &\leq \sqrt{r} \| \bN\| \lesssim \sigma \sqrt{rn}.
\end{align*}
In addition,
\begin{align*}
    \| \bS - \uhat \uhat\t \bS \uhat \uhat\t \|_F &= \| \U \bm{\Lambda} \U\t - \uhat \uhat\t \U \bm{\Lambda} \U\t \uhat \uhat\t \|_F \\
    &\leq \| \U \bm{\Lambda} \U\t \uhat \uhat\t -  \uhat \uhat\t \U \bm{\Lambda} \U\t \uhat \uhat\t \|_F + \| \U \bm{\Lambda}\U\t \big( \uhat\uhat\t - \U \U\t \big) \|_F \\
    &\leq 2\| (\U\U\t - \uhat \uhat\t ) \U \bm{\Lambda} \U\t \|_F  \\
    &\leq 2 \sqrt{r} \lambda_1 \| \U \U\t - \uhat \uhat\t \| \\
    &\lesssim  \kappa \sigma \sqrt{rn},
\end{align*}
where the final bound follows from the Davis-Kahan Theorem, both holding with probability $e^{-cn}$.  Therefore,
\begin{align*}
    \bigg| \| \mathcal{P}_{{\sf upper-diag}}\big(\bm{\hat S} - \uhat\uhat\t \bm{\hat S} \uhat \uhat\t \big) \|_F - \|\mathcal{P}_{{\sf upper-diag}}\big( \bN \big)\|_F \bigg| &\lesssim \kappa \sigma \sqrt{rn}. \numberthis \label{thebound}  
\end{align*}
Next, it is straightforward to demonstrate that (e.g., Lemma 1 of \citet{laurent_adaptive_2000})
\begin{align*}
\bigg|    \frac{\| \mathcal{P}_{{\sf upper-diag}}\big(\bN\big) \|_F^2}{\sigma^2} - \binom{n}{2} \bigg| &\lesssim \sqrt{ \binom{n}{2} \log(n)} + \log(n) \lesssim n \sqrt{\log(n)}
\end{align*}
with probability at least $1- O(n^{-10})$.  Therefore,
\begin{align*}
\bigg|    \frac{\| \mathcal{P}_{{\sf upper-diag}}\big(\bN\big) \|_F}{\sqrt{\binom{n}{2}}} - \sigma   \bigg| &\lesssim \sigma \frac{\sqrt{\log(n)}}{n}.
\end{align*}
Hence, combining this bound with \eqref{thebound}, we obtain
\begin{align*}
    \bigg| \frac{\| \mathcal{P}_{{\sf upper-diag}}\big(\bm{\hat S} - \uhat\uhat\t \bm{\hat S} \uhat \uhat\t \big) \|_F}{\sqrt{\binom{n}{2}}} - \sigma \bigg| &\lesssim    \bigg| \frac{\| \mathcal{P}_{{\sf upper-diag}}\big(\bN) \big) \|_F}{\sqrt{\binom{n}{2}}} - \sigma \bigg| + \frac{1}{n} \kappa \sigma \sqrt{rn} \\
    &\lesssim \sigma \frac{\sqrt{\log(n)}}{n} + \sigma \frac{\kappa \sqrt{r}}{\sqrt{n}} \lesssim \sigma \kappa \sqrt{\frac{r}{n}},
\end{align*}
and hence with probability at least $1 - O(n^{-10})$,
\begin{align*}
    \bigg| \frac{\hat \sigma}{\sigma} - 1 \bigg| \lesssim \kappa \sqrt{\frac{r}{n}}.
\end{align*}
Furthermore, this also implies that $\big|\frac{\hat \sigma^2}{\sigma^2} - 1 \big| \lesssim \kappa \sqrt{\frac{r}{n}}$ with this same probability.
\end{proof}

\subsubsection{Proof of Lemma \ref{lem:biashatmd}} \label{sec:biashatmdproof}

\begin{proof}[Proof of \cref{lem:biashatmd}]
    We note that \cref{lem:sigmahatoversigma} implies that with probability at least $1 - O(n^{-10})$, 
\begin{align*}
    \big| \hat{b_k\md} - b_k\md \big| &\lesssim \big| \hat \sigma^2 - \sigma^2 \big|  \bigg| \sum_{l> r} \frac{1}{\hat \lambda_k - \hat \lambda_l} \bigg|\lesssim \kappa \sqrt\frac{r}{n} |b_k\md|,
\end{align*}
and hence $|\hat{b_k\md}| \lesssim |b_k\md| = o(1)$ with this same probability by \cref{lem:bias_MD}.  Therefore, by Taylor Expansion,
\begin{align*}
    \sqrt{1 + \hat{b_k\md}} &= \sqrt{1 + b_k\md + (\hat{b_k\md} - b_k\md)} = \sqrt{1 + b_k\md} + O\bigg( \kappa | b_k\md| \sqrt{\frac{r}{n}} \bigg) = \sqrt{1 + b_k\md} + O\bigg( \kappa \frac{\sigma^2\sqrt{rn}}{\lambda_{\min}^2} \bigg).
\end{align*}
\end{proof}

\subsubsection{Proof of Lemma \ref{lem:sigma_md_approx}}
\label{sec:sigma_approx_MD_proof}

\begin{proof}[Proof of \cref{lem:sigma_md_approx}]
First, by \cref{lem:biashatmd}, with probability at least $1 - O(n^{-10})$ it holds that
\begin{align*}
     \sqrt{1 + \hat{b_k\md}} &=  
     \sqrt{1 + b_k\md} + O\bigg( \kappa \frac{\sigma^2\sqrt{rn}}{\lambda_{\min}^2} \bigg).
\end{align*}
Therefore by \cref{lem:bias_MD}, with this same probability we also have that 
\begin{align*}
    \bigg| \sqrt{1 + \hat{b_k\md}}\uk\t \bm{\hat u}_k  - 1 \bigg|  &\lesssim 
        \frac{\sigma^2 r \log(n)}{\Delta_k^2} + \frac{\sigma^2 \sqrt{n\log(n)}}{\lambda_k^2} + \frac{\sigma^2 \kappa \sqrt{nr}}{\lambda_{\min}^2} \lesssim \frac{\sigma^2 r \log(n)}{\Delta_{\min}^2} + \frac{\sigma^2\kappa\sqrt{nr\log(n)}}{\lambda_{\min}^2}.
\end{align*}
    We now prove the main result. 
    We first write the decomposition
    \begin{align*}
        \big(\hat{s\md_{\a,j}}\big)^2 - (s\md_{\a,j})^2 &= \sum_{\substack{k\neq j\\k\leq r}} \frac{\hat \sigma^2 (\a\t \bm{\hat u}_k )^2 (1 + \hat{b_k\md})}{(\check \lambda_j - \check \lambda_k)^2}  + \frac{\hat \sigma^2 \|\uhat_{\perp}\t \a \|^2}{\hat \lambda_j^2} - \sum_{\substack{k\neq j\\k\leq r}} \frac{\sigma^2 (\a\t \uk)^2}{(\lambda_j - \lambda_k^2} - \frac{\sigma^2 \|\uperp\t\a\|^2}{\lambda_j^2} \\
        &= \big(\hat \sigma^2 - \sigma^2 \big) \bigg(  \sum_{\substack{k\neq j\\k\leq r}} \frac{ (\a\t \uk)^2}{(\lambda_j - \lambda_k^2} + \frac{ \|\uperp\t\a\|^2}{\lambda_j^2} \bigg) \\
        &\quad + \hat \sigma^2\sum_{\substack{k\neq j\\k\leq r}} \bigg( \frac{ (\a\t \bm{\hat u}_k )^2 (1 + \hat{b_k\md})}{(\check \lambda_j - \check \lambda_k)^2} - \frac{ (\a\t \uk)^2}{(\lambda_j - \lambda_k)^2} \bigg) \\
        &\quad + \hat \sigma^2 \bigg( \frac{ \|\uhat_{\perp}\t \a \|^2}{\hat \lambda_j^2} - \frac{ \|\uperp\t\a\|^2}{\lambda_j^2} \bigg) \\
        &=: \alpha_1 + \alpha_2 + \alpha_3.
    \end{align*}
We analyze each term in turn.
\begin{itemize}
    \item \textbf{The term $\alpha_1$}.  We note that \cref{lem:sigmahatoversigma} demonstrates that
    \begin{align*}
        |\hat \sigma^2 - \sigma^2 | \lesssim \sigma^2 \kappa \sqrt{\frac{r}{n}}.  
    \end{align*}
    Consequently,
    \begin{align*}
        | \alpha_1 | &\lesssim \sigma^2 \kappa \sqrt{\frac{r}{n}} \bigg( \sum_{\substack{k\neq j\\k\leq r}} \frac{(\a\t\uk)^2}{(\lambda_j - \lambda_k)^2} + \frac{\|\uperp\t\a\|^2}{\lambda_j^2} \bigg) \asymp \kappa \sqrt{\frac{r}{n}} (s\md_{\a,j})^2.
    \end{align*}
    \item \textbf{The term $\alpha_2$}.  We note that \cref{lem:sigmahatoversigma} demonstrates that $\hat \sigma^2 \lesssim \sigma^2$ as long as $r \lesssim n/\kappa^2$.  Therefore,
    \begin{align*}
        |\alpha_2| &\lesssim \sigma^2 \bigg|  \sum_{\substack{k\neq j\\k\leq r}} \bigg( \frac{(\a\t \bm{\hat u}_k)^2 (1 + \hat{b_k\md})}{(\check \lambda_j - \check \lambda_k )^2} - \frac{(\a\t\uk)^2}{(\lambda_j - \lambda_k)^2} \bigg)  \bigg|  \\
        &\lesssim   \sigma^2 \bigg| \sum_{\substack{k\neq j\\k\leq r}}  (\a\t\uk)^2 \bigg( \frac{1}{(\lambda_j-\lambda_k)^2} - \frac{1}{(\check \lambda_j - \check \lambda_k)^2} \bigg) \bigg| \\
         &\quad + \sigma^2 \bigg|  \sum_{\substack{k\neq j\\k\leq r}} \frac{1}{(\check \lambda_j - \check\lambda_k)^2} \bigg( (\a\t \bm{\hat u}_k)^2 (1 + \hat{b_k\md}) - (\a\t\uk)^2 \bigg) \bigg| \\
       &=: \beta_1 + \beta_2.
    \end{align*}
    We will bound each quantity $\beta_1$ and $\beta_2$ separately. 
    \begin{itemize}
    \item \textbf{The term $\beta_1$}.  We have that
    \begin{align*}
        |\beta_1| &\lesssim \sigma^2 \sum_{\substack{k\neq j\\k\leq r}} (\a\t\uk)^2 \bigg| \frac{1}{(\lambda_j - \lambda_k)^2} - \frac{1}{(\check \lambda_j - \check \lambda_k)^2} \bigg|.
    \end{align*}
    Recall that $\check \lambda_j = \hat \lambda_j - \hat \sigma^2 \sum_{k > r} \frac{1}{\hat \lambda_j - \hat \lambda_k}$.  Moreover, by \cref{lem:eigenvalues_MD} it holds that
    \begin{align*}
        \big| \hat \lambda_j - \hat \gamma\md(\hat \lambda_j) - \lambda_j \big| &\leq \big| \hat \lambda_j - \lambda_j - \gamma\md(\hat \lambda_j) \big| + \big|\gamma\md(\hat \lambda_j ) - \hat \gamma\md(\hat \lambda_j) \big| \lesssim \sigma \big( \sqrt{r} + \sqrt{\log(n)} \big) + \frac{\sigma^2 r}{\lambda_j}.
    \end{align*}
    Consequently,
    \begin{align*}
        \big| \check \lambda_j  - \lambda_j\big| &\lesssim \sigma \big( \sqrt{r} + \sqrt{\log(n)} \big) + \frac{\sigma^2 r}{\lambda_j} + \big| \sigma^2 - \hat \sigma^2 \big| \frac{n}{\lambda_j^2} \\
        &\lesssim \sigma \sqrt{r\log(n)} + \sigma \sqrt{r} \frac{\sigma \sqrt{r}}{\lambda_j} + \sigma \sqrt{r} \frac{\sigma \kappa \sqrt{n}}{\lambda_j^2} \\
        &\lesssim \sigma \sqrt{r\log(n)}.
    \end{align*}
    This bound holds uniformly over all $j$ and $k$ with probability at least $1- O(n^{-8})$.    
    Moreover, under the assumption $\Delta_j \gg \sigma r \log(n)$, we have that
    \begin{align*}
        \frac{1}{(\check \lambda_j - \check \lambda_k)^2} &= \frac{1}{(\lambda_j - \lambda_k)^2} + O\bigg( \frac{1}{(\lambda_j - \lambda_k)^2} \frac{\sigma \sqrt{r\log(n)}}{\Delta_j} \bigg).
    \end{align*}
    Combining these bounds results in
    \begin{align*}
        |\beta_1| &\lesssim \sigma^2 \sum_{\substack{k\neq j\\k\leq r}} (\a\t \uk)^2 \frac{1}{(\lambda_j - \lambda_k)^2} \frac{\sigma \sqrt{r\log(n)}}{\Delta_j} \lesssim (s\md_{\a,j})^2 \frac{\sigma \sqrt{r\log(n)}}{\Delta_j}.
    \end{align*}
        \item \textbf{The term $\beta_1$}. First, by the analysis leading to the proof of \cref{thm:distributionaltheory_MD}, with $j = k$, we have that
\begin{align*}
    \a\t \bm{\hat u}_k \sqrt{1 + \hat{b_k\md}} -\a\t \uk &=\bigg( \a\t \bm{\hat u}_k - \a\t \uk \uk\t \bm{\hat u}_k \bigg)  \sqrt{1 + \hat{b_k\md}} + \a\t \uk \bigg( \uk\t \bm{\hat u}_k  \sqrt{1 + \hat{b_k\md}}  - 1\bigg) \\
    &= \sqrt{1 + \hat{b_k\md}} \bigg( \sum_{l \neq k} \frac{\bm{u}_k\t \bN \bm{u}_l}{\lambda_k - \lambda_l} \bm{u}_l\t \a + {\sf ErrMD} \times s\md_{\a,k} \bigg) + \a\t \uk \bigg(  \sqrt{1 + \hat{b_k\md}}\uk\t \bm{\hat u}_k  - 1 \bigg).
\end{align*}
Consequently, with probability at least $1 - O(n^{-10})$, as long ${\sf ErrMD} \lesssim \sqrt{\log(n)}$ for each $k \leq r$,
\begin{align*}
 \bigg|    \a\t \bm{\hat u}_k& \sqrt{1 + \hat{b_k\md}} -\a\t \uk \bigg|  \\&\lesssim \bigg| \sum_{l \neq k} \frac{\bm{u}_k\t \bN \bm{u}_l}{\lambda_k - \lambda_l} \bm{u}_l\t \a \bigg| + {\sf ErrMD} s\md_{\a,k}  + | \a\t \uk| \bigg| \uk\t \bm{\hat u}_k - \sqrt{1 + \hat{b_k\md}} \bigg| \\
 &\lesssim s_{\a,k}\md \sqrt{\log(n)} + {\sf ErrMD} s\md_{\a,k}  + |\a\t \uk| \frac{\sigma^2 r \log(n)}{\Delta_{\min}^2} + |\a\t \uk|\frac{\sigma^2\kappa \sqrt{rn\log(n)}}{\lambda_{\min}^2} \\
 &\lesssim s\md_{\a,k} \sqrt{\log(n)}+ |\a\t \uk| \frac{\sigma^2 r \log(n)}{\Delta_{\min}^2} + |\a\t \uk|\frac{\sigma^2\kappa \sqrt{rn\log(n)}}{\lambda_{\min}^2}. 
\end{align*}
In particular, we have that $|\a\t\bm{\hat u}_k \sqrt{1 + \hat{b_k}\md}| \lesssim |\a\t\uk| + s_{\a,k}\md\sqrt{\log(n)}$.  
 Therefore,  
    \begin{align*}
        \bigg|  (\a\t \bm{\hat u}_k)^2 & (1 + \hat{b_k\md}) -(\a\t \uk)^2 \bigg| \\ 
        &\lesssim \big( |\a\t \bm{u}_k | + s_{\a,k}\md\sqrt{\log(n)} \big) \bigg( s\md_{\a,k} \sqrt{\log(n)} + |\a\t \uk| \frac{\sigma^2 r \log(n)}{\Delta_{\min}^2} + |\a\t \uk|\frac{\sigma^2\kappa \sqrt{rn\log(n)}}{\lambda_{\min}^2} \bigg) \\
        &\asymp |\a\t\uk| s_{\a,k}\md \sqrt{\log(n)} + |\a\t\uk|^2 \bigg( \frac{\sigma^2 r\log(n)}{\Delta_{\min}^2} + \frac{\sigma^2 \kappa \sqrt{rn\log(n)}}{\lambda_{\min}^2} \bigg) + (s_{\a,k}\md)^2 \log(n).
    \end{align*}
    Consequently, taking a union bound over all $r$ terms and noting that the previous analysis implies that $|\check \lambda_j - \check \lambda_k| \gtrsim |\lambda_j - \lambda_k|$, we see that
\begin{align*}
    |\beta_2| &\lesssim \sigma^2  \sum_{\substack{k\neq j\\k\leq r}} \frac{1}{(\lambda_j - \lambda_k)^2} \bigg[ |\a\t\uk| s_{\a,k}\md \sqrt{\log(n)} + |\a\t\uk|^2 \bigg( \frac{\sigma^2 r\log(n)}{\Delta_{\min}^2} + \frac{\sigma^2 \kappa \sqrt{rn\log(n)}}{\lambda_{\min}^2} \bigg) + (s_{\a,k}\md)^2 \log(n) \bigg] \\
    &\leq  \max_{\substack{k\neq j\\k\leq r}} \frac{s_{\a,k}\md \sqrt{\log(n)}}{s_{\a,j}\md} \frac{\sigma}{\Delta_j} s_{\a,j}\md \sum_{\substack{k\neq j\\k\leq r}} \frac{\sigma|\a\t\uk|}{|\lambda_j-\lambda_k|} + (s_{\a,j}\md)^2 \bigg( \frac{\sigma^2 r \log(n)}{\Delta_{\min}^2} + \frac{\sigma^2 \kappa \sqrt{rn\log(n)}}{\lambda_{\min}^2} \bigg) \\
    &\quad+ \frac{\sigma^2 r \log(n)}{\Delta_j^2}   \max_{\substack{k\neq j\\k\leq r}} \frac{(s_{\a,k}\md)^2}{(s_{\a,j}\md)^2} (s_{\a,j}\md)^2 \\
    &\lesssim (s_{\a,j}\md)^2 \max_{\substack{k\neq j\\k\leq r}} \bigg( \frac{s_{\a,k}\md}{s_{\a,j}\md} \frac{\sigma \sqrt{r\log(n)}}{\Delta_j} + \frac{\sigma^2 r\log(n)}{\Delta_{\min}^2} + \frac{\sigma^2 \kappa \sqrt{rn\log(n)}}{\lambda_{\min}^2} + \frac{\sigma^2 r \log(n)}{\Delta_j^2} \frac{(s_{\a,k}\md)^2}{(s_{\a,j}\md)^2}  \bigg).
\end{align*}
    \end{itemize}
    Therefore,
    \begin{align*}
        |\alpha_2| &\lesssim (s\md_{\a,j})^2 \Bigg( \max_{k\neq j} \frac{s\md_{\a,k}}{s\md_{\a,j}} \frac{\sigma \sqrt{r\log(n)}}{\Delta_j}  + \frac{\sigma^2 r \log(n)}{\Delta_{\min}^2} + \frac{\sigma^2 \kappa \sqrt{rn\log(n)}}{\lambda_{\min}^2} + \frac{\sigma^2 r\log(n)}{\Delta_j^2} \max_{\substack{k\neq j\\k\leq r}} \frac{(s_{\a,k}\md)^2}{(s_{\a,j}\md)^2} + \frac{\sigma \sqrt{r\log(n)}}{\Delta_j} \Bigg)  .  
    \end{align*}
    \item \textbf{The term $\alpha_3$}.  Turning to $\alpha_3$, we observe again by \cref{lem:sigmahatoversigma} that $\hat \sigma^2 \lesssim \sigma^2$ and hence 
    \begin{align*}
        |\alpha_3| &\lesssim \sigma^2 \bigg| \frac{\|\uhat_{\perp}\t \a\|^2}{\check \lambda_j^2} - \frac{ \|\uperp\t\a\|^2}{\lambda_j^2} \bigg| \sigma^2 \bigg| \frac{\| \uhat_{\perp}\t\a\|^2 - \|\uperp\t\a\|^2}{\check\lambda_j^2} \bigg| +\sigma^2 \|\uperp\t\a\|^2 \bigg| \frac{1}{\lambda_j^2} - \frac{1}{\check \lambda_j^2} \bigg|.
    \end{align*}
By a similar argument as the previous term, we have that
\begin{align*}
    \bigg|\frac{1}{\lambda_j^2} - \frac{1}{\check \lambda_j^2} \bigg| &\lesssim \frac{\sigma \sqrt{r\log(n)}}{|\lambda_j^3|}.
\end{align*}
Consequently,
\begin{align*}
    |\alpha_3| &\lesssim  \frac{\sigma^2}{\lambda_j^2} \bigg|\| \uhat_{\perp}\t\a\|^2 - \|\uperp\t\a\|^2 \bigg| + \frac{\sigma^2 \|\uperp\t\a\|^2 }{\lambda_j^2} \frac{\sigma \sqrt{r\log(n)}}{|\lambda_j|}.
\end{align*}
Therefore, it suffices to analyze the difference $ \bigg|\| \uhat_{\perp}\t\a\|^2 - \|\uperp\t\a\|^2 \bigg|$.  First we will analyze the difference $\big| \|\uhat_{\perp}\t\a \| - \| \uperp\t\a\| \big|$.  By orthonormality it holds that
\begin{align*}
    \big| \|\uhat_{\perp}\t\a \| - \| \uperp\t\a\| \big| &= \big| \|\uhat_{\perp} \uhat_{\perp}\t \a \| - \| \uperp \uperp\t\a \| \big| \\
    &\leq \| \big( \uhat_{\perp} \uhat_{\perp}\t - \uperp \uperp\t\a \big) \a \| \\
    &= \| \big( \uhat \uhat\t - \U \U\t \big) \a \| \\
    &\leq \frac{\sigma \sqrt{n}}{\lambda_{\min}} \|\a \| \\
    &\leq \frac{\sigma \sqrt{n}}{\lambda_{\min}} \bigg( \|\uperp\t \a \| + \| \U\t\a\|\bigg).
\end{align*}
Therefore,
\begin{align*}
    \frac{\sigma^2}{\lambda_j^2} \bigg| \|\uhat_{\perp}\t\a \|^2 - \|\uperp\t\a \|^2 \bigg| &= \frac{\sigma^2}{\lambda_j^2} \bigg( \| \uhat_{\perp}\t\a \| + \|\uperp\t\a \| \bigg) \bigg| \|\uhat_{\perp}\t\a\| - \|\uperp\t\a \| \bigg| \\
    &\lesssim\frac{\sigma^2}{\lambda_j^2} \bigg( \| \uhat_{\perp}\t\a \| + \|\uperp\t\a \| \bigg)  \frac{\sigma \sqrt{n}}{\lambda_{\min}} \bigg( \|\uperp\t \a \| + \| \U\t\a\|\bigg) \\
    &\lesssim \frac{\sigma^2}{\lambda_j^2} \bigg( \|\uperp\t\a \| + \frac{\sigma \sqrt{n}}{\lambda_{\min}} \|\U\t\a\| \bigg) \frac{\sigma \sqrt{n}}{\lambda_{\min}} \bigg( \|\uperp\t \a \| + \| \U\t\a\|\bigg) \\
    &\asymp \frac{\sigma^2}{\lambda_j^2} \|\uperp\t\a \|^2 \frac{\sigma \sqrt{n}}{\lambda_{\min}} + \frac{\sigma^2}{\lambda_j^2} \|\uperp\t\a\| \|\U\t\a\| \frac{\sigma \sqrt{n}}{\lambda_{\min}} + \frac{\sigma^2}{\lambda_j^2} \frac{\sigma^2 n}{\lambda_{\min}^2} \|\U\t\a\|^2.
\end{align*}
Combining these results yields
\begin{align*}
    |\alpha_3| &\lesssim \frac{\sigma \sqrt{n}}{\lambda_{\min}} \frac{\sigma^2}{\lambda_j^2} \|\uperp\t\a\|^2 + \frac{\sigma^2}{\lambda_j^2} \|\uperp\t\a\| \|\U\t\a\| \frac{\sigma \sqrt{n}}{\lambda_{\min}} + \frac{\sigma^2}{\lambda_j^2} \frac{\sigma^2 n}{\lambda_{\min}^2} \|\U\t\a\|^2.
\end{align*}
First, if $\|\U\t\a\| \leq \|\uperp\t\a\|$, then we immediately arrive at  the bound $|\alpha_3| \lesssim  \frac{\sigma \sqrt{n}}{\lambda_{\min}} \frac{\sigma^2}{\lambda_j^2} \|\uperp\t\a\|^2.$  Therefore,  it suffices to consider the setting that $\|\U\t\a\| > \|\uperp\t\a\|$. 
Considering the first term above we have that
\begin{align*}
     \frac{\sigma^2}{\lambda_j^2} \| \uperp\t \a \| \| \U\t \a \| \frac{\sigma \sqrt{n}}{\lambda_{\min}} &=      \frac{\sigma^2}{\lambda_j^2} \| \uperp\t \a \|  \frac{\sigma \sqrt{n}}{\lambda_{\min}} \sqrt{\sum_{\substack{k\neq j\\k\leq r}} (\uk\t\a)^2 + (\uj\t \a)^2} \\
     &\leq \frac{\sigma^2}{\lambda_j^2} \| \uperp\t\a \| \frac{\sigma \sqrt{n}}{\lambda_{\min}} \Delta_{\max} \sqrt{\sum_{k\leq r,k\neq j} \frac{(\uk\t\a)^2}{(\lambda_k - \lambda_j)^2}} + |\uj\t\a | \frac{\sigma^2}{\lambda_j^2} \|\uperp\t\a \| \frac{\sigma\sqrt{n}}{\lambda_{\min}} \\
     &\lesssim \frac{\sigma  \kappa \sqrt{n}}{\lambda_j} (s_{\a,j}\md)^2 + o\bigg(\frac{1}{\sqrt{\log(n)}}\bigg) (s_{\a,j}\md)^2,
\end{align*}
where we used the fact that $\frac{\sigma^2 \sqrt{n\log(n)}}{\lambda_{\min}^2} |\uj\t\a| \ll s_{\a,j}\md$ which follows from the assumption \eqref{atujcondition_ci}.  For the other term,
\begin{align*}
    \frac{\sigma^2}{\lambda_j^2} \frac{\sigma^2 n}{\lambda_{\min}^2} \| \U\t\a\|^2 &=   \frac{\sigma^2}{\lambda_j^2} \frac{\sigma^2 n}{\lambda_{\min}^2} \bigg[ \sum_{\substack{k\neq j\\k\leq r}} (\uk\t\a)^2 + (\uj\t\a)^2 \bigg]\\
    &\leq \frac{\Delta_{\max}^2}{\lambda_j^2} \frac{\sigma^2 n} {\lambda_{\min}^2} (s_{\a,j}\md)^2 +  \frac{\sigma^2}{\lambda_j^2} \frac{\sigma^2 n}{\lambda_{\min}^2}  (\uj\t\a)^2 \\
    &\leq \frac{\kappa^2\sigma^2 n }{\lambda_j^2}  (s_{\a,j}\md)^2 + \frac{(s_{\a,j}\md)^2}{\log(n)}
\end{align*}
which follows from this same assumption.

 \end{itemize}
Combining our bounds for $\alpha_1,\alpha_2$ and $\alpha_3$, we have that 
\begin{align*}
    \big| \big( \hat{s\md_{\a,j}} \big)^2 - \big( s\md_{\a,j} \big)^2 \big| &\lesssim (s\md_{\a,j})^2 \Bigg( \kappa \sqrt{\frac{r}{n}}  +  \max_{k\neq j} \frac{s\md_{\a,k}}{s\md_{\a,j}} \frac{\sigma \sqrt{r\log(n)}}{\Delta_j}  + \frac{\sigma^2 r \log(n)}{\Delta_{\min}^2} + \frac{\sigma^2 \kappa \sqrt{rn\log(n)}}{\lambda_{\min}^2} \\
    &\quad + \max_{k\neq j} \bigg(\frac{s\md_{\a,k}}{s\md_{\a,j}}\bigg)^2 \frac{\sigma^2 r \log(n)}{\Delta_j^2} + \frac{\sigma \sqrt{r\log(n)}}{\Delta_j} +  \frac{\sigma \kappa \sqrt{n}}{\lambda_j} + \frac{\sigma \sqrt{n}}{\lambda_{\min}} + o\big( \frac{1}{\sqrt{\log(n)}} \big) \bigg).
\end{align*}
Under the assumptions in \cref{thm:civalidity_MD}, the term in parentheses can be directly verified to be $o(\frac{1}{\sqrt{\log(n)}})$.
The proof is therefore completed by noting that
\begin{align*}
    \bigg| \frac{\hat{s\md_{\a,j}}}{s\md_{\a,j}} - 1 \bigg| &=  \Bigg|\frac{  \frac{\big( \hat{s\md_{\a,j}} \big)^2}{\big( s\md_{\a,j} \big)^2} -  1  }{\frac{\hat{s\md_{\a,j}} }{s\md_{\a,j} } + 1}  \Bigg| \lesssim \frac{1}{(s\md_{\a,j})^2}  \big| \big( \hat{s\md_{\a,j}} \big)^2 - \big( s\md_{\a,j} \big)^2 \big| = o\bigg( \frac{1}{\sqrt{\log(n)}} \bigg).
\end{align*}
\end{proof}

\section{Proofs for PCA} \label{sec:pcamainproof}

In this section we detail our analysis for the PCA model; namely, \cref{thm:distributionaltheory_PCA,thm:distributionaltheory_PCA,thm:civalidity_pca}. First will define some notation that we will use throughout this section and in associated proofs.  Let 
$\lamtilde= \bm{\Lambda} + \sigma^2 \bm{I}_r$, where $\tilde \lambda_k = \lambda_k + \sigma^2$ for $k \leq r$ and denote $\tilde \lambda_{k} = \sigma^2$ for $k \geq r+1$.   In this manner the eigenvalues of $\bSigma$ are simply $\tilde \lambda_j$.  We also observe that we can write $\x = \bSigma^{1/2} \y$, where $\y \in \mathbb{R}^{p\times n}$ is a matrix of independent $\mathcal{N}(0,1)$ random variables.   Throughout our proofs we assume that $\lambda_1$ through $\lambda_r$ are unique.  The extensions to eigenvalue multiplicity are more cumbersome but not materially different.

We start with the following fact. 

\begin{fact}\label{fact1_PCA}
By (a slight modification of) Lemma 7 of \citet{li_minimax_2025}, with probability at least $1 - O((n\vee p)^{-10})$ it holds that \begin{align*}
    \| \sigmahat - \bSigma \| &\lesssim \lambda_{\max} \sqrt{\frac{r\log(n\vee p)}{n}} + \sqrt{(\lambda_{\max} + \sigma^2)\sigma^2 \frac{p}{n}}\log(n\vee p)  + \sigma^2 \bigg( \sqrt{\frac{p}{n}} + \frac{p}{n} + \sqrt{\frac{\log(n\vee p)}{n}} \bigg) \numberthis \label{epcadef} \\
    &=: \mathcal{E}\pca. 
\end{align*}
\end{fact}
Note that direct comparison implies that $\lambda_{\min} \gg \mathcal{E}\pca \times \log(n\vee p)$.

The next lemma provides an analogue of \cref{fact2_MD} for the PCA setting.
\begin{lemma} \label{fact2_PCA}
    Instate the conditions of \cref{thm:distributionaltheory_PCA}.  Then it holds that
    \begin{align*}
        \| \U\t \big( \sigmahat - \bSigma \big) \U \| \lesssim \big( \lambda_{\max} +\sigma^2 \big) \sqrt{\frac{r\log(n\vee p)}{n}}.
    \end{align*}
\end{lemma}

\begin{proof}
Note that 
\begin{align*}
    \| \U\t \big( \sigmahat - \bSigma) \U \| &= \big\| \lamtilde^{1/2} \U\t \bigg( \frac{\y\t\y}{n} - \bm{I}_p \bigg) \U\lamtilde^{1/2} \big\| \leq (\lambda_{\max} + \sigma^2) 
    \big\| \U\t \bigg( \frac{\y\t\y}{n} - \bm{I}_p \bigg) \U \big\|.
\end{align*}
Note that $\U\t \y$ is a $r \times n$ dimensional standard Gaussian matrix.  The result then follows from standard concentration inequalities for covariance matrices (e.g., Theorem 6.5 of \citet{wainwright_high-dimensional_2019}) together with the assumption that $r \lesssim n$.  
\end{proof}

\begin{fact} \label{fact3_PCA}
By Weyl's inequality and the assumption $\lambda_{\min} \gg \mathcal{E}\pca$, it holds that $\hat \lambda_j \in [2(\lambda_j + \sigma^2)/3, 4(\lambda_j +\sigma^2)/3]$ on the event in \cref{fact1_PCA}.
\end{fact}

\begin{fact}\label{fact4_PCA}
By \cref{fact3_PCA} together with our noise assumption $\lambda_{\min} \gg \mathcal{E}\pca$, it holds that $\lambda\bm{I}_{p-r}- \uperp\t \sigmahat \uperp$ is invertible for any $\lambda \geq 2( \lambda_{\min}+\sigma^2)/3$.  Specifically, $\hat \lambda_j \bm{I}_{p-r} - \uperp\t \sigmahat \uperp$ is invertible on the event $\|\sigmahat - \bSigma \| \lesssim \mathcal{E}\pca$.  In addition, $\big\| \big( \hat \lambda_j \bm{I}_{p-r} - \uperp\t\sigmahat \uperp \big)\inv \big\| \lesssim \lambda_j\inv$, which can be seen since a similar argument to \cref{fact2_PCA} shows that \begin{align*}\|\uperp\t \sigmahat \uperp \| \leq \| \sigma^2 \bm{I}_{p-r} \| + \| \uperp\t \sigmahat \uperp -\sigma^2 \bm{I}_{p-r} \| &\leq \sigma^2 + C\sigma^2 \bigg( \frac{p}{n} + \sqrt{\frac{p}{n}} + \sqrt{\frac{\log(n\vee p)}{n}} \bigg) \\
&\ll \lambda_j + \sigma^2,\end{align*}
where the final inequality holds from the noise assumption \eqref{noiseassumption:pca}.  This argument demonstrates implies that $\hat \lambda_j \geq 2[\lambda_j + \sigma^2] \gg \|\uperp\t \sigmahat \uperp \|$.
\end{fact}

Next we state several results that follow directly from the analysis in \citet{li_minimax_2025}.  
The following lemma studies the bias term $\uhatj\t \uj$ as well as provides the fidelity of the bias correction term $b_j\pca$.  
\begin{lemma} \label{lem:bias_PCA}
   Instate the conditions of \cref{thm:distributionaltheory_PCA}, and define $b_j\pca$ as in \cref{thm:distributionaltheory_PCA}.  Then with probability at least $1 - O((n\vee p)^{-10})$ it holds that
    \begin{align*}
        |1 - (\uhatj\t \uj)^2 | &\lesssim \frac{(\lambda_{\max} + \sigma^2)(\lambda_j + \sigma^2)r\log(n\vee p)}{\Delta_j^2 n} + \frac{(\lambda_j + \sigma^2)\sigma^2 p\log^2(n\vee p)}{\lambda_j^2 n} \\
        &=: (\mathcal{F}\pca)^2. \numberthis \label{fpcadef}
        \end{align*}
In addition, with this same probability,
\begin{align*}
        \big|1 - \sqrt{1 + b_j\pca} \uj\t \uhatj \big|
        &\lesssim \begin{cases} \frac{(\lambda_{\max} + \sigma^2)(\lambda_j + \sigma^2)r\log(n\vee p)}{\Delta_j^2 n} + \frac{\sigma^2 \kappa \sqrt{pr \log(n\vee p)}}{\lambda_j n} & p > n \\ 
 \frac{(\lambda_{\max} + \sigma^2)(\lambda_j + \sigma^2)r\log(n\vee p)}{\Delta_j^2 n} + \frac{\sigma^2 \kappa p \sqrt{r\log(n\vee p)}}{\lambda_j n^{3/2}} \bigg( 1 + \frac{\sigma^2}{\lambda_j} \bigg) + \frac{\sigma^2 \sqrt{p} \log^2(n\vee p)}{\lambda_j n} \bigg( 1 + \frac{\sigma^2}{\lambda_j} \bigg)
  & p \leq n. \end{cases} \\
&= {\sf ErrBiasPCA}.
\end{align*}
\end{lemma}
\begin{proof}
The first inequality holds by equation 5.79 in \citet{li_minimax_2025}.  The second inequality holds by Equations 5.80 and 5.81  of \citet{li_minimax_2025} when $n \leq p$ and $n > p$ respectively. 
\end{proof}
Next, the following result, which is a slight modification from the preliminary analysis in \citet{li_minimax_2025},  characterizes the eigenvalue bias.
\begin{lemma}\label{lem:eigenvalueconcentration_PCA}
Instate the conditions of \cref{thm:distributionaltheory_PCA}.  Define
\begin{align*}
    \gamma\pca( \hat \lambda_j) &\coloneqq \frac{1}{n^2}\tr\Bigg(  \x\t \uperp  \bigg( \hat \lambda_j \bm{I}_{p-r} - \uperp\t \frac{\x \x\t}{n} \uperp \bigg)\inv \uperp\t \x \Bigg); \\
    \hat \gamma\pca( \hat \lambda_j) &\coloneqq \frac{1}{n} \sum_{k = r+ 1}^{\min\{p-r,n\}} \frac{\hat \lambda_k}{\hat \lambda_j - \hat \lambda_k}.
\end{align*}
  Then with probability at least $1 -  O((n\vee p)^{-10})$ it holds that
    \begin{align*}
\big|        \hat \lambda_k - \big( \lambda_k + \sigma^2\big)\big (1 + \gamma\pca(\hat \lambda_k) \big) \big| \lesssim ( \lambda_{\max} + \sigma^2) \sqrt{\frac{r}{n}} \log(n\vee p) =: \delta\pca. \numberthis \label{deltapcadef}
    \end{align*}
    In addition, with this same probability it holds that
    \begin{align*}
        |\gamma\pca(\hat \lambda_j)| &\lesssim  \frac{\sigma^2 p}{\lambda_j n}; \\
      \bigg|     \gamma\pca(\hat \lambda_j) - \hat \gamma\pca(\hat \lambda_j)\bigg|  &\lesssim \frac{\sigma^2 r}{n\lambda_j} \bigg( 1 + \frac{p}{n} \bigg).
    \end{align*}
\end{lemma}

\begin{proof}
  The perturbation bounds follows from Theorem 8 of \citet{li_minimax_2025} where we note that their assumptions are slightly weaker than our assumptions by factors of $r$ and $\log(p\vee n)/\log(n)$.  

 We therefore study the approximation of $\hat \gamma\pca(\hat \lambda_j)$ to $\gamma\pca(\hat \lambda_j)$.  Let $\bm{Z} = \uperp\t \bm{X}$, and suppose that $\bm{ZZ}\t/n$ has eigenvalues $\lambda_k^{\perp}$.  Then observe that
 \begin{align*}
     \gamma\pca(\hat \lambda_j) = \frac{1}{n} \sum_{k=1}^{\min(p-r,n)} \frac{ \lambda_k^{\perp}}{\hat \lambda_j -  \lambda_k^{\perp}}.
 \end{align*}
 By the Poincare Separation Theorem (Corollary 4.3.37 in \citet{horn_matrix_2012}) it holds that
 \begin{align*}
     \hat \lambda_{k + r} \leq \lambda_{k}^{\perp} \leq \hat \lambda_k
 \end{align*}
 for all $1 \leq k \leq p - r$.  In addition, note that the function $x \mapsto \frac{x}{a-x}$ has derivative $\frac{a}{(a-x)^2}$ which is strictly nonnegative for $a$ positive, and hence the function  $\frac{x}{a-x}$ is increasing in $x$ for $x \leq a$ and $a$ positive.  By \cref{fact3_PCA} it holds that $\hat \lambda_j \in [2\lambda_j/3,4\lambda_j/3]$ and hence is positive.  In addition, by \cref{fact1_PCA}, with probability at least $1 - O((n\vee p)^{-10})$, $\hat \lambda_{k+r} < \hat \lambda_j$, and similarly $\lambda_{k}^{\perp} <\hat \lambda_j$ and hence
 \begin{align*}
     \frac{\hat \lambda_{k+r}}{\hat \lambda_j - \hat \lambda_{k+r}} \leq \frac{\lambda_k^{\perp}}{\hat \lambda_j - \lambda_k^{\perp}} \leq \frac{\hat \lambda_k}{\hat \lambda_j - \hat \lambda_k}
 \end{align*}
 for each $1 \leq k \leq \min(p-r,n)$.  Therefore,
 \begin{align*}
     \frac{1}{n} \sum_{k=1}^{\min(p-r,n)} \frac{\lambda_k^{\perp}}{\hat \lambda_j - \lambda_k^{\perp}} &\leq \frac{1}{n} \sum_{k=1}^{r} \frac{\lambda_k^{\perp}}{\hat \lambda_j - \lambda_k^{\perp}} + \frac{1}{n} \sum_{k=r+1}^{\min(p-r,n)} \frac{\lambda_k^{\perp}}{\hat \lambda_j - \lambda_k^{\perp}}
     \leq \frac{1}{n} \sum_{k=1}^{r} \frac{\lambda_k^{\perp}}{\hat \lambda_j - \lambda_k^{\perp}} + \frac{1}{n} \sum_{k=r+1}^{\min(p-r,n)} \frac{\hat \lambda_k}{\hat \lambda_j - \hat \lambda_k}. 
 \end{align*}
 In addition, 
 \begin{align*}
     \frac{1}{n} \sum_{k=1}^{\min(p-r,n)} \frac{\lambda_k^{\perp}}{\hat \lambda_j - \lambda_k^{\perp}} &\geq \frac{1}{n} \sum_{k=1}^{\min(p-r,n)} \frac{\hat \lambda_{k+r}}{\hat \lambda_j - \hat \lambda_{k+r}} 
     = \frac{1}{n} \sum_{k={r+1}}^{\min(p,n+r)} \frac{\hat \lambda_{k}}{\hat \lambda_j - \hat \lambda_{k}} 
     \geq \frac{1}{n} \sum_{k={r+1}}^{\min(p-r,n)} \frac{\hat \lambda_{k}}{\hat \lambda_j - \hat \lambda_{k}}.
 \end{align*}
 Therefore, 
 \begin{align*}
\bigg|     \gamma\pca(\hat \lambda_j) - \hat \gamma\pca(\hat \lambda_j)\bigg| &\leq \frac{1}{n} \sum_{k=1}^{r} \frac{\lambda_k^{\perp}}{\hat \lambda_j - \lambda_k^{\perp}} \lesssim \frac{r}{n\lambda_j} \lambda_1^{\perp},
 \end{align*}
 where we have used the implicit inequality $\hat \lambda_j - \lambda_k^{\perp} \gtrsim \lambda_j$.
We therefore bound $\lambda_1^{\perp}$.  However, we note that by  $\uperp\sigmahat\uperp$ is the empirical covariance for i.i.d. standard Gaussians.  Consequently, by the covariance concentration inequality for Gaussian random variables (e.g., \cref{fact2_PCA}),
\begin{align*}
\lambda_1^{\perp} &= \| \uperp\t\sigmahat \uperp \| \leq \| \uperp\t \sigmahat \uperp - \sigma^2 \bm{I}_{p-r}\| + \sigma^2 \lesssim \sigma^2 \bigg( 1 + \sqrt{\frac{p}{n}} + \frac{p}{n} + \sqrt{\frac{\log(n\vee p)}{n}} \bigg) 
\lesssim \sigma^2 \bigg( 1 + \frac{p}{n}\bigg),
\end{align*}
with probability at least $1- O((n\vee p)^{-10})$. Therefore, with this same probability,
\begin{align*}
    \bigg|     \gamma\pca(\hat \lambda_j) - \hat \gamma\pca(\hat \lambda_j)\bigg|  &\lesssim \frac{\sigma^2 r}{n\lambda_j} \sigma^2 \bigg( 1 + \frac{p}{n}\bigg)
\end{align*}
which completes the proof.
 \end{proof}
We next provide a similar technical result to \cref{lem:gammaapprox}. Define the matrices
\begin{align*}
    \bm{\tilde G}\pca(\lambda) &\coloneqq \U\t \sigmahat \uperp\big( \lambda \bm{I}_{p-r} - \uperp\t \sigmahat \uperp \big)\inv \uperp\t \sigmahat \U;\\
    \bm{G}\pca(\lambda) &\coloneqq  \gamma\pca(\lambda)\bm{\tilde \Lambda},
\end{align*}
where we recall $\lamtilde = \bm{\Lambda} + \sigma^2 \bm{I}_r$. 
The following result is used as an intermediate step in \citet{li_minimax_2025}.
\begin{lemma}\label{lem:Gpca}
Under the conditions of \cref{thm:distributionaltheory_PCA}, with probability at least $1 - O((n\vee p)^{-10})$ it holds that
    \begin{align*}
        \sup_{\lambda \in [2(\lambda_j+\sigma^2)/3, 4(\lambda_j+\sigma^2)/3]} \| \bm{\tilde G}\pca(\lambda) - \bm{G}\pca(\lambda) \| \lesssim (\lambda_{\max} + \sigma^2) \frac{\sigma^2}{\lambda_j} \bigg( \frac{p}{n} + \sqrt{\frac{p}{n}} \bigg) \sqrt{\frac{r}{n}} \log^2(n\vee p).
    \end{align*}
    In addition,
    \begin{align*}
        \sup_{\lambda \in [2(\lambda_j+\sigma^2)/3, 4(\lambda_j+\sigma^2)/3]}\| \bm{G}\pca(\lambda) \| &\lesssim (\lambda_{\max} + \sigma^2) \frac{\sigma^2 p}{\lambda_j n}.
    \end{align*}
\end{lemma}
\begin{proof}
   This follows from the argument en route to the proof of Theorem 8 of \citet{li_minimax_2025} (see their equation 5.57), together with their bound 5.61, replacing the terms $\log(n)$ with $\log(n\vee p)$ throughout.  
\end{proof}

\subsection{Isolating the Leading-Order Term} We are now prepared to establish the leading-order expansion for the difference $\a\t\uhatj - \a\t\uj\uj\t\uhatj$.  
Invoke \cref{thm:mainexpansion} with $\bM = \bSigma$ and $\mhat = \sigmahat$  
to yield
\begin{align*}
         \a\t \uhatj - \a\t \uj \uj\t \uhatj - \sum_{k\neq j} \frac{\uj\t \big(\sigmahat - \bSigma \big) \uk}{\tilde \lambda_j - \tilde \lambda_k} \uk\t \a 
    &= \sum_{k\neq j} \frac{(\uhatj - \uj)\t \big( \sigmahat - \bSigma \big) \uk }{\tilde \lambda_j - \tilde\lambda_k} \a\t \uk    + \sum_{k\neq j} \frac{\tilde \lambda_j - \hat \lambda_j }{\tilde \lambda_j - \tilde \lambda_k } \uhatj\t \uk \uk\t \a.
\end{align*}
Observe that $\uj\t \bSigma \uk = (\lambda_j + \sigma^2) \uj\t \uk = 0$.  Therefore, by a similar analysis for matrix denoising and recognizing $\tilde \lambda_k =\lambda_k + \sigma^2$ we arrive at
\begin{align*}
        \a\t \uhatj - \a\t \uj \uj\t \uhatj - \sum_{k\neq j} \frac{\uj\t \sigmahat  \uk}{ \lambda_j -  \lambda_k} \uk\t \a 
   &= \sum_{\substack{k\neq j\\k\leq r} } \frac{(\uhatj - \uj)\t \big( \sigmahat - \bSigma \big) \uk - \gamma\pca(\hat \lambda_j)(\lambda_k + \sigma^2) \uhatj\t \uk  }{ \lambda_j - \lambda_k} \a\t \uk \\
   &\quad + \frac{(\uhatj - \uj)\t \big( \sigmahat - \bSigma \big) \uperp \uperp\t\a  }{\lambda_j}  \\
    &\quad + \sum_{k\neq j,k\leq  r } \frac{\tilde \lambda_j - \hat \lambda_j + \gamma\pca(\hat \lambda_j)(\lambda_k + \sigma^2)  }{ \lambda_j - \lambda_k } \uhatj\t \uk \uk\t \a +\frac{\tilde \lambda_j -\hat \lambda_j }{\lambda_j  } \uhatj\t \uperp \uperp\t \a  \\
    &=: \rpca_1 + \rpca_2 + \rpca_3 + \rpca_4, 
\end{align*}
where we have added and subtracted $\gamma\pca(\hat \lambda_j)(\lambda_k + \sigma^2) \uhatj\t \uk \uk\t\a$ for all $k \leq r$.  
As in the case of matrix denoising, we will demonstrate first that the residuals $\rpca_1$ through $\rpca_4$ are sufficiently small relative to the standard deviation $s_{\a,j}\pca$.

  We bound each residual in turn.
\begin{itemize}
    \item \textbf{Bounding $\rpca_1$}.
    We first decompose $\rpca_1$ in a similar manner to $\rmd_1$ via
\begin{align*}
    \rpca_1 
    &= \sum_{\substack{k\neq j\\k\leq r} } \frac{(\uhatj - \uj)\t \U \U\t  \big( \sigmahat - \bSigma \big) \uk  }{ \lambda_j - \lambda_k} \a\t \uk \\
    &\quad + \sum_{\substack{k\neq j\\k\leq r} } \frac{\uhatj \t  \U^{-k} (\U^{-k})\t \sigmahat \uperp\big( \hat \lambda_j \bm{I}_{p-r} - \uperp\t \sigmahat \uperp \big)\inv  \uperp\t \sigmahat    \uk  }{ \lambda_j - \lambda_k} \a\t \uk \\
    &\quad + \sum_{\substack{k\neq j\\k\leq r} } \frac{\uhatj \t \uk\uk\t \sigmahat \uperp\big( \hat \lambda_j \bm{I}_{p-r} - \uperp\t \sigmahat \uperp \big)\inv  \uperp\t \sigmahat    \uk - \gamma\pca(\hat \lambda_j)(\lambda_k + \sigma^2) \uhatj\t \uk }{ \lambda_j - \lambda_k} \a\t \uk \\
    &=: \alpha_1 + \alpha_2 + \alpha_3,
\end{align*}
where we have implicitly invoked \cref{lem:uhatjuperpidentity} via \cref{fact4_PCA}.  
We bound $\alpha_1$ through $\alpha_3$ sequentially. First, for $\alpha_1$, by Cauchy-Schwarz,
    \begin{align*}
    |\alpha_1| 
        &\lesssim \|\uhatj - \uj\| \bigg\| \sum_{\substack{k\neq j\\k\leq r} } \frac{\U\t  \big( \sigmahat - \bSigma \big) \uk  }{ \lambda_j - \lambda_k} \a\t \uk \bigg\| \lesssim \mathcal{F}\pca \bigg\| \sum_{\substack{k\neq j\\k\leq r} } \frac{\U\t  \big( \sigmahat - \bSigma \big) \uk  }{ \lambda_j - \lambda_k} \a\t \uk \bigg\|,
    \end{align*}
    which holds with probabilty at least $1 - O((n\vee p)^{-10})$,     where in the final line we have implicitly invoked \cref{lem:bias_PCA} through a similar argument as in the matrix denoising setting.  The following lemma bounds the remaining term above.
    \begin{lemma} \label{lem:pca1}
    Instate the conditions of \cref{thm:distributionaltheory_PCA}.  Then with probability at least $1 - O((n\vee p)^{-10})$ it holds that
        \begin{align*}
            \bigg\| \sum_{\substack{k\neq j\\k\leq r} } \frac{\U\t  \big( \sigmahat - \bSigma \big) \uk  }{ \lambda_j - \lambda_k} \a\t \uk \bigg\| &\lesssim \frac{ r (\lambda_{\max} + \sigma^2)^{1/2} \sqrt{\log(n\vee p)}}{\sqrt{n}}   \sqrt{\sum_{\substack{k\neq j\\k\leq r}} \frac{(\lambda_k + \sigma^2)(\a\t\uk)^2}{(\lambda_j - \lambda_k)^2}}.
        \end{align*}
        
    \end{lemma}
    \begin{proof}
        See \cref{sec:pca_1_proof}.
    \end{proof}
    As a consequence, we obtain that with probability at least $1 - O((n\vee p)^{-10})$,      \begin{align*}
        |\alpha_1| &\lesssim \mathcal{F}\pca r \sqrt{\kappa \log(n\vee p)}\sqrt{\sum_{\substack{k\neq j\\k\leq r}} \frac{(\lambda_k + \sigma^2)(\lambda_j + \sigma^2)(\a\t\uk)^2}{n(\lambda_j - \lambda_k)^2}} . \numberthis \label{alpha1bound}
    \end{align*}
    Next, when it comes to $\alpha_2$ we note that
    \begin{align*}
    |\alpha_2| 
        &\lesssim \max_{\substack{k\neq j\\k\leq r}} \bigg\| (\U^{-k})\t \sigmahat \uperp\big( \hat \lambda_j \bm{I}_{p-r} - \uperp\t \sigmahat \uperp \big)\inv  \uperp\t \sigmahat    \uk  \bigg\| \sum_{\substack{k\neq j\\k\leq r}} \frac{|\a\t \uk|}{|\lambda_j - \lambda_k|}.
    \end{align*}
   Again, the term in the spectral norm above is a submatrix of $\tilde{\bm{G}}\pca(\lambda) - \bm{G}\pca(\lambda)$.  
    Therefore, by \cref{lem:Gpca}, with probability at least $1- O((n\vee p)^{-10})$ it holds that
    \begin{align*}
        |\alpha_2| &\lesssim (\lambda_{\max} + \sigma^2) \frac{\sigma^2}{\lambda_j} \bigg( \frac{p}{n} + \sqrt{\frac{p}{n}} \bigg) \sqrt{\frac{r}{n}}\log^2(n\vee p) \sum_{\substack{k\neq j\\k\leq r}} \frac{|\a\t\uk|}{|\lambda_j - \lambda_k|} \\
        &\lesssim \kappa r \log^2(n\vee p) \frac{\sigma^2}{\lambda_j} \bigg( \frac{p}{n} + \sqrt{\frac{p}{n}} \bigg) \sqrt{\sum_{k\neq ,k\leq r} \frac{(\lambda_j + \sigma^2)(\lambda_k+\sigma^2)(\a\t\uk)^2}{n (\lambda_j - \lambda_k)^2}}. \numberthis \label{alpha2bound}
    \end{align*} Finally, for $\alpha_3$, it holds that 
    \begin{align*}
    |\alpha_3| 
         &\lesssim \max_{\substack{k\neq j\\k\leq r}}  \bigg| \uk\t \sigmahat \uperp \big( \hat \lambda_j \bm{I}_{p-r} - \uperp\t \sigmahat \uperp \big)\inv \uperp\t \sigmahat \uk - (\lambda_k + \sigma^2) \gamma\pca(\hat\lambda_j) \bigg| \sum_{\substack{k\neq j\\k\leq r} } \frac{|\a\t\uk|}{\lambda_j - \lambda_k}. 
    \end{align*}
We may again apply \cref{lem:Gpca} to obtain that with the probability therein,
    \begin{align*}
        |\alpha_3| &\lesssim (\lambda_{\max} + \sigma^2)  \Bigg(\frac{\sigma^2}{\lambda_j} \bigg( \frac{p}{n} + \sqrt{\frac{p}{n}} \bigg) \sqrt{\frac{r}{n}}\log^2(n\vee p) \Bigg) \sum_{\substack{k\neq j\\k\leq r} } \frac{|\a\t\uk|}{\lambda_j - \lambda_k} \\
            &\lesssim \kappa \sqrt{r} \log^2(n\vee p) \Bigg(\frac{\sigma^2}{\lambda_j} \bigg( \frac{p}{n} + \sqrt{\frac{p}{n}} \bigg) \Bigg) \sqrt{\sum_{\substack{k\neq j\\k\leq r} } \frac{(\lambda_j + \sigma^2)(\lambda_k + \sigma^2)(\a\t\uk)}{n(\lambda_j - \lambda_k)^2}}.  \numberthis \label{alpha3boundpca}
    \end{align*}
Combining \eqref{alpha1bound}, \eqref{alpha2bound}, and \eqref{alpha3boundpca} we obtain
\begin{align*}
    |\rpca_1 | &\lesssim  \Bigg( \mathcal{F}\pca r \sqrt{\kappa \log(n\vee p)} + \kappa r \log^2(n\vee p) \frac{\sigma^2}{\lambda_j} \bigg( \frac{p}{n} + \sqrt{\frac{p}{n}}\bigg) \Bigg) \sqrt{\sum_{\substack{k\neq j\\k\leq r} } \frac{(\lambda_j + \sigma^2)(\lambda_k + \sigma^2)(\a\t\uk)}{n(\lambda_j - \lambda_k)^2}},\numberthis \label{eq:rpca1}
\end{align*}
which holds with probability at least $1 - O((n\vee p)^{-10})$
\item \textbf{Bounding $\rpca_2$}.  Observe that
\begin{align*}
  \frac{(\uhatj - \uj)\t \big( \sigmahat - \bSigma \big) \uperp \uperp\t\a  }{\lambda_j}   
  &= \frac{(\uhatj - \uj)\t \U\U\t \big( \sigmahat - \bSigma \big) \uperp \uperp\t\a  }{\lambda_j} +  \frac{(\uhatj - \uj)\t \uperp\uperp\t \big( \sigmahat - \bSigma \big) \uperp \uperp\t\a  }{\lambda_j} \\ 
  &=: \beta_1 + \beta_2.
\end{align*}
To bound $\beta_1$, by Cauchy-Schwarz,
    \begin{align*}
        \bigg| \frac{(\uhatj - \uj)\t \U\U\t \big( \sigmahat - \bSigma \big) \uperp \uperp\t\a  }{\lambda_j} \bigg| 
        &\lesssim \frac{\|\uhatj - \uj \|}{\lambda_j} \| \U\t \sigmahat \uperp \uperp\t \a \| \lesssim \frac{\mathcal{F}\pca}{\lambda_j} \| \U\t \sigmahat \uperp \uperp\t \a \|,
    \end{align*}
    where the final inequality holds with probability at least $1 - O((n\vee p)^{-10})$ by \cref{lem:bias_PCA}.
    The following lemma bounds the remaining term  $\|\U\t \sigmahat \uperp \uperp\t \a\|$.
    \begin{lemma}\label{lem:pca4}
    Instate the conditions of \cref{thm:distributionaltheory_MD}.  Then with probability at least $1- O((n\vee p)^{-10})$ it holds that
        \begin{align*}
            \|\U\t \sigmahat \uperp \uperp\t \a\| &\lesssim \frac{(\lambda_{\max} + \sigma^2)^{1/2} \sigma \|\uperp\t \a\| \sqrt{r}\log(n\vee p)}{\sqrt{n}}.
        \end{align*}
    \end{lemma}
    \begin{proof}
        See \cref{sec:pca4proof}.
    \end{proof}
    Therefore, with probability at least $1 - O((n\vee p)^{-10})$ it holds that
    \begin{align*} 
        |\beta_1| &\lesssim \mathcal{F}\pca \sqrt{r}\log(n\vee p) \frac{(\lambda_{\max} + \sigma^2)^{1/2} \sigma \|\uperp\t\a\|}{\lambda_j \sqrt{n}}  \lesssim 
       \mathcal{F}\pca \sqrt{r\kappa } \log(n\vee p) \frac{(\lambda_j + \sigma^2)^{1/2} \sigma\|\uperp\t\a\|}{\lambda_j \sqrt{n}}. 
        \numberthis \label{beta1bound}
    \end{align*}
    As for $\beta_2$, we may invoke \cref{lem:uhatjuperpidentity} via \cref{fact4_PCA} to yield
    \begin{align*}
    |\beta_2| &= \bigg| \frac{(\uhatj - \uj)\t \uperp\uperp\t \big( \sigmahat - \bSigma \big) \uperp \uperp\t\a  }{\lambda_j} \bigg| \\
    &\lesssim   \frac{1}{\lambda_j} \bigg\| \U\t \sigmahat \uperp \big( \hat \lambda_j - \uperp\t \sigmahat \uperp\big)\inv \uperp\t \bigg(  \sigmahat - \bSigma \bigg) \uperp \uperp\t\a \bigg\|.
    \end{align*}
    The following lemma bounds this remaining term.
    \begin{lemma} \label{lem:pca5}
    Instate the conditions of \cref{thm:distributionaltheory_MD}.  Then with probability at least $1- O((n\vee p)^{-10})$ it  holds that \begin{align*}
            \bigg\| \U\t \sigmahat \uperp &\big( \hat \lambda_j - \uperp\t \sigmahat \uperp\big)\inv \uperp\t \bigg(  \sigmahat - \bSigma \bigg) \uperp \uperp\t\a \bigg\| \\
            &\lesssim (\lambda_{\max} + \sigma^2)^{1/2} \frac{\sigma^3 \sqrt{r} \log(n\vee p) \|\uperp\t\a\|}{\lambda_j \sqrt{n}} \bigg( \frac{p}{n} + \sqrt{\frac{p}{n}} + \sqrt{\frac{\log(n\vee p)}{n}} \bigg).
        \end{align*}
    \end{lemma}
    \begin{proof}
        See \cref{sec:pca5proof}.
    \end{proof}
Therefore, we have shown that with probability at least $1 - O((n\vee p)^{-10})$
\begin{align*}
    |\beta_2| &\lesssim \frac{\sigma^2 \sqrt{\kappa r}\log(n\vee p)}{\lambda_j} \bigg( \frac{p}{n} + \sqrt{\frac{p}{n}} + \sqrt{\frac{\log(n\vee p)}{n}}\bigg) \frac{\sigma (\lambda_j + \sigma^2)^{1/2} \|\uperp\t\a\|}{\lambda_j \sqrt{n}}. \numberthis \label{beta2bound}
\end{align*}
As a consequence of \eqref{beta1bound} and \eqref{beta2bound}, it holds with probability at least $1- O((n\vee p)^{-10})$ that 
\begin{align*}
    | \rpca_2 | &\lesssim \sqrt{r\kappa} \log(n\vee p)\Bigg( \mathcal{F}\pca  + \frac{\sigma^2}{\lambda_j} \bigg( \frac{p}{n} + \sqrt{\frac{p}{n}} + \sqrt{\frac{\log(n\vee p)}{n}} \bigg) \Bigg) \frac{(\lambda_j + \sigma^2)^{1/2}\sigma \|\uperp\t\a\|}{\lambda_j \sqrt{n}}.    
    \numberthis \label{rpca2}
\end{align*}
\item \textbf{Bounding $\rpca_3$}.  We decompose via
\begin{align*}
|\rpca_3| &=   \Bigg|   \sum_{k\neq j,k\leq  r } \frac{\tilde \lambda_j - \hat \lambda_j + \gamma(\hat \lambda_j)(\lambda_k + \sigma^2) }{ \lambda_j - \lambda_k } \uhatj\t \uk \uk\t \a  \Bigg| \\
&= \Bigg|   \sum_{k\neq j,k\leq  r } \frac{\tilde \lambda_j - \hat \lambda_j + \gamma(\hat \lambda_j)(\lambda_j + \sigma^2) + \gamma(\hat \lambda_j)(\lambda_k - \lambda_j) }{ \lambda_j - \lambda_k } \uhatj\t \uk \uk\t \a  \Bigg| \\
&\leq  \big| \tilde \lambda_j - \hat \lambda_j + \gamma(\hat \lambda_j)(\lambda_j + \sigma^2) \big|  \bigg|   \sum_{k\neq j,k\leq  r } \frac{1 }{ \lambda_j - \lambda_k } \uhatj\t \uk \uk\t \a  \bigg|   + \gamma(\hat \lambda_j)  \sum_{k\neq j ,k\leq r} \big| \uhatj\t \uk \uk\t\a \big|\\
&=: \eta_1 + \eta_2.
\end{align*}
Appealing directly to the eigenvalue bounds in \cref{lem:eigenvalueconcentration_PCA}, it holds that
    \begin{align*}
        |\eta_1| &\lesssim \delta\pca \sum_{\substack{k\neq j\\k\leq r}} \frac{|\uk\t \a|}{|\lambda_j - \lambda_k|} | \uhatj\t \uk |,
        \end{align*}
        where $\delta\pca$ is defined in \eqref{deltapcadef}.  
        Therefore, it suffices to provide a bound on $\uhatj\t \uk$ for $k\neq j$, which is accomplished via the following lemma.
    \begin{lemma}
    \label{lem:innerproduct_PCA}
 Instate the conditions in \cref{thm:distributionaltheory_MD}.   Then simultaneously for all $k\neq j$ with $k \leq r$, with probability at least $1 - O((n\vee p)^{-9})$ it holds that \begin{align*}
        | \uhatj\t \uk | &\lesssim \sqrt{\frac{r\log(n\vee p)}{n}} \frac{(\lambda_{\max} + \sigma^2)^{1/2} (\lambda_k + \sigma^2)^{1/2}}{|\lambda_k - \lambda_j|}.
    \end{align*}.
\end{lemma}
\begin{proof}
    See \cref{sec:innerproduct_PCA_proof}.
\end{proof}
As a result of this lemma,
\begin{align*}
        |\eta_1| 
        &\lesssim \frac{\delta\pca (\lambda_{\max} + \sigma^2)^{1/2}\sqrt{r\log(n\vee p)}}{\sqrt{n}\Delta_j}  \sum_{\substack{k\neq j\\k\leq r}} \frac{|\uk\t \a|(\lambda_k + \sigma^2)^{1/2}}{|\lambda_j - \lambda_k|} \\
              &\lesssim  \frac{\delta\pca r \sqrt{\kappa \log(n\vee p)}}{\Delta_j} \sqrt{\sum_{\substack{k\neq j\\k\leq r}} \frac{(\lambda_j + \sigma^2)(\lambda_k + \sigma^2) (\uk\t\a)^2}{n(\lambda_j - \lambda_k)^2}}, \numberthis \label{eta1bound}
    \end{align*}
    with probability at least $1 - O((n\vee p)^{-9})$.
    
For $\eta_2$,  \cref{lem:eigenvalueconcentration_PCA} guarantees that $\gamma\pca(\hat \lambda_j) \lesssim \frac{\sigma^2 p}{\lambda_j n}$ with probability at least $1 - O((n\vee p)^{-10})$.
Consequently,
\begin{align*}
    |\eta_2| &\lesssim \frac{\sigma^2 p}{\lambda_j n} \sum_{\substack{k\neq j\\k\leq r}} |\uk\t\a| | \uhatj\t \uk | \\
  &\lesssim   \frac{\sigma^2 p}{\lambda_j n} \frac{(\lambda_{\max} + \sigma^2 )^{1/2}\sqrt{r\log(n\vee p)}}{\sqrt{n}} \sum_{\substack{k\neq j\\k\leq r}} \frac{(\lambda_k + \sigma^2)^{1/2}|\uk\t\a|}{|\lambda_j - \lambda_k|}\\
  &\lesssim  r \sqrt{\kappa \log(n\vee p)}  \frac{\sigma^2 p  }{\lambda_j n} \sqrt{\sum_{\substack{k\neq j\\k\leq r}} \frac{(\lambda_j + \sigma^2)(\lambda_k + \sigma^2) (\uk\t\a)^2}{n(\lambda_j - \lambda_k)^2}}, \numberthis \label{eta2bound}
\end{align*}
where we have implicitly invoked \cref{lem:innerproduct_PCA} in the first line. 

Therefore, combining \eqref{eta1bound} and \eqref{eta2bound} we have that 
\begin{align*}
|\rpca_3| &\lesssim r \sqrt{\kappa\log(n\vee p)} \Bigg( \frac{\sigma^2 p}{\lambda_j n} + \frac{\delta\pca}{\Delta_j} \bigg) \sqrt{\sum_{\substack{k\neq j\\k\leq r}} \frac{(\lambda_j + \sigma^2)(\lambda_k + \sigma^2) (\uk\t\a)^2}{n(\lambda_j - \lambda_k)^2}}
\numberthis \label{rpca3}
\end{align*}
with probability at least  $1- O((n\vee p)^{-9})$. 
\item \textbf{Bounding $\rpca_4$}. By Weyl's inequality,  \cref{lem:eigenvalueconcentration_PCA}, \cref{lem:uhatjuperpidentity}, and \cref{fact4_PCA} it holds that
\begin{align*}
  |\rpca_4| =   \bigg|\frac{\tilde \lambda_j - \hat \lambda_j }{\lambda_j  } \uhatj\t \uperp \uperp\t \a \bigg| 
    &\lesssim \frac{\mathcal{E}\pca}{\lambda_j} \bigg\| \U\t \sigmahat \uperp \big( \hat \lambda_j \bm{I}_{p-r} - \uperp\t \sigmahat \uperp \big)\inv  \uperp\t \a \bigg\|.
\end{align*}
The following lemma provides a bound for this remaining quantity. 
\begin{lemma} \label{lem:pca6}
Instate the conditions in \cref{thm:distributionaltheory_PCA}.  Then with probability at least $1 - O((n\vee p)^{-10})$ it holds that
\begin{align*}
     \bigg\| \U\t \sigmahat \uperp \big( \hat \lambda_j \bm{I}_{p-r} - \uperp\t \sigmahat \uperp \big)\inv  \uperp\t \a \bigg\| &\lesssim \frac{(\lambda_{\max} + \sigma^2)^{1/2} \sigma \sqrt{r}\log(n\vee p)}{\sqrt{n}\lambda_j} \|\uperp\t\a\|.
\end{align*}
\end{lemma}
\begin{proof}
     The proof follows \emph{mutatis mutandis} the proof of \cref{lem:pca5} only with the replacement $\uperp\t\a$ instead of $ \uperp\t \big( \sigmahat - \bSigma \big) \uperp \uperp\t\a.$
\end{proof}
As a consequence of this lemma, with probability at least $1- O((n\vee p)^{-10})$ it holds that
\begin{align*}
    | \rpca_4 | &\lesssim  \frac{\mathcal{E}\pca \sqrt{r\kappa \log(n\vee p)}}{\lambda_j} \frac{(\lambda_j + \sigma^2)^{1/2} \sigma}{\sqrt{n}\lambda_j} \|\uperp\t\a\|. \numberthis \label{eq:rpca4}
\end{align*}
\end{itemize}
We now complete the proof. By \eqref{eq:rpca1}, \eqref{rpca2}, \eqref{rpca3}, \eqref{eq:rpca4}  with probability at least $1- O((n\vee p)^{-9})$ it holds that
\begin{align*}
    |\rpca_1& + \rpca_2 + \rpca_3 + \rpca_4 | \\
    &\lesssim  \Bigg( \mathcal{F}\pca r \sqrt{\kappa \log(n\vee p)} + \kappa r \log^2(n\vee p) \frac{\sigma^2}{\lambda_j} \bigg( \frac{p}{n} + \sqrt{\frac{p}{n}}\bigg) \Bigg) \sqrt{\sum_{\substack{k\neq j\\k\leq r} } \frac{(\lambda_j + \sigma^2)(\lambda_k + \sigma^2)(\a\t\uk)}{n(\lambda_j - \lambda_k)^2}} \\
    &\quad + \sqrt{r\kappa} \log(n\vee p)\Bigg( \mathcal{F}\pca  + \frac{\sigma^2}{\lambda_j} \bigg( \frac{p}{n} + \sqrt{\frac{p}{n}} + \sqrt{\frac{\log(n\vee p)}{n}} \bigg) \Bigg) \frac{(\lambda_j + \sigma^2)^{1/2}\sigma \|\uperp\t\a\|}{\lambda_j \sqrt{n}} \\
    &\quad + r \sqrt{\kappa\log(n\vee p)} \Bigg( \frac{\sigma^2 p}{\lambda_j n} + \frac{\delta\pca}{\Delta_j} \bigg) \sqrt{\sum_{\substack{k\neq j\\k\leq r}} \frac{(\lambda_j + \sigma^2)(\lambda_k + \sigma^2) (\uk\t\a)^2}{n(\lambda_j - \lambda_k)^2}} \\
    &\quad + 
    \frac{\mathcal{E}\pca \sqrt{r\kappa \log(n\vee p)}}{\lambda_j} \frac{(\lambda_j + \sigma^2)^{1/2} \sigma}{\sqrt{n}\lambda_j} \|\uperp\t\a\|
    \\
    &\lesssim \kappa r \log^2( n\vee p) \Bigg(  \frac{\sigma^2}{\lambda_j} \bigg( \frac{p}{n} + \sqrt{\frac{p}{n}} + \sqrt{\frac{\log(n\vee p)}{n}}\bigg) + \mathcal{F}\pca + \frac{\mathcal{E}\pca}{\lambda_j} + \frac{\delta\pca}{\Delta_j} \Bigg) s_{\a,j}\pca \\
       &=: {\sf \tilde{ErrPCA}} \times s_{\a,j}\pca, 
\end{align*}
where $s_{\a,j}\pca$ takes the form
\begin{align*}
    (s\pca_{\a,j})^2 &= \sum_{\substack{k\neq j\\k\leq r}} \frac{(\lambda_j + \sigma^2)(\lambda_k + \sigma^2)}{n(\lambda_j - \lambda_k)^2} (\uk\t \a)^2 + \frac{(\lambda_j + \sigma^2)\sigma^2}{n\lambda_j^2} \|\uperp\t \a\|^2.
\end{align*}
We now bound ${\sf \tilde{ErrPCA}}$.  Plugging in the definition of $\mathcal{F}\pca$ from \eqref{fpcadef}, $\delta\pca$ from \eqref{deltapcadef}, and $\mathcal{E}\pca$ from \eqref{epcadef}, we see that
\begin{align*}
    {\sf \tilde{ErrPCA}} &\asymp \kappa r \log^2( n\vee p) \Bigg(  \frac{(\lambda_{\max} + \sigma^2)^{1/2} (\lambda_j +\sigma^2)^{1/2} \sqrt{r\log(n\vee p)}}{\Delta_j \sqrt{n}} + \frac{(\lambda_{\max} + \sigma^2) \sqrt{r}\log(n\vee p)}{\sqrt{n}\Delta_j} \Bigg) \\
    &\quad + \kappa r \log^2( n\vee p) \Bigg( \Bigg(  \frac{\sigma^2}{\lambda_j} \bigg( \frac{p}{n} + \sqrt{\frac{p}{n}} + \sqrt{\frac{\log(n\vee p)}{n}}\bigg) + \frac{(\lambda_j + \sigma^2)^{1/2} \sigma \sqrt{p}\log(n\vee p)}{\lambda_j \sqrt{n}} + \frac{\mathcal{E}\pca}{\lambda_j} \Bigg)  \\
     &\lesssim \frac{(\lambda_{\max} + \sigma^2) \kappa r^{3/2} \log^{5/2}(n\vee p)}{\Delta_j \sqrt{n}} + \frac{\kappa^2 r^{3/2} \log^{5/2}(n\vee p)}{\sqrt{n}} \\
    &\quad + \kappa^{3/2} r \log^3(n\vee p) \Bigg( \frac{\sigma^2}{\lambda_j} \bigg( \frac{p}{n} + \sqrt{\frac{p}{n}} + \sqrt{\frac{\log(n\vee p)}{n}} \bigg) +  \frac{\sigma}{\sqrt{\lambda_j}} \sqrt{\frac{p}{n}} \Bigg) \\
   & =: {{\sf ErrPCA}},
\end{align*}
which matches the bound in \cref{thm:distributionaltheory_PCA}.

\subsection{Completing the Proofs of \cref{thm:distributionaltheory_PCA,thm:civalidity_pca}}  
We have shown thus far that with probability at least $1- O( (n\vee p)^{-9})$ that
\begin{align*}
    \frac{1}{s_{\a,j}\pca} \bigg( \a\t \uhatj - \a\t \uj \uj\t \uhatj \bigg) =  \frac{1}{s_{\a,j}\pca}\bigg( \sum_{k\neq j} \frac{\uj\t \big(\sigmahat - \bSigma \big) \uk}{ \lambda_j -  \lambda_k} \uk\t \a \bigg) + {\sf ErrPCA}.
\end{align*}
We now study this leading-order term.  Recalling that $\bm{X}_i = \bSigma^{1/2} \bm{Y}_i$, where $\bm{Y}_i \sim \mathcal{N}(0,\bm{I}_p)$, we can write $\sigmahat = \bSigma^{1/2} \frac{\y \y\t }{n} \bSigma^{1/2}$.  Therefore, recognizing that the eigenvectors of $\bSigma^{1/2}$ are the same as $\bSigma$,  the leading-order term further decomposes as 
\begin{align*}
    \sum_{k\neq j} \frac{\uj\t \big( \sigmahat - \bSigma \big) \uk}{ \lambda_j - \lambda_k} \uk\t \a 
    &=   \sum_{\substack{k\neq j\\k\leq r} } \uj\t \bigg( \frac{\y \y\t }{n} \bigg) \uk    \frac{ (\lambda_j + \sigma^2)^{1/2}(\lambda_k + \sigma^2)^{1/2}\uk\t \a }{\lambda_j - \lambda_k} \\
    &\quad + \uj\t \bigg( \frac{\y \y\t }{n} \bigg) \uperp \frac{\uperp\t \a (\lambda_j + \sigma^2)^{1/2} \sigma}{\lambda_j} 
\end{align*}
Since $\uj\t \y$ and $\uk\t \y$ are independent for $j \neq k$, it is straightforward to demonstrate that the variance of this term is $(s_{\a,j}\pca)^2$.  
Unlike the matrix denoising setting, the leading-order term is not a Gaussian random variable, so we will demonstrate its approximate Gaussianity via the following lemma.
\begin{lemma} \label{lem:approximategaussian}
    Under the conditions of \cref{thm:distributionaltheory_PCA} it holds that
    \begin{align*}
     \sup_{z \in \mathbb{R}} \bigg| \p\bigg\{ \frac{1}{s_{\a,j}\pca} \sum_{k\neq j} \frac{\uj\t \big( \sigmahat - \bSigma \big) \uk}{ \lambda_j - \lambda_k} \uk\t \a \leq z \bigg\} - \Phi(z) \bigg| &\lesssim  \sqrt{\frac{\log(n\vee p)}{n}}.
    \end{align*}
\end{lemma}
\begin{proof}
    See \cref{sec:approximategaussianproof}.
\end{proof}
Therefore, for any $z \in \mathbb{R}$ it holds that
\begin{align*}
    \bigg| &\p\bigg\{ \frac{1}{s_{\a,j}\pca} \bigg( \a\t\uhatj - \a\t\uj\uj\t\uhatj \bigg) \leq z \bigg\} - \Phi(z) \bigg|\\
    &=   \bigg| \p\bigg\{\frac{1}{s_{\a,j}\pca} \sum_{k\neq j} \frac{\uj\t \big( \sigmahat - \bSigma \big) \uk}{ \lambda_j - \lambda_k} \uk\t \a   \leq z - \frac{1}{s_{\a,j}\pca}  \bigg( \rpca_1 + \rpca_2 + \rpca_3 + \rpca_4 \bigg) \bigg\} - \Phi(z) \bigg| \\
    &\lesssim \sqrt{\frac{\log(n\vee p)}{n}} + {\sf ErrPCA} + (n\vee p)^{-9} \lesssim  {\sf ErrPCA}.
\end{align*}
This completes the proof of \eqref{pca_firstresult}.

We now prove \eqref{pca_secondresult}. Let ${\sf ErrBiasPCA}$ be as in \cref{lem:bias_PCA},  absorbing constant factors as necessary. Observe that through a similar argument as the matrix denoising case,
\begin{align*}
    \bigg| &\p\bigg\{ \frac{1}{s_{\a,j}\pca} \bigg( \a\t\uhatj \sqrt{1 + b_j\pca} - \a\t\uj \bigg) \leq z \bigg\} - \Phi(z) \bigg| \\
    &\leq    \bigg| \p\bigg\{ \frac{1}{s_{\a,j}\pca} \bigg( \a\t\uhatj \sqrt{1 + b_j\pca} - \a\t\uj \sqrt{1 + b_j\pca}  + \a\t \uj \big( \uj\t\uhatj \sqrt{1 + b_j\pca} - 1 \big) \bigg) \leq z \bigg\} - \Phi(z) \bigg| \\
&\lesssim {\sf ErrPCA} + \frac{|\a\t\uj| {\sf ErrBiasPCA}}{s_{\a,j}\pca}.
\end{align*}
The conditions \eqref{atujconditionpca1} and \eqref{atujconditionpca2} imply that the right hand side is $o(1).$
This completes the proof of \cref{thm:distributionaltheory_PCA}. 

We now prove \cref{thm:civalidity_pca}.
First we state the following lemma concerning the estimate of the noise variance.
\begin{lemma} \label{lem:sigmahatapproxpca}
Define $\hat \sigma^2$ as in \eqref{alg:noisevarPCA}. 
    Then with probability at least $1 - O((n\vee p)^{-10})$ it holds that
    \begin{align*}
        \bigg|\frac{\hat \sigma^2}{\sigma^2} - 1 \bigg| &\lesssim \mathcal{E}_{\sigma},
    \end{align*}
where
\begin{align*}
    \mathcal{E}_{\sigma} & \coloneqq \begin{cases}
        \sqrt{r} \kappa\frac{\log^4(n\vee p)}{n} + \frac{\sqrt{r}}{p} \frac{\mathcal{E}\pca}{\lambda_{\min}} & n/\log^4(n\vee p) < p; \\
        \sqrt{\frac{p}{n}} + \sqrt{\frac{\log(p\vee n)}{n}} &  n/\log^4(n\vee p) \geq p.
    \end{cases}
\end{align*}
\end{lemma}

\begin{proof}
    See \cref{sec:sigmahatoversigmaproofpca}.
\end{proof}
Next we study the estimated bias-correction value $\hat{b_k\pca}$. 
\begin{lemma}\label{lem:biashatpca}
Suppose that $k \leq r$.  If $p > n$, with probability at least $1 - O((n\vee p)^{-10})$ it holds that
   \begin{align*}
    \bigg|   \sqrt{1+ \hat{b_k\pca}} - \sqrt{1 + b_k\pca} \bigg| &\lesssim \frac{\sigma^2 p}{\lambda_{\min} n} \mathcal{E}_{\sigma}, 
   \end{align*}
   with $\mathcal{E}_{\sigma}$ defined in \cref{lem:sigmahatapproxpca}.  
   If $n \geq p$ then $\hat{b_k\pca} = b_k\pca$.  
\end{lemma}
\begin{proof}
    See \cref{sec:biashatpcaproof}.
\end{proof}
In what follows, let
\begin{align*}
    {\sf ErrBiasApproxPCA} \coloneqq \begin{cases}
        \frac{\sigma^2 p}{\lambda_{\min} n}  \mathcal{E}_{\sigma}
        & p > n; \\
        0 & p\leq n,
    \end{cases} \numberthis \label{errbiasapproxpcadef}
\end{align*}
absorbing constants if necessary. 
As a consequence of \cref{lem:biashatpca}, with probability at least $1 - O( (n\vee p)^{-9})$,
\begin{align*}
   \bigg| \frac{\a\t\uhatj\sqrt{1 + \hat{b_j\pca}} - \a\t\uj}{s\pca_{\a,j}} -  \frac{\a\t\uhatj\sqrt{1 +  b_j\pca} - \a\t\uj}{s\pca_{\a,j}}\bigg| &\lesssim \frac{|\a\t\uhatj|}{s\pca_{\a,j}} {\sf ErrBiasApproxPCA}.
\end{align*}
Our previous analysis implies that
\begin{align*}
    |\a\t \uhatj| \lesssim |\a\t \uj| + \sqrt{\log(n\vee p)} s_{\a,j}\pca + {\sf ErrPCA} \times s\pca_{\a,j} \lesssim |\a\t \uj| + \sqrt{\log(n\vee p)} s_{\a,j}\pca,
\end{align*}
with probability at least $1 - O((n\vee p)^{-8})$, provided that ${\sf ErrPCA} \lesssim \sqrt{\log(n\vee p)}$, which holds by assumption  Therefore, with this same probability, 
\begin{align*}
      \bigg| \frac{\a\t\uhatj\sqrt{1 + \hat{b_k\pca}} - \a\t\uj}{s\pca_{\a,j}}& -  \frac{\a\t\uhatj\sqrt{1 +  b_k\pca} - \a\t\uj}{s\pca_{\a,j}}\bigg| \\
      &\lesssim \frac{|\a\t\uj|}{s\pca_{\a,j}} {\sf ErrBiasApproxPCA} + \sqrt{\log(n\vee p)} {\sf ErrBiasApproxPCA}.
\end{align*}
We next consider the approximated variance. 
The following result establishes the proximity of the estimated variance $\hat{s_{\a,j}\pca}$ to $s_{\a,j}\pca$.  

\begin{lemma} \label{lem:sapproxpca}
 Under the conditions of \cref{thm:civalidity_pca}, with probability at least $1 - O((n\vee p)^{-8})$ it holds that 
 \begin{align*}
       \bigg| \frac{\hat{s_{\a,j}\pca}}{s_{\a,j}\pca} - 1 \bigg| \ll  \frac{1}{\sqrt{\log(n\vee p)}}.
 \end{align*}
\end{lemma}
\begin{proof}
    See \cref{sec:sapproxpca_proof}.
\end{proof}
Let the error from \cref{lem:sapproxpca} be denoted ${\sf ErrCIPCA}$.  Then with probability at least $1 - O((n\vee p)^{-8})$,
\begin{align*}
      \bigg| &\frac{ \a\t \uhatj \sqrt{1 + \hat{b_j\pca}} - \a\t \uj}{\hat{s\pca_{\a,j}}} -\frac{ \a\t \uhatj \sqrt{1 +  b_j\pca} - \a\t \uj}{s\pca_{\a,j}} \bigg| \\
     &\lesssim  \bigg| \frac{ \a\t \uhatj \sqrt{1 + \hat{b_j\pca}} - \a\t \uj}{s\pca_{\a,j}} -\frac{ \a\t \uhatj \sqrt{1 +  b_j\pca} - \a\t \uj}{s\pca_{\a,j}} \bigg|  \\
     &\quad + \bigg| \frac{ \a\t \uhatj \sqrt{1 + \hat{b_j\pca}} - \a\t \uj}{s\pca_{\a,j}}\bigg(1 - \frac{s\pca_{\a,j}}{\hat{s\pca_{\a,j}}} \bigg) \bigg|  \\
       &\lesssim \frac{|\a\t\uj|}{s\pca_{\a,j}} {\sf ErrBiasApproxPCA}  + \sqrt{\log(n\vee p)} {\sf ErrBiasApproxPCA}  + {\sf ErrCIPCA} \sqrt{\log(n\vee p)}, \numberthis \label{boundedboy2}
\end{align*}
as long as ${\sf ErrCIPCA} = o(1)$, which we will verify at the end of the proof.  Here we have used the fact that the analysis leading to \cref{thm:distributionaltheory_PCA} implies that
\begin{align*}
    \bigg| \frac{\a\t\uhatj \sqrt{1 + b_j\pca} - \a\t \uj}{s\pca_{\a,j}} \bigg| \lesssim \sqrt{\log(n\vee p)}
\end{align*}
with probability at least $1 - O((n\vee p)^{-8})$ as long as each of the quantities in \cref{thm:distributionaltheory_PCA} are $o(1)$, which is straightforward to verify directly from our noise assumption \eqref{noisecondpca:inf} and our eigengap assumption \eqref{eigengapcondpca:inf}.  

With these results stated we are prepared to prove \cref{thm:civalidity_pca}.   For a given $z \in \mathbb{R}$, by a similar analysis to the proof of \cref{thm:civalidity_MD},
\begin{align*}
\Bigg|  \p\bigg\{ \frac{\a\t \uhatj \sqrt{1 + \hat{b_j\pca}} - \a\t \uj}{\hat{s\pca_{\a,j}}} \leq z \bigg\} -\Phi(z) \Bigg| 
&\lesssim {\sf ErrPCA} +  \frac{|\a\t \uj| {\sf ErrBiasPCA}}{s_{\a,j}\pca}+\bigg[\frac{|\a\t\uj|}{s\pca_{\a,j}}   + \sqrt{\log(n\vee p)} \bigg]{\sf ErrBiasApproxPCA}  \\
&\quad + {\sf ErrCIPCA} \sqrt{\log(n\vee p)}.
\end{align*}
 The proof of \cref{thm:civalidity_pca} immediately follows by taking $\pm z = \Phi\inv(1-\alpha/2)$, provided that each term on the right hand side above is $o(1)$.    It is immediate from \cref{lem:sapproxpca} that ${\sf ErrCIPCA} \sqrt{\log(n\vee p)} = o(1)$, and the condition \eqref{atujconditionpca1} and \eqref{atujconditionpca2} imply that $\frac{|\a\t\uj|}{s_{\a,j}\pca} {\sf ErrBiasPCA} = o(1)$.  
The proof of \cref{lem:sapproxpca} shows that $\sqrt{\log(n\vee p)} {\sf ErrBiasApproxPCA} = o(1)$.  Furthermore, if $p \leq n$, then ${\sf ErrBiasApproxPCA} = 0$, and hence it suffices to show that when $p > n$
 \begin{align*}
     \frac{|\a\t\uj|}{s_{\a,j}\pca} \frac{\sigma^2 p}{\lambda_{\min} n} \bigg( \frac{\sqrt{r} \kappa \log^4(n\vee p)}{n} + \frac{\sqrt{r}}{p} \frac{\mathcal{E}\pca}{\lambda_{\min}} \bigg) = o(1). \numberthis \label{atujcond3}
 \end{align*}
From the definition of $\mathcal{E}\pca$ in \eqref{fact1_PCA}, when $p > n$, we have
\begin{align*}
    \frac{\mathcal{E}\pca}{\lambda_{\min}} \lesssim \kappa \sqrt{\frac{r\log(p)}{n}} + \frac{\sqrt{\kappa} \sigma}{\sqrt{\lambda_{\min}}} \sqrt{\frac{p}{n}} \log(n\vee p) + \frac{\sigma^2}{\lambda_{\min}} \frac{p}{n} \log(n\vee p).
\end{align*}
It can be shown that since $p > n$ that the condition \eqref{atujconditionpca1} is stronger than \eqref{atujcond3}.  This completes the proof.

\subsection{Proofs of Additional PCA Lemmas} \label{sec:pca_lemmas_proofs}
In this section we prove all the additional lemmas required en route to the proof of \cref{thm:distributionaltheory_PCA,thm:civalidity_pca}.

\subsubsection{Proof of Lemma \ref{lem:pca1}}
\label{sec:pca_1_proof}
\begin{proof}
First we note that
\begin{align*}
    \bigg\| \sum_{\substack{k\neq j\\k\leq r}} \frac{\U\t \big(\sigmahat - \bSigma)\uk}{\lambda_j - \lambda_k} \a\t \uk \bigg\|  &\leq \sqrt{r}   \bigg\| \sum_{\substack{k\neq j\\k\leq r}} \frac{\U\t \big(\sigmahat - \bSigma)\uk}{\lambda_j - \lambda_k} \a\t \uk \bigg\|_{\infty}.
\end{align*}
Therefore, we will bound each of the $r$ entries and take a union bound.  

Fix an index $l$.  Note that $\U\t \sigmahat = \frac{1}{n} \lamtilde^{1/2} \U\t \y \y\t \bSigma^{1/2}$.  Therefore, $$\bm{e}_l\t \U\t \sigmahat  = \frac{1}{n} (\lambda_l + \sigma^2)^{1/2} \bm{u}_l\t \y \y\t \bSigma^{1/2}.$$ 
As consequence,
\begin{align*}
   \bm{e}_l\t  \sum_{\substack{k\neq j\\k\leq r}} \frac{\U\t \big(\sigmahat - \bSigma)\uk}{\lambda_j - \lambda_k} \a\t \uk &= \bm{e}_l\t  \lamtilde^{1/2} \U\t \frac{1}{n} \sum_{\substack{k\neq j\\k\leq r}} \sum_{i=1}^{n} \frac{\big(\y_i \y_i\t  -\bm{I}_{p}\big) \uk }{\lambda_j - \lambda_k }\a\t \uk (\lambda_k + \sigma^2)^{1/2} \\
      &=  \frac{1}{n}\sum_{i=1}^{n}\sum_{\substack{k\neq j\\k\leq r}}\big(\bm{u}_l\t \y_i\y_i\t \uk  - \mathbb{I}_{l=k} \big) \frac{(\lambda_l + \sigma^2)^{1/2}(\lambda_k + \sigma^2)^{1/2} \a\t \uk}{\lambda_j - \lambda_k }.
\end{align*}
Observe that the expression above is a sum of independent mean-zero sub-exponential random variables.  Let 
\begin{align*}
    \bm{S}_i &= \sum_{\substack{k\neq j\\k\leq r}} \big( \bm{u}_l\t \y_i\y_i\t \uk  - (\lambda_k + \sigma^2)^{1/2} \big) \frac{(\lambda_l + \sigma^2)^{1/2}(\lambda_k + \sigma^2)^{1/2}\bm \a\t \uk}{\lambda_j - \lambda_k }.
\end{align*}
To apply Bernstein's inequality, we need to study the $\psi_1$ norm of each $\bm{S}_i$, which by independence are all the same.  We split this into two cases based on whether or not $k = l$.
\begin{itemize}
    \item \textbf{Case 1: $k \neq l$}. For a given $l$ we note $\bm{u}_l\t \y_i$ is independent from $\uk\t \y_i$ by rotational invariance of Gaussian random variables, and moreover, $\uk\t\y_i$ and $\bm{u}_l\t\y_i$ are standard Gaussian random variables.  Consequently, 
    \begin{align*}
    \| \bm{S}_i \|_{\psi_1} &=         \bigg\| \sum_{k\neq ,k\leq r} \bm{u}_l\t \y_i \y_i\t \uk \frac{(\lambda_k + \sigma^2)^{1/2}(\lambda_l + \sigma^2)^{1/2} \a\t \uk}{\lambda_j - \lambda_k }  \bigg\|_{\psi_1}  \\
        &\leq \|  \bm{u}_l\t \y_i\|_{\psi_2 } \bigg\| \sum_{k\neq j,k\neq l,k\leq r}  \y_i\t \uk \frac{(\lambda_k + \sigma^2)^{1/2}(\lambda_l + \sigma^2)^{1/2} \a\t \uk}{\lambda_j - \lambda_k }  \bigg\|_{\psi_2} \\
    &\lesssim \sqrt{ \sum_{k\neq j,k\neq l,k\leq r} \frac{(\lambda_k + \sigma^2) (\lambda_l + \sigma^2) (\a\t\uk)^2}{(\lambda_j - \lambda_k)^2}} \\
        &\lesssim (\lambda_{\max} + \sigma^2)^{1/2} \sqrt{ \sum_{\substack{k\neq j\\k\leq r}} \frac{(\lambda_k + \sigma^2) (\a\t\uk)^2}{(\lambda_j - \lambda_k)^2}},
    \end{align*}
    where the penultimate line uses the standard property for sums of independent subgaussian random variables (Proposition 2.6.1 of \citet{vershynin_high-dimensional_2018}).  
    \item \textbf{Case 2: $k = l$.} Observe that the summation is a sum of independent sub-exponential random variables.  Furthermore, it is straightforward to observe that $\|\uk\t\y_i\y_i\t \uk - 1 \|_{\psi_1} \lesssim 1$.  Therefore,
    \begin{align*}
        \bigg\| \sum_{\substack{k\neq j\\k\leq r}}& \bigg( \uk\t\y_i \y_i\t \uk - 1 \bigg) \frac{(\lambda_k + \sigma^2)\a\t\uk}{\lambda_j - \lambda_k} \bigg\|_{\psi_1} \\
    &\lesssim \max_{\substack{k\neq j\\k\leq r}} \bigg\| \uk\t\y_i \y_i\t \uk - 1  \bigg\|_{\psi_1} \sum_{\substack{k\neq j\\k\leq r}} \bigg|  \frac{(\lambda_k + \sigma^2)\a\t\uk}{\lambda_j - \lambda_k} \bigg|\\
         &\lesssim  \sum_{\substack{k\neq j\\k\leq r}}  \bigg| \frac{(\lambda_k + \sigma^2)\a\t\uk}{\lambda_j - \lambda_k} \bigg|\\
      &\lesssim  \sqrt{r} (\lambda_{\max} + \sigma^2)^{1/2} \sqrt{\sum_{\substack{k\neq j\\k\leq r}} \frac{(\lambda_k + \sigma^2)(\a\t\uk)^2}{(\lambda_j - \lambda_k)^2}},
    \end{align*}
    where the final inequality is due to Cauchy-Schwarz.  
\end{itemize}
Combining these inequalities, we have that
\begin{align*}
    \| \bm{S}_i \|_{\psi_1} &\lesssim \sqrt{r} (\lambda_{\max} + \sigma^2)^{1/2} \sqrt{\sum_{\substack{k\neq j\\k\leq r}} \frac{(\lambda_k + \sigma^2)(\a\t\uk)^2}{(\lambda_j - \lambda_k)^2}} =: K.
\end{align*}
Therefore, by Bernstein's inequality, it holds that
\begin{align*}
    \p\bigg\{ \bigg|\frac{1}{n} \sum_{i=1}^{n} \bm{S}_i \bigg| \geq t K  \bigg\} &\leq 2 \exp\bigg( -c n \min\{ t^2, t\} \bigg).
\end{align*}
Therefore,  with $t = C \frac{ \sqrt{\log(n\vee p)}}{\sqrt{n}}$, together with the assumption $\log(p) \lesssim n$, we obtain that
\begin{align*}
\bigg| \frac{1}{n} \sum_{i=1}^{n} \bm{S}_i \bigg| &\lesssim \frac{\sqrt{r\log(n\vee p)}  (\lambda_{\max} + \sigma^2)^{1/2}}{\sqrt{n}} \sqrt{\sum_{\substack{k\neq j\\k\leq r}} \frac{(\lambda_k + \sigma^2)(\a\t\uk)^2}{(\lambda_j - \lambda_k)^2}} .
\end{align*}
with probability at least $1- O((n\vee p)^{-11})$. As a consequence, taking a union bound over at most $(n\vee p)$ entries completes the proof.
\end{proof}

\subsubsection{Proof of Lemma \ref{lem:pca4}} \label{sec:pca4proof}
\begin{proof}
The proof is by $\eps$-net.  Let  $\bm{x} \in \mathbb{R}^{r}$ be a fixed unit vector.  Note that
\begin{align*}
    \bm{x}\t \U\t \sigmahat \uperp \uperp\t\a &= \frac{\sigma}{n}\bm{x}\t \lamtilde \U\t \y \y\t \uperp \uperp\t \a.
\end{align*}
Note that by rotational invariance $\U\t\y$ and $\uperp\t \y$ are independent and hence equal in distribution to $\bm{W} \in \mathbb{R}^{r \times n}$ and $\bm{Z} \in \mathbb{R}^{p-r \times n}$, where $\bm{W}$ and $\bm{Z}$ have independent standard Gaussian entries.  Furthermore,
\begin{align*}
    \bm{x}\t \lamtilde^{1/2} \bm{W} \sim \mathcal{N}(0, \| \lamtilde \bm{x}\|^2 \bm{I}_n ); \qquad \a\t\uperp \bm{Z} \sim \mathcal{N}(0, \|\uperp\t\a\|^2 \bm{I}_n).  
\end{align*}
Therefore, conditional on $\bm{Z}$, Hoeffding's inequality implies that
\begin{align*}
\bigg|    \frac{\sigma}{n} \bm{x}\t \lamtilde^{1/2} \bm{W} \bm{Z}\t \uperp\t \a \bigg| 
&\lesssim \frac{\sigma (\lambda_{\max} +\sigma^2)^{1/2}}{n} \| \bm{Z}\t \uperp\t\ \a\| \sqrt{r\log(n\vee p)}
\end{align*}
with probability at least $1 - \exp(- c r \log(n\vee p))$. Therefore, by unfixing $\bm{x}$, a standard $\eps$-net argument implies that conditional on $\bm{Z}$ it holds that
\begin{align*}
    \frac{\sigma}{n} \| \lamtilde^{1/2} \bm{W} \bm{Z}\t \uperp\t\a \| &\lesssim \frac{\sigma (\lambda_{\max} + \sigma^2)^{1/2}}{n} \sqrt{r\log(n\vee p)} \| \bm{Z}\t \uperp\t \a\|.
\end{align*}
The result is completed as $\|\bm{Z}\t \uperp\t\a \| \lesssim \sqrt{n\log(n\vee p)} \|\uperp\t\a\|$ with probability at least $1 - O((n\vee p)^{-10})$.  
\end{proof}

\subsubsection{Proof of Lemma \ref{lem:pca5}} \label{sec:pca5proof}
\begin{proof}
First we note that
 \begin{align*}
     \bigg\|& \U\t \sigmahat \uperp \big(  \hat \lambda_j \bm{I}_{p-r} - \uperp\t \sigmahat \uperp\big)\inv \uperp\t \big( \sigmahat -\bSigma \big) \uperp \uperp\t\a \bigg\| \\
    &\lesssim \frac{\sqrt{r} (\lambda_{\max} + \sigma^2)^{1/2}}{\sqrt{n}} \\&\quad 
    \times \sup_{\lambda \in [2(\lambda_j + \sigma^2)/3, 4(\lambda_j + \sigma^2)/3]} \bigg\| \U\t \y \frac{\y\t\uperp}{\sqrt{n}}  \bigg( \lambda \bm{I}_{p-r} - \frac{\uperp\t \y\y\t\uperp}{n} \bigg)\inv \bigg( \frac{\uperp\t \y\y\t\uperp}{n} - \sigma^2 \bm{I}_{p-r} \bigg) \uperp\t\a \bigg\|_{\infty}, \numberthis \label{bigboi}
 \end{align*}
 where we have implicitly applied \cref{fact3_PCA}.   
Therefore, it suffices to bound the quantity above which we will achieve by bounding each entry and then taking a union bound over all entries.  To bound each entry, similar to the proof of \cref{lem:residual2_MD} in \cref{sec:residual2proof} we will first derive a concentration inequality for fixed $\lambda$ and then unfix $\lambda$.  
\\ \ \\
\noindent 
\textbf{Step 1: Bounding each entry for fixed $\lambda$}. Let $\bm{W}$ be a $r \times n$ matrix of i.i.d. $\mathcal{N}(0,1)$ random variables, and let $\bm{Z}$ be a $p-r \times n$ matrix with i.i.d. $\mathcal{N}(0, \sigma^2)$ random variables. Then it holds that
\begin{align*}
\U\t \y& \frac{\y\t \uperp}{\sqrt{n}} \bigg( \lambda \bm{I}_{p-r} - \frac{\uperp\t \y\y\t\uperp}{n} \bigg)\inv \bigg( \frac{\uperp\t \y\y\t\uperp}{n} - \sigma^2 \bm{I}_{p-r} \bigg) \uperp\t\a \\
&\overset{d}{=}
    \bm{W}\t \frac{\bm{Z}}{\sqrt{n}} \bigg( \lambda \bm{I}_{p-r} - \frac{\bm{ZZ}\t}{n} \bigg)\inv \bigg( \frac{\bm{ZZ}\t}{n} - \sigma^2 \bm{I}_{p-r}\bigg) \uperp\t\a ,
\end{align*}
where $\overset{d}{=}$ denotes equality in distribution.  Suppose $\bm{Z}/\sqrt{n}$ has SVD $\bm{U}^{(\bm{Z})} \sqrt{\bm{\Gamma}^{(\bm{Z})}}(\bm{V}^{(\bm{Z})})\t$.  Then it holds that
\begin{align*}
     \bm{W}\t &\frac{\bm{Z}}{\sqrt{n}} \bigg( \lambda \bm{I}_{p-r} - \frac{\bm{ZZ}\t}{n} \bigg)\inv \bigg( \frac{\bm{ZZ}\t}{n} - \sigma^2 \bm{I}_{p-r}\bigg) \uperp\t\a \\
    &\overset{d}{=} \bm{W}\t  \bm{\Gamma}^{(\bm{Z})}  \bigg( \lambda \bm{I}_{p-r} - \bm{\Gamma}^{(\bm{Z})}\bigg)\inv  \bigg( \bm{\Gamma}^{(\bm{Z})} - \sigma^2 \bm{I}_{p-r} \bigg) (\bm{V}^{(\bm{Z})})\t \uperp\t\a,
\end{align*}
where we have used the fact that $\bm{V}^{(\bm{Z})}$ and $\bm{\Gamma}^{(\bm{Z})}$, $\bm{\U}^{(\bm{Z})}$ are independent and, in addition, $\bm{W}\t \bm{U}^{(\bm{Z})} \overset{d}{=}\bm{W}$ by rotational invariance.  Consequently, the terms $\bm{W},\bm{\Gamma}^{(\bm{Z})}$, and $\bm{V}^{(\bm{Z})}$ are mutually independent.

Observe that the $l$'th entry can be written via
\begin{align*}
    \sum_{l'=1}^{n} \bm{W}_{l' l} \Bigg[& \bm{\Gamma}^{(\bm{Z})}  \bigg( \lambda \bm{I}_{p-r} - \bm{\Gamma}^{(\bm{Z})}\bigg)\inv  \bigg( \bm{\Gamma}^{(\bm{Z})} - \sigma^2 \bm{I}_{p-r} \bigg) (\bm{V}^{(\bm{Z})})\t \uperp\t\a \bigg]_{l'} \\
    &= \sum_{l'=1}^{\min(n,p)} \bm{W}_{l'l} \frac{(\bm{\Gamma}^{(\bm{Z})}_{l'l'})^{1/2}(\bm{\Gamma}^{(\bm{Z})}_{l'l'} - \sigma^2)}{\lambda - \bm{\Gamma}^{(\bm{Z})}_{l'l'}} \bigg(\big( \bm{V}^{(\bm{Z})}\big)\t \uperp\t\a \bigg)_{l'}
\end{align*}
where we have used the fact that $\bm{\Gamma}^{(\bm{Z})}$ is diagonal.  Let $\bm{w}$ denote the vector with entries given by
\begin{align*}
    \bm{w}_{l'} &= \bm{W}_{l'l} \frac{(\bm{\Gamma}^{(\bm{Z})}_{l'l'})^{1/2}(\bm{\Gamma}^{(\bm{Z})}_{l'l'} - \sigma^2)}{\lambda - \bm{\Gamma}^{(\bm{Z})}_{l'l'}}.
\end{align*}
Then we can write the above quantity via $\langle \bm{w}, \big( \bm{V}^{(\bm{Z})}\big)\t \uperp\t\a \rangle.$
Conditional on $\bm{W}$ and $\bm{\Gamma}^{(\bm{Z})}$, the random variable $\frac{\big( \bm{V}^{(\bm{Z})}\big)\t \uperp\t\a}{\|\uperp\t\a\|}$ is uniformly distributed on $p-r$-dimensional sphere. Therefore, by standard concentration inequalities for uniform spherical random variables (Theorem 3.4.6 of \citet{vershynin_high-dimensional_2018}) it holds that
\begin{align*}
    \big| \langle \bm{w}, \big( \bm{V}^{(\bm{Z})}\big)\t \uperp\t\a \rangle \big| &\lesssim \sqrt{\frac{\log(n\vee p)}{p-r}} \| \bm{w}\| \| \uperp\t\a \| \lesssim \sqrt{\frac{\log(n\vee p)}{p}} \|\bm{w}\| \|\uperp\t\a\| \numberthis \label{rotationbound}
\end{align*}
with probability at least $1- O((n\vee p)^{-12})$.  In addition,  conditional on $\bm{Z}$, $\bm{w}$ is a Gaussian random vector with diagonal covariance $\bm{C}$ given by 
\begin{align*}
    \bm{C}_{l'l'} &= \frac{(\bm{\Gamma}^{(\bm{Z})}_{l'l'})^{1/2}(\bm{\Gamma}^{(\bm{Z})}_{l'l'} - \sigma^2)}{\lambda - \bm{\Gamma}^{(\bm{Z})}_{l'l'}}.
\end{align*}
Consequently, with probability at least $1- O((n\vee p)^{-12})$, 
\begin{align*}
    \| \bm{w} \| &\lesssim \| \bm{C} \|_F \sqrt{\log(n\vee p)} \lesssim \sqrt{\min(p-r,n) \log(n\vee p)} \| \bm{C} \|, \numberthis \label{wbound}
\end{align*}
where we have used the fact that $\bm{C}$ has only $\min(p-r,n)$ nonzero eigenvalues.  Since $\bm{C}$ is a function of the matrix $\bm{\Gamma}^{(\bm{Z})}$, it therefore suffices to study the singular values of the matrix $\bm{Z}/\sqrt{n}$.

First we note that standard concentration for Gaussian covariance matrices imply that
\begin{align*}
    \bigg\| \frac{1}{n} \bm{ZZ}\t - \sigma^2 \bm{I}_{p-r} \bigg\| &\lesssim \sigma^2 \bigg( \sqrt{\frac{p}{n}} + \frac{p}{n} + \sqrt{\frac{\log(n\vee p)}{n}}\bigg);\numberthis \label{boundable} \\
    \bigg\| \frac{1}{n} \bm{Z}\t\bm{ Z} - \sigma^2 \frac{p-r}{n} \bm{I}_{p-r} \bigg\| &\lesssim \sigma^2 \big( 1  + \sqrt{\frac{p}{n}}\big).
\end{align*}
Consequently, since the eigenvalues of $\bm{ZZ}\t$ and $\bm{Z}\t\bm{Z}$ are the same the above inequalities imply that
\begin{align*}
      \max_{l'} \sqrt{ \bm{\Gamma}^{(\bm{Z})}_{l'l'} } &\lesssim \sigma \bigg( \sqrt{\frac{\max(n,p-r)}{n}} + \bigg( \frac{p}{n}\bigg)^{1/4} \bigg).
\end{align*}
Combining this bound with the bound \eqref{boundable} and \cref{fact3_PCA}, we obtain that
\begin{align*}
    \| \bm{C} \| &\lesssim \max_{l'}  \frac{\sigma^3}{\lambda - \bm{\Gamma}^{(\bm{Z})}} \left( \sqrt{\frac{\max(n,p-r)}{n}} + \left( \frac{p}{n}\right)^{1/4} \right)\bigg( \sqrt{\frac{p}{n}} + \frac{p}{n} + \sqrt{\frac{\log(n\vee p)}{n}}\bigg) \\
    &\lesssim  \frac{\sigma^3}{\lambda_j} \left( \sqrt{\frac{\max(n,p)}{n}} + \left( \frac{p}{n}\right)^{1/4} \right)\bigg( \sqrt{\frac{p}{n}} + \frac{p}{n} + \sqrt{\frac{\log(n\vee p)}{n}}\bigg) \\
    &\asymp  \frac{\sigma^3}{\lambda_j} \bigg( \sqrt{\frac{p}{n}}  + \sqrt{\frac{\log(n\vee p)}{n}} +  \frac{p^{3/2}}{n^{3/2}}\bigg).
\end{align*}
Therefore, plugging in this bound in \eqref{wbound} and combining with \eqref{rotationbound} we obtain that
\begin{align*}
    \big| \langle \bm{w}, \big( \bm{V}^{(\bm{Z})}\big)\t \uperp\t\a \rangle \big| &\lesssim   \frac{\sigma^3\log(n\vee p) \sqrt{\min(p,n)}}{\lambda_j\sqrt{p}} \bigg( \sqrt{\frac{p}{n}} + \frac{p^{3/2}}{n^{3/2}} + \sqrt{\frac{\log(n\vee p)}{n}} \bigg) \|\uperp\t\a\| \\
    &\lesssim \frac{\sigma^3 \log(n\vee p)}{\lambda_j} \bigg( \frac{p}{n} + \sqrt{\frac{p}{n}} + \sqrt{\frac{\log(n\vee p)}{n}} \bigg) \|\uperp\t\a\|.
\end{align*}
These bounds hold cumulatively with probability at least $1 - O((n\vee p)^{-11})$.  Concequently, we obtain that
\begin{align*}
    \bigg\|    \bm{W}\t  \bm{\Gamma}^{(\bm{Z})}  \bigg( \lambda &\bm{I}_{p-r} - \bm{\Gamma}^{(\bm{Z})}\bigg)\inv  \bigg( \bm{\Gamma}^{(\bm{Z})} - \sigma^2 \bm{I}_{p-r} \bigg) (\bm{V}^{(\bm{Z})})\t \uperp\t\a \bigg\|_{\infty} \\
    &\lesssim \frac{\sigma^3 \log(n\vee p)}{\lambda_j} \bigg( \frac{p}{n} + \sqrt{\frac{p}{n}} + \sqrt{\frac{\log(n\vee p)}{n}} \bigg) \|\uperp\t\a\|
\end{align*}
with probability at least $1 - O((n\vee p)^{-12})$.
\\ \ \\ \noindent
\textbf{Step 2: Bounding each entry by unfixing $\lambda$}.  
We now complete the proof for arbitrary $\lambda$.  
Let $\mathcal{E}_{\eps}$ be an $\eps$-net for the space $[2(\lambda_j + \sigma^2)/3, 4(\lambda_j + \sigma^2)/3]$, with $\eps = c \lambda_j/(n\vee p)$.  Let $\lambda' \in \mathcal{E}_{\eps}$ be such that $|\lambda -\lambda'| \leq \eps$.  Denote
\begin{align*}
    \bm{F}(\lambda) &\coloneqq    \bm{W}\t \frac{ \bm{Z}}{\sqrt{n}}\bigg( \lambda \bm{I}_{p-r} - \sigma^2 \frac{\bm{ZZ}\t}{n} \bigg)\inv \bigg( \frac{\bm{ZZ}\t}{n} - \sigma^2\bm{I}_{p-r} \bigg) \uperp\t\a .
\end{align*}
By a similar argument as previously, letting $\frac{\bm{Z}}{\sqrt{n}} = \bm{U}^{(\bm{Z})} \sqrt{\bm{\Gamma}^{(\bm{Z})}} (\bm{V}^{(\bm{Z})})\t$, observe that
\begin{align*}
   \| &\bm{F}(\lambda) - \bm{F}(\lambda') \|_{\infty} \\
     &\leq \bigg\| \bm{W}\t \frac{ \bm{Z}}{\sqrt{n}} \bm{V}^{(\bm{Z})} \left( \left( \lambda\bm{I}_{p-r} - \bm{\Gamma}^{(\bm{Z})} \right)\inv - \left( \lambda'\bm{I}_{p-r} - \bm{\Gamma}^{(\bm{Z})} \right)\inv \right) \left( \bm{\Gamma}^{(\bm{Z})} - \sigma^2 \bm{I}_{p-r} \right) \big(\bm{V}^{(\bm{Z})}\big)\t \uperp\t\a  \bigg\|  \\
    &\lesssim     \bigg\| \bm{W}\t\ \frac{ \bm{Z}}{\sqrt{n}} \bigg\|_{\infty} \bigg\|  \left( \lambda\bm{I}_{p-r} - \bm{\Gamma}^{(\bm{Z})} \right)\inv - \left( \lambda'\bm{I}_{p-r} - \bm{\Gamma}^{(\bm{Z})} \right)\inv \bigg\| \left\| \bm{\Gamma}^{(\bm{Z})} - \sigma^2 \bm{I}_{p-r} \right\| \| \uperp\t\a\| \\
    &\lesssim \sigma^2  \bigg( \frac{p}{n} + \sqrt{\frac{p}{n}} + \sqrt{\frac{\log(n\vee p)}{n}} \bigg) \bigg\| \bm{W}\t\ \frac{ \bm{Z}}{\sqrt{n}} \bigg\|_{\infty} \bigg\|  \left( \lambda\bm{I}_{p-r} - \bm{\Gamma}^{(\bm{Z})} \right)\inv - \left( \lambda'\bm{I}_{p-r} - \bm{\Gamma}^{(\bm{Z})} \right)\inv \bigg\|  \| \uperp\t\a\|, \end{align*}
where in the final line have made use of the previous covariance concentration inequality.  In addition, we note that standard Gaussian concentration inequalities imply
\begin{align*}
    \bigg\| \bm{W}\t\ \frac{ \bm{Z}}{\sqrt{n}} \bigg\|_{\infty} &\lesssim \sqrt{\frac{\log(n\vee p)}{n}} \| \bm{Z} \|_F \lesssim \sqrt{\frac{\log(n\vee p)\min(p-r,n)}{n}} \|\bm{Z} \|  \\
      &\lesssim \sigma \sqrt{\log(n\vee p) \min(p,n)} \bigg( 1 + \sqrt{\frac{p}{n}}\bigg) \asymp \sigma \sqrt{ (n\vee p)\log(n\vee p)}
\end{align*}
which holds by standard bounds on Gaussian random matrices (i.e. $\|\bm{Z}\| \lesssim \sigma (\sqrt{n} + \sqrt{p})$).  Consequently,
\begin{align*}
    \|\bm{F}(\lambda) - \bm{F}(\lambda') \|_{\infty} 
    &\lesssim \|\uperp\t\a\| \sigma^3 \sqrt{ (n\vee p)\log(n\vee p)} \bigg( \frac{p}{n} + \sqrt{\frac{p}{n}} + \sqrt{\frac{\log(n\vee p)}{n}} \bigg) \\
    &\quad \times \Bigg\| \bigg( \lambda \bm{I}_{p-r} - \sigma^2 \frac{\bm{ZZ}\t}{n} \bigg)\inv - \bigg( \lambda' \bm{I}_{p-r} - \sigma^2 \frac{\bm{ZZ}\t}{n} \bigg)\inv\Bigg\|.
\end{align*}
Let $\eta_i$ denote the eigenvalues of $\sigma^2 \frac{\bm{ZZ}\t}{n}$.  Turning to the remaining term it holds that
\begin{align*}
    \Bigg\| \bigg( \lambda \bm{I}_{p-r} - \sigma^2 \frac{\bm{ZZ}\t}{n} \bigg)\inv - \bigg( \lambda' \bm{I}_{p-r} - \sigma^2 \frac{\bm{ZZ}\t}{n} \bigg)\inv\Bigg\| 
      &= \max_{1\leq i \leq p-r} \bigg| \frac{\lambda' - \lambda }{(\lambda - \eta_i)(\lambda' - \eta_i)} \bigg| \lesssim \frac{\eps}{\lambda_j^2} \asymp \frac{1}{\lambda_j (n\vee p)}.
\end{align*}
Therefore, for any $\lambda$ and $\lambda'$ satisfying $|\lambda -\lambda'| \leq \eps$, we have that
\begin{align*}
    \| \bm{F}(\lambda) - \bm{F}(\lambda') \|_{\infty} &\lesssim \|\uperp\t\a\| \sigma^3 \sqrt{ (n\vee p)\log(n\vee p)} \bigg( \frac{p}{n} + \sqrt{\frac{p}{n}} + \sqrt{\frac{\log(n\vee p)}{n}} \bigg)   \frac{1}{\lambda_j (n\vee p)} \\
    &\lesssim   \frac{\|\uperp\t\a\| \sigma^3 \sqrt{\log(n\vee p)}}{\lambda_j \sqrt{n\vee p}} \bigg( \frac{p}{n} + \sqrt{\frac{p}{n}} + \sqrt{\frac{\log(n\vee p)}{n}} \bigg).
\end{align*}
This bound holds independent of $\lambda$ and $\lambda'$ (provided they are less than $\eps$) with probability at least $1- O( (n\vee p)^{-12})$.  Therefore, let $\lambda^*$ denote the maximizer, and pick $\lambda^{*'} \in \mathcal{E}_{\eps}$ satisfying $|\lambda^* - \lambda^{*'}|\leq \eps$.  Then taking a union bound over at most $(n\vee p)$ many terms, we arrive at
\begin{align*}
    \sup_{\lambda \in [2\lambda_j/3,4\lambda_j/3]} \| \bm{F}(\lambda)\|_{\infty} &= \| \bm{F}(\lambda^*) \|_{\infty} \leq\| \bm{F}(\lambda^{*'}) \|_{\infty}   + \| \bm{F}(\lambda^{*'}) - \bm{F}(\lambda) \|_{\infty} \\
    &\lesssim \frac{\sigma^3 \log(n\vee p) \|\uperp\t\a\|}{\lambda_j} \bigg( \frac{p}{n} + \sqrt{\frac{p}{n}} + \sqrt{\frac{\log(n\vee p)}{n}} \bigg).
\end{align*}
with probability at least $1- O((n\vee p)^{-10})$.  Therefore, combining this bound with \eqref{bigboi} we achieve the desired result.
\end{proof}

\subsubsection{Proof of Lemma \ref{lem:innerproduct_PCA}} \label{sec:innerproduct_PCA_proof}
\begin{proof}
    We slightly modify the proof of \cref{lem:innerproduct_MD}.   We note that
    \begin{align*}
        \uhatj\t \uk\big( \hat \lambda_j -(\lambda_k +\sigma^2) - \gamma\pca(\hat \lambda_j) (\lambda_k + \sigma^2) \big) &= \uhatj\t \big( \sigmahat - \bSigma \big) \uk - \gamma\pca(\hat \lambda_j) (\lambda_k + \sigma^2) \uhatj\t \uk.
    \end{align*}
    Consequently,
    \begin{align*}
        \uhatj\t\uk 
        &=\frac{1}{\hat \lambda_j -(\lambda_k +\sigma^2) - \gamma\pca(\hat \lambda_j) (\lambda_k + \sigma^2)}  \Bigg( \uhatj\t \U \U\t \bigg( \sigmahat - \bSigma \bigg) \uk \Bigg) \\
        &\quad + \frac{1}{\hat \lambda_j -(\lambda_k +\sigma^2) - \gamma\pca(\hat \lambda_j) (\lambda_k + \sigma^2)}  \Bigg( \uhatj\t \per \bigg( \sigmahat - \bSigma \bigg) \uk - \gamma\pca(\hat \lambda_j) (\lambda_k + \sigma^2) \uhatj\t \uk  \Bigg),
    \end{align*}
    where we have implicitly assumed that the denominator is nonzero.  However, by \cref{lem:eigenvalueconcentration_PCA}, we note that
    \begin{align*}
\bigg|        \hat \lambda_j  - (\lambda_k +\sigma^2) \big( 1 +  \gamma\pca(\hat \lambda_j) \bigg| &= (1 + \gamma\pca(\hat \lambda_j)) \bigg| \frac{\hat \lambda_j}{1 + \gamma\pca(\hat \lambda_j)} - \lambda_k\bigg| \\
&\gtrsim |\lambda_j - \lambda_k| - \delta\pca  \\
&\gtrsim \Delta_j,
    \end{align*}
    as long as $\Delta_j \gg \delta\pca$, which is guaranteed by the noise  assumption \eqref{noiseassumption:pca}.  Consequently,
    \begin{align*}
        | \uhatj\t \uk| &\lesssim \frac{1}{\Delta_j} \bigg\| \U \U\t \bigg(\sigmahat - \bSigma \bigg) \uk \bigg\| + \frac{1}{\Delta_j} \bigg| \uhatj \t\uperp \uperp\t  \bigg( \sigmahat - \bSigma \bigg) \uk - \gamma\pca(\hat \lambda_j) (\lambda_k + \sigma^2) \uhatj\t \uk \bigg|.
    \end{align*}
    By a similar argument to the proof of \cref{fact2_PCA}  it holds that
    \begin{align*}
        \bigg\| \U \U\t \bigg( \sigmahat- \bSigma \bigg) \uk \bigg\| &\lesssim \sqrt{\frac{r\log(n\vee p)}{n}} (\lambda_{\max} + \sigma^2)^{1/2} (\lambda_k + \sigma^2)^{1/2}.
    \end{align*}
    In addition, by \cref{lem:uhatjuperpidentity} and \cref{fact4_PCA} we have that
    \begin{align*}
        \uhatj\t \uperp \uperp\t \bigg( \sigmahat - \bSigma\bigg) \uk 
             &= \uhatj\t \U \U\t \sigmahat \uperp\big( \hat \lambda_j \bm{I}_{p-r} - \uperp\t \sigmahat \uperp \big)\inv \uperp\t \sigmahat \uk 
    \end{align*}
   Consequently,
    \begin{align*}
        \bigg| \uhatj \t\uperp \uperp\t & \bigg( \sigmahat - \bSigma \bigg) \uk - \gamma\pca(\hat \lambda_j) (\lambda_k + \sigma^2) \uhatj\t \uk \bigg| \\
        &\lesssim \bigg| \uk\t  \sigmahat \uperp\big( \hat \lambda_j \bm{I}_{p-r} - \uperp\t \sigmahat \uperp \big)\inv \uperp\t \sigmahat \uk - \gamma\pca(\hat \lambda_j) (\lambda_k + \sigma^2) \bigg| \\
        &\quad + \bigg\| (\U^{-k})\t  \sigmahat \uperp\big( \hat \lambda_j \bm{I}_{p-r} - \uperp\t \sigmahat \uperp \big)\inv \uperp\t \sigmahat \uk \bigg\| \\
      &\lesssim \| \bm{G}\pca(\hat \lambda_j) - \bm{\tilde G}\pca(\hat \lambda_j) \| 
    \end{align*}
    where $\bm{G}\pca(\lambda)$ and $\bm{\tilde G}\pca(\lambda)$ are defined as in \cref{lem:Gpca}.  By \cref{lem:Gpca} with probability at least $1 - O((n\vee p)^{-10})$ it holds that
    \begin{align*}
        \sup_{\lambda \in [2(\lambda_j+\sigma^2)/3, 4(\lambda_j+\sigma^2)/3]}\| \bm{G}( \lambda) - \bm{\tilde G}( \lambda) \| &\lesssim (\lambda_{\max} + \sigma^2) \frac{\sigma^2}{\lambda_j^2} \bigg( \frac{p}{n} + \sqrt{\frac{p}{n}} \bigg) \sqrt{\frac{r}{n}} \log^2(n\vee p) \\
        &\ll \sqrt{\frac{r\log(n\vee p)}{n}} (\lambda_{\max} + \sigma^2)^{1/2} (\lambda_k + \sigma^2)^{1/2}.
    \end{align*}
    As a consequence,
    \begin{align*}
        | \uhatj\t \uk | &\lesssim \sqrt{\frac{r\log(n\vee p)}{n}} \frac{(\lambda_{\max} + \sigma^2)^{1/2} (\lambda_k + \sigma^2)^{1/2}}{\Delta_j}.
    \end{align*}
    Therefore, the result is completed by taking a union bound over all $k\neq j$ with $k \leq r$.
\end{proof}

\subsubsection{Proof of Lemma \ref{lem:approximategaussian}} \label{sec:approximategaussianproof}
\begin{proof}
Recalling that $\bm{X}_i = \bm{\Sigma}^{1/2} \y_i$ where $\y_i$ are i.i.d $\mathcal{N}(0, \bm{I}_p)$ random variables, we observe that
    \begin{align*}
        \sum_{k\neq j} &\frac{\uj\t \big( \sigmahat - \bSigma\big)\uk}{\lambda_j - \lambda_k} \uk\t \a\\
        &= \frac{1}{n} \sum_{i=1}^{n} \uj\t \y_i \bigg( \sum_{\substack{k\neq j\\k\leq r}} \uk\t \y_i \frac{(\lambda_j + \sigma^2)^{1/2} (\lambda_k + \sigma^2)^{1/2} (\uk\t\a)}{(\lambda_j - \lambda_k)} + \a\t \uperp \uperp\t\y_i \frac{(\lambda_j + \sigma^2)^{1/2} \sigma}{\lambda_j} \bigg).
    \end{align*}
Consider each summand above.  By independence of $\uj\t \y$ and $\uk\t \y$ for $k\neq j$, we can condition on $\uj\t \y$ to observe that
\begin{align*}
     \frac{1}{n} \sum_{i=1}^{n}& \uj\t \y_i \bigg( \sum_{\substack{k\neq j\\k\leq r}} \uk\t \y_i \frac{(\lambda_j + \sigma^2)^{1/2} (\lambda_k + \sigma^2)^{1/2} (\uk\t\a)}{(\lambda_j - \lambda_k)} + \a\t \uperp \uperp\t\y_i \frac{(\lambda_j + \sigma^2)^{1/2} \sigma}{\lambda_j} \bigg) \\
     &\sim \mathcal{N}\big(0, (\tilde{s_{\a,j}\pca})^2 \big),
\end{align*}
where
\begin{align*}
    \big(\tilde{s_{\a,j}\pca} \big)^2 &= \frac{1}{n} \sum_{i=1}^{n} (\uj\t \y_i)^2 \bigg( \sum_{\substack{k\neq j\\k\leq r}}  \frac{(\lambda_j + \sigma^2) (\lambda_k + \sigma^2) (\uk\t\a)^2}{(\lambda_j - \lambda_k)^2 n} + \| \a\t \uperp \| \frac{(\lambda_j + \sigma^2) \sigma^2}{\lambda_j^2 n} \bigg)\\
    &= \frac{(s_{\a,j}\pca)^2}{n} \sum_{i=1}^{n} (\uj\t \y_i)^2.
\end{align*}
Therefore, conditional on $\uj\t \y$, 
\begin{align*}
    \frac{1}{s_{\a,j}\pca}  \sum_{k\neq j} \frac{\uj\t \big( \sigmahat - \bSigma\big)\uk}{\lambda_j - \lambda_k} \uk\t \a &= \frac{\tilde{s_{\a,j}\pca}}{s_{\a,j}\pca} \frac{1}{\tilde{s_{\a,j}\pca}}  \sum_{k\neq j} \frac{\uj\t \big( \sigmahat - \bSigma\big)\uk}{\lambda_j - \lambda_k} \uk\t \a \sim \mathcal{N}\bigg(0,    (\tilde{s_{\a,j}\pca})^2 \frac{( s_{\a,j}\pca)^2 }{(\tilde{s_{\a,j}\pca})^2 } \bigg).
\end{align*}
Consequently,
\begin{align*}
   \p\bigg\{  \frac{1}{s_{\a,j}\pca}  \sum_{k\neq j} \frac{\uj\t \big( \sigmahat - \bSigma\big)\uk}{\lambda_j - \lambda_k} \uk\t \a \leq z  \  \bigg| \  \uj\t \y \bigg\} - \Phi(z) 
   &= \Phi\bigg( z \frac{ s_{\a,j}\pca}{\tilde{s_{\a,j}\pca}} \bigg) - \Phi(z).\numberthis \label{conditionalthing}
\end{align*}
In addition, we note that
\begin{align*}
\bigg|    \frac{  \big(\tilde{s_{\a,j}\pca} \big)^2 }{(s_{\a,j}\pca)^2} - 1 \bigg| &= \bigg| \frac{1}{n} \sum_{i=1}^{n} (\uj\t \y_i)^2 - 1 \bigg| \lesssim \sqrt{\frac{\log(n\vee p)}{n}},
\end{align*}
with probability at least $1-O((n\vee p)^{-10})$ by Bernstein's inequality.  As a consequence, with probability at least $1 - O((n\vee p)^{-10})$,
\begin{align*}
    \bigg| \frac{\tilde{s_{\a,j}\pca}}{s_{\a,j}\pca} - 1 \bigg|&= \frac{1}{1 + \frac{\tilde{s_{\a,j}\pca}}{s_{\a,j}\pca} } \bigg| \frac{  \big(\tilde{s_{\a,j}\pca} \big)^2 }{(s_{\a,j}\pca)^2} - 1 \bigg| \lesssim \bigg| \frac{  \big(\tilde{s_{\a,j}\pca} \big)^2 }{(s_{\a,j}\pca)^2} - 1 \bigg| \lesssim \sqrt{\frac{\log(n\vee p)}{n}}. \numberthis \label{eventholds1}
\end{align*}
Let $\mathcal{E} \coloneqq \{ \eqref{eventholds1} \text{ holds} \}$. Then by \eqref{conditionalthing} on the event $\mathcal{E}$ it holds that
\begin{align*}
 \bigg|   \p\bigg\{  \frac{1}{s_{\a,j}\pca}  \sum_{k\neq j} \frac{\uj\t \big( \sigmahat - \bSigma\big)\uk}{\lambda_j - \lambda_k} \uk\t \a \leq z  \  \bigg| \  \uj\t \y \bigg\} - \Phi(z) \bigg| &= \bigg| \Phi\bigg( z \frac{ s_{\a,j}\pca}{\tilde{s_{\a,j}\pca}} \bigg) - \Phi(z) \bigg| \\
 &\lesssim \bigg| \frac{ s_{\a,j}\pca}{\tilde{s_{\a,j}\pca}} - 1 \bigg| \lesssim \sqrt{\frac{\log(n\vee p)}{n}}, \numberthis \label{conditionalthing2}
\end{align*}
where we have used the Lipschitz property of the Gaussian cumulative distributional function.  

We are now prepared to complete the proof. Define
\begin{align*}
    \xi \coloneqq \sum_{k\neq j} \frac{\uj\t \big( \sigmahat - \bSigma\big)\uk}{\lambda_j - \lambda_k} \uk\t \a.
\end{align*}
Then it holds that
\begin{align*}
    \p\bigg\{ \frac{\xi}{s_{\a,j}\pca}  \leq z \bigg\}= \mathbb{E}\left[ \mathbb{E} \left[ \mathbb{I}_{\xi/s_{\a,j}\pca \leq z}  | \uj\t \y \right] \right] &= \mathbb{E} \left[ \mathbb{E} \left[\mathbb{I}_{\xi/s_{\a,j}\pca \leq z} \mathbb{I}_{\mathcal{E}} \mid \uj\t \y \right] \right] + \mathbb{E} \left[ \mathbb{E} \left[\mathbb{I}_{\xi/s_{\a,j}\pca \leq z} \mathbb{I}_{\mathcal{E}^c} \mid \uj\t \y \right] \right]. \numberthis \label{decomp}
\end{align*}
Note that $\mathcal{E}$ depends only on $\uj\t\y$.  Therefore, it holds that
\begin{align*}
   \mathbb{E} \left[ \mathbb{I}_{\xi/s_{\a,j}\pca \leq z} \mathbb{I}_{\mathcal{E}} \mid \uj\t \y \right]  &=  
 \mathbb{I}_{\mathcal{E}}\mathbb{E} \left[ \mathbb{I}_{\xi/s_{\a,j}\pca \leq z}  \mid  \uj\t \y \right]  \\
 &= 
 \mathbb{I}_{\mathcal{E}} \p\bigg\{ \xi/s_{\a,j}\pca \leq z  \mid  \uj\t \y \bigg\} = \mathbb{I}_{\mathcal{E}} \Bigg( \Phi(z) + O\bigg( \sqrt{\frac{\log(n\vee p)}{n}} \bigg) \Bigg),
\end{align*}
where the final inequality holds by \eqref{conditionalthing2}.  Plugging this into \eqref{decomp} it holds that
\begin{align*}
    \bigg| \p\bigg\{ \frac{\xi}{s_{\a,j}\pca} \leq z \bigg\} - \Phi(z) \bigg| &\leq\mathbb{E} \bigg[ \Phi(z) \mathbb{I}_{\mathcal{E}^c}  \bigg] + O\bigg( \sqrt{\frac{\log(n\vee p)}{n}} \bigg) \\
    &\lesssim (n \vee p)^{-10} + \sqrt{\frac{\log(n\vee p)}{n}} \lesssim \sqrt{\frac{\log(n\vee p)}{n}},
\end{align*}
where the final inequality holds since $\Phi(z) \leq 1$ and $\p\big\{ \mathcal{E}^{c} \big\} \lesssim (n\vee p)^{-10}$.  This completes the proof.
\end{proof}

\subsubsection{Proof of Lemma \ref{lem:sigmahatapproxpca}} \label{sec:sigmahatoversigmaproofpca}

\begin{proof}
First suppose that $p > n/\log^4(n\vee p)$.  Recall that
\begin{align*}
    \bm{X} = \bSigma^{1/2} \y = \U \bm{\Lambda}^{1/2} \U\t \bm{Y} + \sigma \uperp\uperp\t \y,
\end{align*}
where $\y$ consists of independent standard Gaussian random variables. 

Suppose that
\begin{align*}
    \U\t \bm{\Lambda}^{1/2} \U\t \frac{\y\y\t}{n} \U \bm{\Lambda}^{1/2} \U\t &= \bm{\tilde U} \bm{\Gamma} \bm{\tilde U}\t,
\end{align*}
where it is readily seen that $\bm{\Gamma}$ has at most $r$ nonzero eigenvalues.  
 Define the event
\begin{align*}
    \mathcal{A} &\coloneqq \bigg\{ \lambda_{\min} \lesssim \bm{\Gamma}_{ii} \lesssim \lambda_{\max} \bigg\}  \\
    &\qquad  \bigcap \bigg\{ \bigg\| \frac{\y\y\t}{n} - \bm{I}_p \bigg\| \lesssim \frac{p}{n} + \sqrt{\frac{p}{n}} + \sqrt{\frac{\log(n\vee p)}{n}}\bigg\} \\
    &\qquad \bigcap \bigg\{ \| \U\t \bm{\Lambda} ^{1/2} \U\t \bigg( \frac{\y\y\t}{n} - \bm{I}_p \bigg) \U \bm{\Lambda}^{1/2} \U\t \| \lesssim  \lambda_{\min}  \sqrt{\frac{r\log(n\vee p)}{n}} \bigg\}.
\end{align*}
  By  a similar analysis to \cref{fact2_PCA} it holds that
\begin{align*}
    \| \U\t \bm{\Lambda} ^{1/2} \U\t \bigg( \frac{\y\y\t}{n} - \bm{I}_p \bigg) \U \bm{\Lambda}^{1/2} \U\t \| \lesssim \lambda_{\max}  \sqrt{\frac{r\log(n\vee p)}{n}} \lesssim \kappa  \lambda_{\min}  \sqrt{\frac{r\log(n\vee p)}{n}},
\end{align*}
which implies that $\bm{\Gamma}_{ii} \gtrsim \lambda_{\min}$. 
Similarly, $\max_{i} \bm{\Gamma}_{ii} \lesssim \lambda_{\max}$. Therefore, it is readily seen that $\p(\mathcal{A}) \geq 1 - O((n\vee p)^{-10})$ by the union bound. 

 We have that
\begin{align*}
    \sigmahat = \bm{\tilde U \Gamma \tilde U}\t + \sigma^2 \uperp\uperp\t \frac{\y\y\t}{n} \uperp\uperp\t + \sigma^2 \U \U\t \frac{\y\y\t}{n} \U \U\t,
\end{align*}
and hence 
\begin{align*}
    \sigmahat - \sigmahat_r &= \bm{\tilde U \Gamma \tilde U}\t + \sigma^2 \uperp\uperp\t \frac{\y\y\t}{n} \uperp\uperp\t - \uhat \uhat\t \bm{\tilde U \Gamma \tilde U}\t \uhat \uhat\t -  \sigma^2\uhat \uhat\t \uperp\uperp\t \frac{\y\y\t}{n} \uperp\uperp\t \uhat\uhat\t  \\
    &\quad + \sigma^2 \U \U\t \frac{\y\t\y}{n} \U \U\t - \sigma^2 \uhat \uhat\t \U \U\t \frac{\y\t\y}{n} \U \U\t \uhat\uhat\t,
\end{align*}
which implies that
\begin{align*}
    \Bigg| \tr\bigg( \sigmahat &- \sigmahat_r \bigg) - \sigma^2 \tr\bigg( \uperp \uperp\t \frac{\y\y\t}{n} \uperp\uperp\t \bigg) \bigg| \\
    &\lesssim \sqrt{r} \underbrace{\bigg\|  \bm{\tilde U \Gamma \tilde U}\t - \uhat \uhat\t \bm{\tilde U \Gamma \tilde U}\t \uhat \uhat\t \bigg\|}_{\beta_1} \\&\quad
    + \sqrt{r} \sigma^2 \underbrace{\bigg\| \uhat \uhat\t \uperp\uperp\t \frac{\y\y\t}{n} \uperp\uperp\t \uhat\uhat\t \bigg\|}_{\beta_2} \\
    &\quad + \sigma^2 \sqrt{r} \underbrace{\bigg\| \U \U\t \frac{\y\y\t}{n} \U \U\t - \uhat \uhat\t \U \U\t \frac{\y\t\y}{n} \U \U\t \uhat\uhat\t \bigg\|}_{\beta_3},
\end{align*}
where we have used the fact that $| \tr( \bm{M}) | \leq \sqrt{{\sf Rank}(\bm{M})} \|\bm{M}\|$.  We analyze each term in turn. 
\begin{itemize}
    \item \textbf{The term $\beta_1$}. For $\beta_1$, we have that
    \begin{align*}
    \beta_1 &= \bigg\|\bm{\tilde U \Gamma \tilde U}\t - \uhat \uhat\t \bm{\tilde U \Gamma \tilde U}\t \uhat \uhat\t \bigg\| \\
    &\leq \bigg\| \big( \bm{I} - \uhat\uhat\t \big) \bm{\tilde U\Gamma\tilde U}\t \bigg\| + \bigg\| \uhat \uhat\t \bm{\tilde U \Gamma \tilde U}\t \big( \bm{I} - \uhat\uhat\t \big) \bigg\| \\
    &\lesssim \bigg\| \big( \bm{I} - \uhat\uhat\t \big) \bm{\tilde U} \bigg\| \lambda_{\max},
    \end{align*}
    where the final bound holds with probability at least $1 - O( (n\vee p)^{-10})$.  
    
    Observe that $\utilde$ are the eigenvectors of the matrix $\utilde \bm{\Gamma} \utilde\t + \sigma^2 \bm{I}_p$, which has $r$'th largest eigenvalue at least $\lambda_{\min}  + \sigma^2$ on the event $\mathcal{A}$.  In addition, on this event it holds that
    \begin{align*}
        \bigg\| \sigma^2\frac{\y\y\t}{n} - \sigma^2 \bm{I}_p \bigg\| \lesssim \sigma^2 \bigg( \frac{p}{n} + \sqrt{\frac{p}{n}} + \sqrt{\frac{\log(n\vee p)}{n}} \bigg).  
    \end{align*}
    Therefore, viewing $\sigmahat$ as a perturbation of $\utilde \bm{\Gamma} \utilde\t + \sigma^2 \bm{I}_p$, we have that
    \begin{align*}
        \| \sigmahat - \utilde \bm{\Gamma} \utilde\t - \sigma^2 \bm{I}_p \| \leq  \bigg\| \frac{\y\y\t}{n} - \sigma^2 \bm{I}_p \bigg\| \lesssim \sigma^2 \bigg( \frac{p}{n} + \sqrt{\frac{p}{n}} + \sqrt{\frac{\log(n\vee p)}{n}} \bigg).
    \end{align*}
    Consequently, by Weyl's inequality,
    \begin{align*}
        \bigg| \lambda_{r+1} \big( \sigmahat \big) - \sigma^2 \bigg| \lesssim \sigma^2 \bigg( \frac{p}{n} + \sqrt{\frac{p}{n}} + \sqrt{\frac{\log(n\vee p)}{n}} \bigg)
    \end{align*}
    which implies that 
    \begin{align*}
        \lambda_{r+1}\big( \sigmahat ) \lesssim \sigma^2 \bigg(1 +  \frac{p}{n} + \sqrt{\frac{p}{n}} + \sqrt{\frac{\log(n\vee p)}{n}} \bigg) \ll \lambda_{\min} + \sigma^2,
    \end{align*}
    which holds by the noise assumption \eqref{noiseassumption:pca}.  Therefore,
    \begin{align*}
        \| \big( \bm{I} - \uhat\uhat\t \big) \utilde \| \lesssim \frac{\sigma^2 \bigg( \frac{p}{n} + \sqrt{\frac{p}{n}} + \sqrt{\frac{\log(n\vee p)}{n}} \bigg)}{\lambda_{\min}}.
    \end{align*}
    Hence, on this event  it holds that
    \begin{align*}
        \beta_1 \lesssim \kappa \sigma^2 \bigg( \frac{p}{n} + \sqrt{\frac{p}{n}} + \sqrt{\frac{\log(n\vee p)}{n}} \bigg).
    \end{align*}
    This bound is deterministic on the event $\mathcal{A}$, which holds with probability at least $1 - O((n\vee p)^{-10})$.
     \item \textbf{The term $\beta_2$}. For $\beta_2$, we observe that
     \begin{align*}
         \beta_2 
         &\lesssim  \|\uhat\t\uperp \|^2 \bigg\| \frac{\uperp\t \y\y\t \uperp}{n} \bigg\| \\
         &\lesssim \bigg( \frac{\mathcal{E}\pca}{\lambda_{\min}} \bigg)^2 \bigg( 1 + \frac{p}{n} + \sqrt{\frac{p}{n}} + \sqrt{\frac{\log(n\vee p)}{n}} \bigg),
     \end{align*}
     where the final line is due to \cref{fact1_PCA}, the Davis-Kahan Theorem, and the noise assumption \ref{noiseassumption:pca}.  
      \item \textbf{The term $\beta_3$}. Note that
      \begin{align*}
          \beta_3 &\leq \| \big(\bm{I} - \uhat\uhat\t \big) \U \U\t \frac{\y\y\t}{n} \U \U\t \| + \| \uhat\uhat\t \U \U\t \frac{\y\y\t}{n} \U \U\t\big(\bm{I} - \uhat\uhat\t \big) \| \\
          &\leq 2 \| \big( \bm{I} - \uhat\uhat\t \big) \U \| \| \U\t \frac{\y\y\t}{n} \U \| \\
          &\lesssim \frac{\mathcal{E}\pca}{\lambda_{\min}}
      \end{align*}
      where the penultimate inequality follows from the covariance concentration for Gaussian random variables, and the final inequality is due to the assumption on $r$.
\end{itemize}
Combining our bounds for $\beta_1$, $\beta_2$, and $\beta_3$, we obtain that
\begin{align*}
     \Bigg| \tr\bigg( \sigmahat - \sigmahat_r \bigg) - \sigma^2 \tr\bigg( \uperp \uperp\t \frac{\y\y\t}{n} \uperp\uperp\t \bigg) \bigg|  &\lesssim \sqrt{r} \kappa \sigma^2 \bigg( \frac{p}{n} + \sqrt{\frac{p}{n}} + \sqrt{\frac{\log(n\vee p)}{n}} \bigg) \\&\quad + \sqrt{r}\sigma^2  \bigg( \frac{\mathcal{E}\pca}{\lambda_{\min}}\bigg)^2 \bigg( 1 + \frac{p}{n} + \sqrt{\frac{p}{n}} + \sqrt{\frac{\log(n\vee p)}{n}} \bigg) \\
     &\quad + \sqrt{r} \sigma^2  \frac{\mathcal{E}\pca}{\lambda_{\min}},
\end{align*}
which holds with probability at least $1 - O((n\vee p)^{-10})$. As a result,
\begin{align*}
    \Bigg| \frac{\tr\bigg( \sigmahat - \sigmahat_r \bigg)}{\sigma^2 p-r} -  \frac{1}{p-r} \tr\bigg( \uperp\uperp\t \frac{\y\y\t}{n} \uperp\uperp\t \bigg) \Bigg| &\lesssim \sqrt{r} \kappa  \bigg( \frac{1}{n} + \sqrt{\frac{1}{pn}} +\frac{1}{p} \sqrt{\frac{\log(n\vee p)}{n}} \bigg) \\&\quad + \sqrt{r}  \bigg( \frac{\mathcal{E}\pca}{\lambda_{\min}}\bigg)^2 \bigg( \frac{1}{p} + \frac{1}{n} + \sqrt{\frac{1}{pn}} +\frac{1}{p} \sqrt{\frac{\log(n\vee p)}{n}} \bigg)\\
     &\quad + \sqrt{r}   \frac{\mathcal{E}\pca}{\lambda_{\min} p}.
\end{align*}
Furthermore, by Lemma 1 of \citet{laurent_adaptive_2000} with probability at least $1- O((n\vee p)^{-10})$, 
\begin{align*}
    \bigg| \| \uperp\t \y \|_F^2 - \sigma^2 (p-r)n \bigg| &\lesssim \sigma^2 \sqrt{(p-r)n \log(n\vee p)}+ \log(n\vee p) \lesssim \sigma^2 \sqrt{pn\log(n\vee p)}.
\end{align*}
Consequently,
\begin{align*}
    \bigg| \frac{\tr\big( \uperp\t \y\y\t \uperp\big)}{\sigma^2 n(p-r)} - 1 \bigg| &\lesssim \sqrt{\frac{\log(n\vee p)}{np}}.
\end{align*}
Therefore, combining our bounds, we obtain
\begin{align*}
    \bigg| \frac{\hat \sigma^2}{\sigma^2} - 1 \bigg| &\lesssim  \sqrt{\frac{\log(n\vee p)}{np}} + \sqrt{r} \kappa  \bigg( \frac{1}{n} + \sqrt{\frac{1}{pn}} +\frac{1}{p} \sqrt{\frac{\log(n\vee p)}{n}} \bigg)  \\
    &\qquad + \sqrt{r}  \bigg( \frac{\mathcal{E}\pca}{\lambda_{\min}}\bigg)^2 \bigg( \frac{1}{p} + \frac{1}{n}+ \sqrt{\frac{1}{pn}} +\frac{1}{p} \sqrt{\frac{\log(n\vee p)}{n}} \bigg) + \sqrt{r}   \frac{\mathcal{E}\pca}{\lambda_{\min} p} \\
    &\lesssim \sqrt{r} \kappa  \frac{\log^4(n\vee p)}{n} + \frac{\sqrt{r}}{p} \frac{\mathcal{E}\pca}{\lambda_{\min}},
\end{align*}
with probability at least $1- O((n\vee p)^{-10})$, where we have used the assumption that $p > n/\log^4(n\vee p)$.  This proves the first assertion.

We now prove the second assertion, where we assume that $p \leq n/\log^4(n\vee p)$, which in particular implies $p \leq n$.  First, by the Poincare Separation Theorem (Corollary 4.3.37 of \citet{horn_matrix_2012}) it holds that
\begin{align*}
    \hat \lambda_{r+1} \leq \lambda_r \bigg( \uperp\t \frac{\bm{XX}\t}{n} \uperp \bigg) = \sigma^2 \lambda_r\bigg( \frac{\uperp\t \bm{YY}\t \uperp}{n} \bigg).
\end{align*}
Note that $\uperp\t \y$ is a $p -r \times n$-dimensional matrix of independent standard Gaussian random variables.  Therefore, the covariance concentration inequality (e.g., \cref{fact2_PCA}) implies that with probability at least $1 - O((n\vee p)^{-10})$,
\begin{align*}
    \bigg\| \frac{\uperp\t \bm{YY}\t \uperp}{n} - \bm{I}_{p-r} \bigg\| &\lesssim \sqrt{\frac{p}{n}} + \sqrt{\frac{\log(p\vee n)}{n}}, \numberthis \label{bbbbb}
\end{align*}
since $p \leq n$.  Therefore, Weyl's inequality implies that
\begin{align*}
    \sigma^2 \lambda_r\bigg( \frac{\uperp\t \bm{YY}\t \uperp}{n} \bigg) &\leq \sigma^2 + O\bigg( \sigma^2 \sqrt{\frac{p}{n}} + \sigma^2 \sqrt{\frac{\log(p\vee n)}{n}} \bigg).
\end{align*}
Next, we have that
\begin{align*}
    \hat \lambda_{r+1} 
    &\geq \sigma^2 \lambda_{r+1} \bigg( \frac{\y\y\t}{n} \bigg) \\
    &\geq \sigma^2 - O\bigg( \sigma^2 \sqrt{\frac{p}{n}} + \sigma^2 \sqrt{\frac{\log(p\vee n)}{n}} \bigg),
\end{align*}
where we have applied Weyl's inequality and the bound \eqref{bbbbb}.  Consequently, combining the upper and lower bounds completes the proof.
\end{proof}

\subsubsection{Proof of Lemma \ref{lem:biashatpca}} \label{sec:biashatpcaproof}

\begin{proof}[Proof of \cref{lem:biashatpca}]
Observe that $\hat{b_k\pca} = b_k\pca$ when $n \geq p$; therefore, it suffices to consider the setting that $n < p$.  In this case, by \cref{lem:sigmahatoversigma} it holds that
\begin{align*}
    \bigg| \hat \sigma^2 - \sigma^2 \bigg| 
    &\lesssim \sigma^2 \mathcal{E}_{\sigma} = o(\sigma^2)
\end{align*}
with probability at least  $1 - O((n\vee p)^{-10})$.  Let $\Delta_{\sigma} = \hat \sigma^2 - \sigma^2$.  Then
\begin{align*}
    \hat{b_k\pca} &= \frac{\hat \sigma^2 \frac{p}{n}}{\hat \lambda_k - \hat \sigma^2 \frac{p}{n}} + \frac{\hat \lambda_k}{\hat \lambda_k - \hat \sigma^2 \frac{p}{n}} \frac{\hat \lambda_k}{n + \sum_{r < i \leq n} \frac{\hat \lambda_i}{\hat \lambda_k - \hat \lambda_i}} \sum_{r< i \leq n} \frac{\hat \lambda_i - \hat \sigma^2 \frac{p}{n}}{(\hat \lambda_k - \hat \lambda_i)^2} \\
    &= \frac{\big( \sigma^2 + \Delta_{\sigma} \big) \frac{p}{n}}{\hat \lambda_k - \big(\sigma^2 + \Delta_{\sigma}\big) \frac{p}{n}} + \frac{\hat \lambda_k}{\hat \lambda_k - \big( \sigma^2 +\Delta_{\sigma}\big) \frac{p}{n}} \frac{\hat \lambda_k}{n + \sum_{r < i \leq n} \frac{\hat \lambda_i}{\hat \lambda_k - \hat \lambda_i}} \sum_{r< i \leq n} \frac{\hat \lambda_i - \big( \sigma^2 +\Delta_{\sigma} \big) \frac{p}{n}}{(\hat \lambda_k - \hat \lambda_i)^2}.
\end{align*}
We claim that 
\begin{align*}
    \hat{b_k\pca} = b_k\pca + O\bigg( \frac{\Delta_{\sigma} \frac{p}{n}}{\lambda_{\min}} \bigg) = b_k\pca + o(1) . \numberthis \label{taylorguy}
\end{align*}
If this is indeed the case, then
\begin{align*}
\bigg|    \sqrt{1 + \hat{b_k\pca}} - \sqrt{1 + b_k\pca} \bigg| &= O\bigg( \frac{\Delta_{\sigma} \frac{p}{n}}{\lambda_{\min}} \bigg),
\end{align*}
which follows from the fact that $|b_k\pca| \lesssim \frac{\sigma^2 p}{\lambda_k n} = o(1)$.  Therefore, the result will be proven using the bound for $\Delta_{\sigma}$.  It therefore suffices to demonstrate that \eqref{taylorguy} holds.  Therefore, it suffices to demonstrate that one has
\begin{align*}
      \sum_{r < i \leq n} \frac{\hat \lambda_k}{\hat \lambda_k - \big( \sigma^2 +\Delta_{\sigma}\big) \frac{p}{n}}& \frac{\hat \lambda_k}{n + \sum_{r < i \leq n} \frac{\hat \lambda_i}{\hat \lambda_k - \hat \lambda_i}} \frac{\hat \lambda_i - \big( \sigma^2 +\Delta_{\sigma} \big) \frac{p}{n}}{(\hat \lambda_k - \hat \lambda_i)^2} \\
       &=   \sum_{r < i \leq n}\frac{\hat \lambda_k}{\hat \lambda_k -  \sigma^2 \frac{p}{n}} \frac{\hat \lambda_k}{n + \sum_{r < i \leq n} \frac{\hat \lambda_i}{\hat \lambda_k - \hat \lambda_i}} \frac{\hat \lambda_i -  \sigma^2  \frac{p}{n}}{(\hat \lambda_k - \hat \lambda_i)^2} +  O\bigg( \frac{\Delta_{\sigma} \frac{p}{n}}{ \lambda_{\min}} \bigg) ,\numberthis \label{taylor1}
\end{align*}
as well as that
\begin{align*}
     \frac{(\sigma^2 + \Delta_{\sigma})\frac{p}{n}}{\hat \lambda_k - (\sigma^2 + \Delta_{\sigma})\frac{p}{n}} &=  \frac{\sigma^2 \frac{p}{n}}{\hat \lambda_k - \sigma^2 \frac{p}{n}} + O\bigg( \frac{\Delta_{\sigma} \frac{p}{n}}{\lambda_{\min}} \bigg). \numberthis \label{taylor2}
\end{align*}
Note that with  probability at least $1 - O((n\vee p)^{-10})$, $\Delta_{\sigma} = o(\sigma^2)$, and hence $\frac{\Delta_{\sigma} \frac{p}{n}}{\lambda_{\min}} = o(1)$ with this same probability by the noise assumption \eqref{noiseassumption:pca}.  Henceforth we fix this event. We now prove \eqref{taylor1} and \eqref{taylor2} sequentially.
\begin{itemize}
    \item \textbf{Showing \eqref{taylor1}}. Denote for convenience $C = \frac{\hat \lambda_k}{n + \sum_{r < i \leq n} \frac{\hat \lambda_i}{\hat \lambda_k - \hat \lambda_i}}$ which does not depend on $\Delta_{\sigma}$.   We then have that
    \begin{align*}
        C \frac{\hat \lambda_k}{\hat \lambda_k - (\sigma^2 + \Delta_{\sigma})\frac{p}{n}} \frac{\hat \lambda_i - (\sigma^2 + \Delta_{\sigma}) \frac{p}{n}}{(\hat \lambda_k - \hat \lambda_i)^2} 
        &= C \frac{\hat \lambda_k}{(\hat \lambda_k - \hat \lambda_i)^2}  \frac{\hat \lambda_i - (\sigma^2 + \Delta_{\sigma}) \frac{p}{n}}{\hat \lambda_k - (\sigma^2 + \Delta_{\sigma})\frac{p}{n}} \\\numberthis \label{plugbaby}
    \end{align*}
     Observe that by Taylor Series,
        \begin{align*}
          \frac{\hat \lambda_i - (\sigma^2 + \Delta_{\sigma}) \frac{p}{n}}{\hat \lambda_k - (\sigma^2 + \Delta_{\sigma}) \frac{p}{n}} &= \frac{\hat \lambda_i - \sigma^2 \frac{p}{n}}{\hat \lambda_k - \sigma^2 \frac{p}{n}} + \frac{\hat \lambda_i - \hat \lambda_k}{\hat \lambda_k - \sigma^2 \frac{p}{n}} \sum_{n=1}^{\infty } \bigg[ \frac{\Delta_{\sigma} \frac{p}{n}}{\hat \lambda_k - \sigma^2 \frac{p}{n}} \bigg]^{n} \\
          &= \frac{\hat \lambda_i - \sigma^2 \frac{p}{n}}{\hat \lambda_k - \sigma^2 \frac{p}{n}} + \frac{\hat \lambda_i - \hat \lambda_k}{\hat \lambda_k - \sigma^2 \frac{p}{n}}  \frac{\frac{\Delta_{\sigma} \tfrac{p}{n}}{\hat \lambda_k - \sigma^2 \tfrac{p}{n}} }{1 - \tfrac{\Delta_{\sigma} \tfrac{p}{n}}{\hat \lambda_k - \sigma^2 \tfrac{p}{n}} } 
    \end{align*}
    where the final series expansion converges by \cref{fact3_PCA} together with the fact that $\lambda_{\min} \gg \sigma^2 \frac{p}{n}$.  Plugging this into \eqref{plugbaby} we have that
    \begin{align*}
        C \frac{\hat \lambda_k}{(\hat \lambda_k - \hat \lambda_i)^2}  \frac{\hat \lambda_i - (\sigma^2 + \Delta_{\sigma}) \frac{p}{n}}{\hat \lambda_k - (\sigma^2 + \Delta_{\sigma})\frac{p}{n}} &=  C \frac{\hat \lambda_k}{(\hat \lambda_k - \hat \lambda_i)^2}  \frac{\hat \lambda_i - \sigma^2 \frac{p}{n}}{\hat \lambda_k - \sigma^2 \frac{p}{n}} +  C \frac{\hat \lambda_k}{(\hat \lambda_k - \hat \lambda_i)^2}  \frac{\hat \lambda_i - \hat \lambda_k}{\hat \lambda_k - \sigma^2 \frac{p}{n}}  \frac{\frac{\Delta_{\sigma} \tfrac{p}{n}}{\hat \lambda_k - \sigma^2 \tfrac{p}{n}} }{1 - \tfrac{\Delta_{\sigma} \tfrac{p}{n}}{\hat \lambda_k - \sigma^2 \tfrac{p}{n}} }.
    \end{align*}
    We note that
    \begin{align*}
      \bigg|   C \frac{\hat \lambda_k}{(\hat \lambda_k - \hat \lambda_i)^2}  \frac{\hat \lambda_i - \hat \lambda_k}{\hat \lambda_k - \sigma^2 \frac{p}{n}}  \frac{\frac{\Delta_{\sigma} \tfrac{p}{n}}{\hat \lambda_k - \sigma^2 \tfrac{p}{n}} }{1 - \tfrac{\Delta_{\sigma} \tfrac{p}{n}}{\hat \lambda_k - \sigma^2 \tfrac{p}{n}} }  \bigg| &\lesssim   \frac{\hat \lambda_k^2}{n |\hat \lambda_k - \hat \lambda_i| | \hat \lambda_k - \sigma^2 \frac{p}{n}|} \frac{\Delta_{\sigma} \frac{p}{n}}{\hat \lambda_k - \sigma^2 \frac{p}{n}}  \lesssim \frac{1}{n} \frac{\Delta_{\sigma}\frac{p}{n}}{\lambda_{\min}},
    \end{align*}
    where the final inequality holds by \cref{fact3_PCA} and the noise assumption \eqref{noiseassumption:pca} (which implies that $\sigma^2 = o(\lambda_{\min})$ when $p > n$), and hence that $\frac{\hat \lambda_k}{|\hat \lambda_k - \hat \lambda_i|} \lesssim 1$ and that $\frac{\hat \lambda_k}{\hat \lambda_k - \sigma^2 \frac{p}{n}} \lesssim 1$.  Therefore,
    \begin{align*}
         \sum_{r < i \leq n} \frac{\hat \lambda_k}{\hat \lambda_k - \big( \sigma^2 +\Delta_{\sigma}\big) \frac{p}{n}}& \frac{\hat \lambda_k}{n + \sum_{r < i \leq n} \frac{\hat \lambda_i}{\hat \lambda_k - \hat \lambda_i}} \frac{\hat \lambda_i - \big( \sigma^2 +\Delta_{\sigma} \big) \frac{p}{n}}{(\hat \lambda_k - \hat \lambda_i)^2} \\
       &=   \sum_{r < i \leq n}\frac{\hat \lambda_k}{\hat \lambda_k -  \sigma^2 \frac{p}{n}} \frac{\hat \lambda_k}{n + \sum_{r < i \leq n} \frac{\hat \lambda_i}{\hat \lambda_k - \hat \lambda_i}} \frac{\hat \lambda_i -  \sigma^2  \frac{p}{n}}{(\hat \lambda_k - \hat \lambda_i)^2} +   \sum_{r < i \leq n} O\bigg( \frac{\Delta_{\sigma} \frac{p}{n}}{n \lambda_{\min}} \bigg) \\
       &= \sum_{r < i \leq n}\frac{\hat \lambda_k}{\hat \lambda_k -  \sigma^2 \frac{p}{n}} \frac{\hat \lambda_k}{n + \sum_{r < i \leq n} \frac{\hat \lambda_i}{\hat \lambda_k - \hat \lambda_i}} \frac{\hat \lambda_i -  \sigma^2  \frac{p}{n}}{(\hat \lambda_k - \hat \lambda_i)^2} +    O\bigg( \frac{\Delta_{\sigma} \frac{p}{n}}{ \lambda_{\min}} \bigg),
    \end{align*}
 which proves \eqref{taylor1}.  
    \item \textbf{Showing \eqref{taylor2}}. We simply proceed by Taylor expansion to observe that
    \begin{align*}
        \frac{(\sigma^2 + \Delta_{\sigma}) \frac{p}{n}}{\hat \lambda_k - (\sigma^2 + \Delta_{\sigma})\frac{p}{n}} &= \frac{\sigma^2 \frac{p}{n}}{\hat \lambda_k - \sigma^2 \frac{p}{n}} + O\bigg( \Delta_{\sigma} \frac{\hat \lambda_k \frac{p}{n}}{(\hat \lambda_k - \sigma^2 \frac{p}{n})^2} \bigg) = \frac{\sigma^2 \frac{p}{n}}{\hat \lambda_k - \sigma^2 \frac{p}{n}} + O\bigg( \frac{\Delta_{\sigma} \frac{p}{n}}{\lambda_{\min}} \bigg),
    \end{align*}
    where we have again used the noise assumption \eqref{noiseassumption:pca} and \cref{fact3_PCA}.  This proves \eqref{taylor2}.
\end{itemize}
\end{proof}

\subsubsection{Proof of Lemma \ref{lem:sapproxpca}} \label{sec:sapproxpca_proof}

\begin{proof}[Proof of \cref{lem:sapproxpca}]
Motivated by the decomposition in the proof of \cref{lem:sigma_md_approx} in \cref{sec:sigma_approx_MD_proof}, it holds that
\begin{align*}
    \big( \hat{s_{\a,j}\pca} \big)^2 - (s_{\a,j}\pca)^2
    &= \underbrace{\sum_{\substack{k\neq j\\k\leq r}} \frac{\check \lambda_j \check \lambda_k}{(\check \lambda_j - \check \lambda_k )^2 n} (\bm{\hat u}_k\t\a)^2 (1 + \hat{b_k\pca} )  - \sum_{\substack{k\neq j\\k\leq r}} \frac{\check \lambda_j \check \lambda_k}{(\check \lambda_j - \check \lambda_k )^2 n} (\uk\t\a)^2}_{\alpha_1} \\
    &\quad + \underbrace{\sum_{\substack{k\neq j\\k\leq r}} \frac{\check \lambda_j \check \lambda_k}{(\check \lambda_j -\check \lambda_k )^2 n} (\uk\t\a)^2 - \sum_{\substack{k\neq j\\k\leq r}} \frac{(\lambda_j + \sigma^2)(\lambda_k+\sigma^2)}{(\lambda_j - \lambda_k)^2 n}  (\uk\t\a)^2}_{\alpha_2} \\
    &\quad + \underbrace{\frac{\check \lambda_j \hat \sigma^2}{(\check \lambda_j - \hat \sigma^2)^2 n} \| \uperp\t\a \|^2 - \frac{(\lambda_j + \sigma^2)\sigma^2}{\lambda_j^2 n} \|\uperp\t\a\|^2}_{\alpha_3} \\
    &\quad + \underbrace{\frac{\check \lambda_j \hat \sigma^2}{(\check \lambda_j - \hat \sigma^2)^2 n}\bigg(  \| \uperp\t\a \|^2 - \| \uhat_{\perp}\t\a\|^2 \bigg)}_{\alpha_4}.
\end{align*}
We will bound $\alpha_1$ through $\alpha_4$ in turn. At the outset we make note of the following consequences of \cref{lem:eigenvalueconcentration_PCA,lem:sigmahatapproxpca}.  For each $k \in [r]$ with probability at least $1 - O((n\vee p)^{-10})$ it holds that
\begin{align*}
   \big| \check \lambda_k - (\lambda_k + \sigma^2) \big| &= \big|\frac{\hat \lambda_k}{1 + \hat \gamma\pca(\hat \lambda_k)} - (\lambda_k + \sigma^2) \big| \\
   &\lesssim \frac{\hat \lambda_k}{1 + \gamma\pca(\hat \lambda_k)} \big| \hat \gamma\pca(\hat \lambda_k) - \gamma\pca(\hat \lambda_k) \big| + \bigg| \frac{\hat \lambda_k}{1 + \gamma\pca} - (\lambda_k + \sigma^2) \bigg| \\
   &\lesssim (\lambda_{k} +\sigma^2) \kappa\sqrt{\frac{r}{n}}\log(n\vee p). \numberthis \label{eigenvaluebound} 
\end{align*}
As a consequence, $\check\lambda_k \lesssim \lambda_k + \sigma^2$, and $|\check \lambda_j - \check \lambda_k | \gtrsim |\lambda_j - \lambda_k|.$
In addition, by \cref{lem:sigmahatapproxpca}, with this same probability
\begin{align*}
    \big| \check \lambda_j - \hat \sigma^2 - \lambda_j \big| &\lesssim (\lambda_{j} +\sigma^2) \kappa\sqrt{\frac{r}{n}}\log(n\vee p) + |\hat \sigma^2 - \sigma^2 | \\
    &\lesssim (\lambda_{j} +\sigma^2) \kappa\sqrt{\frac{r}{n}}\log(n\vee p) + \sigma^2 \mathcal{E}_{\sigma}.
\end{align*}
As a result, $\big| \check \lambda_j - \hat \sigma^2 \big| \geq \lambda_j + o(\lambda_j) \gtrsim \lambda_j,$
which follows from the noise assumption \eqref{noiseassumption:pca} and the definition of $\mathcal{E}_{\sigma}$.
Finally, by \cref{lem:sigmahatapproxpca} it holds that $\hat \sigma^2 \lesssim \sigma^2$ with probability at least $1- O((n\vee p)^{-10})$.  Armed with these results we are now prepared to bound $\alpha_1$ through $\alpha_4$.
\begin{itemize}
    \item \textbf{The term $\alpha_1$.}
    For a given $k \in [r]$ we have that
    \begin{align*}
        \frac{\check \lambda_j \check \lambda_k}{(\check \lambda_j - \check \lambda_k )^2n} \bigg| (\bm{\hat u}_k\t\a)^2 (1 + \hat{b_k\pca}) - (\uk\t\a)^2 \bigg| &\lesssim \frac{(\lambda_j + \sigma^2)(\lambda_k + \sigma^2)}{(\lambda_j - \lambda_k)^2 n}\bigg| (\bm{\hat u}_k\t\a)^2 (1 + \hat{b_k\pca}) - (\uk\t\a)^2 \bigg|.
    \end{align*}
  First by \cref{lem:biashatpca} it holds that
    \begin{align*}
        \sqrt{1 + \hat{b_k\pca}} = \sqrt{1 + b_k\pca} +  {\sf ErrBiasApproxPCA}. 
    \end{align*}
    Furthermore, by \cref{lem:bias_PCA}, with probability at least $1 - O((n \vee p)^{-10})$ it holds that
    \begin{align*}
         \bigg| 1 - \sqrt{1 + b_k\pca} \uk\t\bm{\hat u}_k \bigg| &\lesssim {\sf ErrBiasPCA}, 
    \end{align*}
    where we slightly redefine ${\sf ErrBiasPCA}$ to depend only on $\Delta_{\min}$ and $\lambda_{\min}$ instead of $\Delta_k$ and $\lambda_k$.  
    Therefore, 
    \begin{align*}
        \bigg| \sqrt{1 + \hat{b_k\pca}} \uk\t \bm{\hat u}_k -1 \bigg| 
     &\lesssim {\sf ErrBiasApproxPCA} + {\sf ErrBiasPCA}.
    \end{align*}
    By the analysis leading to the proof of \cref{thm:distributionaltheory_PCA}, we have that with probability at least $1 - O((n\vee p)^{-8})$,
\begin{align*}
 \big| \a\t \bm{\hat u}_k \sqrt{1 + \hat{b_k\pca}} - \a\t \uk \big| &\leq \bigg| \big(\a\t \bm{\hat u}_k - \a\t\uk\uk\t \bm{\hat u}_k\big) \sqrt{1 + \hat{b_k\pca}}\bigg|  + \big| \a\t \uk \big| \big| 1 - \uk \uk\t \bm{\hat u}_k \sqrt{1 + \hat{b_k\pca}} \big| \\
 &\leq \bigg| \bigg( \sum_{l\neq k} \frac{\bm{u}_k\t \big( \sigmahat -\bSigma \big) \bm{u}_l}{\lambda_k - \lambda_l}(\bm{u}_l\t\a) + s\pca_{\a,k} \times {\sf ErrPCA} \bigg) \sqrt{1 + \hat{b_k\pca}}\bigg|  \\
 &\quad + \big| \a\t \uk \big| \big| 1 - \uk \uk\t \bm{\hat u}_k \sqrt{1 + \hat{b_k\pca}} \big| \\
 &\lesssim s_{\a,k}\pca \bigg( \sqrt{\log(n\vee p)} + {\sf ErrPCA} \bigg) + |\a\t \uk| \bigg( {\sf ErrBiasApproxPCA} + {\sf ErrBiasPCA} \bigg) \\
 &\lesssim s_{\a,k}\pca \sqrt{\log(n\vee p)}  + |\a\t \uk| \bigg( {\sf ErrBiasApproxPCA} + {\sf ErrBiasPCA} \bigg),
\end{align*}
as long as ${\sf ErrPCA} = o(1)$, which is assumed.     
The bound above implies that with probability at least $1 - O((n\vee p)^{-8})$,
    \begin{align*}
        \bigg| (\bm{\hat u}_k\t\a)^2 (1 + \hat{b_k\pca}) - (\uk\t\a)^2 \bigg| &\lesssim  (s_{\a,k}\pca)^2 \log(n\vee p)  + |\a\t\uk|^2 \bigg( {\sf ErrBiasApproxPCA} + {\sf ErrBiasPCA} \bigg)^2 \\
        &\quad + |\a\t\uk| s_{\a,k}\pca \sqrt{\log(n\vee p)}  \bigg( {\sf ErrBiasApproxPCA} + {\sf ErrBiasPCA} \bigg).
    \end{align*}
    As a result,
    \begin{align*}
        |\alpha_1| &\leq \sum_{\substack{k\neq j\\k\leq r}} \frac{(\lambda_j + \sigma^2)(\lambda_k + \sigma^2)}{(\lambda_j - \lambda_k)^2 n}  (s_{\a,k}\pca)^2 \log(n\vee p) \\
        &\quad +  \bigg( {\sf ErrBiasApproxPCA} + {\sf ErrBiasPCA} \bigg)^2\sum_{\substack{k\neq j\\k\leq r}} \frac{(\lambda_j + \sigma^2)(\lambda_k + \sigma^2)}{(\lambda_j - \lambda_k)^2 n}  (\a\t\uk)^2  \\
        &\quad + \sum_{\substack{k\neq j\\k\leq r}} \frac{(\lambda_j + \sigma^2)(\lambda_k + \sigma^2)}{(\lambda_j - \lambda_k)^2 n} |\a\t\uk| s_{\a,k}\pca \sqrt{\log(n\vee p)}\bigg( {\sf ErrBiasApproxPCA} + {\sf ErrBiasPCA} \bigg) 
\\       
&\lesssim (s_{\a,j}\pca)^2\bigg( \max_{k \neq j} \frac{(s_{\a,k}\pca)^2}{(s_{\a,j}\pca)^2} \frac{(\lambda_{\max} + \sigma^2)^2}{\Delta_j^2 n} \log(n\vee p) + \big( {\sf ErrBiasApproxPCA} + {\sf ErrBiasPCA} \big)^2 \\
&\qquad + \max_{k \neq j}\frac{s_{\a,k}\pca \sqrt{\log(n\vee p)}}{s_{\a,j}\pca} \frac{\lambda_{\max} + \sigma^2}{\Delta_j \sqrt{n}} \big(  {\sf ErrBiasApproxPCA} + {\sf ErrBiasPCA} \big) \bigg). 
    \end{align*}
        \item \textbf{The term $\alpha_2$.} 
        For a given $k \in [r]$, it holds that
        \begin{align*}
\bigg|            \frac{\check \lambda_j \check \lambda_k}{(\check\lambda_j - \check\lambda_k)^2 n} - \frac{(\lambda_j + \sigma^2)(\lambda_k + \sigma^2)}{(\lambda_j  -\lambda_k)^2 n} \bigg| &\leq \bigg| \frac{(\check \lambda_j - \lambda_j - \sigma^2) \check \lambda_k}{(\check \lambda_j - \check \lambda_k)^2n} \bigg| + \bigg| \frac{ (\lambda_j + \sigma^2)(\check\lambda_k - \lambda_k - \sigma^2)}{(\check \lambda_j - \check \lambda_k)^2n} \bigg|\\
&\quad + \bigg| \frac{ (\lambda_j + \sigma^2)(\lambda_k + \sigma^2)}{(\check \lambda_j - \check \lambda_k)^2n} -\frac{ (\lambda_j + \sigma^2)(\lambda_k + \sigma^2)}{(\lambda_j -  \lambda_k)^2n} \bigg|.
        \end{align*}
By the fact that $|\check \lambda_k - \lambda_k - \sigma^2| \lesssim (\lambda_k + \sigma^2) \kappa \sqrt{\frac{r}{n}}\log(n\vee p)$, the first two terms above satisfy
\begin{align*}
    \bigg| \frac{(\check \lambda_j - \lambda_j - \sigma^2) \check \lambda_k}{(\check \lambda_j - \check \lambda_k)^2n} \bigg| + \bigg| \frac{ (\lambda_j + \sigma^2)(\check\lambda_k - \lambda_k - \sigma^2)}{(\check \lambda_j - \check \lambda_k)^2n} \bigg| &\lesssim \kappa \sqrt{\frac{r}{n}}\log(n\vee p) \frac{(\lambda_j + \sigma^2)(\lambda_k + \sigma^2)}{(\lambda_j - \lambda_k)^2 n}.
\end{align*}
    For the remaining term, by  Taylor expansion we have that
    \begin{align*}
         \bigg| \frac{ (\lambda_j + \sigma^2)(\lambda_k + \sigma^2)}{(\check \lambda_j - \check \lambda_k)^2n} -\frac{ (\lambda_j + \sigma^2)(\lambda_k + \sigma^2)}{(\lambda_j -  \lambda_k)^2n} \bigg| &\lesssim \frac{(\lambda_j + \sigma^2)(\lambda_k + \sigma^2)}{n}  \frac{|\check \lambda_j - \check \lambda_k - \lambda_j - \sigma^2 + \lambda_k + \sigma^2|}{|\lambda_j - \lambda_k|^3} \\
         &\lesssim \frac{(\lambda_j + \sigma^2)(\lambda_k + \sigma^2)}{n(\lambda_j - \lambda_k)^2} \frac{(\lambda_{\max} + \sigma^2)}{\Delta_j} \sqrt{\frac{r}{n}}\log(n\vee p).
    \end{align*}
        As a consequence, with probability at least $1 - O((n\vee p)^{-10})$,
        \begin{align*}
            |\alpha_2| &\lesssim \bigg( \kappa \sqrt{\frac{r}{n}}\log(n\vee p) +\frac{(\lambda_{\max} + \sigma^2)}{\Delta_j} \sqrt{\frac{r}{n}} \bigg)  (s\pca_{\a,j})^2.
        \end{align*}
    \item \textbf{The term $\alpha_3$.} By decomposing we have that
    \begin{align*}
        |\alpha_3| &\lesssim \underbrace{\bigg| \frac{(\check \lambda_j - \lambda_j - \sigma^2)\hat \sigma^2}{(\check \lambda_j - \hat \sigma^2)^2 n} \bigg|}_{\beta_1} \|\uperp\t\a\|^2  + \underbrace{\bigg| \frac{(\lambda_j + \sigma^2)(\hat \sigma^2 - \sigma^2)}{(\check \lambda_j - \hat \sigma^2)^2n}\bigg|}_{\beta_2} \|\uperp\t\a\|^2  \\
        &\quad + \underbrace{\bigg|\frac{(\lambda_j + \sigma^2)\sigma^2}{\lambda_j^2 n} - \frac{(\lambda_j + \sigma^2)\sigma^2}{(\check\lambda_j - \hat \sigma^2)^2 n} \bigg|}_{\beta_3} \|\uperp\t\a\|^2.
        \end{align*}
    We bound $\beta_1$,$\beta_2$, and $\beta_3$ in turn.
    \begin{itemize}
        \item \textbf{Bounding $\beta_1$}. By the eigenvalue bounds in \eqref{eigenvaluebound}, with probability at least $1- O((n\vee p)^{-10})$ it holds that
        \begin{align*}
            \beta_1 &= \bigg| \frac{(\check \lambda_j - \lambda_j - \sigma^2)\hat \sigma^2}{(\check \lambda_j - \hat \sigma^2)^2 n} \bigg| \lesssim \kappa \sqrt{\frac{r}{n}}\log(n\vee p) \frac{(\lambda_j + \sigma^2)  \sigma^2}{\lambda_j^2 n}.
        \end{align*}
          \item \textbf{Bounding $\beta_2$}.  By \cref{lem:sigmahatapproxpca}, it holds that
          \begin{align*}
              |\hat \sigma^2 - \sigma^2| &\lesssim \sigma^2  \mathcal{E}_{\sigma}.
          \end{align*}
          Consequently,
          \begin{align*}
              |\beta_2| &\lesssim \mathcal{E}_{\sigma} \frac{(\lambda_j + \sigma^2) \sigma^2}{\lambda_j^2 n}.
          \end{align*}
            \item \textbf{Bounding $\beta_3$}. Finally, for $\beta_3$, we note that 
            \begin{align*}
\big|                \check \lambda_j - \hat \sigma^2 - \lambda_j \big| &\lesssim (\lambda_j + \sigma^2) \kappa \sqrt{\frac{r}{n}}\log(n\vee p) + \sigma^2 \mathcal{E}_{\sigma} = \lambda_j\big(1 + o(1) \big).
            \end{align*}
            Therefore,
            \begin{align*}
\bigg|                \frac{1}{\lambda_j^2} - \frac{1}{(\check \lambda_j - \hat \sigma^2)^2} \bigg| &\lesssim  \frac{1}{\lambda_j^3} \bigg( (\lambda_j + \sigma^2)\kappa \sqrt{\frac{r}{n}}\log(n\vee p) + \sigma^2 \mathcal{E}_{\sigma} \bigg),
            \end{align*}
        and hence
        \begin{align*}
            \beta_3 &\lesssim \frac{(\lambda_j + \sigma^2)\sigma^2}{\lambda_j^2 n} \bigg( \big(1 + \frac{\sigma^2}{\lambda_j} \big) \kappa \sqrt{\frac{r}{n}}\log(n\vee p) + \frac{\sigma^2}{\lambda_j } \mathcal{E}_{\sigma} \bigg).
        \end{align*}
    \end{itemize}
    Therefore, combining these bounds, we obtain that with probability at least $1- O((n\vee p)^{-10})$,
    \begin{align*}
        |\alpha_3| &\lesssim \frac{(\lambda_j + \sigma^2)\sigma^2}{\lambda_j^2 n} \|\uperp\t\a\|^2  \bigg(  1 +\frac{\sigma^2}{\lambda_j} \bigg)\bigg( \kappa \sqrt{\frac{r}{n}}\log(n\vee p) +  \mathcal{E}_{\sigma} \bigg)  
    \end{align*}
    \item \textbf{The term $\alpha_4$.} We observe that
    \begin{align*}
        |\alpha_4 | 
        &\lesssim \frac{(\lambda_j + \sigma^2) \sigma^2}{\lambda_j^2 n}\bigg| \|\uperp\t\a\|^2 - \|\uhat_{\perp}\t\a\|^2 \bigg|.
    \end{align*}
By orthonormality, we have that
    \begin{align*}
        \big| \|\uhat_{\perp}\t\a\| - \|\uperp\t\a\|\big| &= \big| \| \uhat_{\perp} \uhat_{\perp}\t\a\| - \|\uperp\uperp\t\a\|\big| \\
        &\leq \frac{\mathcal{E}\pca}{\lambda_{\min} } \| \a \| \\
        &\leq \frac{\mathcal{E}\pca}{\lambda_{\min}} \bigg( \|\uperp\t\a\| + \|\U\t\a\| \bigg),
    \end{align*}
    where the final inequality holds on the event  in \cref{fact1_PCA} by the Davis-Kahan Theorem.  
    Consequently, 
    \begin{align*}
        \bigg| \|\uhat_{\perp}\t\a\|^2 - \|\uperp\t\a\|^2 \bigg| &= \bigg( \| \uhat_{\perp}\t\a \| + \|\uperp\t\a\| \bigg) \frac{\mathcal{E}\pca}{\lambda_{\min}} \bigg( \|\uperp\t\a\| + \|\U\t\a\| \bigg) \\
        &\lesssim \frac{\mathcal{E}\pca}{\lambda_{\min}} \bigg( \| \uperp\t\a \| + \frac{\mathcal{E}\pca}{\lambda_{\min}} \|\U\t\a\| \bigg) \bigg( \|\uperp\t\a\| + \|\U\t\a\| \bigg) \\
        &\asymp \frac{\mathcal{E}\pca}{\lambda_{\min}}  \| \uperp\t\a \|^2  +\bigg(\frac{\mathcal{E}\pca}{\lambda_{\min}}\bigg)^2 \|\U\t\a\| \|\uperp\t\a \| +   \bigg(\frac{\mathcal{E}\pca}{\lambda_{\min}}\bigg)^2 \|\U\t\a\|^2.
    \end{align*}
    We therefore have that
    \begin{align*}
        |\alpha_4| &\lesssim ( s_{\a,j}\pca)^2 \frac{\mathcal{E}\pca}{\lambda_{\min}}  + \frac{(\lambda_j + \sigma^2)\sigma^2}{\lambda_j^2 n} \big(\frac{\mathcal{E}\pca}{\lambda_{\min}}\big)^2 \|\uperp\t\a\| \|\U\t\a\| + \frac{(\lambda_j + \sigma^2)\sigma^2}{\lambda_j^2 n}\big(\frac{\mathcal{E}\pca}{\lambda_{\min}}\big)^2 \|\U\t\a\|^2.
    \end{align*}
    If $\|\U\t\a\| \leq \|\uperp\t\a \|$, then it holds that $|\alpha_4| \lesssim  (s_{\a,j}\pca)^2 \frac{\mathcal{E}\pca}{\lambda_{\min}}$.  Therefore, it suffices to consider when $\|\U\t\a\| > \|\uperp\t\a\|$.  In this case, considering the cross-term, we have that
\begin{align*}
    \frac{(\lambda_j+\sigma^2)\sigma^2}{\lambda_j^2 n}& \bigg(\frac{\mathcal{E}\pca}{\lambda_{\min}} \bigg)^2 \| \uperp\t\a\| \|\U\t\a\| \\
    &=  \frac{(\lambda_j+\sigma^2)\sigma^2}{\lambda_j^2 n} \bigg(\frac{\mathcal{E}\pca}{\lambda_{\min}} \bigg)^2 \| \uperp\t\a\| \sqrt{\sum_{\substack{k\neq j\\k\leq r}} (\uk\t\a)^2 + (\uj\t\a)^2} \\
    &\leq \frac{\kappa \sigma^2}{\lambda_j^2 \sqrt{n}} \bigg( \frac{\mathcal{E}\pca}{\lambda_{\min}} \bigg)^2 \|\uperp\t\a\| \Delta_{\max} \sqrt{\sum_{\substack{k\neq j\\k\leq r}} \frac{(\lambda_j + \sigma^2)(\lambda_k + \sigma^2)}{(\lambda_j - \lambda_k)^2 n} (\uk\t\a)^2} + \frac{(\lambda_j + \sigma^2)\sigma^2}{\lambda_j^2 n} \| \uperp\t\a \| |\uj\t\a| \bigg( \frac{\mathcal{E}\pca}{\lambda_{\min}} \bigg)^2 \\
    &\lesssim (s_{\a,j}\pca)^2 \frac{(\mathcal{E}\pca)^2 \kappa^{3/2}}{\lambda_{\min}^2} + o\bigg( \frac{1}{\sqrt{\log(n\vee p)}} \bigg) (s_{\a,j}\pca)^2
\end{align*}
where we used the assumption that
\begin{align*}
    \frac{\sigma \sqrt{\lambda_j + \sigma^2}}{\lambda_j \sqrt{n}} \frac{\mathcal{E}\pca}{\lambda_{\min}} |\uj\t\a| = o \bigg( \frac{s_{\a,j}\pca}{\sqrt{\log(n\vee p)}} \bigg).
\end{align*}
Similarly,
\begin{align*}
  \frac{(\lambda_j + \sigma^2 )\sigma^2}{\lambda_j^2 n}  &\bigg( \frac{\mathcal{E}\pca}{\lambda_{\min}} \bigg)^2 \|\U\t\a \|^2 \\
  &=   \frac{(\lambda_j + \sigma^2 )\sigma^2}{\lambda_j^2 n}  \bigg( \frac{\mathcal{E}\pca}{\lambda_{\min}} \bigg)^2 \bigg[\sum_{\substack{k\neq j\\k\leq r}} (\uk\t\a)^2 + (\uj\t\a)^2 \bigg] \\
  &\leq \frac{\sigma^2 \Delta_{\max}^2}{\lambda_j^2(\lambda_{\min} + \sigma^2)} \bigg( \frac{\mathcal{E}\pca}{\lambda_{\min}} \bigg)^2  \sum_{\substack{k\neq j\\k\leq r}} \frac{(\lambda_j + \sigma^2)(\lambda_k + \sigma^2)}{(\lambda_j - \lambda_k)^2 n} (\uk\t\a)^2 + \bigg( \frac{\mathcal{E}\pca}{\lambda_{\min}} \bigg)^2 \frac{(\lambda_j + \sigma^2 )\sigma^2}{\lambda_j^2 n} (\uj\t\a)^2\\
  &\leq \frac{(\mathcal{E}\pca)^2 \kappa^2}{\lambda_{\min}^2} (s_{\a,j}\pca)^2 + \frac{(s_{\a,j}\pca)^2}{\log(n\vee p)}
\end{align*}
    Consequently, 
\begin{align*}
    |\alpha_4| \lesssim (s_{\a,j}\pca)^2 \bigg( \frac{\kappa^2 (\mathcal{E}\pca)^2}{\lambda_{\min}^2} + \frac{\mathcal{E}\pca}{\lambda_{\min}} + o\bigg( \frac{1}{\sqrt{\log(n\vee p)}} \bigg) \bigg).
\end{align*}
\end{itemize}
Combining all of our bounds we arrive at the bound
\begin{align*}
    \big| \big( \hat{s\pca_{\a,j}}\big)^2 - \big( s\pca_{\a,j} \big)^2 \big| &\lesssim (s_{\a,j}\pca)^2 \Bigg( \max_{k \neq j} \frac{(s_{\a,k}\pca)^2}{(s_{\a,j}\pca)^2} \frac{(\lambda_{\max} + \sigma^2)^2}{\Delta_j^2 n} \log(n\vee p) + \big( {\sf ErrBiasApproxPCA} + {\sf ErrBiasPCA} \big)^2 \\
&\qquad \qquad + \max_{k \neq j}\frac{s_{\a,k}\pca \sqrt{\log(n\vee p)}}{s_{\a,j}\pca} \frac{\lambda_{\max} + \sigma^2}{\Delta_j \sqrt{n}} \big(  {\sf ErrBiasApproxPCA} + {\sf ErrBiasPCA} \big) \bigg) \\
&\qquad \qquad +  \big( 1 + \frac{\sigma^2}{\lambda_j} \big) \kappa \sqrt{\frac{r}{n}} \log(n\vee p) + \frac{(\lambda_{\max} + \sigma^2)\sqrt{r}}{\Delta_j \sqrt{n}} + \big( 1 + \frac{\sigma^2}{\lambda_j} \big) \mathcal{E}_{\sigma} \\
&\qquad \qquad + \frac{\kappa\mathcal{E}\pca}{\lambda_{\min}} + o \bigg( \frac{1}{\sqrt{\log(n\vee p)}} \bigg) \Bigg)
\end{align*}
which holds with probability at least $1- O((n\vee p)^{-8})$.  Let the right hand side above be denoted as $\Delta_s$.  Then we have that
\begin{align*}
    |\hat{s_{\a,j}\pca} - s_{\a,j}\pca| \leq \frac{\big| \big( \hat{s\pca_{\a,j}}\big)^2 - \big( s\pca_{\a,j} \big)^2 \big|}{|\hat{s_{\a,j}\pca} + s_{\a,j}\pca|} \lesssim s_{\a,j}\pca \Delta_s.
\end{align*}
Therefore, it suffices to show that $\Delta_s \ll \frac{1}{\sqrt{\log(n\vee p)}}$ to complete the proof.  

Thus, it suffices to have that
\begin{align}
    \frac{\kappa \mathcal{E}\pca \sqrt{\log(n\vee p)}}{\lambda_{\min}} &\ll 1; \label{firstcondition} \\
    \frac{(s_{\a,k}\pca)^2}{(s_{\a,j}\pca)^2} \frac{(\lambda_{\max} + \sigma^2)}{\Delta_j^2 n} \log^{3/2}(n\vee p) &\ll 1; \label{secondcond} \\
    {\sf ErrBiasApproxPCA}\sqrt{\log(n\vee p)} &\ll 1 \label{thirdcond} \\
    {\sf ErrBiasPCA} \sqrt{\log(n\vee p)} &\ll 1; \label{fourthcond} \\
    \max_{k \neq j,k\leq r} \frac{s_{a,k}\pca}{s_{\a,j}\pca} \frac{\log(n\vee p) (\lambda_{\max} + \sigma^2)}{\Delta_j \sqrt{n}} {\sf ErrBiasApproxPCA} &\ll 1;  \label{fifthcond} \\
     \max_{k \neq j,k\leq r} \frac{s_{a,k}\pca}{s_{\a,j}\pca} \frac{\log(n\vee p) (\lambda_{\max} + \sigma^2)}{\Delta_j \sqrt{n}} {\sf ErrBiasPCA} &\ll 1 ; \label{sixthcond} \\
     \frac{(\lambda_{\max} + \sigma^2) \sqrt{r\log(n\vee p)}}{\Delta_j \sqrt{n}} &\ll 1; \label{seventhcond} \\
     (1 + \frac{\sigma^2}{\lambda_j} ) \mathcal{E}_{\sigma}\sqrt{\log(n\vee p)} &\ll 1. \label{eighthcond}
    \end{align}
Recall that ${\sf ErrBiasApproxPCA}$ is defined as in \eqref{errbiasapproxpcadef}, $\mathcal{E}\pca$ is defined in \cref{fact1_PCA},
$\mathcal{E}_{\sigma}$ is defined in \cref{lem:sigmahatapproxpca}, and ${\sf ErrBiasPCA}$ is defined in \cref{lem:bias_PCA}.  

The condition \eqref{firstcondition} holds by our assumption \eqref{noisecondpca:inf} and the assumption on $r$.  The condition \eqref{secondcond} holds directly by assumption \eqref{sakassumption:pca}.  The condition \eqref{thirdcond} is immediately satisfied when $p \leq n$, since then ${\sf ErrBiasApproxPCA} = 0$.  Therefore, it suffices to consider when $p > n$, in which case it holds that ${\sf ErrBiasApproxPCA}$ satisfies
\begin{align*}
    {\sf ErrBiasApproxPCA} = \frac{\sigma^2 p}{\lambda_{\min} n} \mathcal{E}_{\sigma} = \frac{\sigma^2 p}{\lambda_{\min}n} \frac{\sqrt{r} \kappa \log^4(n\vee p)}{n} + \frac{\sigma^2 p}{\lambda_{\min}n} \frac{\sqrt{r}}{p} \frac{\mathcal{E}\pca}{\lambda_{\min}}.
\end{align*}
Indeed, this quantity is $o(\sqrt{\log(n\vee p)})$ by the assumption \eqref{noisecondpca:inf} and the assumption that $r^3 \lesssim \frac{n}{\kappa^6\log^{12}(n\vee p)}$.  The condition \eqref{fourthcond} holds by our assumption on $\Delta_{\min}$ in \eqref{eigengapcondpca:inf} and our assumption on $\lambda_{\min}$ in \eqref{noisecondpca:inf}.  The next three conditions \eqref{fifthcond}, \eqref{sixthcond}, and \eqref{seventhcond}  hold by our previous discussions.  Finally, the condition \eqref{eighthcond} is significantly weaker than the condition on $\lambda_{\min}$ in \eqref{noisecondpca:inf}.  This completes the proof. 
\end{proof}

\section{Proofs of Lower Bounds} \label{sec:lowerboundproofs}
In this section we prove \cref{thm:minimaxlowerbd} and \cref{thm:minimaxlowerbound_PCA}.

\subsection{Proof of \cref{thm:minimaxlowerbd}} 
\label{sec:minimaxlowerbound}
\begin{proof}
Suppose we observe $\bS + \bN$, and consider observing $\tilde{\bS} + \bN$, where $\tilde{\bS} = \sum_{i=1}^{r} \lambda_i \tilde{\bm{u}}_i \tilde{\bm{u}_i}\t$, where $\tilde{\bm{u}}_i$ will be chosen momentarily.  Let $P$ and $\tilde P$ denote the corresponding distribution with $\bS$ and $\tilde{\bS}$ respectively.  By (a minor generalization of) Lemma 1 of \citet{cai_confidence_2017}, 
    \begin{align*}
        \inf_{{\sf C.I.} \in \mathcal{I}_{\alpha,\a}(\mathcal{P})} {\sf L}_{{\sf C.I.}}(\bS) \geq |\a\t \bm{u}_j \pm \a\t \bm{\tilde u}_j | \bigg(1 - 2\alpha - {\sf TV}\big( P, \tilde P) \bigg)_+,
    \end{align*}
where ${\sf TV}(\cdot,\cdot)$ denotes the total variation distance. (The generalization only requires considering $\pm \a\t\uj$ instead, and yields the same result above.) Thus, since ${\sf TV}(\cdot,\cdot) \leq \sqrt{\frac{1}{2}{\sf KL}(\cdot,\cdot)}$, it suffices to upper bound the ${\sf KL}$ divergence and lower bound $|\a\t \uj \pm \a\t \tilde{\uj}|$.

     First, assume that $\sqrt{\sum_{\substack{k\neq j\\k\leq r}} \frac{\sigma^2}{(\lambda_j - \lambda_k)^2} (\a\t\uk)^2} \geq \frac{\sigma^2 \| \uperp\t\a\|^2}{\lambda_j^2}.$
    Then the left hand side dominates $s_{\a,j}\md$.    
   We let $\bm{U}$ and $\bm{\tilde U}$ be defined via $\bm{\tilde U} = \bm{U}\bm{R}$, where $\bm{R}$ is an orthogonal matrix that we construct explicitly as follows. 
      For $k \neq j$ set \begin{align*}
    \theta_k =  \pm \frac{1}{4} \frac{\sigma^2 \a\t\uk}{(\lambda_j - \lambda_k)^2 \sqrt{\sum_{l\neq j}\frac{\sigma^2(\a\t\bm{u}_l)^2}{(\lambda_j - \lambda_l)^2}}},
\end{align*}
and let $\bm{\theta}$ be the vector with $\theta_j = 0$, where the sign is chosen to be  $\mathrm{sgn}(\a\t\uj)$ if $|\a\t\uj| \leq \sqrt{\sum_{\substack{k\neq j\\k\leq r}} \frac{(\a\t\uk)^2 \sigma^2 }{(\lambda_j - \lambda_k)^2}}, $ and $-\mathrm{sgn}(\a\t\uj)$ otherwise, where the convention that $\mathrm{sgn}(0)= 1$.  
Define
\begin{align*}
    \bm{R} = \bm{I}_r + \bigg( \cos( \| \bm{\theta}\|) - 1  \bigg) \bigg( e_j e_j\t + \frac{\bm{\theta}\bm{\theta}\t}{\|\bm{\theta}\|^2} \bigg) + \frac{\sin\|\bm{\theta}\|}{\|\bm{\theta}\|} (\bm{\theta} e_j\t - e_j \bm{\theta}\t ),
\end{align*}
where $e_j$ is a standard basis vector. 
We then have that
\begin{align*}
    \tilde{\uj} &= \bm{U}\bm{R} e_j = \cos(\|\bm{\theta}\|) \uj + \frac{\sin\|\bm{\theta}\|}{4\|\bm{\theta}\|} \sum_{k\neq j}  \theta_k \uk\\
    &= \cos \|\bm{\theta} \| \uj \pm \frac{\sin\|\bm{\theta}\|}{2\|\bm{\theta}\|} \sum_{k\neq j}  \frac{\sigma^2 \a\t\uk }{(\lambda_j - \lambda_k)^2 \sqrt{\sum_{l\neq j}\frac{\sigma^2 (\a\t \bm{u}_l)^2}{(\lambda_j - \lambda_l)^2}}} \uk,
\end{align*}
and hence that
\begin{align*}
    \a\t\tilde{\uj} &= \a\t\uj \cos\|\bm{\theta} \| \pm \frac{\sin\|\bm{\theta}\|}{4\|\bm{\theta}\|}  \sqrt{\sum_{k\neq j}  \frac{\sigma^2 (\uk\t\a)^2}{(\lambda_j - \lambda_k)^2 }}.
\end{align*}
If $|\a\t\uj| \leq T_j$, then both terms above have the same sign, which shows that 
\begin{align*}
    |\a\t\uj| \geq \frac{\sin \|\bm{\theta}\|}{4\|\bm{\theta}\|}  \sqrt{\sum_{k\neq j}  \frac{\sigma^2(\uk\t\a)^2}{(\lambda_j - \lambda_k)^2 }} + |\a\t\uj| \cos \|\bm{\theta}\|
\end{align*}
and thus
\begin{align*}
    \big| |\a\t\uj| - |\a\t \tilde{\uj} | \big| \geq \frac{\sin \|\bm{\theta}\|}{4\|\bm{\theta}\|} \sqrt{\sum_{k\neq j}  \frac{\sigma^2(\uk\t\a)^2}{(\lambda_j - \lambda_k)^2 }} - (1 - \cos \|\bm{\theta} \|) |\a\t\uj|.
\end{align*}
Next, we note that since $\|\bm{\theta}\| \leq \frac{\sigma}{\Delta_j} = o(1)$, it holds that $\frac{\sin \|\bm{\theta}\|}{\|\bm{\theta}\|} \geq \frac{3}{4}$ and $\cos(\|\bm{\theta}\|) \geq \frac{3}{4}$, then the above is at least $\frac{1}{8}  \sqrt{\sum_{k\neq j}  \frac{\sigma^2(\uk\t\a)^2}{(\lambda_j - \lambda_k)^2 }}.$  

We now bound the ${\sf KL}$ divergence.  We note that because $\bN$ is a ${\sf GOE}$ matrix, it holds that
\begin{align*}
    {\sf KL}(P, \tilde P) &= \frac{1}{4\sigma^2} \| \bm{U\Lambda U}\t - \bm{UR \Lambda R\t U}\t \|_F^2      = \frac{1}{2\sigma^2}\| \bm{R} \bm{\Lambda} - \bm{\Lambda} \bm{R} \|_F^2 \\
    &= \frac{1}{2\sigma^2}\bigg\|\bigg[ (\cos(\|\bm{\theta}\|) - 1) \bigg(e_j e_j\t + \frac{\bm{\theta} \bm{\theta}\t}{\|\bm{\theta}\|^2} \bigg) + \frac{\sin\bm{\theta}}{\|\bm{\theta}\|} ( \bm{\theta} e_j\t - e_j \bm{\theta}\t ) \bigg] \bm{\Lambda} \\
    &\qquad \qquad - \bm{\Lambda} \bigg[ (\cos(\|\bm{\theta}\|) - 1) \bigg(e_j e_j\t + \frac{\bm{\theta} \bm{\theta}\t}{\|\bm{\theta}\|^2}\bigg) + \frac{\sin\bm{\Theta}}{\|\bm{\Theta}\|} ( \bm{\theta} e_j\t - e_j \bm{\theta}\t ) \bigg)  \bigg] \bigg\|_F^2 \\
    &\leq \frac{1}{2\sigma^2}\frac{(\cos\|\bm{\theta}\| - 1)^2}{\|\bm{\theta}\|^4} \sum_{i,k\neq j,i \neq k} \theta_i^2 \theta_k^2 (\lambda_i - \lambda_k)^2 + 2 \frac{1}{2\sigma^2}\frac{\sin^2 \|\bm{\theta}\|}{\|\bm{\theta}\|^2} \sum_{i \neq j} \theta_i^2 (\lambda_i - \lambda_j)^2.
    \end{align*}
Noting that $\sin^2\|\bm{\theta}\| \leq \|\bm{\theta}\|^2$ and $( 1 - \cos \|\bm{\theta}\|)^2 \leq \| \bm{\theta}\|^4$ for small values of $\|\bm{\theta} \| \leq \frac{\sigma}{\Delta_j} = o(1)$, by plugging in the definition of $\theta_k$ we can bound
\begin{align*}
{\sf KL}(P,\tilde P) 
    &\leq \frac{1}{8}  + \frac{\sigma^2}{16 (\sum_{k\neq j} \frac{\sigma^2 (\a\t\uk)^2}{(\lambda_k - \lambda_j)^2})^2} \sum_{i,k\neq j,i\neq k} \frac{\sigma^4 (\a\t\uk)^2 (\a\t \bm{u}_i)^2}{(\lambda_i - \lambda_j)^4(\lambda_k - \lambda_j)^4} (\lambda_i - \lambda_k)^2 \\
    &\leq \frac{1}{8} + \frac{\sigma^2}{4( \sum_{k\neq j} \frac{\sigma^2 (\a\t\uk)^2}{(\lambda_k - \lambda_j)^2})^2} \sum_{i,k\neq j,i\neq k} \frac{\sigma^4 (\a\t\uk)^2 (\a\t \bm{u}_i)^2}{(\lambda_i - \lambda_j)^2(\lambda_k - \lambda_j)^4}  \\
    &\leq \frac{1}{8} + \frac{\sigma^2}{4\Delta_j^2}  = \frac{1}{8}(1 + o(1)).
\end{align*}

Next, assume that $\frac{\sigma^2}{\lambda_j^2} \|\uperp\t\a\|^2 \geq \sigma^2 \sum_{\substack{k\neq j\\k\leq r}} \frac{(\a\t\uk)^2}{(\lambda_j - \lambda_k)^2}.$  We consider all $\uk$ the same except for $\tilde{\uj}$, which we now define as
\begin{align*}
    \tilde{\uj} &= \frac{\uj + \delta \frac{\uperp \uperp\t \a}{\|\uperp\t\a\|} }{\sqrt{1 + \delta^2}},
\end{align*}
which is a unit vector, and is orthogonal to $\uk$ for $k \neq j$. Without loss of generality we assume $\a\t\uj \geq 0$.
We note that
\begin{align*}
    \a\t \tilde{\uj} &= \frac{\a\t\uj}{\sqrt{1 + \delta^2 }} + \frac{\delta}{\sqrt{1 + \delta^2 }} \| \uperp\t\a\|.
\end{align*}
Let $\delta = c\frac{\sigma}{\lambda_j}$, where $c > 0$ is some constant to be chosen later.   First, suppose that $c$ is fixed as $c = \frac{1}{1600}$.  We note that with this definition
\begin{align*}
    | \a\t (\uj - \tilde{\uj})| &= \bigg| \a\t\uj \bigg(1 - \frac{1}{\sqrt{1 + \delta^2}}\bigg) - \frac{\delta}{\sqrt{1 + \delta^2 }} \|\uperp\t\a \| \bigg|.
\end{align*}
If $ \frac{\delta}{\sqrt{1 + \delta^2 }} \|\uperp\t\a \| \geq 2 \a\t\uj (1 - \frac{1}{\sqrt{1 + \delta^2}})$, then $|\a\t(\uj - \tilde{\uj})| \geq\frac{1}{2} \frac{\delta}{\sqrt{1 + \delta^2 }} \|\uperp\t\a \| \geq \frac{1}{4} \delta \|\uperp\t\a\|$.  Instead, if $\a\t\uj (1 - \frac{1}{\sqrt{1 + \delta^2}}) \geq 2\frac{\delta}{\sqrt{1 + \delta^2 }} \|\uperp\t\a \| $, then $|\a\t(\uj - \tilde{\uj})| \geq  \frac{1}{2} \delta \|\uperp\t\a\|$.  If neither of these conditions hold, we can either change $c$ to be $\frac{1}{400}$ or $\frac{1}{6400},$ which rescales $\a\t\uj ( 1 - \frac{1}{\sqrt{1 + \delta^2}})$ by at least $8$ and rescales $\frac{\delta}{\sqrt{1 + \delta^2 }} \|\uperp\t\a \|$ by at most 4, which again falls into the first case.  Since $\frac{\delta}{\sqrt{1 + \delta^2 }} \|\uperp\t\a \|$ changes at most linearly in $c$ and $\a\t\uj ( 1 - \frac{1}{\sqrt{1 + \delta^2}})$ grows quadratically in $c$, we can adjust $c$ so that we fall into one of these two cases.  Regardless, we have $\frac{1}{6400} \leq  c\leq \frac{1}{400}$.

Similarly,
\begin{align*}
    |\a\t(\uj + \tilde{\uj}) | &= \bigg|\a\t\uj + \frac{\a\t\uj}{\sqrt{1 + \delta^2}} + \frac{\delta}{\sqrt{1 + \delta^2 }} \|\uperp\t\a \| \bigg| \\
    &\geq |\a\t\uj| + \frac{1}{2}\delta \|\uperp\t\a\| \\
    &\geq \frac{1}{2} \delta \|\uperp\t\a\|.
\end{align*}
In either case we have 
\begin{align*}
    |\a\t\uj \pm \a\t \tilde{\uj} | \gtrsim  \frac{\sigma}{\lambda_j} \|\uperp\t\a\|.
\end{align*}
We now compute the ${\sf KL}$ divergence. We have that \begin{align*}
    {\sf KL}(P,\tilde P) &\leq \frac{\lambda_j^2}{4\sigma^2} \| \uj \uj\t - \tilde{\uj} \tilde{\uj}\t \|_F^2 \leq \frac{\lambda_j^2}{2\sigma^2} \| \uj -\tilde{\uj} \|^2.
\end{align*}
Finally, one can straightforwardly check that $\| \uj - \tilde{\uj} \| \leq  \delta \leq \frac{1}{400}\frac{\sigma}{\lambda_j}$.  
  This completes the proof.
\end{proof}

\subsection{Proof of \cref{thm:minimaxlowerbound_PCA}}
\label{sec:minimaxlowerbound_pca}
\begin{proof}
    The proof is similar to the previous proof, only with different calculations for the constructions.  Consider observing $\bm{\Sigma}$ and $\tilde{\bSigma}$ of the form
    \begin{align*}
        \bSigma = \sum_{k=1}^{r} \lambda_k \uk\uk\t + \sigma^2 \bm{I}_p; \qquad \tilde{\bSigma} = \sum_{k=1}^{r} \lambda_k \tilde{\uk\uk}\t + \sigma^2 \bm{I}_p.
    \end{align*}
By Lemma 1 of \citet{cai_confidence_2017} it suffices to construct $\tilde{\bSigma}$ such that the quantity $|\a\t\uj \pm \a\t \tilde{\uj}|$ is lower bounded by $s_{\a,j}\pca$ and the ${\sf KL}$ divergence is upper bounded.

First, if $\sqrt{\sum_{\substack{k\neq j\\k\leq r}} \frac{(\a\t\uk)^2 (\lambda_j +\sigma^2(\lambda_k + \sigma^2)}{n(\lambda_j - \lambda_k)^2}} \leq \sqrt{\frac{(\lambda_j+\sigma^2)\sigma^2 \|\uperp\t\a\|^2}{n\lambda_j^2}}$, then by applying the same construction from \citet{li_minimax_2025} in the proof of Equation E.2 in the appendix, we can prove the result.  

Therefore, it suffices to prove the result when  $\sqrt{\sum_{\substack{k\neq j\\k\leq r}} \frac{(\a\t\uk)^2 (\lambda_j +\sigma^2(\lambda_k + \sigma^2)}{n(\lambda_j - \lambda_k)^2}} \geq \sqrt{\frac{(\lambda_j+\sigma^2)\sigma^2 \|\uperp\t\a\|^2}{n\lambda_j^2}}$
which implies that the left hand side is larger than $C s_{\a,j}\pca$.  As in the previous proof, define $\tilde{\U} = \bm{UR}$, where $\bm{R}$ is a rotation matrix defined as follows. 
For $k \neq j$ set \begin{align*}
    \theta_k =  \pm \frac{1}{4} \frac{(\lambda_j + \sigma^2)(\lambda_k + \sigma^2) \a\t\uk}{n (\lambda_j - \lambda_k)^2 \sqrt{\sum_{l\neq j}\frac{(\lambda_l + \sigma^2)(\lambda_j + \sigma^2)(\a\t\bm{u}_l)^2}{n(\lambda_j - \lambda_l)^2}}},
\end{align*}
and let $\bm{\theta}$ be the vector with $\theta_j = 0$, where the sign is chosen to be  $\mathrm{sgn}(\a\t\uj)$ if $|\a\t\uj| \leq \sqrt{\sum_{\substack{k\neq j\\k\leq r}} \frac{(\a\t\uk)^2 (\lambda_j +\sigma^2(\lambda_k + \sigma^2)}{n(\lambda_j - \lambda_k)^2}}, $ and $-\mathrm{sgn}(\a\t\uj)$ otherwise, where the convention that $\mathrm{sgn}(0)= 1$.  
Define
\begin{align*}
    \bm{R} = \bm{I}_r + \bigg( \cos( \| \bm{\theta}\|) - 1  \bigg) \bigg( e_j e_j\t + \frac{\bm{\theta}\bm{\theta}\t}{\|\bm{\theta}\|^2} \bigg) + \frac{\sin\|\bm{\theta}\|}{\|\bm{\theta}\|} (\bm{\theta} e_j\t - e_j \bm{\theta}\t ),
\end{align*}
where $e_j$ is a standard basis vector. 
We can similarly show that 
\begin{align*}
    \a\t\tilde{\uj} &= \a\t\uj \cos\|\bm{\theta} \| \pm \frac{\sin\|\bm{\theta}\|}{4\|\bm{\theta}\|}  \sqrt{\sum_{k\neq j}  \frac{(\lambda_j + \sigma^2)(\lambda_k +\sigma^2)(\uk\t\a)^2}{n(\lambda_j - \lambda_k)^2 }},
\end{align*}
and through the same analysis as the previous proof we can demonstrate 
that $|\a\t\uj \pm \a\t\tilde{\uj}| \gtrsim  \sqrt{\sum_{k\neq j}  \frac{(\lambda_j + \sigma^2)(\lambda_k +\sigma^2)(\uk\t\a)^2}{n(\lambda_j - \lambda_k)^2 }}.$  

Therefore, it suffices to calculate the ${\sf KL}$ divergence and show it is bounded.  It is well-known that the ${\sf KL}$ divergence between mean-zero multivariate Gaussians is given by
\begin{align*}
    {\sf KL}(\tilde P, P) = \frac{n}{2} \bigg( {\sf Tr} \big( \bm{\Sigma}\inv \tilde{\bSigma} \big) - p \bigg),
\end{align*}
where we have implicitly used the fact that the determinants of both matrices are the same since their eigenvalues are the same.  We further have that since $\bSigma\inv = \sum_{\substack{k\neq j\\k\leq r}} \frac{1}{\lambda_k + \sigma^2} \uk\uk\t + \frac{1}{\sigma^2} \uperp\uperp\t$, then
\begin{align*}
    {\sf Tr}\big( \bSigma\inv \tilde{\bSigma} \big) - p &= {\sf Tr} \big( \tilde{\U} ( \bm{\Lambda} + \sigma^2 \bm{I}_r )\inv \tilde{\U}\t \U (\bm{\Lambda} + \sigma^2 \bm{I}_r  ) \U\t \big) + p - r - p \\
    &= {\sf Tr} \bigg( \bm{R} ( \bm{\Lambda} + \sigma^2 \bm{I}_r )\inv \bm{R}\t ( \bm{\Lambda} + \sigma^2 \bm{I}_r ) \bigg) - r \\
     &= \bigg( \frac{\sin\|\bm{\theta}\|}{4\|\bm{\theta}\|} \bigg)^2 \sum_{k\neq j} \frac{(\lambda_k - \lambda_j)^2}{(\lambda_j + \sigma^2)(\lambda_k + \sigma^2)} \theta_k^2 \leq \frac{1}{16n},
\end{align*}
where the final bound comes from the definition of $\theta_k$.  
\end{proof}

\section{Numerical Simulations}
\begin{figure}[t]
\centering
  \begin{subfigure}[b]{0.7\textwidth}
            \subfloat[$\lambda_{\min} = 2 \sqrt{n}$, \\ \ \ \ \ $\Delta_1 =  \log(n)$]
                {%
            \includegraphics[width=.33\linewidth]{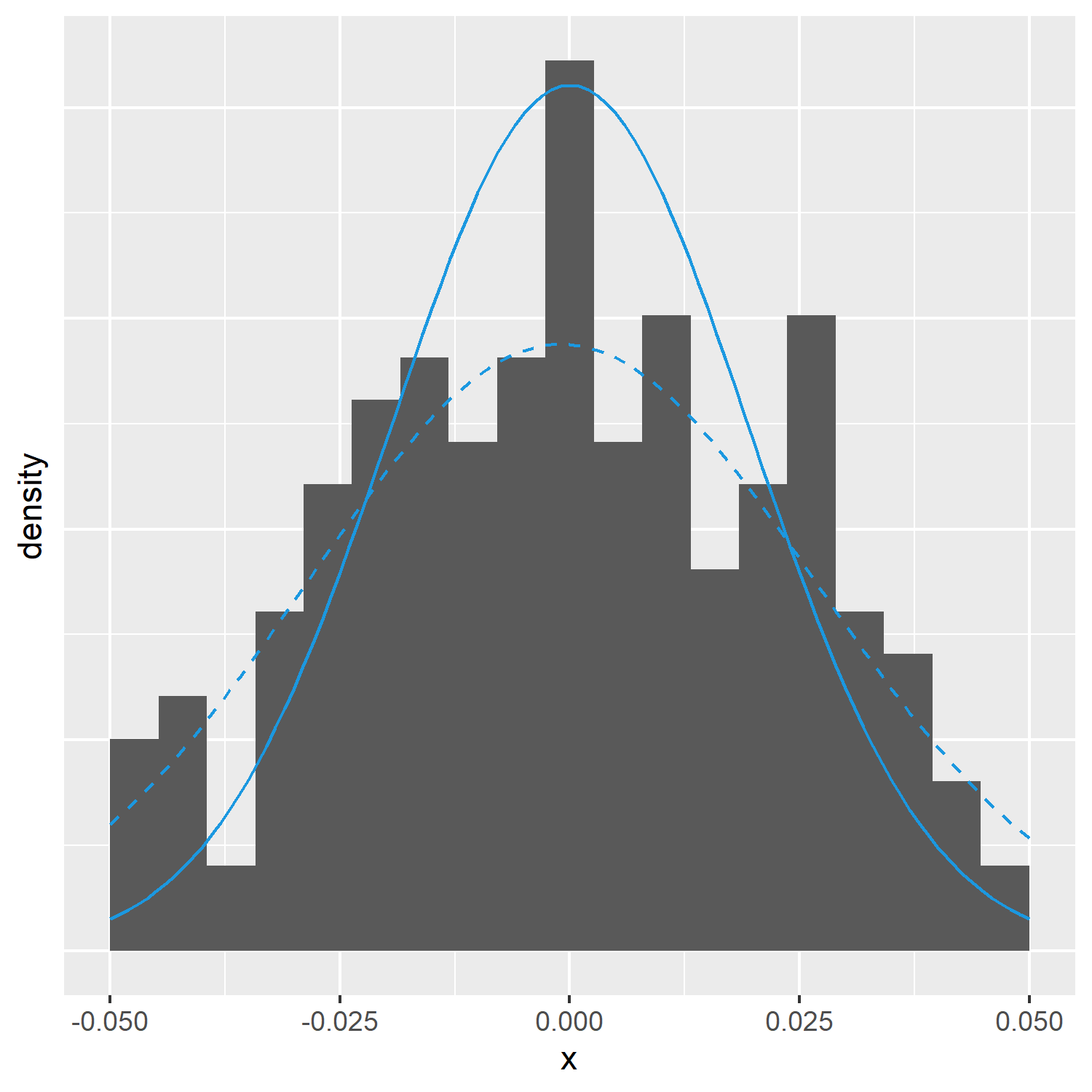}%
            \label{subfig:aveerror}%
        }\hfill
         \subfloat[$\lambda_{\min} = 4\sqrt{n}$, \\ \ \ \ \ $\Delta_1 =  \log(n)$]
                {%
            \includegraphics[width=.33\linewidth]{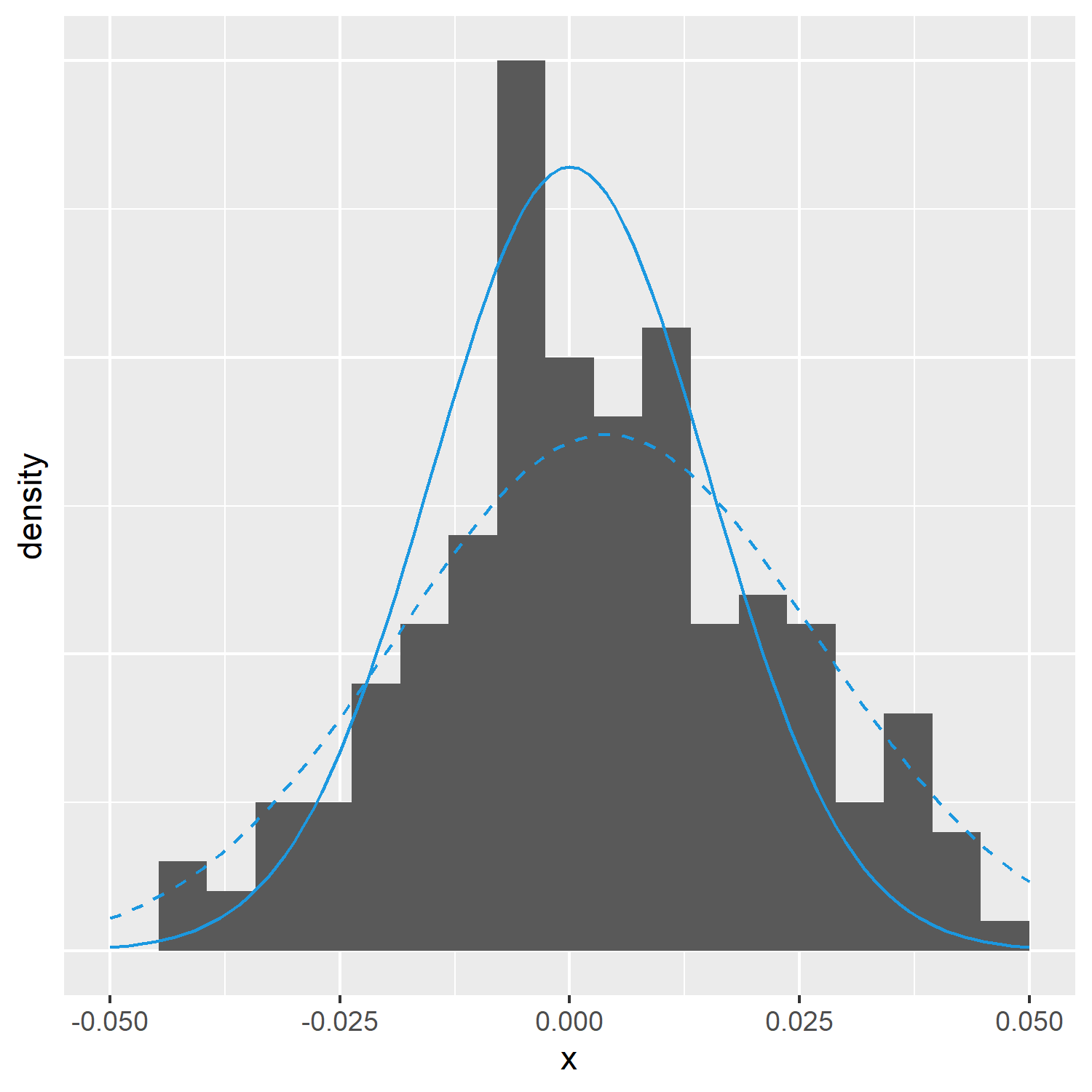}%
            \label{subfig:aveerror}%
        }   \hfill
         \subfloat[$\lambda_{\min} = 6 \sqrt{n}$, \\ \ \ \ \ $\Delta_1 = \log(n)$]
                {%
            \includegraphics[width=.33\linewidth]{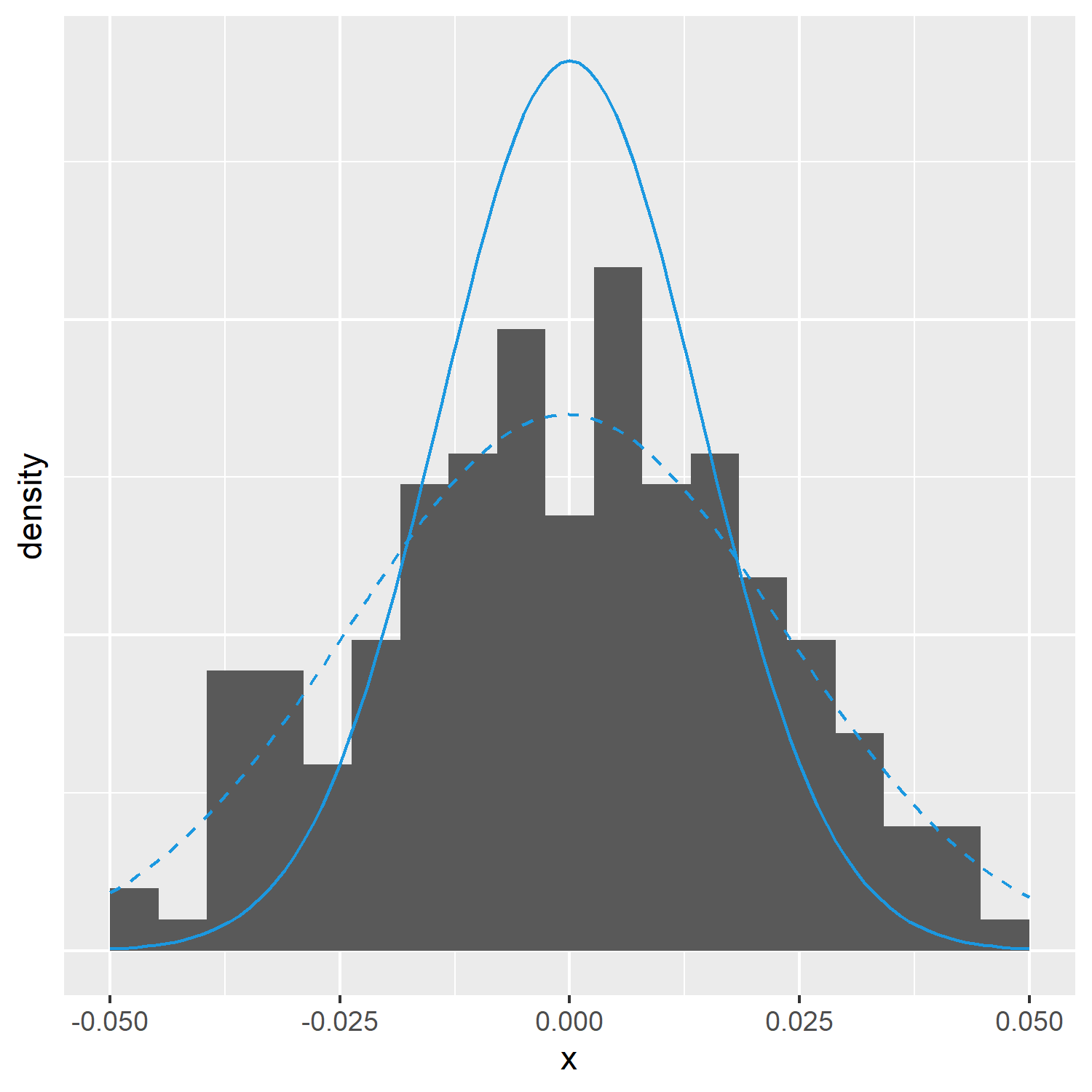}%
            \label{subfig:aveerror}%
        } \\
       \subfloat[$\lambda_{\min} = 2 \sqrt{n}$, \\ \ \ \ \ $\Delta_1 = 2\log(n)$]
                {%
            \includegraphics[width=.33\linewidth]{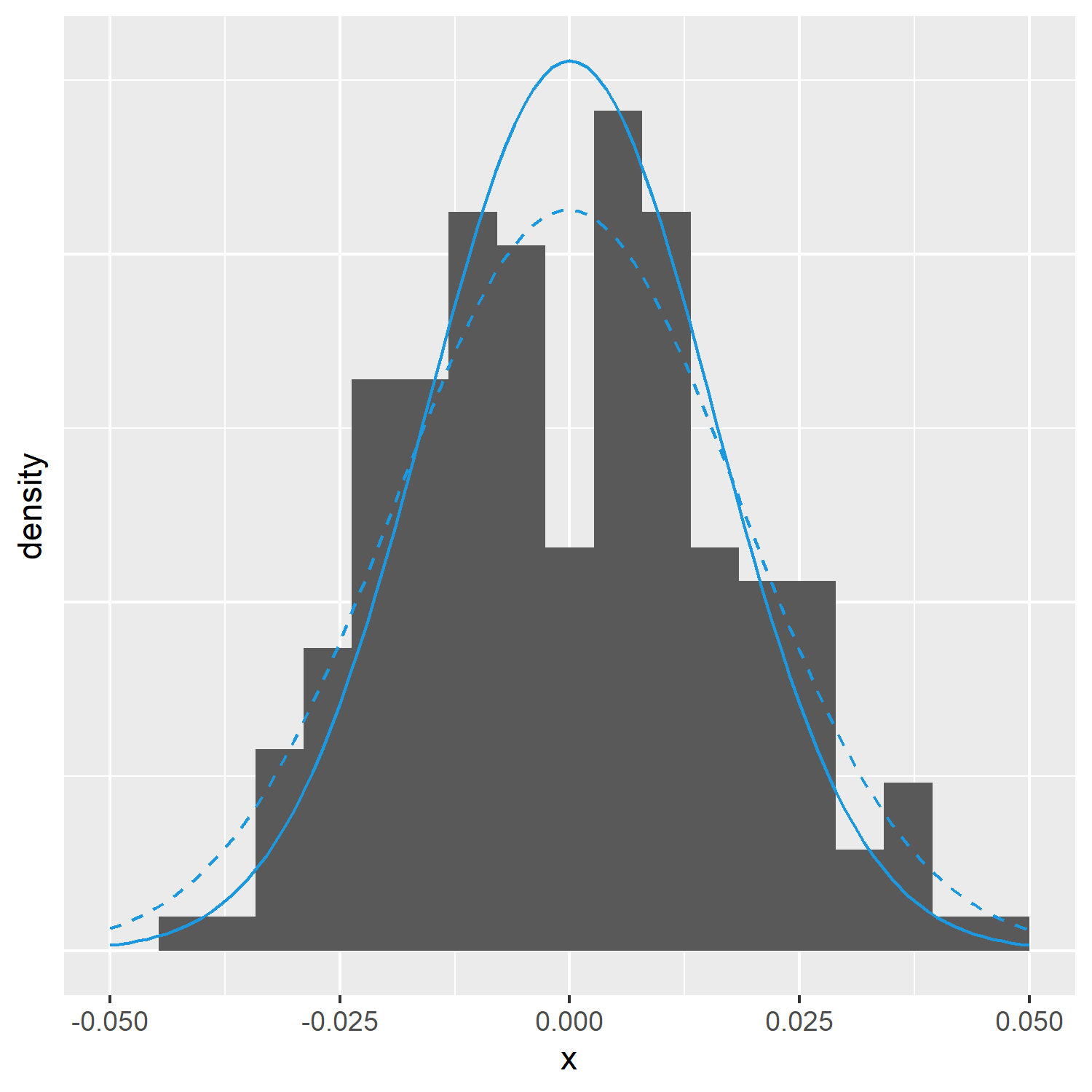}%
            \label{subfig:aveerror}%
        }\hfill
         \subfloat[$\lambda_{\min} = 4 \sqrt{n}$, \\ \ \ \ \ $\Delta_1 = 2\log(n)$]
                {%
            \includegraphics[width=.33\linewidth]{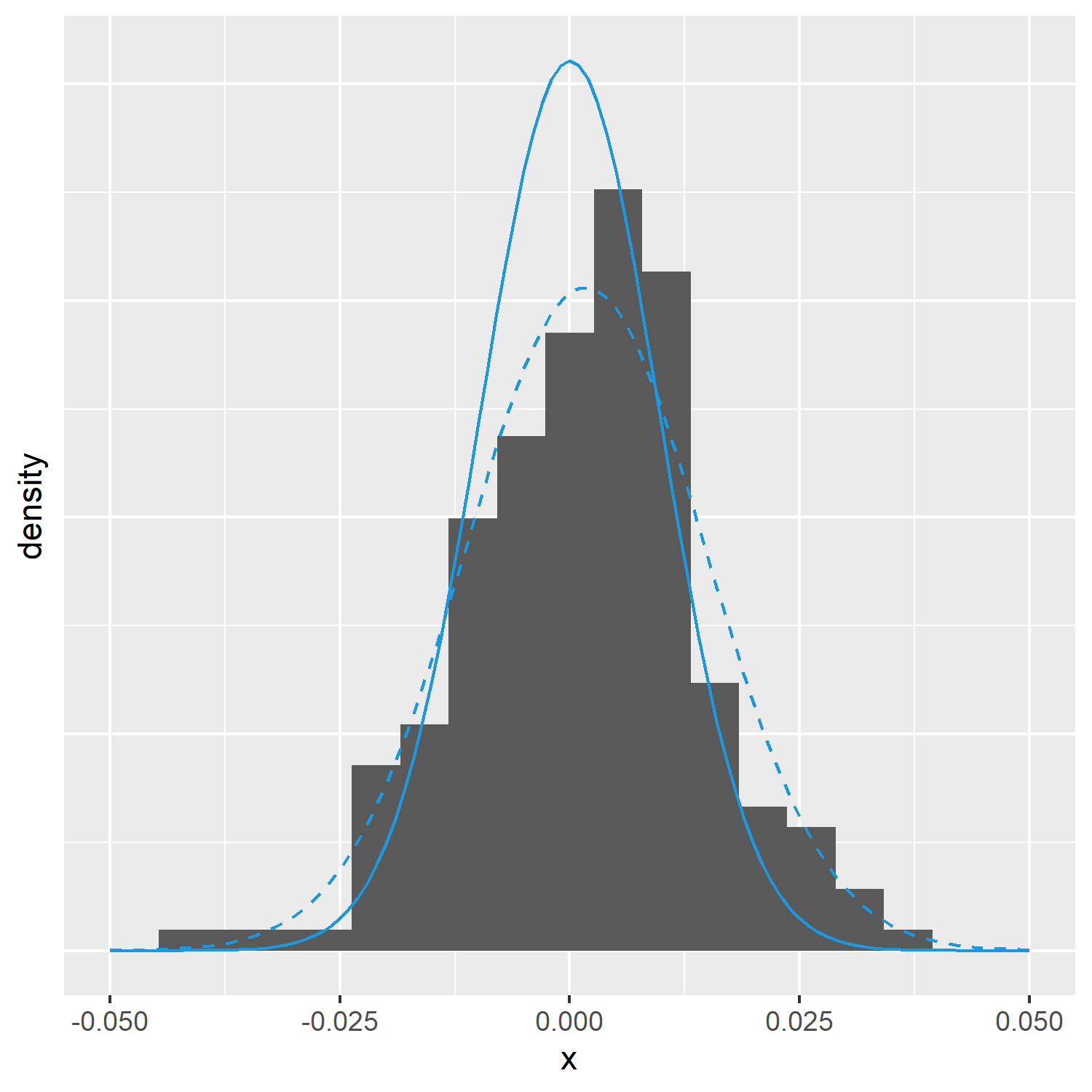}%
            \label{subfig:aveerror}%
        }   \hfill
         \subfloat[$\lambda_{\min}= 6\sqrt{n}$, \\ \ \ \ \ $\Delta_1 = 2\log(n)$]
                {%
            \includegraphics[width=.33\linewidth]{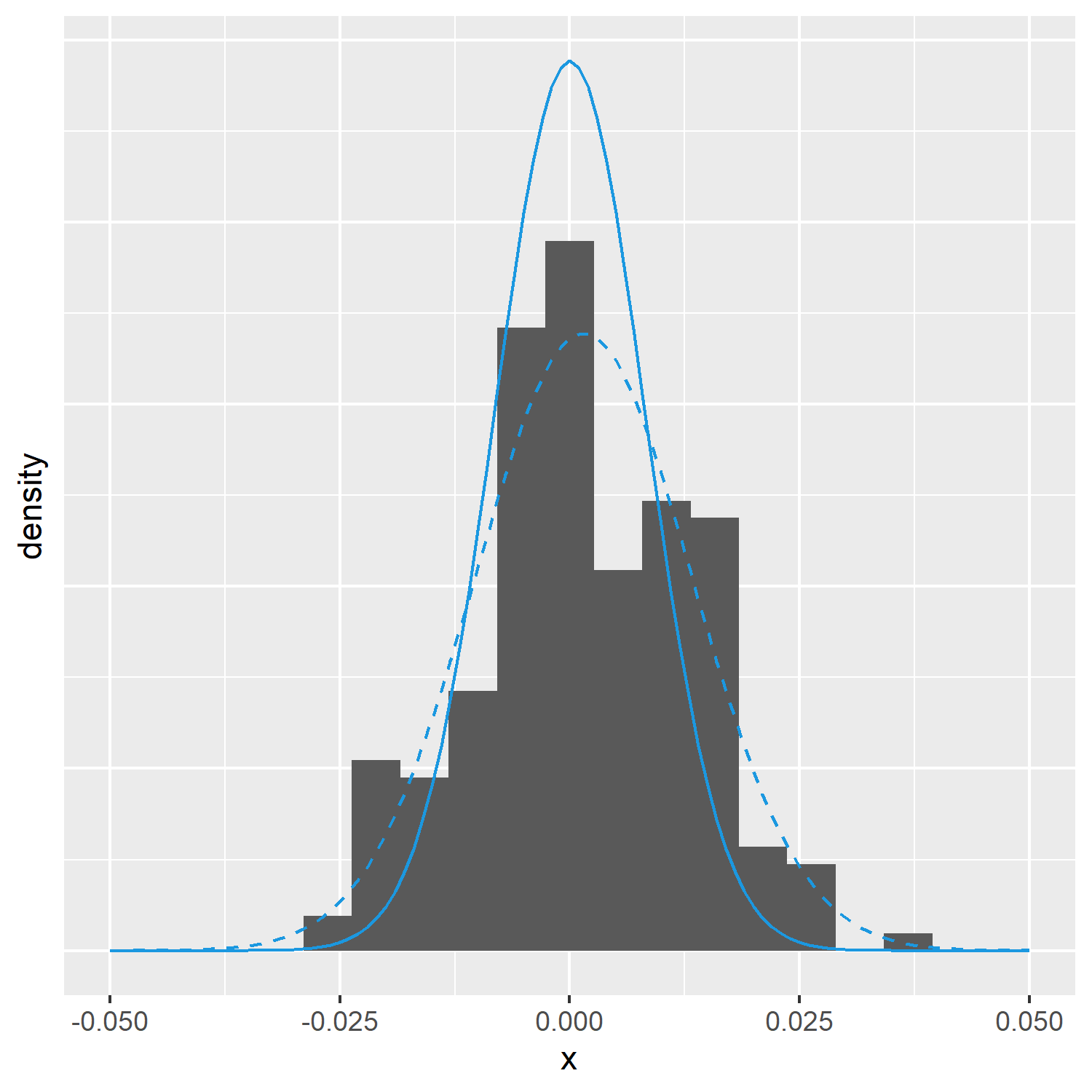}%
            \label{subfig:aveerror}%
        }\\
       \subfloat[$\lambda_{\min}= 2 \sqrt{n}$, \\ \ \ \ \ $\Delta_1 = 3\log(n)$]
                {%
            \includegraphics[width=.33\linewidth]{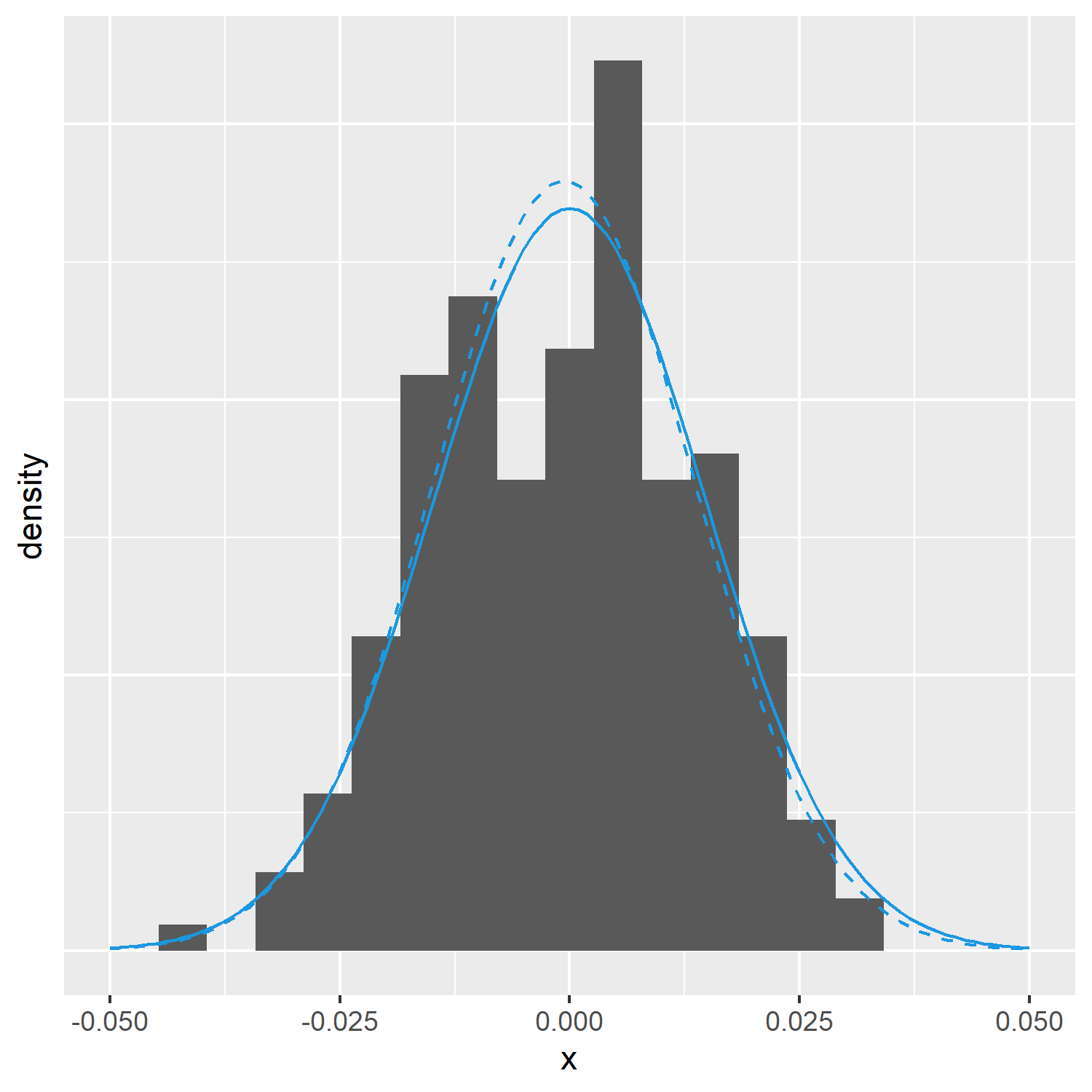}%
            \label{subfig:aveerror}%
        }\hfill
         \subfloat[$\lambda_{\min} = 4\sqrt{n}$, \\ \ \ \ \ $\Delta_1 = 3\log(n)$]
                {%
            \includegraphics[width=.33\linewidth]{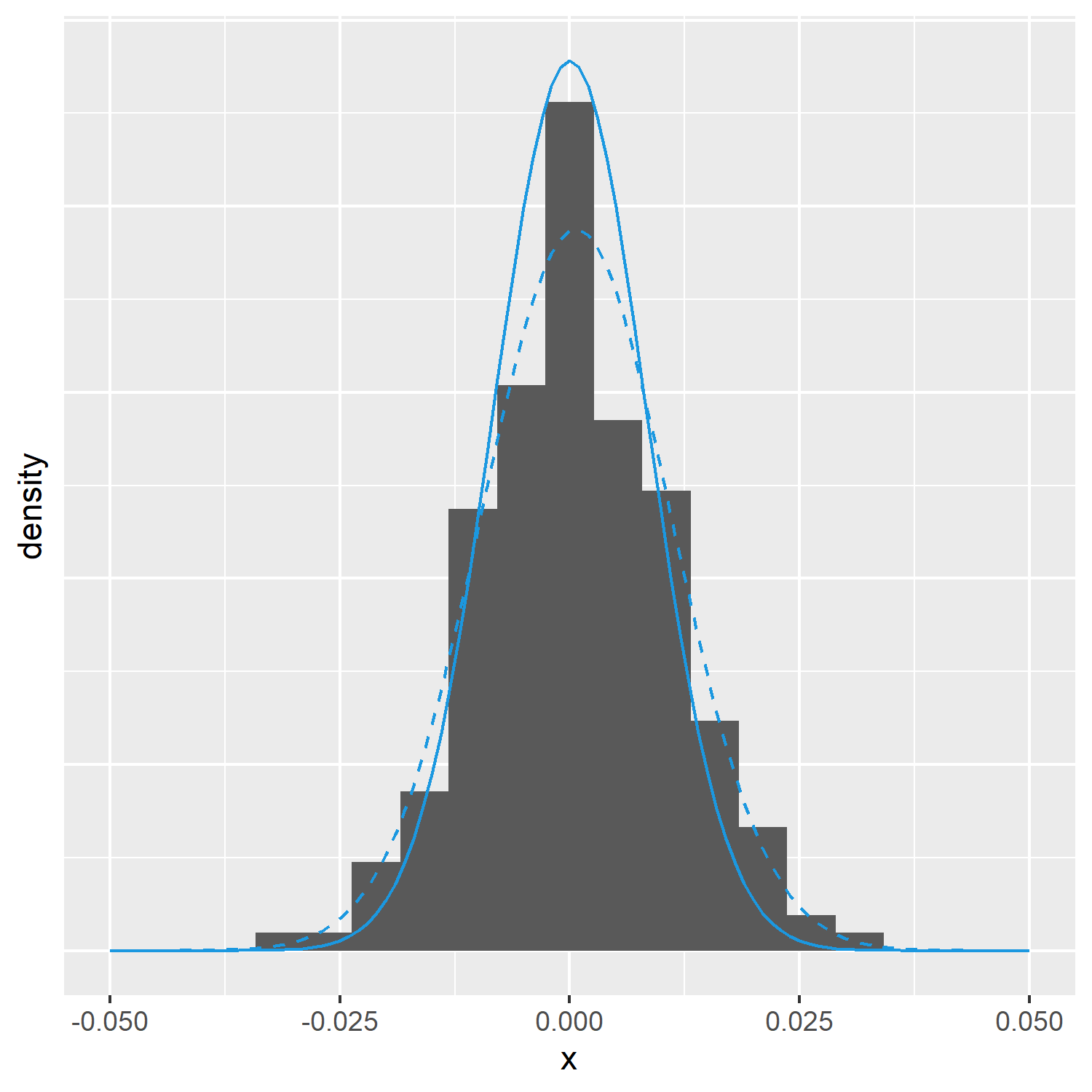}%
            \label{subfig:aveerror}%
        }   \hfill
         \subfloat[$\lambda_{\min} = 6 \sqrt{n}$, \\ \ \ \ \ $\Delta_1 = 3\log(n)$]
                {%
            \includegraphics[width=.33\linewidth]{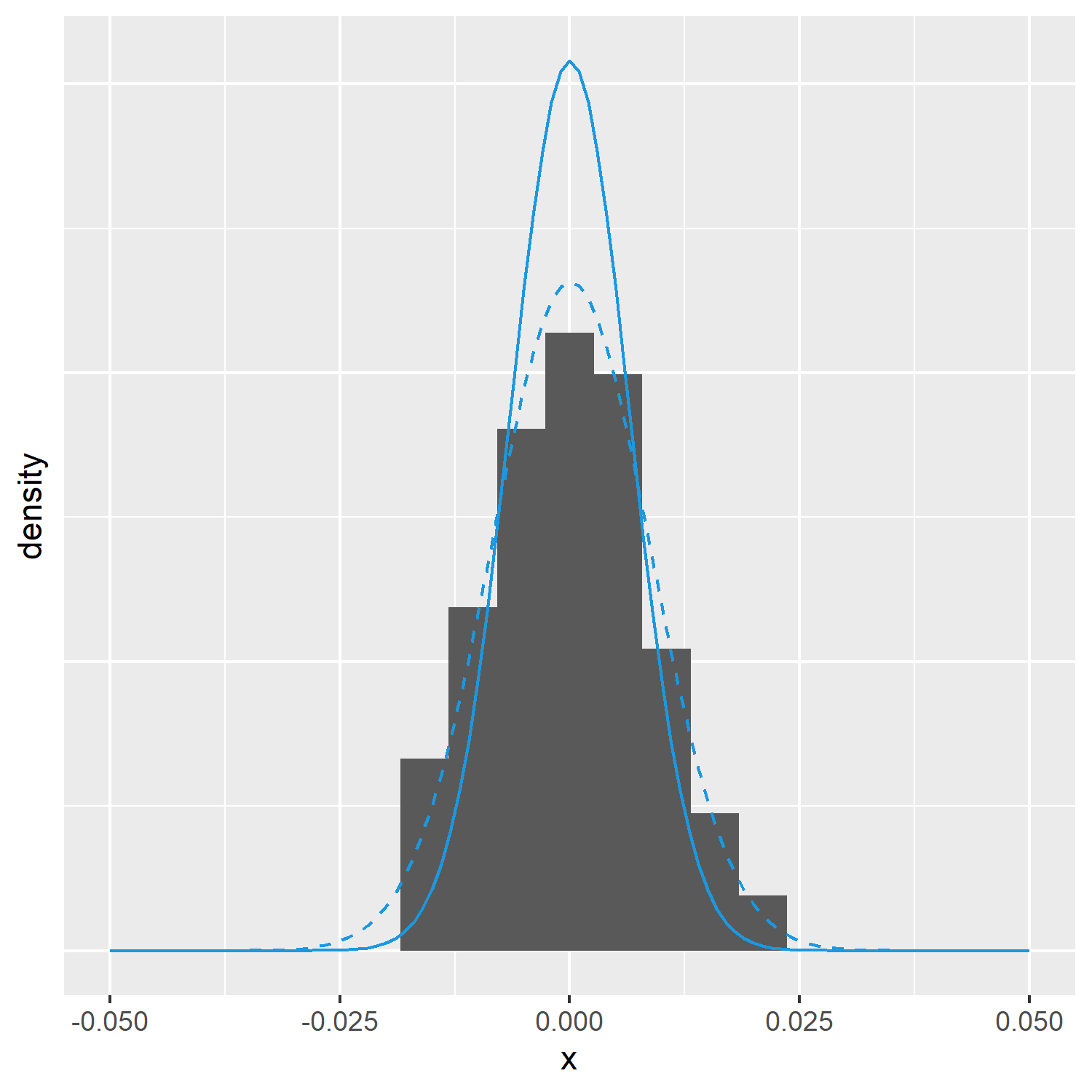}%
            \label{subfig:aveerror}%
        }
\end{subfigure}
        \caption{Empirical (dotted) and theoretical (solid) ellipses for the quantity $\a\t \bm{\hat u}_1 - \a\t \bm{u}_1 \bm{u}_1\t \bm{\hat u}_1$ with varying signal-strength and eigengaps under the matrix denoising model \eqref{MD}. From top to bottom the eigengap increases, and from left to right the signal strength increases. } \label{fig:mddistributionaltheory}
    \end{figure}
\subsection{Numerical Simulations for Matrix Denoising}

In this section we demonstrate our theory through simulations.  Each result is generated through 200 Monte Carlo iterations.   
\\ \ \\
\noindent
\textbf{Setup:} In all simulations we set $n = 200$ and $\sigma^2 = 1$, and $r = 3$.  To generate the matrix $\bm{S}$, we first define a vector $\bm{v}$ with entries consisting of $n$ independent $U(-1,1)$ random variables, and we define $\bm{u}_1= \frac{\bm{v}}{\|\bm{v}\|}$.  We then draw $\bm{u}_2$ by drawing a new vector $\bm{v}'$ with $U(-1,1)$ entries and setting $\bm{u}_2 = \big( \bm{I} - \bm{u}_1 \bm{u}_1\t\big) \frac{\bm{v}'}{\|\bm{v}'\|}$ to ensure orthogonality, and then normalizing.  Finally we set $\bm{u}_3$ by drawing another vector $\bm{v}''$ with $U(-1,1)$ entries and setting $\bm{u}_3 =\big( \bm{I} - \bm{u}_1 \bm{u}_1\t - \bm{u}_2 \bm{u}_2\t\big) \frac{\bm{v}}{\|\bm{v}\|}$.  Finally, for the varying levels of signal-to-noise ratio $\lambda_{\min}/\sigma \equiv \lambda_{\min}$ and eigengap $\Delta_1$, we define $\lambda_1 = 5 \lambda_{\min}$, $\lambda_2 = \lambda_1 - \Delta_1$, $\lambda_3 = \lambda_{\min}$.  We first draw the matrix $\bm{S}$ before running simulations, and we keep it fixed throughout all simulations.  The matrix $\bm{N}$ is re-drawn at each Monte Carlo iteration.
\\ \ \\
\textbf{Distributional Theory:} Plotted in \cref{fig:mddistributionaltheory} is the empirical histogram of the random variable $\a\t \bm{\hat u}_1 - \a\t \bm{u}_1 \bm{u}_1\t \bm{\hat u}_1$, with solid line corresponding to the theoretical distribution (mean zero and variance $s\md_{\a,1}$) and dotted line corresponding to the empirical distribution, with $\bm{a}$ equal to the constant vector with entries $\frac{1}{\sqrt{n}}$. By \cref{thm:distributionaltheory_MD} the approximate Gaussianity is governed by the eigengap $\Delta_1/\sigma$ and the signal strength $\lambda_{\min}/\sigma$.  From left to right we consider increasing signal strength and from top to bottom we consider increasing eigengap, so that top left has the smallest eigengap and signal strength, and bottom right has the largest eigengap and signal strength.  Furthermore, by \cref{thm:distributionaltheory_MD}, the asymptotic variance decreases with larger eigengap and signal strength, and observe that the theoretical and empirical histograms become narrower as these quantities increase.  We also ran simulations for the difference $\a\t \bm{\hat u}_1 \sqrt{1 + b_j\md} - \a\t\bm{u}_1$, but the figures are not materially different. \\ \ \\
    \noindent
\textbf{Approximate Coverage Rates:} We also consider approximate coverage rates for $\a\t \bm{u}_1$ using \cref{alg:ciMD}. 
Similar to the previous setting we examine increasing eigengaps and signal-strengths with $\alpha = .05$.  The column denoted ``Mean'' corresponds to the number of times the true value of $\a\t \bm{u}_1$ 
was in the outputted confidence interval, and the column denoted ``Std'' denotes the standard deviation of this value.  Due to the sign ambiguity, we multiplied $\a\t \bm{u}_1$ 
by ${\sf sign}(\bm{u}_1\t \bm{\hat u}_1)$, as we assume this quantity is positive throughout our theory. The coverage remains close to $.95$, though it appears to be slightly conservative. 
\begin{table}[htb]
    \centering
    \begin{tabular}{|c|c | c |c|}
    \hline
\multicolumn{4}{|c|}{Coverage Rates for $\a\t \bm{u}_1$}   \\ 
\hline 
  $\Delta_1$ & $ \lambda_{\min} $ & Mean  & Std  \\
  \hline 
  $\log(n)$ & $2 \sqrt{n}$ & 0.955   &  0.015 
\\ 
\hline 
  
 $\log(n)$ & $4 \sqrt{n}$ &   0.980  
   &    0.010  
\\
\hline 
  
 $\log(n)$ & $6 \sqrt{n}$ & 0.990 
 & 0.007
\\
\hline 
  
 $2\log(n)$ & $2 \sqrt{n}$ & 0.935  
  &  0.017  
 \\
 \hline 
  
 $2\log(n)$ & $4 \sqrt{n}$ & 0.965 
& 0.013   
\\\hline 
  
 $2\log(n)$  & $6 \sqrt{n}$ &   0.975 & 0.011
\\ \hline 
 $3\log(n)$ & $2 \sqrt{n}$ & 0.900 
  &  0.021  
 \\
 \hline 
  
 $3\log(n)$ & $4 \sqrt{n}$ &  0.965 
& 0.013  
\\\hline 
  
 $3\log(n)$  & $6\sqrt{n}$ &   0.990 & 0.007
\\ \hline 
  
\end{tabular} 
  
  
  
  
  
  
  
  
    \caption{Matrix Denoising empirical coverage rates for
    $\a\t \bm{u}_1$
    \cref{alg:ciMD} 
    for varying signal $\lambda_{\min}$ and eigengap $\Delta_1$.  Here $\a \equiv \frac{1}{\sqrt{n}}$.  The column ``Mean'' represents the empirical probability of coverage averaged over $200$ Monte Carlo iterations, and the column ``Std'' denotes the standard deviation of this coverage rate.}
    \label{table:coverageMD}
\end{table}

\subsection{Numerical Simulations for PCA}
\begin{figure}[H]
\centering
  \begin{subfigure}[b]{0.7\textwidth}
        \subfloat[{$\lambda_{\min}/\sigma^2 = 2\log(n)$, \\ \ \ \ \ $\Delta_1 = (\lambda_{\min} + \sigma^2) \frac{\log(n)}{\sqrt{n}}$}]
                {%
            \includegraphics[width=.33\linewidth]{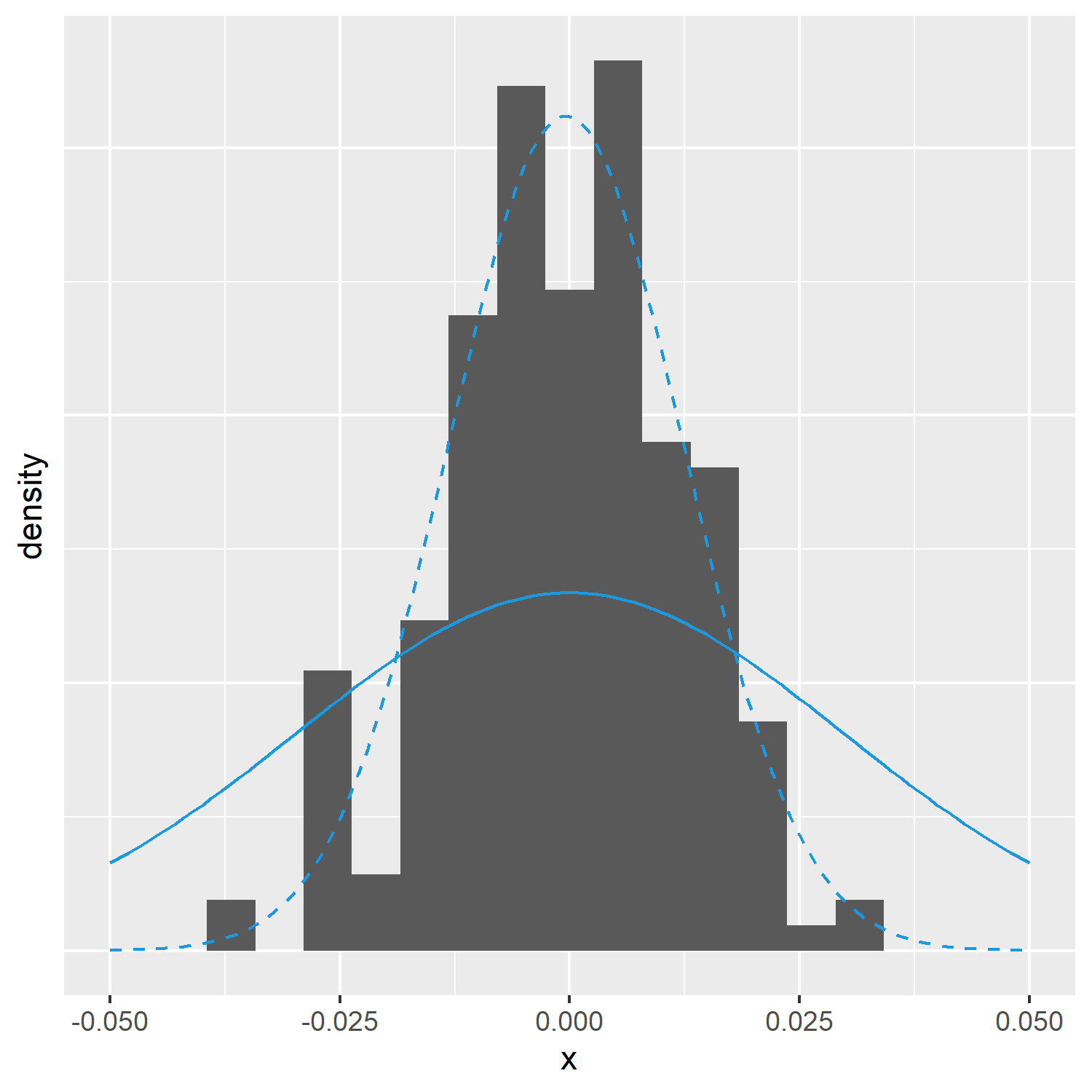}%
            \label{subfig:aveerror}%
        }\hfill
         \subfloat[{ $\lambda_{\min}/\sigma^2 = 4\log(n)$,  \\ \ \ \ \ $\Delta_1 = (\lambda_{\min} + \sigma^2) \frac{\log(n)}{\sqrt{n}}$}]
                {%
            \includegraphics[width=.33\linewidth]{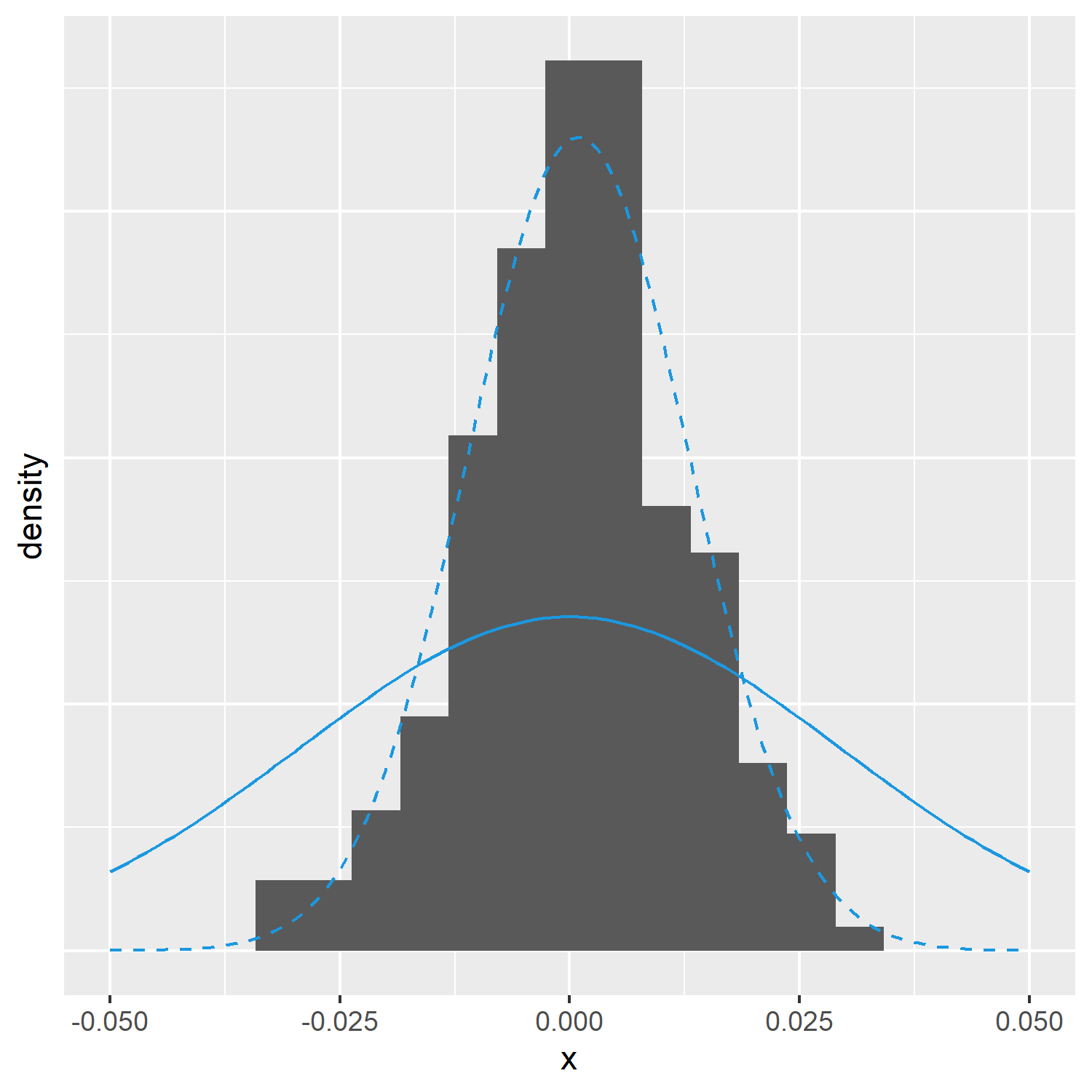}%
            \label{subfig:aveerror}%
        }   \hfill
         \subfloat[{ $\lambda_{\min}/\sigma^2 = 6\log(n)$,  \\ \ \ \ \ $\Delta_1 = (\lambda_{\min} + \sigma^2) \frac{\log(n)}{\sqrt{n}}$}]
                {%
            \includegraphics[width=.33\linewidth]{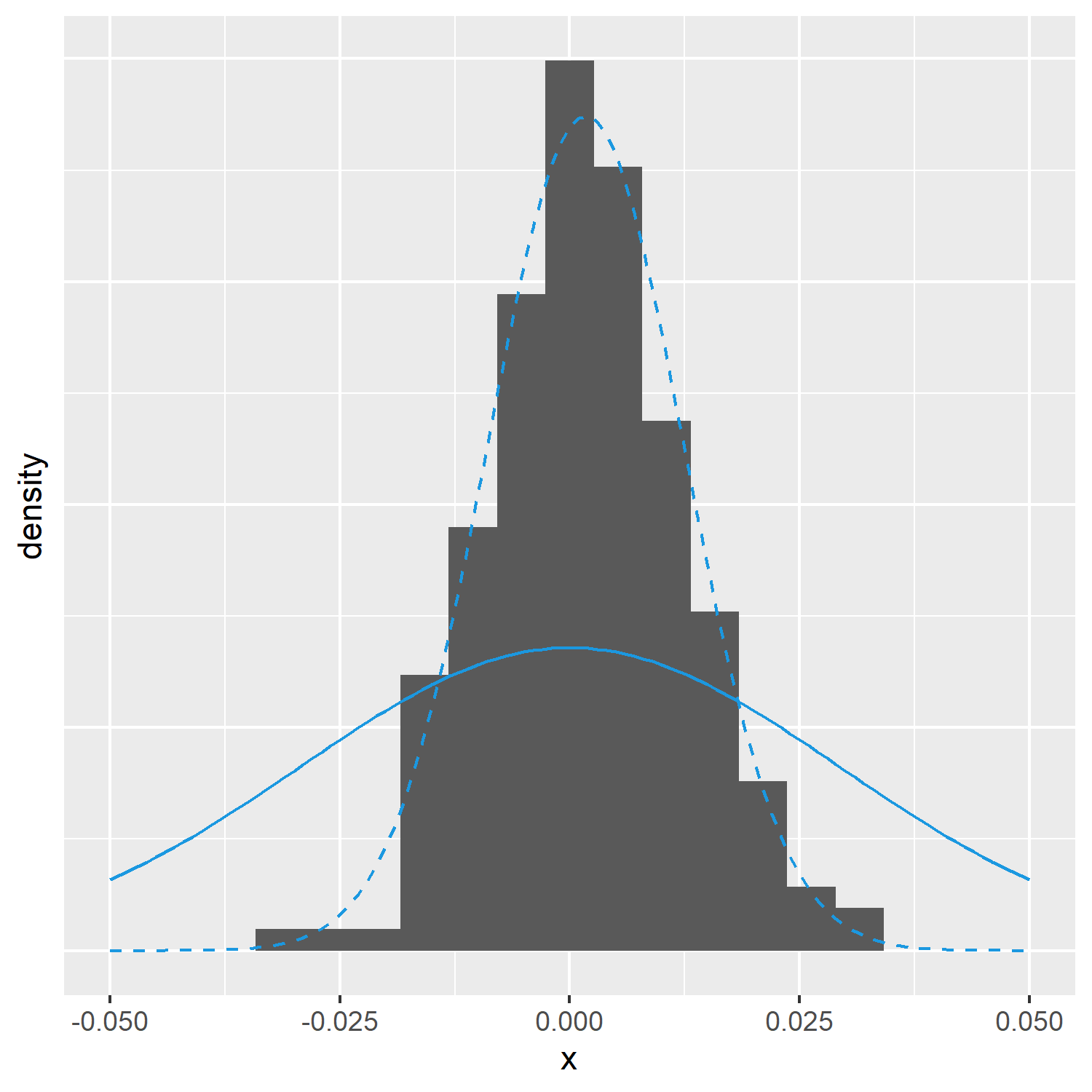}%
            \label{subfig:aveerror}%
        } \\
       \subfloat[{ $\lambda_{\min}/\sigma^2 = 2\log(n)$,  \\ \ \ \ \ $\Delta_1 = 2(\lambda_{\min} + \sigma^2) \frac{\log(n)}{\sqrt{n}}$}]
                {%
            \includegraphics[width=.33\linewidth]{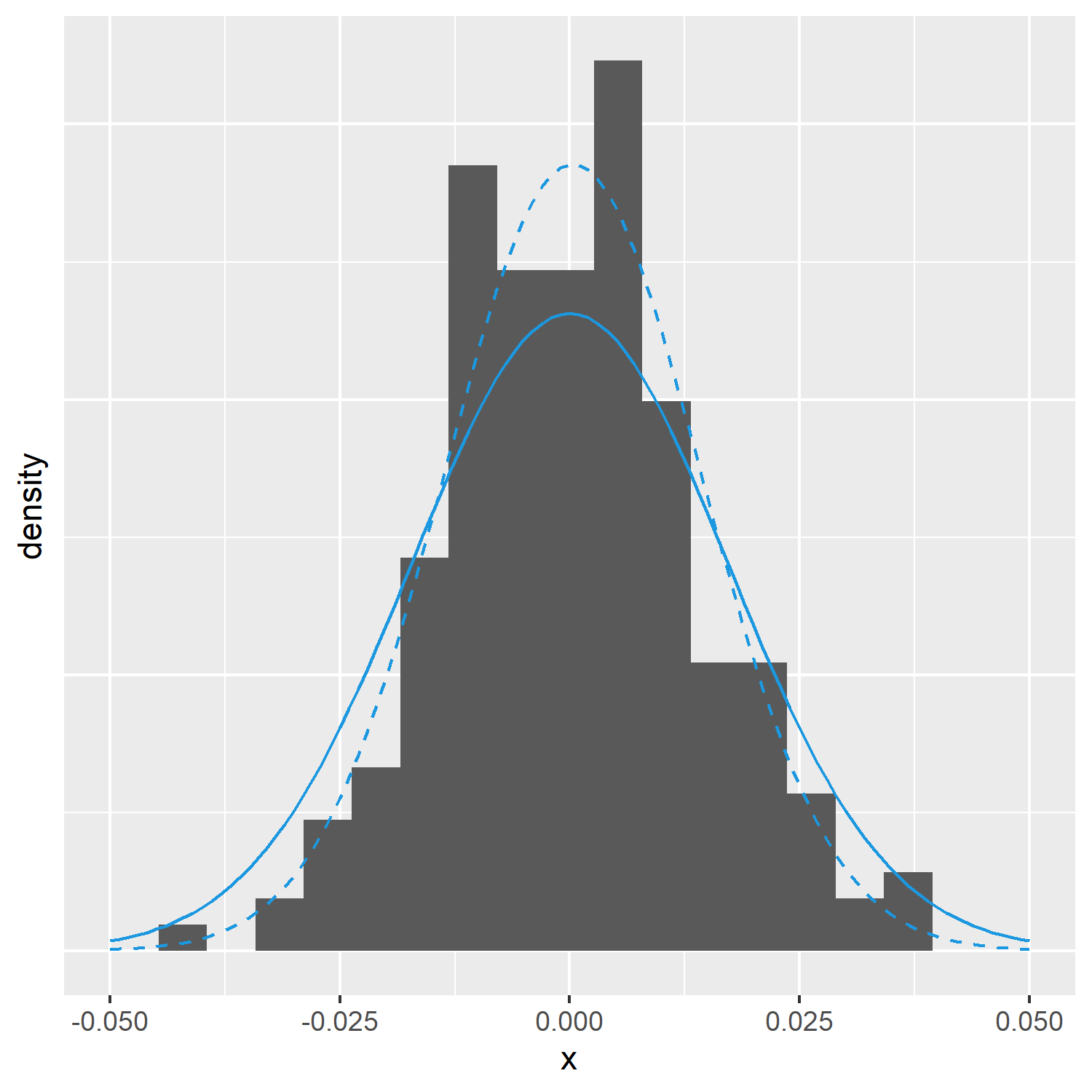}%
            \label{subfig:aveerror}%
        }\hfill
         \subfloat[{ $\lambda_{\min}/\sigma^2 = 4\log(n)$,  \\ \ \ \ \ $\Delta_1 = 2(\lambda_{\min} + \sigma^2) \frac{\log(n)}{\sqrt{n}}$}]
                {%
            \includegraphics[width=.33\linewidth]{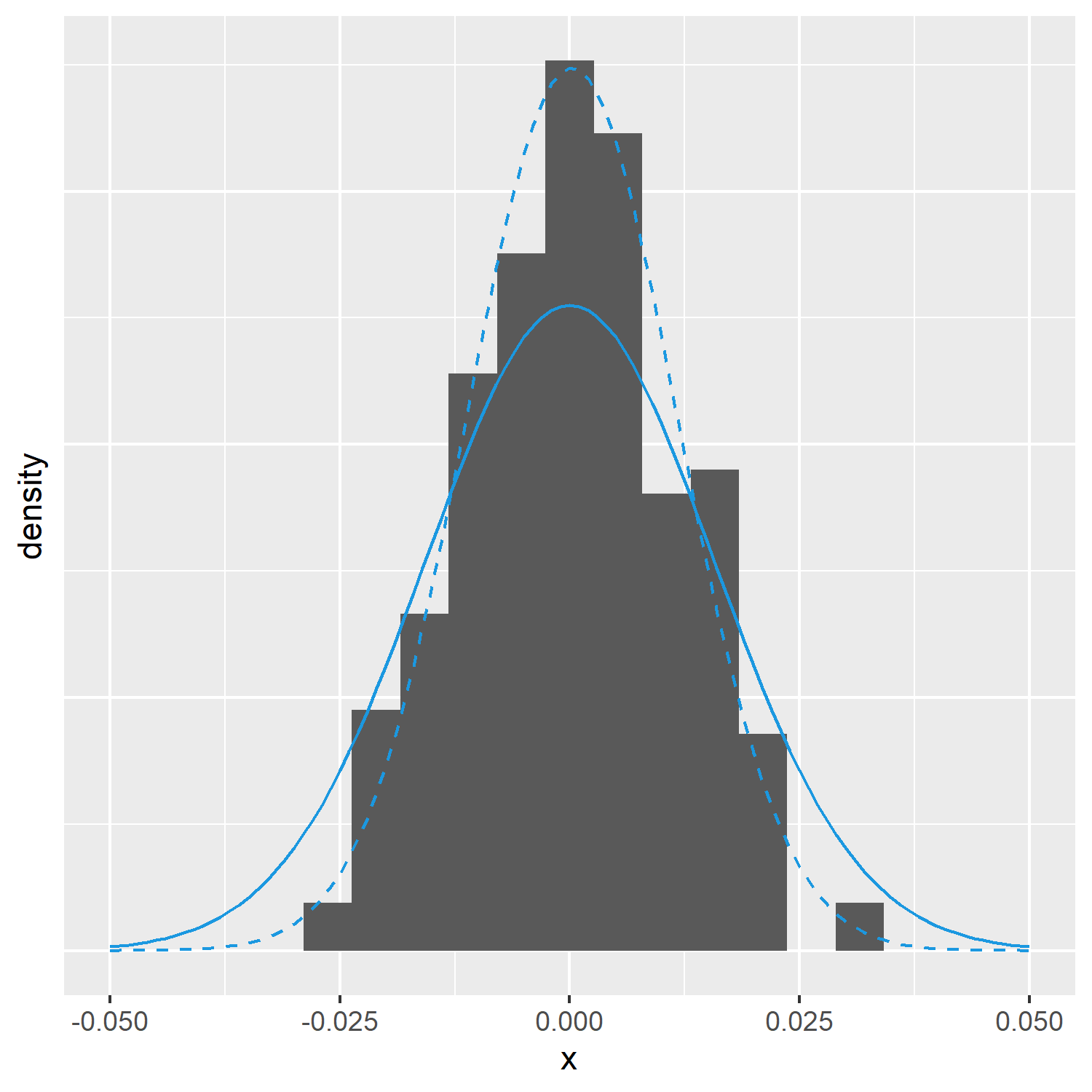}%
            \label{subfig:aveerror}%
        }   \hfill
         \subfloat[{ $\lambda_{\min}/\sigma^2 = 6\log(n)$,  \\ \ \ \ \ $\Delta_1 = 2(\lambda_{\min} + \sigma^2) \frac{\log(n)}{\sqrt{n}}$}]
                {%
            \includegraphics[width=.33\linewidth]{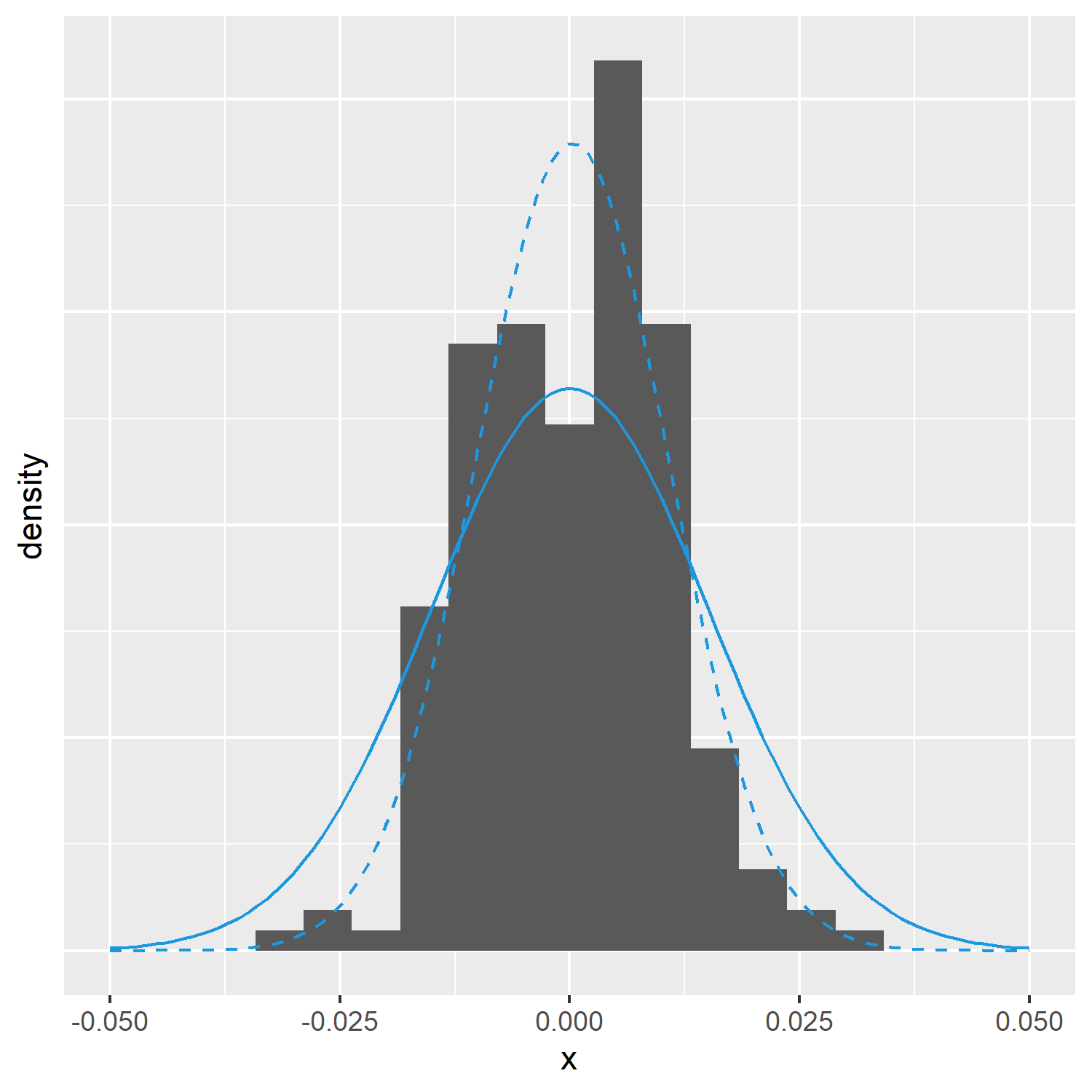}%
            \label{subfig:aveerror}%
        }\\
       \subfloat[{ $\lambda_{\min}/\sigma^2 = 2\log(n)$, \\ \ \ \ \ $\Delta_1 = 3(\lambda_{\min} + \sigma^2) \frac{\log(n)}{\sqrt{n}}$}]
                {%
            \includegraphics[width=.33\linewidth]{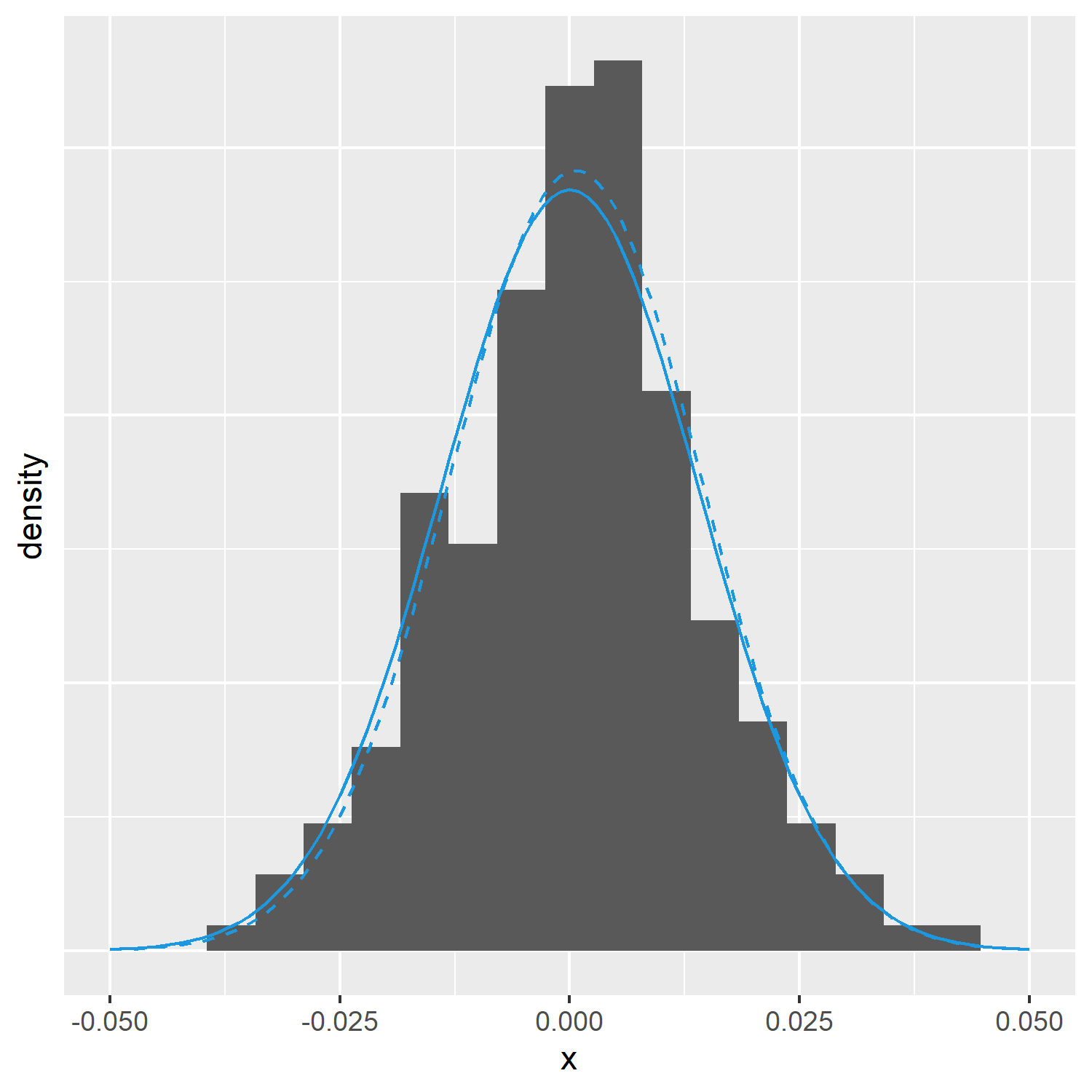}%
            \label{subfig:aveerror}%
        }\hfill
         \subfloat[{ $\lambda_{\min}/\sigma^2 = 4\log(n)$,  \\ \ \ \ \ $\Delta_1 = 3(\lambda_{\min} + \sigma^2) \frac{\log(n)}{\sqrt{n}}$}]
                {%
            \includegraphics[width=.33\linewidth]{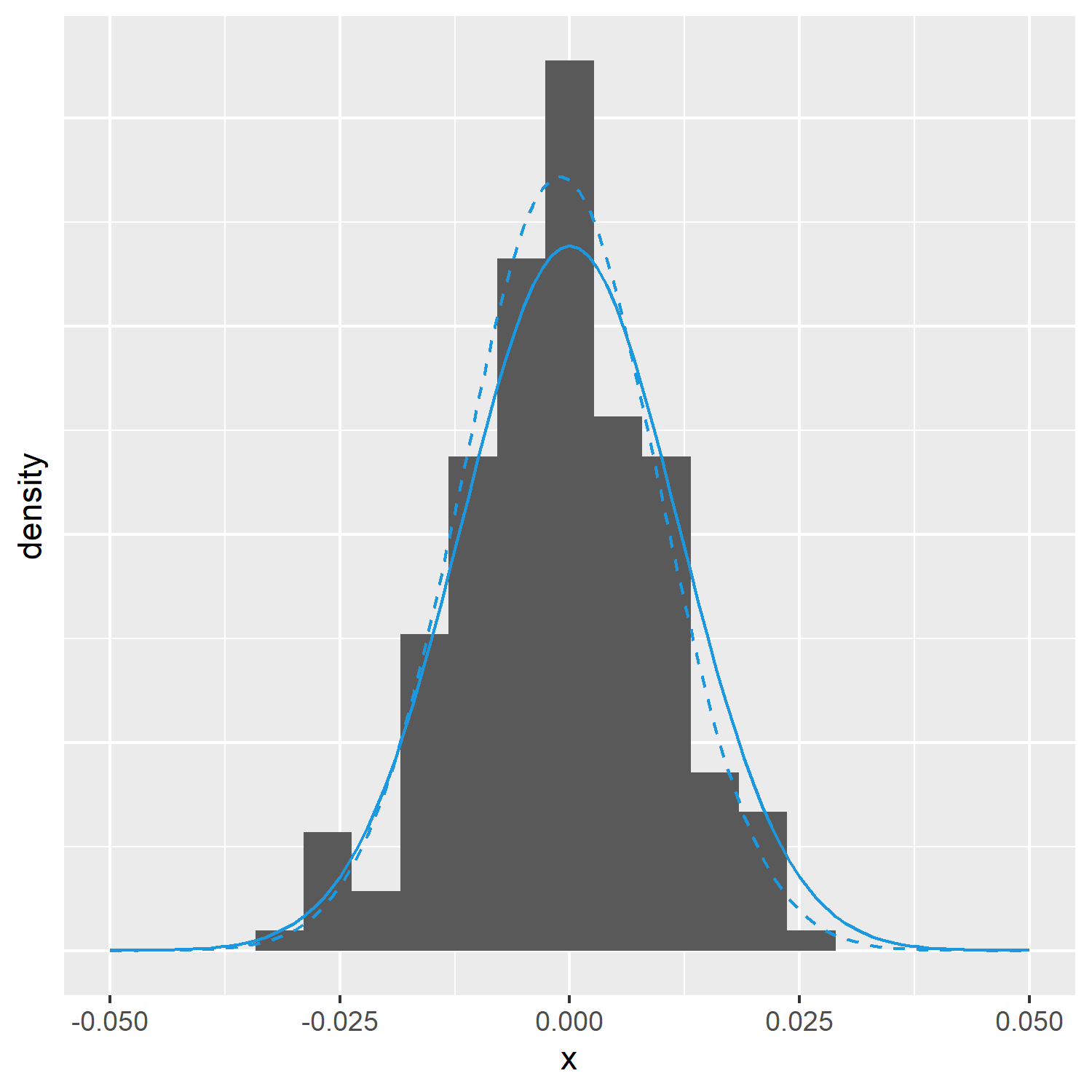}%
            \label{subfig:aveerror}%
        }   \hfill
         \subfloat[{ $\lambda_{\min}/\sigma^2 = 6\log(n)$, \\ \ \ \ \  $\Delta_1 = 3(\lambda_{\min} + \sigma^2) \frac{\log(n)}{\sqrt{n}}$}]
                {%
            \includegraphics[width=.33\linewidth]{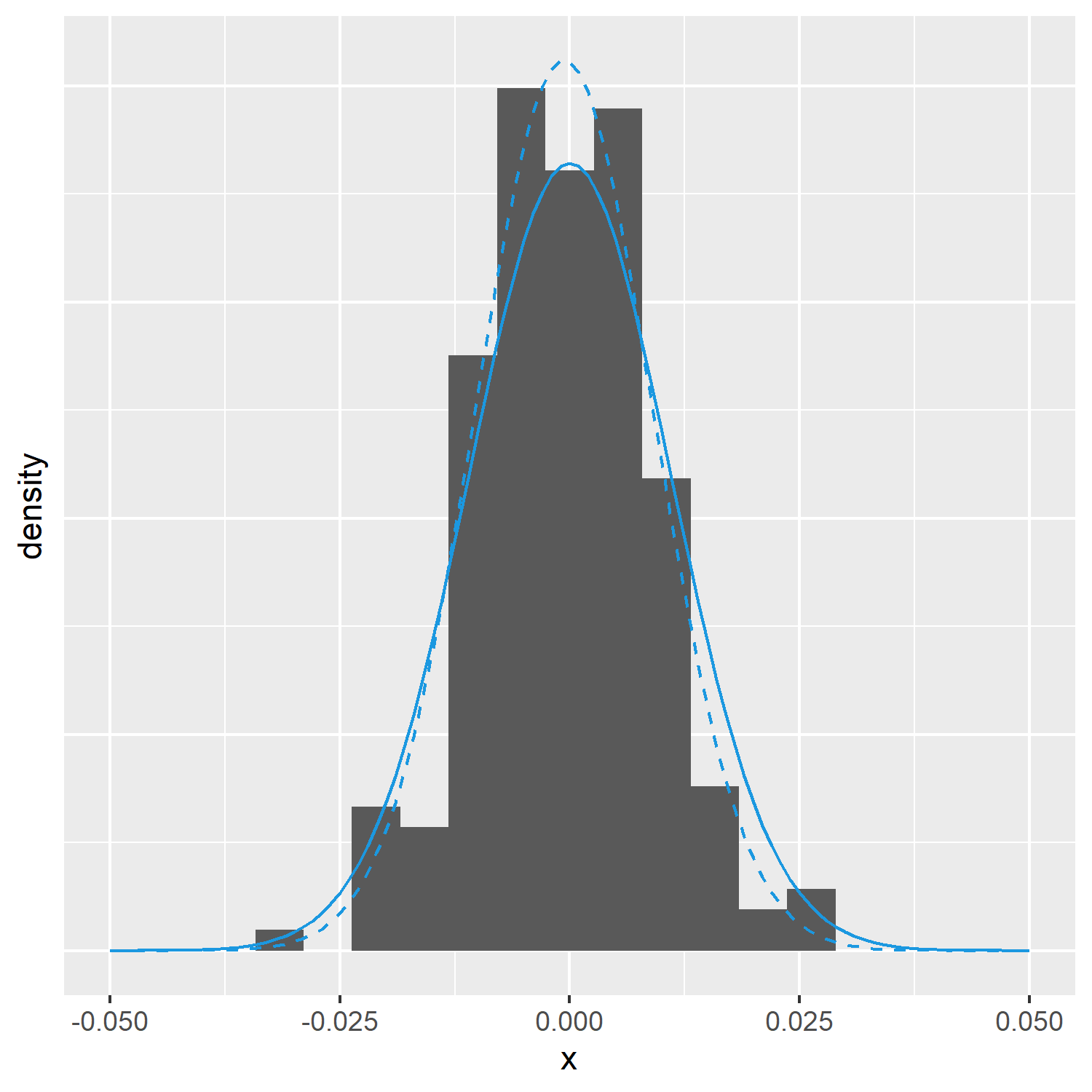}%
            \label{subfig:aveerror}%
        }
  \end{subfigure}
        \caption{Empirical (dotted) and theoretical (solid) ellipses for the quantity $\a\t \bm{\hat u}_1 - \a\t \bm{u}_1 \bm{u}_1\t \bm{\hat u}_1$ with varying signal-strength and eigengaps under the matrix denoising model \eqref{pcadef} From top to bottom the eigengap increases, and from left to right the signal strength increases.} \label{fig:pca}
    \end{figure}
In this section we demonstrate our theory through simulations.  Each result is generated through 200 Monte Carlo iterations.  \\ \ \\
\noindent \textbf{Setup: } In all simulations we set $n = 300$, $p = 200$ and $r= 3$, and we draw $\bm{u}_1,\bm{u}_2,\bm{u}_3$ the same way as in the matrix denoising simulations.  We keep $\sigma^2 = 1$, and for a given eigengap $\Delta_1$ and signal strength $\lambda_{\min}/\sigma^2 \equiv \lambda_{\min}$, we set $\lambda_1 = 5 \lambda_{\min}$, $\lambda_2 = \lambda_1 - \Delta_1$, and $\lambda_3 = \lambda_{\min}$.  \\ \ \\ \noindent
\noindent \textbf{Distributional Theory:} Plotted in \cref{fig:pca} is the empirical histogram of the random variable $\a\t \bm{\hat u}_1 - \a\t \bm{u}_1 \bm{u}_1\t \bm{\hat u}_1$, with solid line corresponding to the theoretical distribution (mean zero and variance $s_{\a,1}\pca$), and dotted line corresponding to the empirical distribution with $\bm{a}$ equal to the constant vector with entries $\frac{1}{\sqrt{p}}$.  From left to right we consider increasing signal strength and from top to bottom we consider increasing eigengap.  Similar to matrix denoising, the Gaussian approximation improves and the theoretical variance decreases as the eigengap and signal strength increase, which we see demonstrated in the figure.
\\ \ \\
    \noindent
\textbf{Approximate Coverage Rates:} We also consider approximate coverage rates for $\a\t \bm{u}_1$ using \cref{alg:ciPCA}.  Similar to the matrix denoising we examine increasing eigengaps and signal-strengths with $\alpha = .05$.  The column denoted ``Mean'' corresponds to the number of times the true value of $\a\t \bm{u}_1$ was in the outputted confidence interval, and the column denoted ``Std'' denotes the standard deviation of this value.  Due to the sign ambiguity, we multiplied $\a\t \bm{u}_1$ by ${\sf sign}(\bm{u}_1\t \bm{\hat u}_1)$, as we assume this quantity is positive throughout our theory.  Observe that empirically the coverage improves as the signal strength and eigengap increases.  
 \begin{table}[H]
    \centering
    \begin{tabular}{|c|c | c | c |}
    \hline
    \multicolumn{4}{|c|}{PCA} \\
    \hline 
\multicolumn{4}{|c|}{Coverage Rates for $\a\t \bm{u}_1$}   \\ [1ex]
\hline 
  { \small $\frac{\Delta_1}{\lambda_{\min} + \sigma^2}$ } & $ \lambda_{\min}/\sigma^2 $ & Mean  & Std  \\  [1ex]
  \hline 
  $\frac{\log(n)}{\sqrt{n}}$ & $2\log(n)$ & 0.935   &  0.017 
\\ 
\hline 
  
 $\frac{\log(n)}{\sqrt{n}}$ & $4\log(n)$ &   0.92 
   &     0.019 
\\
\hline 
  
 $\frac{\log(n)}{\sqrt{n}}$ & $6\log(n)$ & 0.93 
 &0.018
\\
\hline 
  
 $2\frac{\log(n)}{\sqrt{n}}$ & $2\log(n)$ &  0.860 
  &  0.025 
 \\
 \hline 
  
 $2 \frac{\log(n)}{\sqrt{n}}$ & $4\log(n)$ &  0.90 
&    0.021 
\\\hline 
  
 $2 \frac{\log(n)}{\sqrt{n}}$  & $6\log(n)$ & 0.92   & 0.019
\\ \hline 
 $3\frac{\log(n)}{\sqrt{n}}$ & $2\log(n)$ &0.885 
  &   0.023 
 \\
 \hline 
  
 $3\frac{\log(n)}{\sqrt{n}}$ & $4\log(n)$ &   0.92 
&   0.019 
\\\hline 
  
 $3 \frac{\log(n)}{\sqrt{n}}$  & $6\log(n)$ &  0.92 & 0.019
\\ \hline 
  
\end{tabular} 
  
  
  
  
  
  
  
  
    \caption{PCA empirical coverage rates for   $\a\t \bm{u}_1$  using \cref{alg:ciMD} for varying signal $\lambda_{\min}/\sigma^2$ and eigengap $\Delta_1$.  Here $\a \equiv \frac{1}{\sqrt{p}}$.  The column ``Mean'' represents the empirical probability of coverage averaged over $200$ Monte Carlo iterations, and the column ``Std'' denotes the standard deviation of this coverage rate.}
    \label{fig:coverageratespca}
\end{table}

\bibliography{linear_forms.bib}

\end{document}